\documentclass{amsart}%
\usepackage{amsfonts}
\usepackage{graphicx}
\usepackage{amscd}
\usepackage{amsmath}
\usepackage{amssymb}%
\providecommand{\U}[1]{\protect\rule{.1in}{.1in}}
\newtheorem{theorem}{Theorem}
\theoremstyle{plain}

\newtheorem{corollary}{Corollary}

\newtheorem{lemma}{Lemma}

\newtheorem{proposition}{Proposition}
\newtheorem{remark}{Remark}

\numberwithin{equation}{section}

\theoremstyle{theorem}
\newtheorem{theoremA}{Theorem}

\begin{document}
\title[New Fej\'{e}r-type Inequalities]{New Fej\'{e}r-type Inequalities Related to Hermite-Hadamard-type Inequalities
and Their Applications}
\author{Kuei-Lin Tseng}
\address{Department of Computer Science and Information Engineering\\
Aletheia University\\
Tamsui, Taiwan 25103.}
\email{kltseng1@gmail.com}
\maketitle

\begin{abstract}
In this project, I shall establish a large number of new Fej\'{e}r-type
inequalities which reduce all results of the following Subsections $1.1-1.3.$
\ A number of applications for the extended incomplete Beta function and the
extended special means are given.

\end{abstract}

\section{Introduction}

Throughout this section, let $a\leq x<y\leq y^{\prime}<x^{\prime}\leq b,$
$x+x^{\prime}=y+y^{\prime}$, $\ \Omega=\left[  x,y\right]  \cup\left[
y^{\prime},x^{\prime}\right]  $, $f:[a,b]\rightarrow\mathbb{R}$ be convex and
$g:[a,b]\rightarrow\left[  0,\infty\right)  $ be integrable and symmetric to
$\frac{a+b}{2}.$

Hadamard established the following inequality (see \cite{8b}):

\begin{theoremA}
\label{A1}\cite{8b} Let $f$\ be defined as above. Then:%
\begin{equation}
f\left(  \frac{a+b}{2}\right)  \leq\frac{1}{b-a}\int\nolimits_{a}^{b}f\left(
s\right)  ds\leq\frac{f\left(  a\right)  +f\left(  b\right)  }{2} \label{1.1}%
\end{equation}
is known as Hermite-Hadamard inequality.
\end{theoremA}

Fej\'{e}r established the following weighted generalization of
Hermite-Hadamard inequality $\left(  \ref{1.1}\right)  $ (see \cite{7b}).

\begin{theoremA}
\label{A2}\cite{7b} Let $f,g$\ be defined as above. Then%
\begin{equation}
f\left(  \frac{a+b}{2}\right)  \int\nolimits_{a}^{b}g\left(  s\right)
ds\leq\int\nolimits_{a}^{b}f\left(  s\right)  g\left(  s\right)  ds\leq
\frac{f\left(  a\right)  +f\left(  b\right)  }{2}\int\nolimits_{a}^{b}g\left(
s\right)  ds \label{1.2}%
\end{equation}
is known as Fej\'{e}r inequality.
\end{theoremA}

See \cite{1b}--\cite{6b} and \cite{9b}--\cite{24b}, the results of which are
the generalization, improvement and extension of the famous integral
inequalities $\left(  \ref{1.1}\right)  $ and $\left(  \ref{1.2}\right)  .$

In Subsections $1.1-1.2,$ we present a number of Hermite-Hadamard-type
inequalities which refine and generalize Hermite-Hadamard inequality $\left(
\ref{1.1}\right)  .$

In Subsection $1.3,$ we present a number of Fej\'{e}r-type inequalities which
refine and generalize Fej\'{e}r inequality $\left(  \ref{1.2}\right)  .$

In Section $2$, we shall establish a number of new Fej\'{e}r-type inequalities
which reduce all results of Subsections $1.1-1.3.$

In Section $3$, we shall establish a number of applications for the extended
incomplete Beta function and the extended special means.

\subsection{Hermite-Hadamard-type Inequalities}

In this subsection, we present nine theorems related Hermite-Hadamard-type
inequalities and which are proved by Dragomir in \cite{1b}, \cite{4b},
Dragomir \textit{et al.} in \cite{6b}, Yang and Hong in \cite{20b} and Tseng
\textit{et al.} in \cite{19b}.

Throughout this subsection, let $f$ be defined as above and define the
following functions on $[0,1].$%
\[
H\left(  t\right)  =\frac{1}{b-a}\int\nolimits_{a}^{b}f\left(  ts+\left(
1-t\right)  \frac{a+b}{2}\right)  ds,
\]%
\begin{align*}
P\left(  t\right)   &  =\frac{1}{2\left(  b-a\right)  }\int\nolimits_{a}%
^{b}\left[  f\left(  \left(  \frac{1+t}{2}\right)  a+\left(  \frac{1-t}%
{2}\right)  s\right)  \right. \\
&  \text{ \ \ \ \ \ \ \ \ \ \ \ \ \ \ \ \ \ \ \ }\left.  +f\left(  \left(
\frac{1+t}{2}\right)  b+\left(  \frac{1-t}{2}\right)  s\right)  \right]  ds,
\end{align*}%
\[
F\left(  t\right)  =\frac{1}{\left(  b-a\right)  ^{2}}\int_{a}^{b}\int_{a}%
^{b}f\left(  ts+\left(  1-t\right)  u\right)  dsdu;
\]%
\[
G\left(  t\right)  =\frac{1}{2}\left[  f\left(  ta+\left(  1-t\right)
\frac{a+b}{2}\right)  +f\left(  tb+\left(  1-t\right)  \frac{a+b}{2}\right)
\right]  ,
\]%
\[
L\left(  t\right)  =\frac{1}{2\left(  b-a\right)  }\int\nolimits_{a}%
^{b}\left[  f\left(  ta+\left(  1-t\right)  s\right)  +f\left(  tb+\left(
1-t\right)  s\right)  \right]  ds
\]
and%
\[
Q\left(  t\right)  =\frac{1}{2}\left[  f\left(  ta+\left(  1-t\right)
b\right)  +f\left(  tb+\left(  1-t\right)  a\right)  \right]  .
\]

Dragomir established the following Hermite-Hadamard-type inequalities related
to the functions $H,F$ which refine the first inequality of $\left(
\ref{1.1}\right)  $ (see \cite{1b}).

\begin{theoremA}
\bigskip\label{A3}\cite{1b} Let $f,H$\ be defined as above. Then $H$\ is
convex, increasing on $\left[  0,1\right]  ,$\ and for all $t\in\left[
0,1\right]  $, we have%
\[
f\left(  \frac{a+b}{2}\right)  =H\left(  0\right)  \leq H\left(  t\right)
\leq H\left(  1\right)  =\frac{1}{b-a}\int\nolimits_{a}^{b}f\left(  s\right)
ds.
\]

\end{theoremA}

\begin{theoremA}
\label{A4}\cite{1b} Let $f,F$\ be defined as above. Then:
\end{theoremA}

\begin{enumerate}
\item $F$\textit{\ is convex on }$\left[  0,1\right]  $\textit{, symmetric to
}$\frac{1}{2}$\textit{, }$F$\textit{\ is decreasing on }$\left[  0,\frac{1}%
{2}\right]  $\textit{\ and increasing on }$\left[  \frac{1}{2},1\right]
,$\textit{\ and we have}%
\[
\sup\limits_{t\in\left[  0,1\right]  }F\left(  t\right)  =F\left(  0\right)
=F\left(  1\right)  =\frac{1}{b-a}\int\nolimits_{a}^{b}f\left(  s\right)  ds
\]
\textit{and }%
\[
\inf\limits_{t\in\left[  0,1\right]  }F\left(  t\right)  =F\left(  \frac{1}%
{2}\right)  =\frac{1}{\left(  b-a\right)  ^{2}}\int_{a}^{b}\int_{a}%
^{b}f\left(  \frac{s+u}{2}\right)  dsdu.
\]

\item \textit{We have: }%
\[
f\left(  \frac{a+b}{2}\right)  \leq F\left(  \frac{1}{2}\right)  \text{;\quad
}H\left(  t\right)  \leq F\left(  t\right)  \text{, \qquad\ }t\in\left[
0,1\right]  \text{.}%
\]
\noindent
\end{enumerate}

Yang and Hong established the following Hermite-Hadamard-type inequality
related to the function $P$\ and which refines the second inequality of
$\left(  \ref{1.1}\right)  $ (see \cite{20b}).

\begin{theoremA}
\label{A5}\cite{20b} Let $f,P$\ be defined as above. Then $P$\ is convex,
increasing on $\left[  0,1\right]  ,$\ and for all $t\in\left[  0,1\right]  $,
we have%
\[
\frac{1}{b-a}\int\nolimits_{a}^{b}f\left(  s\right)  ds=P\left(  0\right)
\leq P\left(  t\right)  \leq P\left(  1\right)  =\frac{f\left(  a\right)
+f\left(  b\right)  }{2}.
\]

\end{theoremA}

Dragomir \textit{et al.} established the following Hermite-Hadamard-type
inequalities related to the functions $H,G,L$ (see \cite{6b}).

\begin{theoremA}
\label{A6}\cite{6b} Let $f,H$\ be defined as above. Then:
\end{theoremA}

\begin{enumerate}
\item \textit{The inequality }%
\begin{align*}
f\left(  \frac{a+b}{2}\right)   &  \leq\frac{2}{b-a}\int\nolimits_{\frac
{3a+b}{4}}^{\frac{a+3b}{4}}f\left(  s\right)  ds\\
&  \leq\int\nolimits_{0}^{1}H\left(  t\right)  dt\\
&  \leq\frac{1}{2}\left[  f\left(  \frac{a+b}{2}\right)  +\frac{1}{b-a}%
\int\nolimits_{a}^{b}f\left(  s\right)  ds\right]
\end{align*}
\textit{holds.}

\item \textit{If }$f$\textit{\ is differentiable on }$\left[  a,b\right]
,$\textit{\ then we have the inequalities }%
\begin{align*}
0  &  \leq\frac{1}{b-a}\int\nolimits_{a}^{b}f\left(  s\right)  ds-H\left(
t\right) \\
&  \leq\left(  1-t\right)  \left[  \frac{f\left(  a\right)  +f\left(
b\right)  }{2}-\frac{1}{b-a}\int\nolimits_{a}^{b}f\left(  s\right)  ds\right]
\end{align*}
\textit{and }%
\[
0\leq\frac{f\left(  a\right)  +f\left(  b\right)  }{2}-H\left(  t\right)
\leq\frac{\left(  f^{\prime}\left(  b\right)  -f^{\prime}\left(  a\right)
\right)  \left(  b-a\right)  }{4}%
\]
\textit{for all }$t\in\left[  0,1\right]  .$
\end{enumerate}

\begin{theoremA}
\label{A7}\cite{6b} Let $f,H,G$\ be defined as above. Then:
\end{theoremA}

\begin{enumerate}
\item $G$\textit{\ is convex and increasing on }$\left[  0,1\right]
$\textit{.}

\item \textit{We have }%
\[
\inf\limits_{t\in\left[  0,1\right]  }G\left(  t\right)  =G\left(  0\right)
=f\left(  \frac{a+b}{2}\right)
\]
\textit{and }%
\[
\sup\limits_{t\in\left[  0,1\right]  }G\left(  t\right)  =G\left(  1\right)
=\frac{f\left(  a\right)  +f\left(  b\right)  }{2}\text{.}%
\]

\item \textit{The inequality }%
\begin{equation}
H\left(  t\right)  \leq G\left(  t\right)  \label{1.3}%
\end{equation}
\textit{holds for all }$t\in\left[  0,1\right]  $\textit{.}

\item \textit{The inequality }%
\begin{align}
\frac{2}{b-a}\int_{\frac{3a+b}{4}}^{\frac{a+3b}{4}}f\left(  s\right)  ds  &
\leq\frac{1}{2}\left[  f\left(  \frac{3a+b}{4}\right)  +f\left(  \frac
{a+3b}{4}\right)  \right] \label{1.4}\\
&  \leq\int_{0}^{1}G\left(  t\right)  dt\nonumber\\
&  \leq\frac{1}{2}\left[  f\left(  \frac{a+b}{2}\right)  +\frac{f\left(
a\right)  +f\left(  b\right)  }{2}\right]  .\nonumber
\end{align}

\item \textit{If }$f$\textit{\ is differentiable on }$\left[  a,b\right]
,$\textit{\ then we have the inequality }%
\begin{equation}
0\leq H\left(  t\right)  -f\left(  \frac{a+b}{2}\right)  \leq G\left(
t\right)  -H\left(  t\right)  \label{1.5}%
\end{equation}
\textit{for all }$t\in\left[  0,1\right]  .$
\end{enumerate}

\begin{theoremA}
\label{A8}\cite{6b} Let $f,H,G,L$ be defined as above. Then:
\end{theoremA}

\begin{enumerate}
\item $L$\textit{\ is convex on }$\left[  0,1\right]  $\textit{.}

\item \textit{We have the inequality: }%
\[
G\left(  t\right)  \leq L\left(  t\right)  \leq\frac{1-t}{b-a}\int%
\nolimits_{a}^{b}f\left(  s\right)  ds+t\cdot\frac{f\left(  a\right)
+f\left(  b\right)  }{2}\leq\frac{f\left(  a\right)  +f\left(  b\right)  }{2}%
\]
\textit{for all }$t\in\left[  0,1\right]  $\textit{\ and }%
\[
\sup\limits_{t\in\left[  0,1\right]  }L\left(  t\right)  =\frac{f\left(
a\right)  +f\left(  b\right)  }{2}.
\]

\item \textit{One has the inequalities: }%
\begin{equation}
H\left(  1-t\right)  \leq L\left(  t\right)  \text{\textit{\ }} \label{1.6}%
\end{equation}
\textit{and}%
\begin{equation}
\frac{H\left(  t\right)  +H\left(  1-t\right)  }{2}\leq L\left(  t\right)
\label{1.7}%
\end{equation}
\textit{for all }$t\in\left[  0,1\right]  .$
\end{enumerate}

Tseng \textit{et al.} established the following Hermite-Hadamard-type
inequalities related to the functions $H,P,L,G,Q$ (see \cite{19b}).

\begin{theoremA}
\label{A9}\cite{19b} Let $f,H,P$\ be defined as above. Then we have the
following results:
\end{theoremA}

\begin{enumerate}
\item \textit{The inequality }%
\begin{align*}
\frac{1}{b-a}\int\nolimits_{a}^{b}f\left(  s\right)  ds  &  \leq\frac{2}%
{b-a}\int_{\left[  a,\frac{3a+b}{4}\right]  \cup\left[  \frac{a+3b}%
{4},b\right]  }f\left(  s\right)  ds\\
&  \leq\int\nolimits_{0}^{1}P\left(  t\right)  dt\\
&  \leq\frac{1}{2}\left[  \frac{1}{b-a}\int\nolimits_{a}^{b}f\left(  s\right)
ds+\frac{f\left(  a\right)  +f\left(  b\right)  }{2}\right]
\end{align*}
\textit{holds.}

\item \textit{The inequalities}%
\[
L\left(  t\right)  \leq P\left(  t\right)  \leq\frac{1-t}{b-a}\int%
\nolimits_{a}^{b}f\left(  s\right)  ds+t\cdot\frac{f\left(  a\right)
+f\left(  b\right)  }{2}\leq\frac{f\left(  a\right)  +f\left(  b\right)  }{2}%
\]
\textit{and}%
\[
0\leq P\left(  t\right)  -G\left(  t\right)  \leq\frac{f\left(  a\right)
+f\left(  b\right)  }{2}-P\left(  t\right)
\]
\textit{hold for all }$t\in\left[  0,1\right]  .$

\item \textit{If }$f$\textit{\ is differentiable on }$\left[  a,b\right]
,$\textit{\ then we have the inequalities }%
\[
0\leq t\left[  \frac{1}{b-a}\int\nolimits_{a}^{b}f\left(  s\right)
ds-f\left(  \frac{a+b}{2}\right)  \right]  \leq P\left(  t\right)  -\frac
{1}{b-a}\int\nolimits_{a}^{b}f\left(  s\right)  ds,
\]%
\[
0\leq P\left(  t\right)  -f\left(  \frac{a+b}{2}\right)  \leq\frac{\left(
f^{\prime}\left(  b\right)  -f^{\prime}\left(  a\right)  \right)  \left(
b-a\right)  }{4}%
\]
\textit{and }%
\[
0\leq P\left(  t\right)  -H\left(  t\right)  \leq\frac{\left(  f^{\prime
}\left(  b\right)  -f^{\prime}\left(  a\right)  \right)  \left(  b-a\right)
}{4}%
\]
\textit{for all }$t\in\left[  0,1\right]  .$
\end{enumerate}

\begin{theoremA}
\label{A10}\cite{19b} Let $f,G,Q$\ be defined as above. Then $Q$\ is symmetric
about $\frac{1}{2}$, $Q$ is decreasing on $\left[  0,\frac{1}{2}\right]  $ and
increasing on $\left[  \frac{1}{2},1\right]  $,%
\[
G\left(  2t\right)  \leq Q\left(  t\right)  \text{ \quad}\left(  t\in\left[
0,\frac{1}{4}\right]  \right)  ,
\]%
\[
G\left(  2t\right)  \geq Q\left(  t\right)  \quad\left(  t\in\left[  \frac
{1}{4},\frac{1}{2}\right]  \right)  ,
\]%
\[
G\left(  2\left(  1-t\right)  \right)  \geq Q\left(  t\right)  \text{ \quad
}\left(  t\in\left[  \frac{1}{2},\frac{3}{4}\right]  \right)
\]
\textit{and}%
\[
G\left(  2\left(  1-t\right)  \right)  \leq Q\left(  t\right)  \quad\left(
t\in\left[  \frac{3}{4},1\right]  \right)  .
\]

\end{theoremA}

Dragomir established the following Hermite-Hadamard-type inequality related to
the functions $H,F,L$ (see \cite{4b})

\begin{theoremA}
\label{A11}\cite{4b} Let $f,F,H,L$ be defined as above. Then we have the
inequality%
\[
0\leq F\left(  t\right)  -H\left(  t\right)  \leq L\left(  1-t\right)
-F\left(  t\right)
\]
for all $t\in\left[  0,1\right]  .$
\end{theoremA}

\subsection{New Hermite-Hadamard-type Inequalities}

In this subsection, we present eight theorems related new
Hermite-Hadamard-type inequalities and which reduce Theorems \ref{A3}
--\ \ref{A9} and \ref{A11} (see \cite{12b} and \cite{13b}).

Throughout this subsection, let $a,b.x,y,y^{\prime},x^{\prime},\Omega
,f,H,P,F,G,L$ be defined as above and define the following functions on
$[0,1].$%
\[
H_{1}\left(  t\right)  =\frac{1}{2\left(  y-x\right)  }\int_{x}^{y}\left[
f\left(  ts+\left(  1-t\right)  y\right)  +f\left(  t\left(  y+y^{\prime
}-s\right)  +\left(  1-t\right)  y^{\prime}\right)  \right]  ds,
\]%
\[
H_{2}\left(  t\right)  =\frac{1}{2\left(  y-x\right)  }\int_{x}^{y}\left[
f\left(  ts+\left(  1-t\right)  y^{\prime}\right)  +f\left(  t\left(
y+y^{\prime}-s\right)  +\left(  1-t\right)  y\right)  \right]  ds,
\]%
\[
F_{1}\left(  t\right)  =\frac{1}{4\left(  y-x\right)  ^{2}}\int_{\Omega}%
\int_{\Omega}f\left(  ts+\left(  1-t\right)  u\right)  dsdu,
\]%
\[
P_{1}\left(  t\right)  =\frac{1}{2\left(  y-x\right)  }\int_{x}^{y}\left[
f\left(  tx+\left(  1-t\right)  s\right)  +f\left(  tx^{\prime}+\left(
1-t\right)  \left(  x+x^{\prime}-s\right)  \right)  \right]  ds,
\]%
\[
G_{1}\left(  t\right)  =\frac{1}{2}\left[  f\left(  tx+\left(  1-t\right)
y\right)  +f\left(  tx^{\prime}+\left(  1-t\right)  y^{\prime}\right)
\right]  ,
\]%
\[
G_{2}\left(  t\right)  =\frac{1}{2}\left[  f\left(  tx+\left(  1-t\right)
y^{\prime}\right)  +f\left(  tx^{\prime}+\left(  1-t\right)  y\right)
\right]
\]
and%
\[
L_{1}\left(  t\right)  =\frac{1}{4\left(  y-x\right)  }%
{\displaystyle\int\nolimits_{\Omega}}
\left[  f\left(  tx+\left(  1-t\right)  s\right)  +f\left(  tx^{\prime
}+\left(  1-t\right)  s\right)  \right]  ds.
\]

\begin{remark}
\label{r1}\textit{We note that }$\Omega=\left[  a,b\right]  $\textit{\ and
}$H_{1}\left(  t\right)  =H_{2}\left(  t\right)  =H\left(  t\right)  ,$
$F_{1}\left(  t\right)  =F\left(  t\right)  ,$ $P_{1}\left(  t\right)
=P\left(  t\right)  ,$ $G_{1}\left(  t\right)  =G_{2}\left(  t\right)
=G\left(  t\right)  ,$ $L_{1}\left(  t\right)  =L\left(  t\right)
$\textit{\ on }$\left[  0,1\right]  $\textit{\ as }$x=a,$ $y=y^{\prime}%
=\frac{a+b}{2}$\textit{\ and }$x^{\prime}=b.$
\end{remark}

Tseng \textit{et al.} established the following four theorems related to the
functions $H_{1},H_{2},P_{1},F_{1}$ and which reduce Theorems \ref{A3}
--\ \ref{A5} (see \cite{12b}).

\begin{theoremA}
\label{A12}\cite{12b} Let $x,y,y^{\prime},x^{\prime},\Omega,f,H_{1}$\ be
defined as above. Then:
\end{theoremA}

\begin{enumerate}
\item $H_{1}$\textit{\ is convex on }$\left[  0,1\right]  .$

\item $H_{1}$\textit{\ is increasing on }$\left[  0,1\right]  $\textit{\ and
the following inequalities }%
\[
\frac{f\left(  y\right)  +f\left(  y^{\prime}\right)  }{2}=H_{1}\left(
0\right)  \leq H_{1}\left(  t\right)  \leq H_{1}\left(  1\right)  =\frac
{1}{2\left(  y-x\right)  }\int_{\Omega}f\left(  s\right)  ds\
\]
\textit{and}%
\begin{align*}
H_{1}\left(  t\right)   &  \leq t\cdot\frac{1}{2\left(  y-x\right)  }%
\int_{\Omega}f\left(  s\right)  ds+\left(  1-t\right)  \cdot\frac{f\left(
y\right)  +f\left(  y^{\prime}\right)  }{2}\\
&  \leq\frac{1}{2\left(  y-x\right)  }\int_{\Omega}f\left(  s\right)
ds\leq\frac{f\left(  x\right)  +f\left(  x^{\prime}\right)  }{2}%
\end{align*}
\textit{hold for all }$t\in\left[  0,1\right]  .$
\end{enumerate}

\begin{theoremA}
\label{A13}\cite{12b} Let $x,y,y^{\prime},x^{\prime},\Omega,f,H_{2}$\ be
defined as above. Then:
\end{theoremA}

\begin{enumerate}
\item $H_{2}$\ is convex on $\left[  0,1\right]  .$

\item \textit{The following inequalities }%
\begin{align*}
f\left(  \frac{y+y^{\prime}}{2}\right)   &  \leq H_{2}\left(  t\right) \\
&  \leq t\cdot\frac{1}{2\left(  y-x\right)  }\int_{\Omega}f\left(  s\right)
ds+\left(  1-t\right)  \cdot\frac{f\left(  y\right)  +f\left(  y^{\prime
}\right)  }{2}\\
&  \leq\frac{1}{2\left(  y-x\right)  }\int_{\Omega}f\left(  s\right)  ds
\end{align*}
\textit{and}%
\[
H_{2}\left(  t\right)  \leq H_{1}\left(  t\right)
\]
\textit{hold for all }$t\in\left[  0,1\right]  .$
\end{enumerate}

\begin{theoremA}
\label{A14}\cite{12b} Let $x,y,y^{\prime},x^{\prime},\Omega,f,H_{1},P_{1}$\ be
defined as above. Then we have the following results:
\end{theoremA}

\begin{enumerate}
\item $P_{1}$\textit{\ is convex on }$\left[  0,1\right]  .$

\item $P_{1}$\textit{\ is increasing on }$\left[  0,1\right]  $\textit{\ and
the following inequalities}%
\[
\frac{1}{2\left(  y-x\right)  }\int_{\Omega}f\left(  s\right)  ds=P_{1}\left(
0\right)  \leq P_{1}\left(  t\right)  \leq P_{1}\left(  1\right)
=\frac{f\left(  x\right)  +f\left(  x^{\prime}\right)  }{2}%
\]
\textit{and}%
\begin{align}
P_{1}\left(  t\right)   &  \leq\left(  1-t\right)  \cdot\frac{1}{2\left(
y-x\right)  }\int_{\Omega}f\left(  s\right)  ds+t\cdot\frac{f\left(  x\right)
+f\left(  x^{\prime}\right)  }{2}\label{1.8}\\
&  \leq\frac{f\left(  x\right)  +f\left(  x^{\prime}\right)  }{2}\nonumber
\end{align}
\textit{hold for all }$t\in\left[  0,1\right]  .$

\item \textit{The inequality}%
\[
H_{1}\left(  t\right)  \leq P_{1}\left(  t\right)  \text{ \qquad\ \ }\left(
t\in\left[  0,1\right]  \right)
\]
\textit{holds.}
\end{enumerate}

\begin{theoremA}
\label{A15}\cite{12b} Let $x,y,y^{\prime},x^{\prime},\Omega,f,H_{1}%
,H_{2},F_{1}$\ be defined as above. Then we have the following results:
\end{theoremA}

\begin{enumerate}
\item $F_{1}$\textit{\ is convex on }$\left[  0,1\right]  $\textit{\ and
symmetric about }$\frac{1}{2}.$

\item $F_{1}$\textit{\ is decreasing on }$\left[  0,\frac{1}{2}\right]
$\textit{\ and increasing on }$\left[  \frac{1}{2},1\right]  ,$%
\[
\sup\limits_{t\in\left[  0,1\right]  }F_{1}\left(  t\right)  =F_{1}\left(
0\right)  =F_{1}\left(  1\right)  =\frac{1}{2\left(  y-x\right)  }\int%
_{\Omega}f\left(  s\right)  ds\
\]
\textit{and }%
\[
\inf\limits_{t\in\left[  0,1\right]  }F_{1}\left(  t\right)  =F_{1}\left(
\frac{1}{2}\right)  =\frac{1}{4\left(  y-x\right)  ^{2}}\int_{\Omega}%
\int_{\Omega}f\left(  \frac{s+u}{2}\right)  dsdu.
\]

\item \textit{We have:}%
\[
\frac{H_{1}\left(  t\right)  +H_{2}\left(  t\right)  }{2}\leq F_{1}\left(
t\right)
\]
\textit{and}%
\[
\frac{f\left(  y\right)  +2f\left(  \frac{y+y^{\prime}}{2}\right)  +f\left(
y^{\prime}\right)  }{4}\leq F_{1}\left(  \frac{1}{2}\right)
\]
for all $t\in\left[  0,1\right]  .$
\end{enumerate}

Tseng \textit{et al.} established the following four Theorems related to the
functions $H_{1},H_{2},P_{1},F_{1},G_{1},G_{2},L_{1}$ and which reduce
Theorems \ref{A6} --\ \ref{A9} and \ref{A11} (see \cite{13b}).

\begin{theoremA}
\label{A16}\cite{13b} Let $x,y,y^{\prime},x^{\prime},\Omega,f,H_{1},H_{2}$\ be
defined as above. Then:
\end{theoremA}

\begin{enumerate}
\item \textit{The inequalities }%
\begin{align*}
\frac{f\left(  y\right)  +f\left(  y^{\prime}\right)  }{2}  &  \leq\frac
{1}{y-x}\int_{\left[  \frac{x+y}{2},y\right]  \cup\left[  y^{\prime}%
,\frac{x^{\prime}+y^{\prime}}{2}\right]  }f\left(  s\right)  ds\\
&  \leq\int_{0}^{1}H_{1}\left(  t\right)  dt\\
&  \leq\frac{1}{2}\left[  \frac{f\left(  y\right)  +f\left(  y^{\prime
}\right)  }{2}+\frac{1}{2\left(  y-x\right)  }\int_{\Omega}f\left(  s\right)
ds\right]
\end{align*}
\textit{and}%
\begin{align*}
f\left(  \frac{y+y^{\prime}}{2}\right)   &  \leq\frac{1}{y-x}\int%
_{\frac{x+y^{\prime}}{2}}^{\frac{y+x^{\prime}}{2}}f\left(  s\right)  ds\\
&  \leq\int_{0}^{1}H_{2}\left(  t\right)  dt\\
&  \leq\frac{1}{2}\left[  \frac{f\left(  y\right)  +f\left(  y^{\prime
}\right)  }{2}+\frac{1}{2\left(  y-x\right)  }\int_{\Omega}f\left(  s\right)
ds\right]
\end{align*}
\textit{hold.}

\item \textit{If }$f$\textit{\ is differentiable on }$\left[  a,b\right]
,$\textit{\ then the inequalities}%
\begin{align*}
0  &  \leq\frac{1}{2\left(  y-x\right)  }\int_{\Omega}f\left(  s\right)
ds-H_{1}\left(  t\right) \\
&  \leq\left(  1-t\right)  \left[  \frac{f\left(  x\right)  +f\left(
x^{\prime}\right)  }{2}-\frac{1}{2\left(  y-x\right)  }\int_{\Omega}f\left(
s\right)  ds\right]  ,
\end{align*}%
\[
0\leq\frac{f\left(  x\right)  +f\left(  x^{\prime}\right)  }{2}-H_{1}\left(
t\right)  \leq\left(  y-x\right)  \frac{f^{\prime}\left(  x^{\prime}\right)
-f^{\prime}\left(  x\right)  }{2},
\]
\textit{and}%
\[
0\leq H_{1}\left(  t\right)  -\frac{f\left(  y\right)  +f\left(  y^{\prime
}\right)  }{2}\leq\left(  y-x\right)  \frac{f^{\prime}\left(  x^{\prime
}\right)  -f^{\prime}\left(  x\right)  }{2}%
\]
\textit{hold for all }$t\in\left[  0,1\right]  .$
\end{enumerate}

\begin{theoremA}
\textit{\label{A17}\cite{13b} Let }$x,y,y^{\prime},x^{\prime},\Omega
,f,H_{1},P_{1}$\textit{\ be defined as above. Then we have the following
results:}
\end{theoremA}

\begin{enumerate}
\item \textit{The inequality}%
\begin{align*}
\frac{1}{2\left(  y-x\right)  }\int_{\Omega}f\left(  s\right)  ds  &
\leq\frac{1}{y-x}\int_{\left[  x,\frac{x+y}{2}\right]  \cup\left[
\frac{x^{\prime}+y^{\prime}}{2},x^{\prime}\right]  }f\left(  s\right)  ds\\
&  \leq\int_{0}^{1}P_{1}\left(  t\right)  dt\\
&  \leq\frac{1}{2}\left[  \frac{f\left(  x\right)  +f\left(  x^{\prime
}\right)  }{2}+\frac{1}{2\left(  y-x\right)  }\int_{\Omega}f\left(  s\right)
ds\right]
\end{align*}
\textit{holds.}

\item \textit{If }$f$\textit{\ is differentiable on }$\left[  a,b\right]
,$\textit{\ then the inequalities}%
\begin{align*}
0  &  \leq t\left[  \frac{1}{2\left(  y-x\right)  }\int_{\Omega}f\left(
s\right)  ds-\frac{f\left(  y\right)  +f\left(  y^{\prime}\right)  }{2}\right]
\\
&  \leq P_{1}\left(  t\right)  -\frac{1}{2\left(  y-x\right)  }\int_{\Omega
}f\left(  s\right)  ds,
\end{align*}%
\[
0\leq P_{1}\left(  t\right)  -\frac{f\left(  y\right)  +f\left(  y^{\prime
}\right)  }{2}\leq\left(  y-x\right)  \frac{f^{\prime}\left(  x^{\prime
}\right)  -f^{\prime}\left(  x\right)  }{2},
\]%
\[
0\leq\frac{f\left(  x\right)  +f\left(  x^{\prime}\right)  }{2}-P_{1}\left(
t\right)  \leq\left(  y-x\right)  \frac{f^{\prime}\left(  x^{\prime}\right)
-f^{\prime}\left(  x\right)  }{2}%
\]
\textit{and}%
\[
0\leq P_{1}\left(  t\right)  -H_{1}\left(  t\right)  \leq\left(  y-x\right)
\frac{f^{\prime}\left(  x^{\prime}\right)  -f^{\prime}\left(  x\right)  }{2}%
\]
\textit{hold for all }$t\in\left[  0,1\right]  .$
\end{enumerate}

\begin{theoremA}
\label{A18}\cite{13b} Let $x,y,y^{\prime},x^{\prime},f,H_{1},P_{1},G_{1}%
,G_{2}$\ be defined as above. Then we have the following results:
\end{theoremA}

\begin{enumerate}
\item $G_{1}$\textit{\ and }$G_{2}$\textit{\ are convex on }$\left[
0,1\right]  .$

\item $G_{1}$\textit{\ is increasing on }$\left[  0,1\right]  ,$%
\textit{\ }$G_{2}$\textit{\ is decreasing on }$\left[  0,\frac{y^{\prime}%
-y}{2\left(  y^{\prime}-x\right)  }\right]  $\textit{\ and increasing on
}$\left[  \frac{y^{\prime}-y}{2\left(  y^{\prime}-x\right)  },1\right]
,$\textit{\ and the inequalities}%
\[
\frac{f\left(  y\right)  +f\left(  y^{\prime}\right)  }{2}=G_{1}\left(
0\right)  \leq G_{1}\left(  t\right)  \leq G_{1}\left(  1\right)
=\frac{f\left(  x\right)  +f\left(  x^{\prime}\right)  }{2},
\]%
\[
f\left(  \frac{x+x^{\prime}}{2}\right)  =G_{2}\left(  \frac{y^{\prime}%
-y}{2\left(  y^{\prime}-x\right)  }\right)  \leq G_{2}\left(  t\right)  \leq
G_{2}\left(  1\right)  =\frac{f\left(  x\right)  +f\left(  x^{\prime}\right)
}{2}%
\]
\textit{and}%
\[
G_{2}\left(  t\right)  \leq G_{1}\left(  t\right)
\]
\textit{hold for all }$t\in\left[  0,1\right]  .$

\item \textit{The inequality}%
\[
H_{1}\left(  t\right)  \leq G_{1}\left(  t\right)  \leq P_{1}\left(  t\right)
\]
\textit{holds for all }$t\in\left[  0,1\right]  .$

\item \textit{The inequalities }%
\begin{align*}
\frac{1}{y-x}\int_{\left[  \frac{x+y}{2},y\right]  \cup\left[  y^{\prime
},\frac{x^{\prime}+y^{\prime}}{2}\right]  }f\left(  s\right)  ds  &  \leq
\frac{1}{2}\left[  f\left(  \frac{x+y}{2}\right)  +f\left(  \frac{x^{\prime
}+y^{\prime}}{2}\right)  \right] \\
&  \leq\int_{0}^{1}G_{1}\left(  t\right)  dt\\
&  \leq\frac{f\left(  x\right)  +f\left(  x^{\prime}\right)  +f\left(
y\right)  +f\left(  y^{\prime}\right)  }{4}%
\end{align*}
\textit{and }%
\begin{align*}
\frac{1}{y-x}\int_{\frac{x+y^{\prime}}{2}}^{\frac{x^{\prime}+y}{2}}f\left(
s\right)  ds  &  \leq\frac{1}{2}\left[  f\left(  \frac{x+y^{\prime}}%
{2}\right)  +f\left(  \frac{x^{\prime}+y}{2}\right)  \right] \\
&  \leq\int_{0}^{1}G_{2}\left(  t\right)  dt\\
&  \leq\frac{f\left(  x\right)  +f\left(  x^{\prime}\right)  +f\left(
y\right)  +f\left(  y^{\prime}\right)  }{4}%
\end{align*}
\textit{hold.}

\item \textit{The inequalities }%
\[
0\leq H_{1}\left(  t\right)  -\frac{f\left(  y\right)  +f\left(  y^{\prime
}\right)  }{2}\leq G_{1}\left(  t\right)  -H_{1}\left(  t\right)
\]
\textit{and }%
\begin{equation}
0\leq P_{1}\left(  t\right)  -G_{1}\left(  t\right)  \leq\frac{f\left(
x\right)  +f\left(  x^{\prime}\right)  }{2}-P_{1}\left(  t\right)  \label{1.9}%
\end{equation}
\textit{hold for all }$t\in\left[  0,1\right]  .$
\end{enumerate}

\begin{theoremA}
\textit{\label{A19}\cite{13b} Let }$x,y,y^{\prime},x^{\prime},\Omega
,f,L_{1},P_{1},G_{1},G_{2},H_{1},H_{2},F_{1}$\textit{\ be defined as above.
Then we have the following results:}
\end{theoremA}

\begin{enumerate}
\item $L_{1}$\textit{\ is convex on }$\left[  0,1\right]  $\textit{.}

\item \textit{The following inequalities}%
\begin{equation}
\frac{G_{1}\left(  t\right)  +G_{2}\left(  t\right)  }{2}\leq L_{1}\left(
t\right)  \leq P_{1}\left(  t\right)  \text{ \qquad\ }\left(  t\in\left[
0,1\right]  \right)  \label{1.10}%
\end{equation}
\textit{and}%
\[
\sup\limits_{t\in\left[  0,1\right]  }L_{1}\left(  t\right)  =L_{1}\left(
1\right)  =\frac{f\left(  x\right)  +f\left(  x^{\prime}\right)  }{2}%
\]
\textit{hold.}

\item \textit{The following inequalities}%
\begin{equation}
\frac{H_{1}\left(  1-t\right)  +H_{2}\left(  1-t\right)  }{2}\leq F_{1}\left(
t\right)  \leq L_{1}\left(  t\right)  , \label{1.11}%
\end{equation}%
\[
\frac{H_{1}\left(  t\right)  +H_{2}\left(  t\right)  }{2}\leq F_{1}\left(
t\right)  \leq L_{1}\left(  t\right)
\]
\textit{and}%
\begin{equation}
\frac{H_{1}\left(  t\right)  +H_{2}\left(  t\right)  +H_{1}\left(  1-t\right)
+H_{2}\left(  1-t\right)  }{4}\leq F_{1}\left(  t\right)  \leq L_{1}\left(
t\right)  \label{1.12}%
\end{equation}
\textit{hold for all }$t\in\left[  0,1\right]  .$

\item \textit{The following inequality}%
\begin{equation}
0\leq F_{1}\left(  t\right)  -\frac{H_{1}\left(  t\right)  +H_{2}\left(
t\right)  }{2}\leq L_{1}\left(  1-t\right)  -F_{1}\left(  t\right)
\label{1.13}%
\end{equation}
\textit{holds for all }$t\in\left[  0,1\right]  .$
\end{enumerate}

\subsection{Fej\'{e}r-type Inequalities}

In this subsection, we present nineteen theorems related Fej\'{e}r-type
inequalities and which are weighted generalizations of Theorems \ref{A3}
--\ \ref{A9} and \ref{A11} (see \cite{14b} --\ \cite{18b}, \cite{21b}\ and
\cite{24b}).

Throughout this subsection, let $f,g,H,P,F,G,L,Q$ be defined as above and
define the following functions on $[0,1].$%
\[
Hg\left(  t\right)  =\int\nolimits_{a}^{b}f\left(  ts+\left(  1-t\right)
\frac{a+b}{2}\right)  g\left(  s\right)  ds,
\]%
\[
Fg\left(  t\right)  =\int_{a}^{b}\int_{a}^{b}f\left(  ts+\left(  1-t\right)
u\right)  g\left(  s\right)  g\left(  u\right)  dsdu,
\]%
\begin{align*}
Pg\left(  t\right)   &  =\int\nolimits_{a}^{b}\frac{1}{2}\left[  f\left(
\left(  \frac{1+t}{2}\right)  a+\left(  \frac{1-t}{2}\right)  s\right)
g\left(  \frac{s+a}{2}\right)  \right. \\
&  \text{ \ \ \ \ \ \ \ \ }\left.  +f\left(  \left(  \frac{1+t}{2}\right)
b+\left(  \frac{1-t}{2}\right)  s\right)  g\left(  \frac{s+b}{2}\right)
\right]  ds,
\end{align*}%
\[
Lg\left(  t\right)  =\frac{1}{2}\int\nolimits_{a}^{b}\left[  f\left(
ta+\left(  1-t\right)  s\right)  +f\left(  tb+\left(  1-t\right)  s\right)
\right]  g\left(  s\right)  ds,
\]%
\begin{align*}
Ig\left(  t\right)   &  =%
{\displaystyle\int\nolimits_{a}^{b}}
\frac{1}{2}\left[  f\left(  t\frac{s+a}{2}+\left(  1-t\right)  \frac{a+b}%
{2}\right)  \right. \\
&  \text{ \ \ \ \ \ \ \ \ }+\left.  f\left(  t\frac{s+b}{2}+\left(
1-t\right)  \frac{a+b}{2}\right)  \right]  g\left(  s\right)  ds,
\end{align*}%
\begin{align*}
Jg\left(  t\right)   &  =%
{\displaystyle\int\nolimits_{a}^{b}}
\frac{1}{2}\left[  f\left(  t\frac{s+a}{2}+\left(  1-t\right)  \frac{3a+b}%
{4}\right)  \right. \\
&  \text{ \ \ \ \ \ \ \ \ }+\left.  f\left(  t\frac{s+b}{2}+\left(
1-t\right)  \frac{a+3b}{4}\right)  \right]  g\left(  s\right)  ds,
\end{align*}%
\begin{align*}
Mg\left(  t\right)   &  =%
{\displaystyle\int\nolimits_{a}^{\frac{a+b}{2}}}
\frac{1}{2}\left[  f\left(  ta+\left(  1-t\right)  \frac{s+a}{2}\right)
+f\left(  t\frac{a+b}{2}+\left(  1-t\right)  \frac{s+b}{2}\right)  \right]
g\left(  s\right)  ds\\
&  +%
{\displaystyle\int\nolimits_{\frac{a+b}{2}}^{b}}
\frac{1}{2}\left[  f\left(  t\frac{a+b}{2}+\left(  1-t\right)  \frac{s+a}%
{2}\right)  +f\left(  tb+\left(  1-t\right)  \frac{s+b}{2}\right)  \right]
g\left(  s\right)  ds,
\end{align*}%
\[
Ng\left(  t\right)  =%
{\displaystyle\int\nolimits_{a}^{b}}
\frac{1}{2}\left[  f\left(  ta+\left(  1-t\right)  \frac{s+a}{2}\right)
+f\left(  tb+\left(  1-t\right)  \frac{s+b}{2}\right)  \right]  g\left(
s\right)  ds,
\]%
\begin{align*}
Sg\left(  t\right)   &  =\frac{1}{4}\int\nolimits_{a}^{b}\left[  f\left(
ta+\left(  1-t\right)  \frac{s+a}{2}\right)  +f\left(  ta+\left(  1-t\right)
\frac{s+b}{2}\right)  \right. \\
&  \text{ \ \ \ \ \ \ \ \ }\left.  +f\left(  tb+\left(  1-t\right)  \frac
{s+a}{2}\right)  +f\left(  tb+\left(  1-t\right)  \frac{s+b}{2}\right)
\right]  g\left(  s\right)  ds,
\end{align*}
and%
\begin{align*}
Kg\left(  t\right)   &  =\int_{a}^{b}\int_{a}^{b}\frac{1}{4}\left[  f\left(
t\frac{s+a}{2}+\left(  1-t\right)  \frac{u+a}{2}\right)  \right. \\
&  +f\left(  t\frac{s+a}{2}+\left(  1-t\right)  \frac{u+b}{2}\right)
+f\left(  t\frac{s+b}{2}+\left(  1-t\right)  \frac{u+a}{2}\right) \\
&  \text{ \ \ \ \ \ \ \ \ \ \ \ \ \ \ \ \ \ \ \ \ \ \ \ \ \ \ \ \ \ \ \ \ }%
+\left.  f\left(  t\frac{s+b}{2}+\left(  1-t\right)  \frac{u+b}{2}\right)
\right]  g\left(  s\right)  g\left(  u\right)  dsdu.
\end{align*}

\begin{remark}
\label{r2}\textit{We note that }$g\left(  s\right)  \equiv\frac{1}{b-a}%
$\ $\left(  s\in\left[  a,b\right]  \right)  $\textit{\ and }$Hg\left(
t\right)  =Ig\left(  t\right)  =H\left(  t\right)  ,$ $Fg\left(  t\right)
=Kg\left(  t\right)  =F\left(  t\right)  ,$ $Pg\left(  t\right)  =Ng\left(
t\right)  =P\left(  t\right)  ,$ $G\left(  t\right)  =G_{1}\left(  t\right)
=G_{2}\left(  t\right)  ,$ $Lg\left(  t\right)  =Sg\left(  t\right)  =L\left(
t\right)  $\textit{\ on }$\left[  0,1\right]  .$
\end{remark}

Yang and Tseng established the following three theorems related to $Hg,Pg,Fg$
and which are weighted generalizations of Theorems \ref{A3} --\ \ref{A5} (see
\cite{21b} and \cite{24b}).

\begin{theoremA}
\label{A20}\cite{21b} Let $f,g,Hg$\ be defined as above. Then $Hg$\ is convex,
increasing on $\left[  0,1\right]  ,$\ and for all $t\in\left[  0,1\right]  $,
we have%
\[
f\left(  \frac{a+b}{2}\right)  \int\nolimits_{a}^{b}g\left(  s\right)
ds=Hg\left(  0\right)  \leq Hg\left(  t\right)  \leq Hg\left(  1\right)
=\int\nolimits_{a}^{b}f\left(  s\right)  g\left(  s\right)  ds.
\]

\end{theoremA}

\begin{theoremA}
\label{A21}\cite{21b} Let $f,g,Pg$\ be defined as above. Then $Pg$\ is convex,
increasing on $\left[  0,1\right]  ,$\ and for all $t\in\left[  0,1\right]  $,
we have%
\[
\int\nolimits_{a}^{b}f\left(  s\right)  g\left(  s\right)  ds=Pg\left(
0\right)  \leq Pg\left(  t\right)  \leq Pg\left(  1\right)  =\frac{f\left(
a\right)  +f\left(  b\right)  }{2}\int\nolimits_{a}^{b}g\left(  s\right)  ds.
\]

\end{theoremA}

\begin{theoremA}
\label{A22}\cite{24b} Let $f,g,Hg,Fg$\ be defined as above. Then we have the
following results:
\end{theoremA}

\begin{enumerate}
\item $Fg$\textit{\ is convex on }$\left[  0,1\right]  $\textit{\ and
symmetric about }$\frac{1}{2}.$

\item $Fg$\textit{\ is decreasing on }$\left[  0,\frac{1}{2}\right]
$\textit{\ and increasing on }$\left[  \frac{1}{2},1\right]  ,$%
\[
\sup\limits_{t\in\left[  0,1\right]  }Fg\left(  t\right)  =Fg\left(  0\right)
=Fg\left(  1\right)  =\int_{a}^{b}f\left(  s\right)  g\left(  s\right)  ds\
\]
\textit{and }%
\[
\inf\limits_{t\in\left[  0,1\right]  }Fg\left(  t\right)  =Fg\left(  \frac
{1}{2}\right)  =\int_{a}^{b}\int_{a}^{b}f\left(  \frac{s+u}{2}\right)
g\left(  s\right)  g\left(  u\right)  dsdu.
\]

\item \textit{We have:}%
\[
f\left(  \frac{a+b}{2}\right)  \left(  \int_{a}^{b}g\left(  s\right)
ds\right)  ^{2}\leq F_{1}\left(  \frac{1}{2}\right)
\]
\textit{and}%
\[
Hg\left(  t\right)  \int_{a}^{b}g\left(  s\right)  ds\leq F_{1}\left(
t\right)
\]
\textit{for all }$t\in\left[  0,1\right]  .$
\end{enumerate}

Tseng \textit{et al.} established the following three theorems related to
$G,Hg,Fg$ and which are weighted generalizations of Theorems \ref{A6}
--\ \ref{A8} (see \cite{15b}).

\begin{theoremA}
\label{A23}\cite{15b} \textit{Let }$f,g,Hg$\textit{\ be defined as above.
Then} \textit{we have the following Fej\'{e}r-type inequalities:}
\end{theoremA}

\begin{enumerate}
\item \textit{The following inequality holds}:%
\begin{align*}
&  f\left(  \frac{a+b}{2}\right)  \int\nolimits_{a}^{b}g\left(  s\right)  ds\\
&  \leq2\int_{\frac{3a+b}{4}}^{\frac{a+3b}{4}}f\left(  s\right)  g\left(
2s-\frac{a+b}{2}\right)  ds\text{ }\\
&  \leq\int\nolimits_{0}^{1}Hg\left(  t\right)  dt\\
&  \leq\frac{1}{2}\left[  f\left(  \frac{a+b}{2}\right)  \int\nolimits_{a}%
^{b}g\left(  s\right)  ds+\int\nolimits_{a}^{b}f\left(  s\right)  g\left(
s\right)  ds\right]  .\text{\ \ \ \ \ \ \ \ \ \ \ \ }%
\end{align*}
\ \ 

\item \textit{If }$f$ \textit{is differentiable on }$\left[  a,b\right]  $
\textit{and }$g$\textit{\ is bounded on }$\left[  a,b\right]  $\textit{, then,
for all }$t\in\left[  0,1\right]  ,$\textit{\ we have the inequality}%
\begin{align*}
0  &  \leq\int\nolimits_{a}^{b}f\left(  s\right)  g\left(  s\right)
ds-Hg\left(  t\right) \\
&  \leq\left(  1-t\right)  \left[  \frac{f\left(  a\right)  +f\left(
b\right)  }{2}\left(  b-a\right)  -\int\nolimits_{a}^{b}f\left(  s\right)
ds\right]  \left\Vert g\right\Vert _{\infty}%
\end{align*}
\textit{where }$\left\Vert g\right\Vert _{\infty}=\sup\limits_{s\in\left[
a,b\right]  }\left\vert g\left(  s\right)  \right\vert .$

\item \textit{If }$f$ \textit{is differentiable on }$\left[  a,b\right]
$\textit{, then, for all }$t\in\left[  0,1\right]  ,$\textit{\ we have the
inequality}
\begin{align*}
0  &  \leq\frac{f\left(  a\right)  +f\left(  b\right)  }{2}\int\nolimits_{a}%
^{b}g\left(  s\right)  ds-H_{g}\left(  t\right) \\
&  \leq\frac{\left(  f^{\prime}\left(  b\right)  -f^{\prime}\left(  a\right)
\right)  \left(  b-a\right)  }{4}\int\nolimits_{a}^{b}g\left(  s\right)  ds.
\end{align*}

\end{enumerate}

\begin{theoremA}
\label{A24}\cite{15b} \textit{Let }$f,g,G,H,Hg$\textit{\ be defined as above.
Then} \textit{we have the following Fej\'{e}r-type inequalities:}
\end{theoremA}

\begin{enumerate}
\item \textit{The following inequality holds for all }$t\in\left[  0,1\right]
$:%
\begin{equation}
Hg\left(  t\right)  \leq G\left(  t\right)  \int\nolimits_{a}^{b}g\left(
s\right)  ds. \label{1.14}%
\end{equation}

\item \textit{The following inequality holds}:%
\begin{align}
&  2\int_{\frac{3a+b}{4}}^{\frac{a+3b}{4}}f\left(  s\right)  g\left(
2s-\frac{a+b}{2}\right)  ds\label{1.15}\\
&  \leq\frac{1}{2}\left[  f\left(  \frac{3a+b}{4}\right)  +f\left(
\frac{a+3b}{4}\right)  \right]  \int\nolimits_{a}^{b}g\left(  s\right)
ds\nonumber\\
&  \leq\left(  b-a\right)  \int_{0}^{1}G\left(  t\right)  g\left(  \left(
1-t\right)  a+tb\right)  dt\nonumber\\
&  \leq\frac{1}{2}\left[  f\left(  \frac{a+b}{2}\right)  +\frac{f\left(
a\right)  +f\left(  b\right)  }{2}\right]  \int\nolimits_{a}^{b}g\left(
s\right)  ds.\nonumber
\end{align}

\item \textit{If }$f$ \textit{is differentiable on }$\left[  a,b\right]  $
\textit{and }$g$\textit{\ is bounded on }$\left[  a,b\right]  ,$%
\textit{\ then, for all }$t\in\left[  0,1\right]  ,$ \textit{we have the
inequality}
\begin{equation}
0\leq Hg\left(  t\right)  -f\left(  \frac{a+b}{2}\right)  \int\nolimits_{a}%
^{b}g\left(  s\right)  ds\leq\left(  b-a\right)  \left[  G\left(  t\right)
-H\left(  t\right)  \right]  \left\Vert g\right\Vert _{\infty} \label{1.16}%
\end{equation}
\textit{where }$\left\Vert g\right\Vert _{\infty}=\sup\limits_{s\in\left[
a,b\right]  }\left\vert g\left(  s\right)  \right\vert .$
\end{enumerate}

\begin{theoremA}
\label{A25}\cite{15b} \textit{Let }$f,g,G,Hg,Lg$\textit{\ be defined as above.
Then} \textit{we have the following results:}
\end{theoremA}

\begin{enumerate}
\item $Lg$\textit{\ is convex on }$\left[  0,1\right]  .$

\item \textit{The following inequalities hold for all }$t\in\left[
0,1\right]  $:%
\begin{align*}
&  G\left(  t\right)  \int\nolimits_{a}^{b}g\left(  s\right)  ds\\
&  \leq L_{g}\left(  t\right) \\
&  \leq\left(  1-t\right)  \int\nolimits_{a}^{b}f\left(  s\right)  g\left(
s\right)  ds+t\cdot\frac{f\left(  a\right)  +f\left(  b\right)  }{2}%
\int\nolimits_{a}^{b}g\left(  s\right)  ds\\
&  \leq\frac{f\left(  a\right)  +f\left(  b\right)  }{2}\int\nolimits_{a}%
^{b}g\left(  s\right)  ds;
\end{align*}%
\[
Hg\left(  1-t\right)  \leq Lg\left(  t\right)  ;
\]%
\[
\frac{Hg\left(  t\right)  +Hg\left(  1-t\right)  }{2}\leq Lg\left(  t\right)
.
\]

\item \textit{The following inequality holds:}%
\[
\sup\limits_{t\in\left[  0,1\right]  }Lg\left(  t\right)  =\frac{f\left(
a\right)  +f\left(  b\right)  }{2}\int\nolimits_{a}^{b}g\left(  s\right)  ds.
\]

\end{enumerate}

Tseng \textit{et al.} established the following six theorems related to
$Q,G,Hg,Pg,Lg,Ig,Ng,Sg$ and which are weighted generalizations of Theorems
\ref{A3} and \ref{A5} (see \cite{14b}).

\begin{theoremA}
\label{A26}\cite{14b} \textit{Let }$f,g,Ig$\textit{\ be defined as above. Then
}$Ig$\textit{\ is convex, increasing on }$\left[  0,1\right]  ,$\textit{\ and
for all }$t\in\left[  0,1\right]  $\textit{, we have the following
Fej\'{e}r-type inequality}%
\begin{align*}
f\left(  \frac{a+b}{2}\right)  \int\nolimits_{a}^{b}g\left(  s\right)  ds  &
=Ig\left(  0\right)  \leq Ig\left(  t\right)  \leq Ig\left(  1\right) \\
&  =%
{\displaystyle\int\nolimits_{a}^{b}}
\frac{1}{2}\left[  f\left(  \frac{s+a}{2}\right)  +f\left(  \frac{s+b}%
{2}\right)  \right]  g\left(  s\right)  ds.
\end{align*}

\end{theoremA}

\begin{theoremA}
\label{A27}\cite{14b} \textit{Let }$f,g,Jg$\textit{\ be defined as above. Then
}$Jg$\textit{\ is convex, increasing on }$\left[  0,1\right]  ,$\textit{\ and
for all }$t\in\left[  0,1\right]  $\textit{, we have the following
Fej\'{e}r-type inequality}%
\begin{multline*}
\frac{f\left(  \frac{3a+b}{4}\right)  +f\left(  \frac{a+3b}{4}\right)  }%
{2}\int\nolimits_{a}^{b}g\left(  s\right)  ds=Jg\left(  0\right)  \leq
Jg\left(  t\right)  \leq Jg\left(  1\right) \\
=\frac{1}{2}%
{\displaystyle\int\nolimits_{a}^{b}}
\left[  f\left(  \frac{s+a}{2}\right)  +f\left(  \frac{s+b}{2}\right)
\right]  g\left(  s\right)  ds.
\end{multline*}

\end{theoremA}

\begin{theoremA}
\label{A28}\cite{14b} \textit{Let }$f,g,Ig,Jg$\textit{\ be defined as above.
Then }$Ig\left(  t\right)  \leq Jg\left(  t\right)  $\textit{\ on }$\left[
0,1\right]  .$
\end{theoremA}

\begin{theoremA}
\label{A29}\cite{14b} \textit{Let }$f,g,Mg$\textit{\ be defined as above. Then
}$Mg$\textit{\ is convex, increasing on }$\left[  0,1\right]  ,$\textit{\ and
for all }$t\in\left[  0,1\right]  $\textit{, we have the following
Fej\'{e}r-type inequality}%
\begin{multline*}%
{\displaystyle\int\nolimits_{a}^{b}}
\frac{1}{2}\left[  f\left(  \frac{s+a}{2}\right)  +f\left(  \frac{s+b}%
{2}\right)  \right]  g\left(  s\right)  ds\\
=Mg\left(  0\right)  \leq Mg\left(  t\right)  \leq Mg\left(  1\right)
=\frac{1}{2}\left[  f\left(  \frac{a+b}{2}\right)  +\frac{f\left(  a\right)
+f\left(  b\right)  }{2}\right]  \int\nolimits_{a}^{b}g\left(  s\right)  ds.
\end{multline*}

\end{theoremA}

\begin{theoremA}
\label{A30}\cite{14b} \textit{Let }$f,g,Ng$\textit{\ be defined as above. Then
}$Ng$\textit{\ is convex, increasing on }$\left[  0,1\right]  ,$\textit{\ and
for all }$t\in\left[  0,1\right]  $\textit{, we have the following
Fej\'{e}r-type inequality}%
\begin{multline*}%
{\displaystyle\int\nolimits_{a}^{b}}
\frac{1}{2}\left[  f\left(  \frac{s+a}{2}\right)  +f\left(  \frac{s+b}%
{2}\right)  \right]  g\left(  s\right)  ds\\
=Ng\left(  0\right)  \leq Ng\left(  t\right)  \leq Ng\left(  1\right)
=\frac{f\left(  a\right)  +f\left(  b\right)  }{2}\int\nolimits_{a}%
^{b}g\left(  s\right)  ds.
\end{multline*}

\end{theoremA}

\begin{theoremA}
\label{A31}\cite{14b} \textit{Let }$f,g,Mg,Ng$\textit{\ be defined as above.
Then }$Mg\left(  t\right)  \leq Ng\left(  t\right)  $\textit{\ on }$\left[
0,1\right]  .$
\end{theoremA}

Tseng \textit{et al.} established the following three theorems related to
$Q,G,H,Ig,Sg$ and which are weighted generalizations of Theorems \ref{A6}
--\ \ref{A8} (see \cite{16b}).

\begin{theoremA}
\label{A32}\cite{16b} Let $f,g,Ig$\ be defined as above. Then:
\end{theoremA}

\begin{enumerate}
\item \textit{The following inequality holds:}%
\begin{align*}
&  f\left(  \frac{a+b}{2}\right)  \int\nolimits_{a}^{b}g\left(  s\right)  ds\\
&  \leq2\left[  \int_{\frac{3a+b}{4}}^{\frac{a+b}{2}}f\left(  s\right)
g\left(  4s-2a-b\right)  ds+\int_{\frac{a+b}{2}}^{\frac{a+3b}{4}}f\left(
s\right)  g\left(  4s-a-2b\right)  ds\right] \\
&  \leq\int\nolimits_{0}^{1}Ig\left(  t\right)  dt\\
&  \leq\frac{1}{2}\left[  f\left(  \frac{a+b}{2}\right)  \int\nolimits_{a}%
^{b}g\left(  s\right)  ds\right. \\
&  \qquad\qquad+\left.
{\displaystyle\int\nolimits_{a}^{b}}
\frac{1}{2}\left[  f\left(  \frac{s+a}{2}\right)  +f\left(  \frac{s+b}%
{2}\right)  \right]  g\left(  s\right)  ds\right]  .
\end{align*}

\item \textit{If }$f$\textit{\ is differentiable on }$\left[  a,b\right]
$\textit{\ and }$g$\textit{\ is bounded on }$\left[  a,b\right]  $\textit{,
then, for all }$t\in\left[  0,1\right]  ,$\textit{\ we have the inequality}%
\begin{align*}
0  &  \leq%
{\displaystyle\int\nolimits_{a}^{b}}
\frac{1}{2}\left[  f\left(  \frac{s+a}{2}\right)  +f\left(  \frac{s+b}%
{2}\right)  \right]  g\left(  s\right)  ds-Ig\left(  t\right) \\
&  \leq\left(  1-t\right)  \left[  \frac{f\left(  a\right)  +f\left(
b\right)  }{2}\left(  b-a\right)  -\int\nolimits_{a}^{b}f\left(  s\right)
ds\right]  \left\Vert g\right\Vert _{\infty},
\end{align*}
\textit{where }$\left\Vert g\right\Vert _{\infty}=\sup\limits_{s\in\left[
a,b\right]  }\left\vert g\left(  s\right)  \right\vert .$

\item \textit{If }$f$\textit{\ is differentiable on }$\left[  a,b\right]
$\textit{, then, for all }$t\in\left[  0,1\right]  ,$\textit{\ we have the
inequality }%
\begin{align*}
0  &  \leq\frac{f\left(  a\right)  +f\left(  b\right)  }{2}\int\nolimits_{a}%
^{b}g\left(  s\right)  ds-Ig\left(  t\right) \\
&  \leq\frac{\left(  f^{\prime}\left(  b\right)  -f^{\prime}\left(  a\right)
\right)  \left(  b-a\right)  }{4}\int\nolimits_{a}^{b}g\left(  s\right)  ds.
\end{align*}

\end{enumerate}

\begin{theoremA}
\label{A33}\cite{16b} \textit{Let }$f,g,G,Ig$\textit{\ be defined as above.
Then:}
\end{theoremA}

\begin{enumerate}
\item \textit{The following inequality holds for all }$t\in\left[  0,1\right]
$:%
\begin{equation}
Ig\left(  t\right)  \leq G\left(  t\right)  \int\nolimits_{a}^{b}g\left(
s\right)  ds. \label{1.17}%
\end{equation}

\item \textit{If }$f$ \textit{is differentiable on }$\left[  a,b\right]  $
\textit{and }$g$\textit{\ is bounded on }$\left[  a,b\right]  ,$%
\textit{\ then, for all }$t\in\left[  0,1\right]  ,$ \textit{we have the
inequality}
\begin{equation}
0\leq Ig\left(  t\right)  -f\left(  \frac{a+b}{2}\right)  \int\nolimits_{a}%
^{b}g\left(  s\right)  ds\leq\left(  b-a\right)  \left[  G\left(  t\right)
-H\left(  t\right)  \right]  \left\Vert g\right\Vert _{\infty} \label{1.18}%
\end{equation}
\textit{where }$\left\Vert g\right\Vert _{\infty}=\sup\limits_{s\in\left[
a,b\right]  }\left\vert g\left(  s\right)  \right\vert .$
\end{enumerate}

\begin{theoremA}
\label{A34}\cite{16b} \textit{Let }$f,g,G,Ig,Sg$\textit{\ be defined as above.
Then} \textit{we have the following results:}
\end{theoremA}

\begin{enumerate}
\item $Sg$\textit{\ is convex on }$\left[  0,1\right]  .$

\item \textit{The following inequalities hold for all }$t\in\left[
0,1\right]  $:
\begin{align*}
G\left(  t\right)  \int\nolimits_{a}^{b}g\left(  x\right)  dx  &  \leq
Sg\left(  t\right) \\
&  \leq\left(  1-t\right)
{\displaystyle\int\nolimits_{a}^{b}}
\frac{1}{2}\left[  f\left(  \frac{s+a}{2}\right)  +f\left(  \frac{s+b}%
{2}\right)  \right]  g\left(  s\right)  ds\\
&  \qquad\qquad+t\cdot\frac{f\left(  a\right)  +f\left(  b\right)  }{2}%
\int\nolimits_{a}^{b}g\left(  s\right)  ds\\
&  \leq\frac{f\left(  a\right)  +f\left(  b\right)  }{2}\int\nolimits_{a}%
^{b}g\left(  s\right)  ds;
\end{align*}%
\[
Ig\left(  1-t\right)  \leq Sg\left(  t\right)  ;
\]%
\[
\frac{Ig\left(  t\right)  +Ig\left(  1-t\right)  }{2}\leq Sg\left(  t\right)
.
\]

\item \textit{The following inequality holds:}%
\[
\sup\limits_{t\in\left[  0,1\right]  }Sg\left(  t\right)  =\frac{f\left(
a\right)  +f\left(  b\right)  }{2}\int\nolimits_{a}^{b}g\left(  s\right)  ds.
\]

\end{enumerate}

Tseng \textit{et al.} established the following two theorems related to
$Q,G,Hg,Pg,Lg,Ig,Ng,Sg$ and which are weighted generalizations of Theorem
\ref{A9} (see \cite{17b}).

\begin{theoremA}
\label{A35}\cite{17b} Let $f,g,G,Hg,Pg,Lg,Ig,Ng$\ be defined as above. Then:
\end{theoremA}

\begin{enumerate}
\item \textit{The inequalities }%
\begin{align*}
&  \int\nolimits_{a}^{b}f\left(  s\right)  g\left(  s\right)  ds\\
&  \leq2\left[  \int\nolimits_{a}^{\frac{3a+b}{4}}f\left(  s\right)  g\left(
2s-a\right)  ds+\int_{\frac{a+3b}{4}}^{b}f\left(  s\right)  g\left(
2s-b\right)  ds\right] \\
&  \leq\int\nolimits_{0}^{1}Pg\left(  t\right)  dt\\
&  \leq\frac{1}{2}\left[  \int\nolimits_{a}^{b}f\left(  s\right)  g\left(
s\right)  ds+\frac{f\left(  a\right)  +f\left(  b\right)  }{2}\int%
\nolimits_{a}^{b}g\left(  s\right)  ds\right]
\end{align*}
\textit{holds.}

\item \textit{The inequalities}%
\begin{align*}
Lg\left(  t\right)   &  \leq Pg\left(  t\right) \\
&  \leq\left(  1-t\right)  \int\nolimits_{a}^{b}f\left(  s\right)  g\left(
s\right)  ds+t\cdot\frac{f\left(  a\right)  +f\left(  b\right)  }{2}%
\int\nolimits_{a}^{b}g\left(  s\right)  ds\\
&  \leq\frac{f\left(  a\right)  +f\left(  b\right)  }{2}\int\nolimits_{a}%
^{b}g\left(  s\right)  ds
\end{align*}
\textit{and}%
\[
0\leq Ng\left(  t\right)  -G\left(  t\right)  \int\nolimits_{a}^{b}g\left(
s\right)  ds\leq\frac{f\left(  a\right)  +f\left(  b\right)  }{2}%
\int\nolimits_{a}^{b}g\left(  s\right)  ds-Ng\left(  t\right)
\]
\textit{hold for all }$t\in\left[  0,1\right]  .$

\item \textit{If }$f$\textit{\ is differentiable on }$\left[  a,b\right]
,$\textit{\ then we have the inequalities}
\begin{align*}
0  &  \leq t\left[  \frac{1}{b-a}\int\nolimits_{a}^{b}f\left(  s\right)
ds-f\left(  \frac{a+b}{2}\right)  \right]  \cdot\inf\limits_{s\in\left[
a,b\right]  }g\left(  s\right) \\
&  \leq Pg\left(  t\right)  -\int\nolimits_{a}^{b}f\left(  s\right)  g\left(
s\right)  ds;
\end{align*}%
\begin{align*}
0  &  \leq Pg\left(  t\right)  -f\left(  \frac{a+b}{2}\right)  \int%
\nolimits_{a}^{b}g\left(  s\right)  ds\\
&  \leq\frac{\left(  f^{\prime}\left(  b\right)  -f^{\prime}\left(  a\right)
\right)  \left(  b-a\right)  }{4}\int\nolimits_{a}^{b}g\left(  s\right)  ds;
\end{align*}%
\[
0\leq Lg\left(  t\right)  -Hg\left(  t\right)  \leq\frac{\left(  f^{\prime
}\left(  b\right)  -f^{\prime}\left(  a\right)  \right)  \left(  b-a\right)
}{4}\int\nolimits_{a}^{b}g\left(  s\right)  ds;
\]%
\[
0\leq Pg\left(  t\right)  -Lg\left(  t\right)  \leq\frac{\left(  f^{\prime
}\left(  b\right)  -f^{\prime}\left(  a\right)  \right)  \left(  b-a\right)
}{4}\int\nolimits_{a}^{b}g\left(  s\right)  ds;
\]%
\[
0\leq Pg\left(  t\right)  -Hg\left(  t\right)  \leq\frac{\left(  f^{\prime
}\left(  b\right)  -f^{\prime}\left(  a\right)  \right)  \left(  b-a\right)
}{4}\int\nolimits_{a}^{b}g\left(  s\right)  ds;
\]%
\[
0\leq Ng\left(  t\right)  -Ig\left(  t\right)  \leq\frac{\left(  f^{\prime
}\left(  b\right)  -f^{\prime}\left(  a\right)  \right)  \left(  b-a\right)
}{4}\int\nolimits_{a}^{b}g\left(  s\right)  ds
\]
\textit{and}%
\[
0\leq Sg\left(  t\right)  -Ig\left(  t\right)  \leq\frac{\left(  f^{\prime
}\left(  b\right)  -f^{\prime}\left(  a\right)  \right)  \left(  b-a\right)
}{4}\int\nolimits_{a}^{b}g\left(  s\right)  ds
\]
\textit{for all }$t\in\left[  0,1\right]  .$
\end{enumerate}

\begin{theoremA}
\label{A36}\cite{17b} Let $f,g,G,Q,Hg,Pg,Sg$\ be defined as above. Then:
\end{theoremA}

\begin{enumerate}
\item \textit{The inequalities}%
\begin{align*}
Hg\left(  t\right)   &  \leq Q\left(  t\right)  \int\nolimits_{a}^{b}g\left(
s\right)  ds\\
&  \leq\frac{f\left(  a\right)  +f\left(  b\right)  }{2}\text{ }%
\int\nolimits_{a}^{b}g\left(  s\right)  ds\qquad\left(  t\in\left[  0,\frac
{1}{3}\right]  \right)
\end{align*}
\textit{and}%
\begin{align*}
f\left(  \frac{a+b}{2}\right)  \int\nolimits_{a}^{b}g\left(  s\right)  dx  &
\leq Q\left(  t\right)  \int\nolimits_{a}^{b}g\left(  s\right)  ds\\
&  \leq Pg\left(  t\right)  \text{ \qquad}\left(  t\in\left[  \frac{1}%
{3},1\right]  \right)
\end{align*}
\textit{hold for all }$t\in\left[  0,1\right]  .$

\item \textit{The inequality }%
\begin{align*}
0  &  \leq Sg\left(  t\right)  -G\left(  t\right)  \int\nolimits_{a}%
^{b}g\left(  s\right)  ds\\
&  \leq\frac{1}{2}\left[  \frac{f\left(  a\right)  +f\left(  b\right)  }%
{2}+Q\left(  t\right)  \right]  \int\nolimits_{a}^{b}g\left(  s\right)
ds-Sg\left(  t\right)
\end{align*}
holds for all $t\in\left[  0,1\right]  .$
\end{enumerate}

Tseng \textit{et al.} established the following two theorems related to
$Ig,Kg,Sg$ and which are weighted generalizations of Theorems \ref{A4}
and\ \ref{A11} (see \cite{18b}).

\begin{theoremA}
\label{A37}\cite{18b} Let $f,g,Ig,Kg$\ be defined as above. Then:
\end{theoremA}

\begin{enumerate}
\item $Kg$\textit{\ is convex on }$\left[  0,1\right]  $\textit{\ and
symmetric about }$\frac{1}{2}.$

\item $Kg$\textit{\ is decreasing on }$\left[  0,\frac{1}{2}\right]
$\textit{\ and increasing on }$\left[  \frac{1}{2},1\right]  ,$%
\begin{align*}
\sup\limits_{t\in\left[  0,1\right]  }Kg\left(  t\right)   &  =Kg\left(
0\right)  =Kg\left(  1\right) \\
&  =%
{\displaystyle\int\nolimits_{a}^{b}}
\frac{1}{2}\left[  f\left(  \frac{s+a}{2}\right)  +f\left(  \frac{s+b}%
{2}\right)  \right]  g\left(  s\right)  ds\cdot\int\nolimits_{a}^{b}g\left(
s\right)  ds
\end{align*}
\textit{and}%
\begin{align*}
\inf\limits_{t\in\left[  0,1\right]  }Kg\left(  t\right)   &  =Kg\left(
\frac{1}{2}\right) \\
&  =\int_{a}^{b}\int_{a}^{b}\frac{1}{4}\left[  f\left(  \frac{s+u+2a}%
{4}\right)  +2f\left(  \frac{s+u+a+b}{4}\right)  \right. \\
&  \qquad\qquad+\left.  f\left(  \frac{s+u+2b}{4}\right)  \right]  g\left(
s\right)  g\left(  u\right)  dsdu.
\end{align*}

\item \textit{We have}%
\[
Ig\left(  t\right)  \int\nolimits_{a}^{b}g\left(  s\right)  ds\leq Kg\left(
t\right)
\]
\textit{and}%
\[
f\left(  \frac{a+b}{2}\right)  \left(  \int\nolimits_{a}^{b}g\left(  s\right)
ds\right)  ^{2}\leq Kg\left(  \frac{1}{2}\right)
\]
\textit{for all }$t\in\left[  0,1\right]  .$
\end{enumerate}

\begin{theoremA}
\label{A38}\cite{18b} Let $f,g,Ig,Kg,Sg$ be defined as above. Then we have the
inequality%
\[
0\leq Kg\left(  t\right)  -Ig\left(  t\right)  \int\nolimits_{a}^{b}g\left(
s\right)  ds\leq Sg\left(  1-t\right)  \int\nolimits_{a}^{b}g\left(  s\right)
ds-Kg\left(  t\right)
\]
for all $t\in\left[  0,1\right]  .$
\end{theoremA}

\section{Main Results}

In this section, we shall establish a number of new Fej\'{e}r-type
inequalities which reduce all results of Subsections $1.1-1.3.$

\subsection{New Fejer-type Inequalities I}

Throughout this subsection, let $x,y,y^{\prime},x^{\prime},$ $\Omega,f,g,$
$H,H_{1},H_{2},Hg,$ $P,P_{1},Pg,$ $F,F_{1},Fg,$ $L,L_{1},Lg,$ $G,G_{1},G_{2}%
$\ be defined as above, the function $p:\left[  x,x^{\prime}\right]
\rightarrow\left[  0,\infty\right)  $ be integrable and symmertic to
$\frac{x+x^{\prime}}{2}$ and define the following functions on $[0,1].$
\[
Hp_{1}\left(  t\right)  =\int_{x}^{y}\left[  f\left(  ts+\left(  1-t\right)
y\right)  +f\left(  t\left(  y+y^{\prime}-s\right)  +\left(  1-t\right)
y^{\prime}\right)  \right]  p\left(  s\right)  ds;
\]%
\[
Hp_{2}\left(  t\right)  =\int_{x}^{y}\left[  f\left(  ts+\left(  1-t\right)
y^{\prime}\right)  +f\left(  t\left(  y+y^{\prime}-s\right)  +\left(
1-t\right)  y\right)  \right]  p\left(  s\right)  ds;
\]%
\[
Fp_{1}\left(  t\right)  =\int_{\Omega}\int_{\Omega}f\left(  ts+\left(
1-t\right)  u\right)  p\left(  s\right)  p\left(  u\right)  dsdu;
\]%
\[
Pp_{1}\left(  t\right)  =\int_{x}^{y}\left[  f\left(  tx+\left(  1-t\right)
s\right)  +f\left(  tx^{\prime}+\left(  1-t\right)  \left(  x+x^{\prime
}-s\right)  \right)  \right]  p\left(  s\right)  ds
\]
and%
\[
Lp_{1}\left(  t\right)  =\frac{1}{2}%
{\displaystyle\int\nolimits_{\Omega}}
\left[  f\left(  tx+\left(  1-t\right)  s\right)  +f\left(  tx^{\prime
}+\left(  1-t\right)  s\right)  \right]  p\left(  s\right)  ds.
\]

\begin{remark}
\label{r3}\textit{We note that :}
\end{remark}

\begin{enumerate}
\item $\int_{x}^{y}p\left(  s\right)  ds=\frac{1}{2}\int_{\Omega}p\left(
s\right)  ds.$

\item $Hp_{1}\left(  t\right)  =H_{1}\left(  t\right)  ,$ $Hp_{2}\left(
t\right)  =H_{2}\left(  t\right)  ,$ $Fp_{1}\left(  t\right)  =F_{1}\left(
t\right)  ,$ $Pp_{1}\left(  t\right)  =P_{1}\left(  t\right)  ,$
$Lp_{1}\left(  t\right)  =L_{1}\left(  t\right)  $\textit{\ on }$\left[
0,1\right]  $\textit{\ as }$p\left(  s\right)  \equiv\frac{1}{2\left(
y-x\right)  }$ \ $\left(  s\in\left[  x,x^{\prime}\right]  \right)  .$

\item $\Omega=\left[  a,b\right]  $\textit{\ and }$Hp_{1}\left(  t\right)
=Hp_{2}\left(  t\right)  =H\left(  t\right)  ,$ $Fp_{1}\left(  t\right)
=F\left(  t\right)  ,$ $Pp_{1}\left(  t\right)  =P\left(  t\right)  ,$
$Lp_{1}\left(  t\right)  =L\left(  t\right)  $\textit{\ on }$\left[
0,1\right]  $\textit{\ as }$x=a,$ $y=y^{\prime}=\frac{a+b}{2},x^{\prime}=b$
\textit{and} $p\left(  s\right)  \equiv\frac{1}{b-a}$ $\left(  s\in\left[
a,b\right]  \right)  .$

\item $\Omega=\left[  a,b\right]  $\textit{\ and }$Hp_{1}\left(  t\right)
=Hp_{2}\left(  t\right)  =Hg\left(  t\right)  ,$ $Fp_{1}\left(  t\right)
=Fg\left(  t\right)  ,$ $Pp_{1}\left(  t\right)  =Pg\left(  t\right)  ,$
$Lp_{1}\left(  t\right)  =Lg\left(  t\right)  $\textit{\ on }$\left[
0,1\right]  $\textit{\ as }$x=a,$ $y=y^{\prime}=\frac{a+b}{2},x^{\prime}=b$
\textit{and }$p\left(  s\right)  =g\left(  s\right)  $ $\left(  s\in\left[
a,b\right]  \right)  .$
\end{enumerate}

To prove the main results of this section, we need the following lemmas:

\begin{lemma}
[see \cite{9b}]\label{l1}Let $f$\ be defined as above and let $a\leq A\leq
C\leq D\leq B\leq b$ with $A+B=C+D.$ Then
\[
f\left(  C\right)  +f\left(  D\right)  \leq f\left(  A\right)  +f\left(
B\right)  .
\]

\end{lemma}

The assumptions in Lemma \ref{l1} can be weakened as in the following lemma:

\begin{lemma}
\label{l2}Let $f,p$\ be defined as above and let $E\in\left[  x,x^{\prime
}\right]  $ and $a\leq A\leq C\leq B\leq b$ with $A+B=C+D.$ Then
\[
\left[  f\left(  C\right)  +f\left(  D\right)  \right]  p\left(  E\right)
\leq\left[  f\left(  A\right)  +f\left(  B\right)  \right]  p\left(  E\right)
.
\]

\end{lemma}

Now, we are ready to state and prove the main results of this section.

\begin{theorem}
\label{t1}Let $x,y,y^{\prime},x^{\prime},\Omega,f,p,Hp_{1},Hp_{2}$\ be defined
as above. Then:

\begin{enumerate}
\item $Hp_{1}$ and $Hp_{2}$\ are convex on $\left[  0,1\right]  .$

\item $Hp_{1}$\ is increasing on $\left[  0,1\right]  $\ and the following
inequalities
\begin{equation}
\frac{f\left(  y\right)  +f\left(  y^{\prime}\right)  }{2}\int_{\Omega
}p\left(  s\right)  ds=Hp_{1}\left(  0\right)  \leq Hp_{1}\left(  t\right)
\leq Hp_{1}\left(  1\right)  =\int_{\Omega}f\left(  s\right)  p\left(
s\right)  ds\ , \label{2.1}%
\end{equation}%
\begin{align}
Hp_{1}\left(  t\right)   &  \leq t\int_{\Omega}f\left(  s\right)  p\left(
s\right)  ds+\left(  1-t\right)  \cdot\frac{f\left(  y\right)  +f\left(
y^{\prime}\right)  }{2}\int_{\Omega}p\left(  s\right)  ds\label{2.2}\\
&  \leq\int_{\Omega}f\left(  s\right)  p\left(  s\right)  ds\leq\frac{f\left(
x\right)  +f\left(  x^{\prime}\right)  }{2}\int_{\Omega}p\left(  s\right)
ds,\nonumber
\end{align}%
\begin{align}
&  f\left(  \frac{y+y^{\prime}}{2}\right)  \int_{\Omega}p\left(  s\right)
ds\label{2.3}\\
&  \leq Hp_{2}\left(  t\right) \nonumber\\
&  \leq t\int_{\Omega}f\left(  s\right)  p\left(  s\right)  ds+\left(
1-t\right)  \cdot\frac{f\left(  y\right)  +f\left(  y^{\prime}\right)  }%
{2}\int_{\Omega}p\left(  s\right)  ds\nonumber\\
&  \leq\int_{\Omega}f\left(  s\right)  p\left(  s\right)  ds\nonumber
\end{align}
and%
\begin{equation}
Hp_{2}\left(  t\right)  \leq Hp_{1}\left(  t\right)  \label{2.4}%
\end{equation}
hold for all $t\in\left[  0,1\right]  .$
\end{enumerate}
\end{theorem}

\begin{proof}
$\left(  1\right)  $ It is easily observed from the convexity of $f$ and the
hypothesis of $p$ that $Hp_{1}$ and $Hp_{2}$ are convex on $\left[
0,1\right]  .$

$\left(  2\right)  $ Let $t_{1}<t_{2}$ in $\left[  0,1\right]  .$ By Lemma
\ref{l2}, the following inequality holds for all $s\in\left[  x,y\right]  :$%
\begin{multline*}
\left[  f\left(  t_{1}s+\left(  1-t_{1}\right)  y\right)  +f\left(
t_{1}\left(  y+y^{\prime}-s\right)  +\left(  1-t_{1}\right)  y^{\prime
}\right)  \right]  p\left(  s\right) \\
\leq\left[  f\left(  t_{2}s+\left(  1-t_{2}\right)  y\right)  +f\left(
t_{2}\left(  y+y^{\prime}-s\right)  +\left(  1-t_{2}\right)  y^{\prime
}\right)  \right]  p\left(  s\right)  .
\end{multline*}

Integrating the above inequality over $s$ on $\left[  x,y\right]  $ and using
the definition of $Hp_{1},$ we have
\[
Hp_{1}\left(  t_{1}\right)  \leq Hp_{1}\left(  t_{2}\right)  .
\]

Thus, $Hp_{1}$ is increasing on $\left[  0,1\right]  $ and $\left(
\ref{2.1}\right)  $ holds. Using the convexity of $f,$ the inequality $\left(
\ref{2.1}\right)  $ and the substitution rule for integration, we obtain the
first and second inequalities of $\left(  \ref{2.2}\right)  $\ and the
inequality $\left(  \ref{2.3}\right)  $. Using simple techniques of
integration and the hypothesis of $p$, we have the following identity%
\[
\int_{x}^{y}f\left(  s\right)  p\left(  s\right)  ds=\int_{x}^{y}f\left(
y+y^{\prime}-s\right)  p\left(  s\right)  ds=\frac{1}{2}\int_{\Omega}f\left(
s\right)  p\left(  s\right)  ds.
\]

By Lemma \ref{l2}, the inequality%
\[
\left[  f\left(  s\right)  +f\left(  y+y^{\prime}-s\right)  \right]  p\left(
s\right)  \leq\left[  f\left(  x\right)  +f\left(  x^{\prime}\right)  \right]
p\left(  s\right)
\]
holds for all $s\in\left[  x,y\right]  $. Integrating the above inequality
over $s$ on $\left[  x,y\right]  $ and using the above identity, we derive the
last inequality of $\left(  \ref{2.2}\right)  .$

Again, using Lemma \ref{l2}, the inequality%
\begin{multline*}
\left[  f\left(  ts+\left(  1-t\right)  y^{\prime}\right)  +f\left(  t\left(
y+y^{\prime}-s\right)  +\left(  1-t\right)  y\right)  \right]  p\left(
s\right) \\
\leq\left[  f\left(  ts+\left(  1-t\right)  y\right)  +f\left(  t\left(
y+y^{\prime}-s\right)  +\left(  1-t\right)  y^{\prime}\right)  \right]
p\left(  s\right)
\end{multline*}
holds for all $t\in\left[  0,1\right]  $ and $s\in\left[  x,y\right]  $.
Integrating the above inequality over $s$ on $\left[  x,y\right]  $ and using
the definitions of $Hp_{1}$ and $Hp_{2}$, we derive $\left(  \ref{2.4}\right)
.$

This completes the proof.
\end{proof}

\begin{remark}
\label{r4}\textit{Using Remark \ref{r3}, }we have the following results:
\end{remark}

\begin{enumerate}
\item \textit{Theorem \ref{t1}} \textit{reduces to Theorem \ref{A3} as }$x=a,
$\textit{\ }$y=y^{\prime}=\frac{a+b}{2},$\textit{\ }$x^{\prime}=b$ \textit{and
}$p\left(  s\right)  \equiv\frac{1}{b-a}$ $\left(  s\in\left[  a,b\right]
\right)  .$

\item \textit{Theorem \ref{t1}} \textit{reduces to Theorems \ref{A12} and
\ref{A13} as }$p\left(  s\right)  \equiv\frac{1}{2\left(  y-x\right)  }%
$\ $\left(  s\in\Omega\right)  .$

\item \textit{Theorem \ref{t1}} \textit{reduces to Theorem \ref{A20} as
}$x=a,$ $y=y^{\prime}=\frac{a+b}{2},x^{\prime}=b$ \textit{and }$p\left(
s\right)  =g\left(  s\right)  $\ $\left(  s\in\left[  a,b\right]  \right)  .$
\end{enumerate}

\begin{theorem}
\label{t2}Let $x,y,y^{\prime},x^{\prime},\Omega,f,p,Hp_{1},Hp_{2}$\ be defined
as above. Then:

\begin{enumerate}
\item The inequalities
\begin{align}
&  \frac{f\left(  y\right)  +f\left(  y^{\prime}\right)  }{2}\int_{\Omega
}p\left(  s\right)  ds\label{2.5}\\
&  \leq\int_{\frac{x+y}{2}}^{y}2f\left(  s\right)  p\left(  2s-y\right)
ds+\int_{y^{\prime}}^{\frac{x^{\prime}+y^{\prime}}{2}}2f\left(  s\right)
p\left(  2s-y^{\prime}\right)  ds\nonumber\\
&  \leq\int_{0}^{1}Hp_{1}\left(  t\right)  dt\nonumber\\
&  \leq\frac{1}{2}\left[  \frac{f\left(  y\right)  +f\left(  y^{\prime
}\right)  }{2}\int_{\Omega}p\left(  s\right)  ds+\int_{\Omega}f\left(
s\right)  p\left(  s\right)  ds\right] \nonumber
\end{align}
and%
\begin{align}
&  f\left(  \frac{y+y^{\prime}}{2}\right)  \int_{\Omega}p\left(  s\right)
ds\label{2.6}\\
&  \leq\int_{\frac{x+y^{\prime}}{2}}^{\frac{y+y^{\prime}}{2}}2f\left(
s\right)  p\left(  2s-y^{\prime}\right)  ds+\int_{\frac{y+y^{\prime}}{2}%
}^{\frac{y+x^{\prime}}{2}}2f\left(  s\right)  p\left(  2s-y\right)
ds\nonumber\\
&  \leq\int_{0}^{1}Hp_{2}\left(  t\right)  dt\nonumber\\
&  \leq\frac{1}{2}\left[  \frac{f\left(  y\right)  +f\left(  y^{\prime
}\right)  }{2}\int_{\Omega}p\left(  s\right)  ds+\int_{\Omega}f\left(
s\right)  p\left(  s\right)  ds\right] \nonumber
\end{align}
hold.

\item If $f$\ is differentiable on $\Omega$ and $p$ is bounded on $\Omega,$
then the inequality
\begin{align}
0  &  \leq\int_{\Omega}f\left(  s\right)  p\left(  s\right)  ds-Hp_{1}\left(
t\right) \label{2.7}\\
&  \leq\left(  1-t\right)  \left[  \left(  f\left(  x\right)  +f\left(
x^{\prime}\right)  \right)  \left(  y-x\right)  -\int_{\Omega}f\left(
s\right)  ds\right]  \sup\limits_{s\in\Omega}p\left(  s\right) \nonumber
\end{align}
holds for all $t\in\left[  0,1\right]  .$

\item If $f$\ is differentiable on $\Omega,$ then the inequalities%
\begin{align}
0  &  \leq\frac{f\left(  x\right)  +f\left(  x^{\prime}\right)  }{2}%
\int_{\Omega}p\left(  s\right)  ds-Hp_{1}\left(  t\right) \label{2.8}\\
&  \leq\left(  y-x\right)  \frac{f^{\prime}\left(  x^{\prime}\right)
-f^{\prime}\left(  x\right)  }{2}\int_{\Omega}p\left(  s\right)  ds\nonumber
\end{align}
and%
\begin{align}
0  &  \leq Hp_{1}\left(  t\right)  -\frac{f\left(  y\right)  +f\left(
y^{\prime}\right)  }{2}\int_{\Omega}p\left(  s\right)  ds\label{2.9}\\
&  \leq\left(  y-x\right)  \frac{f^{\prime}\left(  x^{\prime}\right)
-f^{\prime}\left(  x\right)  }{2}\int_{\Omega}p\left(  s\right)  ds\nonumber
\end{align}
holds for all $t\in\left[  0,1\right]  .$
\end{enumerate}
\end{theorem}

\begin{proof}
$\left(  1\right)  $ Using simple techniques of integration and the hypothesis
of $p$, we have the following identities
\[
\frac{f\left(  y\right)  +f\left(  y^{\prime}\right)  }{2}\int_{\Omega
}p\left(  s\right)  ds=\int_{x}^{y}\int_{0}^{\frac{1}{2}}2\left[  f\left(
y\right)  +f\left(  y^{\prime}\right)  \right]  p\left(  s\right)  dtds,
\]%
\begin{multline*}
\int_{\frac{x+y}{2}}^{y}2f\left(  s\right)  p\left(  2s-y\right)
ds+\int_{y^{\prime}}^{\frac{x^{\prime}+y^{\prime}}{2}}2f\left(  s\right)
p\left(  2s-y^{\prime}\right)  ds\\
=\int_{x}^{y}\int_{0}^{\frac{1}{2}}2\left[  f\left(  \frac{s+y}{2}\right)
+f\left(  \frac{y+2y^{\prime}-s}{2}\right)  \right]  p\left(  s\right)  dtds,
\end{multline*}%
\begin{multline*}
\int_{0}^{1}Hp_{1}\left(  t\right)  dt=\int_{x}^{y}\int_{0}^{\frac{1}{2}%
}\left[  f\left(  ty+\left(  1-t\right)  s\right)  +f\left(  ts+\left(
1-t\right)  y\right)  \right]  p\left(  s\right)  dtds\\
+\int_{x}^{y}\int_{0}^{\frac{1}{2}}\left[  f\left(  t\left(  y+y^{\prime
}-s\right)  +\left(  1-t\right)  y^{\prime}\right)  \right. \\
\left.  +f\left(  ty^{\prime}+\left(  1-t\right)  \left(  y+y^{\prime
}-s\right)  \right)  \right]  p\left(  s\right)  dtds
\end{multline*}
and%
\begin{multline*}
\frac{1}{2}\left[  \frac{f\left(  y\right)  +f\left(  y^{\prime}\right)  }%
{2}\int_{\Omega}p\left(  s\right)  ds+\int_{\Omega}f\left(  s\right)  p\left(
s\right)  ds\right] \\
=\int_{x}^{y}\int_{0}^{\frac{1}{2}}\left[  f\left(  s\right)  +f\left(
y\right)  \right]  p\left(  s\right)  dtds\\
+\int_{x}^{y}\int_{0}^{\frac{1}{2}}\left[  f\left(  y^{\prime}\right)
+f\left(  y+y^{\prime}-s\right)  \right]  p\left(  s\right)  dtds.
\end{multline*}

By Lemma \ref{l2}, the following inequalities hold for all $t\in\left[
0,\frac{1}{2}\right]  $ and $s\in\left[  x,y\right]  :$%
\[
2\left[  f\left(  y\right)  +f\left(  y^{\prime}\right)  \right]  p\left(
s\right)  \leq2\left[  f\left(  \frac{s+y}{2}\right)  +f\left(  \frac
{y+2y^{\prime}-s}{2}\right)  \right]  p\left(  s\right)  ,
\]%
\[
2f\left(  \frac{s+y}{2}\right)  p\left(  s\right)  \leq\left[  f\left(
ty+\left(  1-t\right)  s\right)  +f\left(  ts+\left(  1-t\right)  y\right)
\right]  p\left(  s\right)  ,
\]%
\begin{align*}
&  2f\left(  \frac{y+2y^{\prime}-s}{2}\right)  p\left(  s\right) \\
&  \leq\left[  f\left(  t\left(  y+\text{ }y^{\prime}-s\right)  +\left(
1-t\right)  y^{\prime}\right)  +f\left(  ty^{\prime}+\left(  1-t\right)
\left(  y+y^{\prime}-s\right)  \right)  \right]  p\left(  s\right)  ,
\end{align*}%
\[
\left[  f\left(  ty+\left(  1-t\right)  s\right)  +f\left(  ts+\left(
1-t\right)  y\right)  \right]  p\left(  s\right)  \leq\left[  f\left(
s\right)  +f\left(  y\right)  \right]  p\left(  s\right)
\]
and%
\begin{multline*}
\left[  f\left(  t\left(  y+y^{\prime}-s\right)  +\left(  1-t\right)
y^{\prime}\right)  +f\left(  ty^{\prime}+\left(  1-t\right)  \left(
y+y^{\prime}-s\right)  \right)  \right]  p\left(  s\right) \\
\leq\left[  f\left(  y^{\prime}\right)  +f\left(  y+y^{\prime}-s\right)
\right]  p\left(  s\right)  .
\end{multline*}

Integrating the above inequalities over $t$ on $\left[  0,\frac{1}{2}\right]
, $ over $s$ on $\left[  x,y\right]  $ and using the above identities, we
derive $\left(  \ref{2.5}\right)  .$

Again, using simple techniques of integration and the hypothesis of $p$, we
have the following identities
\[
f\left(  \frac{y+y^{\prime}}{2}\right)  \int_{\Omega}p\left(  s\right)
ds=\int_{x}^{y}\int_{0}^{\frac{1}{2}}4f\left(  \frac{y+y^{\prime}}{2}\right)
p\left(  s\right)  dtds,
\]%
\begin{align*}
&  \int_{\frac{x+y^{\prime}}{2}}^{\frac{y+y^{\prime}}{2}}2f\left(  s\right)
p\left(  2s-y^{\prime}\right)  ds+\int_{\frac{y+y^{\prime}}{2}}^{\frac
{y+x^{\prime}}{2}}2f\left(  s\right)  p\left(  2s-y\right)  ds\\
&  =\int_{\frac{x+y^{\prime}}{2}}^{\frac{y+y^{\prime}}{2}}2f\left(  s\right)
p\left(  2s-y^{\prime}\right)  ds+\int_{\frac{y+y^{\prime}}{2}}^{\frac
{y+x^{\prime}}{2}}2f\left(  s\right)  p\left(  2y+y^{\prime}-2s\right)  ds\\
&  =\int_{x}^{y}\int_{0}^{\frac{1}{2}}2\left[  f\left(  \frac{s+y^{\prime}}%
{2}\right)  +f\left(  \frac{y^{\prime}+2y-s}{2}\right)  \right]  p\left(
s\right)  dtds,
\end{align*}%
\begin{multline*}
\int_{0}^{1}Hp_{2}\left(  t\right)  dt=\int_{x}^{y}\int_{0}^{\frac{1}{2}%
}\left[  f\left(  ty^{\prime}+\left(  1-t\right)  s\right)  +f\left(
ts+\left(  1-t\right)  y^{\prime}\right)  \right]  p\left(  s\right)  dtds\\
+\int_{x}^{y}\int_{0}^{\frac{1}{2}}\left[  f\left(  t\left(  y+y^{\prime
}-s\right)  +\left(  1-t\right)  y\right)  \right. \\
\left.  +f\left(  ty+\left(  1-t\right)  \left(  y+y^{\prime}-s\right)
\right)  \right]  p\left(  s\right)  dtds
\end{multline*}
and%
\begin{multline*}
\frac{1}{2}\left[  \frac{f\left(  y\right)  +f\left(  y^{\prime}\right)  }%
{2}\int_{\Omega}p\left(  s\right)  ds+\int_{\Omega}f\left(  s\right)  p\left(
s\right)  ds\right] \\
=\int_{x}^{y}\int_{0}^{\frac{1}{2}}\left[  f\left(  s\right)  +f\left(
y^{\prime}\right)  \right]  p\left(  s\right)  dtds\\
+\int_{x}^{y}\int_{0}^{\frac{1}{2}}\left[  f\left(  y\right)  +f\left(
y+y^{\prime}-s\right)  \right]  p\left(  s\right)  dtds.
\end{multline*}

By Lemma \ref{l2}, the following inequalities hold for all $t\in\left[
0,\frac{1}{2}\right]  $ and $s\in\left[  x,y\right]  :$%
\[
4f\left(  \frac{y+y^{\prime}}{2}\right)  p\left(  s\right)  \leq2\left[
f\left(  \frac{s+y^{\prime}}{2}\right)  +f\left(  \frac{y^{\prime}+2y-s}%
{2}\right)  \right]  p\left(  s\right)  ,
\]%
\[
2f\left(  \frac{s+y^{\prime}}{2}\right)  p\left(  s\right)  \leq\left[
f\left(  ty^{\prime}+\left(  1-t\right)  s\right)  +f\left(  ts+\left(
1-t\right)  y^{\prime}\right)  \right]  p\left(  s\right)  ,
\]%
\begin{align*}
&  2f\left(  \frac{y^{\prime}+2y-s}{2}\right)  p\left(  s\right) \\
&  \leq\left[  f\left(  t\left(  y+\text{ }y^{\prime}-s\right)  +\left(
1-t\right)  y\right)  +f\left(  ty+\left(  1-t\right)  \left(  y+y^{\prime
}-s\right)  \right)  \right]  p\left(  s\right)  ,
\end{align*}%
\[
\left[  f\left(  ty^{\prime}+\left(  1-t\right)  s\right)  +f\left(
ts+\left(  1-t\right)  y^{\prime}\right)  \right]  p\left(  s\right)
\leq\left[  f\left(  s\right)  +f\left(  y^{\prime}\right)  \right]  p\left(
s\right)
\]
and%
\begin{multline*}
\left[  f\left(  t\left(  y+y^{\prime}-s\right)  +\left(  1-t\right)
y\right)  +f\left(  ty+\left(  1-t\right)  \left(  y+y^{\prime}-s\right)
\right)  \right]  p\left(  s\right) \\
\leq\left[  f\left(  y\right)  +f\left(  y+y^{\prime}-s\right)  \right]
p\left(  s\right)  .
\end{multline*}

Integrating the above inequalities over $t$ on $\left[  0,\frac{1}{2}\right]
, $ over $s$ on $\left[  x,y\right]  $ and using the above identities, we
derive $\left(  \ref{2.6}\right)  .$

$\left(  2\right)  $ By integration by parts, we have the following identity%
\begin{align*}
&  \int_{x}^{y}\left[  \left(  s-y\right)  f^{\prime}\left(  s\right)
+\left(  y-s\right)  f^{\prime}\left(  y+y^{\prime}-s\right)  \right]  ds\\
&  =\left(  f\left(  x\right)  +f\left(  x^{\prime}\right)  \right)  \left(
y-x\right)  -\int_{\Omega}f\left(  s\right)  ds.
\end{align*}

Now, using the convexity of $f$, the inequalities
\[
f\left(  s\right)  -f\left(  ts+\left(  1-t\right)  y\right)  \leq\left(
1-t\right)  \left(  s-y\right)  f^{\prime}\left(  s\right)
\]
and%
\[
f\left(  y+y^{\prime}-s\right)  -f\left(  t\left(  y+y^{\prime}-s\right)
+\left(  1-t\right)  y^{\prime}\right)  \leq\left(  1-t\right)  \left(
y-s\right)  f^{\prime}\left(  y+y^{\prime}-s\right)
\]
hold for all $t\in\left[  0,1\right]  $ and $s\in\left[  x,y\right]  $. Using
the definition of $Hp_{1},$ the hypothesis of $p,$ the above inequalities, the
convexity of $f$ and the above identity, we have%
\begin{align*}
&  \int_{\Omega}f\left(  s\right)  p\left(  s\right)  ds-Hp_{1}\left(
t\right) \\
&  =\int_{x}^{y}\left[  \left(  f\left(  s\right)  -f\left(  ts+\left(
1-t\right)  y\right)  \right)  p\left(  s\right)  \right. \\
&  \text{ \ \ \ \ \ \ \ \ \ \ \ \ \ \ \ \ \ \ }+\left.  \left(  f\left(
y+y^{\prime}-s\right)  -f\left(  t\left(  y+y^{\prime}-s\right)  +\left(
1-t\right)  y^{\prime}\right)  \right)  p\left(  s\right)  \right]  ds\\
&  \leq\int_{x}^{y}\left[  \left(  1-t\right)  \left(  s-y\right)  f^{\prime
}\left(  s\right)  p\left(  s\right)  \right. \\
&  \text{ \ \ \ \ \ \ \ \ \ \ \ \ \ \ \ \ \ \ \ }+\left.  \left(  1-t\right)
\left(  y-s\right)  f^{\prime}\left(  y+y^{\prime}-s\right)  p\left(
s\right)  \right]  ds\\
&  =\left(  1-t\right)  \int_{x}^{y}\left(  y-s\right)  \left(  f^{\prime
}\left(  y+y^{\prime}-s\right)  -f^{\prime}\left(  s\right)  \right)  p\left(
s\right)  ds\\
&  \leq\left(  1-t\right)  \int_{x}^{y}\left(  y-s\right)  \left(  f^{\prime
}\left(  y+y^{\prime}-s\right)  -f^{\prime}\left(  s\right)  \right)
ds\sup\limits_{s\in\Omega}p\left(  s\right) \\
&  =\left(  1-t\right)  \left[  \left(  f\left(  x\right)  +f\left(
x^{\prime}\right)  \right)  \left(  y-x\right)  -\int_{\Omega}f\left(
s\right)  ds\right]  \sup\limits_{s\in\Omega}p\left(  s\right)  .
\end{align*}
\ 

Using the above inequality and $\left(  \ref{2.1}\right)  $, we have derive
$\left(  \ref{2.7}\right)  .$

On the other hand, we have
\[
\frac{f\left(  x\right)  -f\left(  y\right)  }{2}\leq\frac{1}{2}\left(
x-y\right)  f^{\prime}\left(  x\right)
\]
and%
\[
\frac{f\left(  x^{\prime}\right)  -f\left(  y^{\prime}\right)  }{2}\leq
\frac{1}{2}\left(  x^{\prime}-y^{\prime}\right)  f^{\prime}\left(  x^{\prime
}\right)  .
\]
Multiplying the above inequalities by $\int_{\Omega}p\left(  s\right)  ds$ and
taking their sum:
\begin{align}
&  \frac{f\left(  x\right)  +f\left(  x^{\prime}\right)  }{2}\int_{\Omega
}p\left(  s\right)  ds-\frac{f\left(  y\right)  +f\left(  y^{\prime}\right)
}{2}\int_{\Omega}p\left(  s\right)  ds\label{2.10}\\
&  \leq\left[  \frac{1}{2}\left(  x-y\right)  f^{\prime}\left(  x\right)
+\frac{1}{2}\left(  x^{\prime}-y^{\prime}\right)  f^{\prime}\left(  x^{\prime
}\right)  \right]  \int_{\Omega}p\left(  s\right)  ds\nonumber\\
&  =\left(  y-x\right)  \frac{f^{\prime}\left(  x^{\prime}\right)  -f^{\prime
}\left(  x\right)  }{2}\int_{\Omega}p\left(  s\right)  ds.\nonumber
\end{align}
$\left(  3\right)  $ The inequalities $\left(  \ref{2.8}\right)  $ and
$\left(  \ref{2.9}\right)  $ follow from $\left(  \ref{2.1}\right)  ,$
$\left(  \ref{2.2}\right)  $ and $\left(  \ref{2.10}\right)  .$

This completes the proof.
\end{proof}

\begin{remark}
\label{r5}\textit{Using Remark \ref{r3}, }we have the following results:
\end{remark}

\begin{enumerate}
\item \textit{Theorem \ref{t2}} \textit{reduces to Theorem \ref{A6} as }$x=a,
$\textit{\ }$y=y^{\prime}=\frac{a+b}{2},$\textit{\ }$x^{\prime}=b$ \textit{and
}$p\left(  s\right)  \equiv\frac{1}{b-a}$ $\left(  s\in\left[  a,b\right]
\right)  .$

\item \textit{Theorem \ref{t2}} \textit{reduces to Theorem \ref{A16} as
}$p\left(  s\right)  \equiv\frac{1}{2\left(  y-x\right)  }$\ $\left(
s\in\Omega\right)  .$

\item \textit{Theorem \ref{t2}} \textit{reduces to Theorem \ref{A23} as
}$x=a,$ $y=y^{\prime}=\frac{a+b}{2},x^{\prime}=b$ \textit{and }$p\left(
s\right)  =g\left(  s\right)  $\ $\left(  s\in\left[  a,b\right]  \right)  .$
\end{enumerate}

\begin{theorem}
\label{t3}Let $x,y,y^{\prime},x^{\prime},\Omega,f,p,Hp_{1},Pp_{1}$\ be defined
as above. Then we have the following results:

\begin{enumerate}
\item $Pp_{1}$\ is convex on $\left[  0,1\right]  .$

\item $Pp_{1}$ is increasing on $\left[  0,1\right]  $\ and the following
inequalities%
\begin{equation}
\int_{\Omega}f\left(  s\right)  p\left(  s\right)  ds=Pp_{1}\left(  0\right)
\leq Pp_{1}\left(  t\right)  \leq Pp_{1}\left(  1\right)  =\frac{f\left(
x\right)  +f\left(  x^{\prime}\right)  }{2}\int_{\Omega}p\left(  s\right)  ds,
\label{2.11}%
\end{equation}%
\begin{align}
Pp_{1}\left(  t\right)   &  \leq\left(  1-t\right)  \int_{\Omega}f\left(
s\right)  p\left(  s\right)  ds+t\cdot\frac{f\left(  x\right)  +f\left(
x^{\prime}\right)  }{2}\int_{\Omega}p\left(  s\right)  ds\label{2.12}\\
&  \leq\frac{f\left(  x\right)  +f\left(  x^{\prime}\right)  }{2}\int_{\Omega
}p\left(  s\right)  ds\nonumber
\end{align}
and%
\begin{equation}
Hp_{1}\left(  t\right)  \leq Pp_{1}\left(  t\right)  \text{ } \label{2.13}%
\end{equation}
hold for all $t\in\left[  0,1\right]  .$
\end{enumerate}
\end{theorem}

\begin{proof}
$\left(  1\right)  $ It is easily observed from the convexity of $f$ and the
hypothesis of $p$ that $Pp_{1}$ is convex on $\left[  0,1\right]  .$

$(2)$ Let $t_{1}<t_{2}$ in $\left[  0,1\right]  .$ By Lemma \ref{l2}, the
following inequality holds for all $s\in\left[  x,y\right]  $%
\begin{multline*}
\left[  f\left(  t_{1}x+\left(  1-t_{1}\right)  s\right)  +f\left(
t_{1}x^{\prime}+\left(  1-t_{1}\right)  \left(  x+x^{\prime}-s\right)
\right)  \right]  p\left(  s\right) \\
\leq\left[  f\left(  t_{2}x+\left(  1-t_{2}\right)  s\right)  +f\left(
t_{2}x^{\prime}+\left(  1-t_{2}\right)  \left(  x+x^{\prime}-s\right)
\right)  \right]  p\left(  s\right)  .
\end{multline*}

Integrating the above inequality over $s$ on $\left[  x,y\right]  $ and using
the definition of $Pp_{1},$ we have
\[
Pp_{1}\left(  t_{1}\right)  \leq Pp_{1}\left(  t_{2}\right)  .
\]

Thus, $Pp_{1}$\ is increasing on $\left[  0,1\right]  $ and $\left(
\ref{2.11}\right)  $ holds.

Using the convexity of $f,$ the hypothesis of $p,$ the inequality $\left(
\ref{2.11}\right)  $ and the substitution rule for integration, the inequality
$\left(  \ref{2.12}\right)  $\ holds.

Finally, $\left(  \ref{2.13}\right)  $ follows from $\left(  \ref{2.1}\right)
$ and $\left(  \ref{2.11}\right)  .$

This completes the proof.
\end{proof}

\begin{remark}
\label{r6}\textit{Using Remark \ref{r3}, }we have the following results:
\end{remark}

\begin{enumerate}
\item \textit{Theorem \ref{t3}} \textit{reduces to Theorem \ref{A5} as }$x=a,
$\textit{\ }$y=y^{\prime}=\frac{a+b}{2},$\textit{\ }$x^{\prime}=b$ \textit{and
}$p\left(  s\right)  \equiv\frac{1}{b-a}$ $\left(  s\in\left[  a,b\right]
\right)  .$

\item \textit{Theorem \ref{t3}} \textit{reduces to Theorem \ref{A14} as
}$p\left(  s\right)  \equiv\frac{1}{2\left(  y-x\right)  }$\ $\left(
s\in\Omega\right)  .$

\item \textit{Theorem \ref{t3}} \textit{reduces to Theorem \ref{A21} as
}$x=a,$ $y=y^{\prime}=\frac{a+b}{2},x^{\prime}=b$ \textit{and }$p\left(
s\right)  =g\left(  s\right)  $\ $\left(  s\in\left[  a,b\right]  \right)  .$
\end{enumerate}

\begin{theorem}
\label{t4}Let $x,y,y^{\prime},x^{\prime},\Omega,f,p,Hp_{1},Pp_{1}$\ be defined
as above. Then we have the following results:

\begin{enumerate}
\item The inequality%
\begin{align}
\int_{\Omega}f\left(  s\right)  p\left(  s\right)  ds  &  \leq%
{\displaystyle\int\nolimits_{x}^{\frac{x+y}{2}}}
2f\left(  s\right)  p\left(  2s-x\right)  ds+\int_{\frac{x^{\prime}+y^{\prime
}}{2}}^{x^{\prime}}2f\left(  s\right)  p\left(  2s-x^{\prime}\right)
ds\label{2.14}\\
&  \leq\int_{0}^{1}Pp_{1}\left(  t\right)  dt\nonumber\\
&  \leq\frac{1}{2}\left[  \frac{f\left(  x\right)  +f\left(  x^{\prime
}\right)  }{2}\int_{\Omega}p\left(  s\right)  ds+\int_{\Omega}f\left(
s\right)  p\left(  s\right)  ds\right] \nonumber
\end{align}
\textit{holds.}

\item If $f$\ is differentiable on $\Omega$ and $p$ is bounded on $\Omega,$
then the inequalities%
\begin{align}
0  &  \leq t\left[  \int_{\Omega}f\left(  s\right)  ds-\left(  f\left(
y\right)  +f\left(  y^{\prime}\right)  \right)  \left(  y-x\right)  \right]
\cdot\inf\limits_{s\in\Omega}p\left(  s\right) \label{2.15}\\
&  \leq Pp_{1}\left(  t\right)  -\int_{\Omega}f\left(  s\right)  p\left(
s\right)  ds,\nonumber
\end{align}%
\begin{align}
0  &  \leq Pp_{1}\left(  t\right)  -\frac{f\left(  y\right)  +f\left(
y^{\prime}\right)  }{2}\int_{\Omega}p\left(  s\right)  ds\label{2.16}\\
&  \leq\left(  y-x\right)  \frac{f^{\prime}\left(  x^{\prime}\right)
-f^{\prime}\left(  x\right)  }{2}\int_{\Omega}p\left(  s\right)  ds,\nonumber
\end{align}%
\begin{align}
0  &  \leq\frac{f\left(  x\right)  +f\left(  x^{\prime}\right)  }{2}%
\int_{\Omega}p\left(  s\right)  ds-Pp_{1}\left(  t\right) \label{2.17}\\
&  \leq\left(  y-x\right)  \frac{f^{\prime}\left(  x^{\prime}\right)
-f^{\prime}\left(  x\right)  }{2}\int_{\Omega}p\left(  s\right)  ds\nonumber
\end{align}
and%
\begin{equation}
0\leq Pp_{1}\left(  t\right)  -Hp_{1}\left(  t\right)  \leq\left(  y-x\right)
\frac{f^{\prime}\left(  x^{\prime}\right)  -f^{\prime}\left(  x\right)  }%
{2}\int_{\Omega}p\left(  s\right)  ds \label{2.18}%
\end{equation}
hold for all $t\in\left[  0,1\right]  .$
\end{enumerate}
\end{theorem}

\begin{proof}
$\left(  1\right)  $ Using simple techniques of integration and the hypothesis
of $p$, we have the following identities%
\[
\int_{\Omega}f\left(  s\right)  p\left(  s\right)  ds=\int_{x}^{y}\int%
_{0}^{\frac{1}{2}}2\left[  f\left(  s\right)  +f\left(  x+x^{\prime}-s\right)
\right]  p\left(  s\right)  dtds,
\]%
\begin{align*}
&
{\displaystyle\int\nolimits_{x}^{\frac{x+y}{2}}}
2f\left(  s\right)  p\left(  2s-x\right)  ds+\int_{\frac{x^{\prime}+y^{\prime
}}{2}}^{x^{\prime}}2f\left(  s\right)  p\left(  2s-x^{\prime}\right)  ds\\
&  =%
{\displaystyle\int\nolimits_{x}^{\frac{x+y}{2}}}
2f\left(  s\right)  p\left(  2s-x\right)  ds+\int_{\frac{x^{\prime}+y^{\prime
}}{2}}^{x^{\prime}}2f\left(  s\right)  p\left(  x+2x^{\prime}-2s\right)  ds\\
&  =\int_{x}^{y}\int_{0}^{\frac{1}{2}}2\left[  f\left(  \frac{x+s}{2}\right)
+f\left(  \frac{x+2x^{\prime}-s}{2}\right)  \right]  p\left(  s\right)  dtds,
\end{align*}%
\begin{multline*}
\int_{0}^{1}Pp_{1}\left(  t\right)  dt=\int_{x}^{y}\int_{0}^{\frac{1}{2}%
}\left[  f\left(  ts+\left(  1-t\right)  x\right)  +f\left(  tx+\left(
1-t\right)  s\right)  \right]  p\left(  s\right)  dtds\\
+\int_{x}^{y}\int_{0}^{\frac{1}{2}}\left[  f\left(  tx^{\prime}+\left(
1-t\right)  \left(  x+x^{\prime}-s\right)  \right)  \right. \\
\left.  +f\left(  t\left(  x+x^{\prime}-s\right)  +\left(  1-t\right)
x^{\prime}\right)  \right]  p\left(  s\right)  dtds
\end{multline*}
and
\begin{multline*}
\frac{1}{2}\left[  \frac{f\left(  x\right)  +f\left(  x^{\prime}\right)  }%
{2}\int_{\Omega}g\left(  s\right)  ds+\int_{\Omega}f\left(  s\right)  g\left(
s\right)  ds\right] \\
=\int_{x}^{y}\int_{0}^{\frac{1}{2}}\left[  f\left(  x\right)  +f\left(
s\right)  \right]  p\left(  s\right)  dtds\\
+\int_{x}^{y}\int_{0}^{\frac{1}{2}}\left[  f\left(  x+x^{\prime}-s\right)
+f\left(  x^{\prime}\right)  \right]  p\left(  s\right)  dtds.
\end{multline*}

By Lemma \ref{l2}, the following inequalities hold for all $t\in\left[
0,\frac{1}{2}\right]  $ and $s\in\left[  x,y\right]  $%
\[
2\left[  f\left(  s\right)  +f\left(  x+x^{\prime}-s\right)  \right]  p\left(
s\right)  \leq2\left[  f\left(  \frac{x+s}{2}\right)  +f\left(  \frac
{x+2x^{\prime}-s}{2}\right)  \right]  p\left(  s\right)  ,
\]%
\[
2f\left(  \frac{x+s}{2}\right)  p\left(  s\right)  \leq\left[  f\left(
ts+\left(  1-t\right)  x\right)  +f\left(  tx+\left(  1-t\right)  s\right)
\right]  p\left(  s\right)  ,
\]%
\begin{align*}
&  2f\left(  \frac{x+2x^{\prime}-s}{2}\right)  p\left(  s\right) \\
&  \leq\left[  f\left(  tx^{\prime}+\left(  1-t\right)  \left(  x+x^{\prime
}-s\right)  \right)  +f\left(  t\left(  x+x^{\prime}-s\right)  +\left(
1-t\right)  x^{\prime}\right)  \right]  p\left(  s\right)  ,
\end{align*}%
\[
\left[  f\left(  ts+\left(  1-t\right)  x\right)  +f\left(  tx+\left(
1-t\right)  s\right)  \right]  p\left(  s\right)  \leq\left[  f\left(
x\right)  +f\left(  s\right)  \right]  p\left(  s\right)
\]
and%
\begin{multline*}
\left[  f\left(  tx^{\prime}+\left(  1-t\right)  \left(  x+x^{\prime
}-s\right)  \right)  +f\left(  t\left(  x+x^{\prime}-s\right)  +\left(
1-t\right)  x^{\prime}\right)  \right]  p\left(  s\right) \\
\leq\left[  f\left(  x+x^{\prime}-s\right)  +f\left(  x^{\prime}\right)
\right]  p\left(  s\right)  .
\end{multline*}

Integrating the above inequalities over $t$ on $\left[  0,\frac{1}{2}\right]
, $ over $s$ on $\left[  x,y\right]  $ and using the above identities, we
derive $\left(  \ref{2.14}\right)  .$

$\left(  2\right)  $ Using the integrating by parts, we have the following
identitiy%
\begin{multline*}
\int_{x}^{y}\left[  \left(  x-s\right)  f^{\prime}\left(  s\right)  +\left(
s-x\right)  f^{\prime}\left(  x+x^{\prime}-s\right)  \right]  ds\\
=\int_{\Omega}f\left(  s\right)  ds-\left(  f\left(  y\right)  +f\left(
y^{\prime}\right)  \right)  \left(  y-x\right)  .
\end{multline*}

Now, using the convexity of $f$, the inequalities%
\[
f\left(  tx+\left(  1-t\right)  s\right)  -f\left(  s\right)  \geq t\left(
x-s\right)  f^{\prime}\left(  s\right)
\]
and
\[
f\left(  tx^{\prime}+\left(  1-t\right)  \left(  x+x^{\prime}-s\right)
\right)  -f\left(  x+x^{\prime}-s\right)  \geq t\left(  s-x\right)  f^{\prime
}\left(  x+x^{\prime}-s\right)
\]
hold for all $t\in\left[  0,1\right]  $ and $s\in\left[  x,y\right]  $. Let
$t\in\left[  0,1\right]  .$ Using the definition of $Pp_{1},$ the hypothesis
of $p,$ the above inequalities, the convexity of $f$ and the above identity,
we have%
\begin{align*}
&  Pp_{1}\left(  t\right)  -\int_{\Omega}f\left(  s\right)  p\left(  s\right)
ds\\
&  =\int_{x}^{y}\left[  \left(  f\left(  tx+\left(  1-t\right)  s\right)
-f\left(  s\right)  \right)  p\left(  s\right)  \right. \\
&  \text{ \ \ \ \ \ \ \ \ \ \ }+\left.  \left(  f\left(  tx^{\prime}+\left(
1-t\right)  \left(  x+x^{\prime}-s\right)  \right)  -f\left(  x+x^{\prime
}-s\right)  \right)  p\left(  s\right)  \right]  ds\\
&  \geq\int_{x}^{y}\left[  t\left(  x-s\right)  f^{\prime}\left(  s\right)
g\left(  s\right)  +t\left(  s-x\right)  f^{\prime}\left(  x+x^{\prime
}-s\right)  p\left(  s\right)  \right]  ds\\
&  =\int_{x}^{y}t\left(  s-x\right)  \left(  f^{\prime}\left(  x+x^{\prime
}-s\right)  -f^{\prime}\left(  s\right)  \right)  p\left(  s\right)  ds\\
&  \geq\int_{x}^{y}t\left(  s-x\right)  \left(  f^{\prime}\left(  x+x^{\prime
}-s\right)  -f^{\prime}\left(  s\right)  \right)  ds\cdot\inf\limits_{s\in
\left[  x,y\right]  }p\left(  s\right) \\
&  =t\int_{x}^{y}\left[  \left(  x-s\right)  f^{\prime}\left(  s\right)
g\left(  s\right)  +\left(  s-x\right)  f^{\prime}\left(  x+x^{\prime
}-s\right)  \right]  ds\cdot\inf\limits_{s\in\left[  x,y\right]  }p\left(
s\right) \\
&  =t\left[  \int_{\Omega}f\left(  s\right)  ds-\left(  f\left(  y\right)
+f\left(  y^{\prime}\right)  \right)  \left(  y-x\right)  \right]  \cdot
\inf\limits_{s\in\left[  x,y\right]  }p\left(  s\right)  .
\end{align*}

Using the above inequality and Theorem \ref{A12}, we derive $\left(
\ref{2.15}\right)  .$

Finally, $\left(  \ref{2.16}\right)  -\left(  \ref{2.18}\right)  $ follow from
$\left(  \ref{2.1}\right)  ,$ $\left(  \ref{2.10}\right)  ,$ $\left(
\ref{2.11}\right)  $ and $\left(  \ref{2.13}\right)  .$

This completes the proof.
\end{proof}

\begin{remark}
\label{r7}\textit{Using Remark \ref{r3}, Theorem \ref{t4}} \textit{reduces to
Theorem \ref{A17} as }$p\left(  s\right)  \equiv\frac{1}{2\left(  y-x\right)
}$\ $\left(  s\in\Omega\right)  .$
\end{remark}

\begin{theorem}
\label{t5}Let $x,y,y^{\prime},x^{\prime},\Omega,f,p,Hp_{1},Hp_{2},Fp_{1}$\ be
defined as above. Then we have the following results:

\begin{enumerate}
\item $Fp_{1}$\ is convex on $\left[  0,1\right]  $\ and symmetric about
$\frac{1}{2}.$

\item $Fp_{1}$ is decreasing on $\left[  0,\frac{1}{2}\right]  $ and
increasing on $\left[  \frac{1}{2},1\right]  ,$%
\begin{equation}
\sup\limits_{t\in\left[  0,1\right]  }Fp_{1}\left(  t\right)  =Fp_{1}\left(
0\right)  =Fp_{1}\left(  1\right)  =\int_{\Omega}f\left(  s\right)  p\left(
s\right)  ds\ \int_{\Omega}p\left(  s\right)  ds \label{2.19}%
\end{equation}
and
\begin{equation}
\inf\limits_{t\in\left[  0,1\right]  }Fp_{1}\left(  t\right)  =Fp_{1}\left(
\frac{1}{2}\right)  =\int_{\Omega}\int_{\Omega}f\left(  \frac{s+u}{2}\right)
p\left(  s\right)  p\left(  u\right)  dsdu. \label{2.20}%
\end{equation}

\item We have:%
\begin{equation}
\frac{Hp_{1}\left(  t\right)  +Hp_{2}\left(  t\right)  }{2}\int_{\Omega
}p\left(  s\right)  ds\leq Fp_{1}\left(  t\right)  \text{ \ \ }\left(
t\in\left[  0,1\right]  \right)  \label{2.21}%
\end{equation}
and%
\begin{equation}
\frac{f\left(  y\right)  +2f\left(  \frac{y+y^{\prime}}{2}\right)  +f\left(
y^{\prime}\right)  }{4}\left[  \int_{\Omega}p\left(  s\right)  ds\right]
^{2}\leq Fp_{1}\left(  \frac{1}{2}\right)  . \label{2.22}%
\end{equation}

\end{enumerate}
\end{theorem}

\begin{proof}
$\left(  1\right)  $ It is easily observed from the convexity of $f$ and the
hypothesis of $p$ that $Fp_{1}$ is convex on $\left[  0,1\right]  .$

By changing variables, we have
\[
Fp_{1}\left(  t\right)  =Fp_{1}\left(  1-t\right)  ,\text{ \qquad\ }%
t\in\left[  0,1\right]
\]
and from which we get that $Fp_{1}$ is symmetric about $\frac{1}{2}.$

$\left(  2\right)  $ Let $t_{1}<t_{2}$ in $\left[  0,\frac{1}{2}\right]  .$
Using the symmetry of $F_{1}$, we have%
\begin{equation}
Fp_{1}\left(  t_{1}\right)  =\frac{1}{2}\left[  Fp_{1}\left(  t_{1}\right)
+Fp_{1}\left(  1-t_{1}\right)  \right]  , \label{2.23}%
\end{equation}%
\begin{equation}
Fp_{1}\left(  t_{2}\right)  =\frac{1}{2}\left[  Fp_{1}\left(  t_{2}\right)
+Fp_{1}\left(  1-t_{2}\right)  \right]  \label{2.24}%
\end{equation}
and, by Lemma \ref{l1}, we obtain%
\begin{equation}
\frac{1}{2}\left[  Fp_{1}\left(  t_{2}\right)  +Fp_{1}\left(  1-t_{2}\right)
\right]  \leq\frac{1}{2}\left[  Fp_{1}\left(  t_{1}\right)  +Fp_{1}\left(
1-t_{1}\right)  \right]  . \label{2.25}%
\end{equation}

From $\left(  \ref{2.23}\right)  -\left(  \ref{2.25}\right)  $, we obtain that
$Fp_{1}$ is decreasing on $\left[  0,\frac{1}{2}\right]  .$ Since $Fp_{1}$ is
symmetric about $\frac{1}{2}$ and $Fp_{1}$ is decreasing on $\left[
0,\frac{1}{2}\right]  $, we get that $Fp_{1}$ is increasing on $\left[
\frac{1}{2},1\right]  .$ Using the symmetry and monotonicity of $Fp_{1},$ we
derive $\left(  \ref{2.19}\right)  $ and $\left(  \ref{2.20}\right)  .$

$\left(  3\right)  $ Using the substitution rules for integration and the
hypothesis of $p$, we have the following identities:%
\begin{multline*}
Fp_{1}\left(  t\right)  =\left\{  \int_{x}^{y}\int_{x}^{y}\left[  f\left(
ts+\left(  1-t\right)  u\right)  +f\left(  ts+\left(  1-t\right)  \left(
y+y^{\prime}-u\right)  \right)  \right]  p\left(  s\right)  p\left(  u\right)
dsdu\right. \\
+\int_{x}^{y}\int_{x}^{y}\left[  f\left(  t\left(  y+y^{\prime}-s\right)
+\left(  1-t\right)  u\right)  \right. \\
+\left.  f\left(  t\left(  y+y^{\prime}-s\right)  +\left(  1-t\right)  \left(
y+y^{\prime}-u\right)  \right)  \right]  p\left(  s\right)  p\left(  u\right)
dsdu\Bigg\}
\end{multline*}
and%
\begin{multline*}
\frac{Hp_{1}\left(  t\right)  +Hp_{2}\left(  t\right)  }{2}\int_{\Omega
}p\left(  s\right)  ds\\
=\left\{  \int_{x}^{y}\int_{x}^{y}\left[  f\left(  ts+\left(  1-t\right)
y\right)  +f\left(  ts+\left(  1-t\right)  y^{\prime}\right)  \right]
p\left(  s\right)  p\left(  u\right)  dsdu\right. \\
+\int_{x}^{y}\int_{x}^{y}\left[  f\left(  t\left(  y+y^{\prime}-s\right)
+\left(  1-t\right)  y\right)  \right. \\
+\left.  f\left(  t\left(  y+y^{\prime}-s\right)  +\left(  1-t\right)
y^{\prime}\right)  \right]  p\left(  s\right)  p\left(  u\right)  dsdu\Bigg\}
\end{multline*}
for all $t\in\left[  0,1\right]  .$ By Lemma \ref{l2}, the following
inequalities hold for all $t\in\left[  0,1\right]  ,$ $s\in\left[  x,y\right]
$ and $u\in\left[  x,y\right]  :$%
\begin{align*}
&  \left[  f\left(  ts+\left(  1-t\right)  y\right)  +f\left(  ts+\left(
1-t\right)  y^{\prime}\right)  \right]  p\left(  s\right)  p\left(  u\right)
\\
&  \leq\left[  f\left(  ts+\left(  1-t\right)  u\right)  +f\left(  ts+\left(
1-t\right)  \left(  y+y^{\prime}-u\right)  \right)  \right]  p\left(
s\right)  p\left(  u\right)  ,
\end{align*}
and%
\begin{multline*}
\left[  f\left(  t\left(  y+y^{\prime}-s\right)  +\left(  1-t\right)
y\right)  +f\left(  t\left(  y+y^{\prime}-s\right)  +\left(  1-t\right)
y^{\prime}\right)  \right]  p\left(  s\right)  p\left(  u\right) \\
\leq\left[  f\left(  t\left(  y+y^{\prime}-s\right)  +\left(  1-t\right)
u\right)  +f\left(  t\left(  y+y^{\prime}-s\right)  +\left(  1-t\right)
\left(  y+y^{\prime}-u\right)  \right)  \right]  p\left(  s\right)  p\left(
u\right)  .
\end{multline*}

Integrating the above inequalities over $s$ on $\left[  x,y\right]  $, over
$u$ on $\left[  x,y\right]  $ and using the above identities, we derive the
inequality $\left(  \ref{2.21}\right)  .$

From the inequalities $\left(  \ref{2.1}\right)  ,$ $\left(  \ref{2.3}\right)
,$ $\left(  \ref{2.21}\right)  $ and the monotonicity of $Hp_{1}$, we have%
\begin{align*}
&  \frac{f\left(  y\right)  +2f\left(  \frac{y+y^{\prime}}{2}\right)
+f\left(  y^{\prime}\right)  }{4}\left[  \int_{\Omega}p\left(  s\right)
ds\right]  ^{2}\\
&  \leq\frac{Hp_{1}\left(  \frac{1}{2}\right)  +Hp_{2}\left(  \frac{1}%
{2}\right)  }{2}\int_{\Omega}p\left(  s\right)  ds\\
&  \leq Fp_{1}\left(  \frac{1}{2}\right)
\end{align*}
and from which we derive the inequality $\left(  \ref{2.22}\right)  .$

This completes the proof.
\end{proof}

\begin{remark}
\label{r8}\textit{Using Remark \ref{r3}, }we have the following results:
\end{remark}

\begin{enumerate}
\item \textit{Theorem \ref{t5}} \textit{reduces to Theorem \ref{A4} as }$x=a,
$\textit{\ }$y=y^{\prime}=\frac{a+b}{2},$\textit{\ }$x^{\prime}=b$ \textit{and
}$p\left(  s\right)  \equiv\frac{1}{b-a}$ $\left(  s\in\left[  a,b\right]
\right)  .$

\item \textit{Theorem \ref{t5}} \textit{reduces to Theorem \ref{A15} as
}$p\left(  s\right)  \equiv\frac{1}{2\left(  y-x\right)  }$\ $\left(
s\in\Omega\right)  .$

\item \textit{Theorem \ref{t5}} \textit{reduces to Theorem \ref{A22} as
}$x=a,$ $y=y^{\prime}=\frac{a+b}{2},x^{\prime}=b$ \textit{and }$p\left(
s\right)  =g\left(  s\right)  $\ $\left(  s\in\left[  a,b\right]  \right)  .$
\end{enumerate}

\begin{theorem}
\label{t6}Let $x,y,y^{\prime},x^{\prime},f,p,Hp_{1},Pp_{1},G_{1},G_{2}$\ be
defined as above. Then we have the following results:

\begin{enumerate}
\item The inequality%
\begin{equation}
Hp_{1}\left(  t\right)  \leq G_{1}\left(  t\right)  \int_{\Omega}p\left(
s\right)  ds\leq Pp_{1}\left(  t\right)  \label{2.26}%
\end{equation}
holds for all $t\in\left[  0,1\right]  .$

\item The inequalities
\begin{align}
&  \int_{\frac{x+y}{2}}^{y}2f\left(  s\right)  p\left(  2s-y\right)
ds+\int_{y^{\prime}}^{\frac{x^{\prime}+y^{\prime}}{2}}2f\left(  s\right)
p\left(  2s-y^{\prime}\right)  ds\label{2.27}\\
&  \leq\frac{1}{2}\left[  f\left(  \frac{x+y}{2}\right)  +f\left(
\frac{x^{\prime}+y^{\prime}}{2}\right)  \right]  \int_{\Omega}p\left(
s\right)  ds\nonumber\\
&  \leq\frac{1}{2}\int_{x}^{y}\left[  f\left(  \frac{s+x}{2}\right)  +f\left(
\frac{x+2y-s}{2}\right)  \right]  p\left(  s\right)  ds\nonumber\\
&  \text{ \ \ \ \ \ \ \ \ \ \ \ \ \ \ \ }+\frac{1}{2}\int_{y^{\prime}%
}^{x^{\prime}}\left[  f\left(  \frac{s+x^{\prime}}{2}\right)  +f\left(
\frac{x^{\prime}+2y^{\prime}-s}{2}\right)  \right]  p\left(  s\right)
ds\nonumber\\
&  \leq\frac{1}{4}\left[  f\left(  x\right)  +f\left(  y\right)  +f\left(
y^{\prime}\right)  +f\left(  x^{\prime}\right)  \right]  \int_{\Omega}p\left(
s\right)  ds\nonumber
\end{align}
and
\begin{align}
&  \int_{\frac{x+y^{\prime}}{2}}^{\frac{y+y^{\prime}}{2}}2f\left(  s\right)
p\left(  2s-y^{\prime}\right)  ds+\int_{\frac{y+y^{\prime}}{2}}^{\frac
{x^{\prime}+y}{2}}2f\left(  s\right)  p\left(  2s-y\right)  ds\label{2.28}\\
&  \leq\frac{1}{2}\left[  f\left(  \frac{x+y^{\prime}}{2}\right)  +f\left(
\frac{x^{\prime}+y}{2}\right)  \right]  \int_{\Omega}p\left(  s\right)
ds\nonumber\\
&  \leq\frac{1}{2}\int_{x}^{y}\left[  f\left(  \frac{s+x}{2}\right)  +f\left(
\frac{x+2y^{\prime}-s}{2}\right)  \right]  p\left(  s\right)  ds\nonumber\\
&  \text{ \ \ \ \ \ \ \ \ \ \ \ \ \ \ \ }+\frac{1}{2}\int_{y^{\prime}%
}^{x^{\prime}}\left[  f\left(  \frac{s+x^{\prime}}{2}\right)  +f\left(
\frac{x^{\prime}+2y-s}{2}\right)  \right]  p\left(  s\right)  ds\nonumber\\
&  \leq\frac{1}{4}\left[  f\left(  x\right)  +f\left(  y\right)  +f\left(
y^{\prime}\right)  +f\left(  x^{\prime}\right)  \right]  \int_{\Omega}p\left(
s\right)  ds\nonumber
\end{align}
hold.
\end{enumerate}
\end{theorem}

\textit{The inequalities}%
\begin{align}
0  &  \leq Hp_{1}\left(  t\right)  -\frac{f\left(  y\right)  +f\left(
y^{\prime}\right)  }{2}\int_{\Omega}p\left(  s\right)  ds\label{2.29}\\
&  \leq G_{1}\left(  t\right)  \int_{\Omega}p\left(  s\right)  ds-Hp_{1}%
\left(  t\right) \nonumber
\end{align}
\textit{and }%
\begin{align}
0  &  \leq Pp_{1}\left(  t\right)  -G_{1}\left(  t\right)  \int_{\Omega
}p\left(  s\right)  ds\label{2.30}\\
&  \leq\frac{f\left(  x\right)  +f\left(  x^{\prime}\right)  }{2}\int_{\Omega
}p\left(  s\right)  ds-Pp_{1}\left(  t\right) \nonumber
\end{align}
\textit{hold for all }$t\in\left[  0,1\right]  .$

\begin{proof}
$\left(  1\right)  $ By Lemma \ref{l2}, the following inequalities hold for
all $t\in\left[  0,1\right]  $ and $s\in\left[  x,y\right]  :$%
\begin{multline*}
\left[  f\left(  ts+\left(  1-t\right)  y\right)  +f\left(  t\left(
y+y^{\prime}-s\right)  +\left(  1-t\right)  y^{\prime}\right)  \right]
p\left(  s\right) \\
\leq\left[  f\left(  tx+\left(  1-t\right)  y\right)  +f\left(  tx^{\prime
}+\left(  1-t\right)  y^{\prime}\right)  \right]  p\left(  s\right)
\end{multline*}
and%
\begin{align*}
&  \left[  f\left(  tx+\left(  1-t\right)  y\right)  +f\left(  tx^{\prime
}+\left(  1-t\right)  y^{\prime}\right)  \right]  p\left(  s\right) \\
&  \leq\left[  f\left(  tx+\left(  1-t\right)  s\right)  +f\left(  tx^{\prime
}+\left(  1-t\right)  \left(  x+x^{\prime}-s\right)  \right)  \right]
p\left(  s\right)  .
\end{align*}

Integrating the above inequalities over $s$ on $\left[  x,y\right]  $, using
the hypothesis of $p$ and the definitions of $G_{1},Hp_{1}$ and $Pp_{1}$, we
have
\[
Hp_{1}\left(  t\right)  \leq G_{1}\left(  t\right)  \int_{\Omega}p\left(
s\right)  ds\leq Pp_{1}\left(  t\right)
\]
for all $t\in\left[  0,1\right]  .$

$(2)$ Using simple techniques of integration and the hypothesis of $p$, we
have the following identities
\begin{align*}
&  \int_{\frac{x+y}{2}}^{y}2f\left(  s\right)  p\left(  2s-y\right)
ds+\int_{y^{\prime}}^{\frac{x^{\prime}+y^{\prime}}{2}}2f\left(  s\right)
p\left(  2s-y^{\prime}\right)  ds\\
&  =\int_{\frac{x+y}{2}}^{y}2f\left(  s\right)  p\left(  2s-y\right)
ds+\int_{y^{\prime}}^{\frac{x^{\prime}+y^{\prime}}{2}}2f\left(  s\right)
p\left(  y+2y^{\prime}-2s\right)  ds\\
&  =\int_{x}^{y}\left[  f\left(  \frac{s+y}{2}\right)  +f\left(
\frac{y+2y^{\prime}-s}{2}\right)  \right]  p\left(  s\right)  ds,
\end{align*}%
\begin{align*}
&  \frac{1}{2}\left[  f\left(  \frac{x+y}{2}\right)  +f\left(  \frac
{x^{\prime}+y^{\prime}}{2}\right)  \right]  \int_{\Omega}p\left(  s\right)
ds\\
&  =\int_{x}^{y}\left[  f\left(  \frac{x+y}{2}\right)  +f\left(
\frac{x^{\prime}+y^{\prime}}{2}\right)  \right]  p\left(  s\right)  ds,
\end{align*}%
\begin{align*}
&  \frac{1}{2}\int_{x}^{y}\left[  f\left(  \frac{s+x}{2}\right)  +f\left(
\frac{x+2y-s}{2}\right)  \right]  p\left(  s\right)  ds\\
&  +\frac{1}{2}\int_{y^{\prime}}^{x^{\prime}}\left[  f\left(  \frac
{s+x^{\prime}}{2}\right)  +f\left(  \frac{x^{\prime}+2y^{\prime}-s}{2}\right)
\right]  p\left(  s\right)  ds\\
&  =\int_{x}^{y}\frac{1}{2}\left[  f\left(  \frac{s+x}{2}\right)  +f\left(
\frac{x+2y-s}{2}\right)  \right]  p\left(  s\right)  ds\\
&  +\int_{x}^{y}\frac{1}{2}\left[  f\left(  \frac{x+2x^{\prime}-s}{2}\right)
+f\left(  \frac{s-x+2y^{\prime}}{2}\right)  \right]  p\left(  s\right)  ds
\end{align*}
and%
\begin{align*}
&  \frac{1}{4}\left[  f\left(  x\right)  +f\left(  y\right)  +f\left(
y^{\prime}\right)  +f\left(  x^{\prime}\right)  \right]  \int_{\Omega}p\left(
s\right)  ds\\
&  =\int_{x}^{y}\frac{1}{2}\left[  f\left(  x\right)  +f\left(  y\right)
\right]  p\left(  s\right)  ds+\int_{x}^{y}\frac{1}{2}\left[  f\left(
y^{\prime}\right)  +f\left(  x^{\prime}\right)  \right]  p\left(  s\right)
ds.
\end{align*}

By Lemma \ref{l2}, the following inequalities hold for all $s\in\left[
x,y\right]  :$%
\[
\left[  f\left(  \frac{s+y}{2}\right)  +f\left(  \frac{y+2y^{\prime}-s}%
{2}\right)  \right]  p\left(  s\right)  \leq\left[  f\left(  \frac{x+y}%
{2}\right)  +f\left(  \frac{x^{\prime}+y^{\prime}}{2}\right)  \right]
p\left(  s\right)  .
\]%
\begin{align*}
f\left(  \frac{x+y}{2}\right)  p\left(  s\right)   &  =\frac{1}{2}\left[
f\left(  \frac{x+y}{2}\right)  +f\left(  \frac{x+y}{2}\right)  \right]
p\left(  s\right) \\
&  \leq\frac{1}{2}\left[  f\left(  \frac{s+x}{2}\right)  +f\left(
\frac{x+2y-s}{2}\right)  \right]  p\left(  s\right)  .
\end{align*}%
\begin{align*}
f\left(  \frac{x^{\prime}+y^{\prime}}{2}\right)  p\left(  s\right)   &
=\frac{1}{2}\left[  f\left(  \frac{x^{\prime}+y^{\prime}}{2}\right)  +f\left(
\frac{x^{\prime}+y^{\prime}}{2}\right)  \right]  p\left(  s\right) \\
&  \leq\frac{1}{2}\left[  f\left(  \frac{x^{\prime}+s}{2}\right)  +f\left(
\frac{x^{\prime}+2y^{\prime}-s}{2}\right)  \right]  p\left(  s\right)  .
\end{align*}%
\[
\frac{1}{2}\left[  f\left(  \frac{s+x}{2}\right)  +f\left(  \frac{x+2y-s}%
{2}\right)  \right]  p\left(  s\right)  \leq\frac{1}{2}\left[  f\left(
x\right)  +f\left(  y\right)  \right]  p\left(  s\right)  .
\]%
\[
\frac{1}{2}\left[  f\left(  \frac{x+2x^{\prime}-s}{2}\right)  +f\left(
\frac{s-x+2y^{\prime}}{2}\right)  \right]  p\left(  s\right)  \leq\frac{1}%
{2}\left[  f\left(  y^{\prime}\right)  +f\left(  x^{\prime}\right)  \right]
p\left(  s\right)  .
\]

Integrating the above inequalities over $s$ on $\left[  x,y\right]  $ and
using the above identities, we derive $\left(  \ref{2.27}\right)  .$

Again, using simple techniques of integration and the hypothesis of $p$, we
have the following identities%
\begin{align*}
&  \int_{\frac{x+y^{\prime}}{2}}^{\frac{y+y^{\prime}}{2}}2f\left(  s\right)
p\left(  2s-y^{\prime}\right)  ds+\int_{\frac{y+y^{\prime}}{2}}^{\frac
{x^{\prime}+y}{2}}2f\left(  s\right)  p\left(  2s-y\right)  ds\\
&  =\int_{\frac{x+y^{\prime}}{2}}^{\frac{y+y^{\prime}}{2}}2f\left(  s\right)
p\left(  2s-y^{\prime}\right)  ds+\int_{\frac{y+y^{\prime}}{2}}^{\frac
{x^{\prime}+y}{2}}2f\left(  s\right)  p\left(  2y+y^{\prime}-2s\right)  ds\\
&  =\int_{x}^{y}\left[  f\left(  \frac{s+y^{\prime}}{2}\right)  +f\left(
\frac{2y+y^{\prime}-s}{2}\right)  \right]  p\left(  s\right)  ds,
\end{align*}%
\begin{align*}
&  \frac{1}{2}\left[  f\left(  \frac{x+y^{\prime}}{2}\right)  +f\left(
\frac{x^{\prime}+y}{2}\right)  \right]  \int_{\Omega}p\left(  s\right)  ds\\
&  =\int_{x}^{y}\left[  f\left(  \frac{x+y^{\prime}}{2}\right)  +f\left(
\frac{x^{\prime}+y}{2}\right)  \right]  p\left(  s\right)  ds,
\end{align*}%
\begin{align*}
&  \frac{1}{2}\int_{x}^{y}\left[  f\left(  \frac{s+x}{2}\right)  +f\left(
\frac{x+2y^{\prime}-s}{2}\right)  \right]  p\left(  s\right)  ds\\
&  +\frac{1}{2}\int_{y^{\prime}}^{x^{\prime}}\left[  f\left(  \frac
{s+x^{\prime}}{2}\right)  +f\left(  \frac{x^{\prime}+2y-s}{2}\right)  \right]
p\left(  s\right)  ds\\
&  =\int_{x}^{y}\frac{1}{2}\left[  f\left(  \frac{s+x}{2}\right)  +f\left(
\frac{x+2y^{\prime}-s}{2}\right)  \right]  p\left(  s\right)  ds\\
&  +\int_{x}^{y}\frac{1}{2}\left[  f\left(  \frac{x+2x^{\prime}-s}{2}\right)
+f\left(  \frac{s-x+2y}{2}\right)  \right]  p\left(  s\right)  ds
\end{align*}
and%
\begin{align*}
&  \frac{1}{4}\left[  f\left(  x\right)  +f\left(  y\right)  +f\left(
y^{\prime}\right)  +f\left(  x^{\prime}\right)  \right]  \int_{\Omega}p\left(
s\right)  ds\\
&  =\int_{x}^{y}\frac{1}{2}\left[  f\left(  x\right)  +f\left(  y^{\prime
}\right)  \right]  p\left(  s\right)  ds+\int_{x}^{y}\frac{1}{2}\left[
f\left(  y\right)  +f\left(  x^{\prime}\right)  \right]  p\left(  s\right)
ds.
\end{align*}

By Lemma \ref{l2}, the following inequalities hold for all $s\in\left[
x,y\right]  :$%
\[
\left[  f\left(  \frac{s+y^{\prime}}{2}\right)  +f\left(  \frac{2y+y^{\prime
}-s}{2}\right)  \right]  p\left(  s\right)  \leq\left[  f\left(
\frac{x+y^{\prime}}{2}\right)  +f\left(  \frac{x^{\prime}+y}{2}\right)
\right]  p\left(  s\right)  .
\]%
\begin{align*}
f\left(  \frac{x+y^{\prime}}{2}\right)  p\left(  s\right)   &  =\frac{1}%
{2}\left[  f\left(  \frac{x+y^{\prime}}{2}\right)  +f\left(  \frac
{x+y^{\prime}}{2}\right)  \right]  p\left(  s\right) \\
&  \leq\frac{1}{2}\left[  f\left(  \frac{s+x}{2}\right)  +f\left(
\frac{x+2y^{\prime}-s}{2}\right)  \right]  p\left(  s\right)  .
\end{align*}%
\begin{align*}
f\left(  \frac{x^{\prime}+y}{2}\right)  p\left(  s\right)   &  =\frac{1}%
{2}\left[  f\left(  \frac{x^{\prime}+y}{2}\right)  +f\left(  \frac{x^{\prime
}+y}{2}\right)  \right]  p\left(  s\right) \\
&  \leq\frac{1}{2}\left[  f\left(  \frac{x+2x^{\prime}-s}{2}\right)  +f\left(
\frac{s-x+2y}{2}\right)  \right]  p\left(  s\right)  .
\end{align*}%
\[
\frac{1}{2}\left[  f\left(  \frac{s+x}{2}\right)  +f\left(  \frac
{x+2y^{\prime}-s}{2}\right)  \right]  p\left(  s\right)  \leq\frac{1}%
{2}\left[  f\left(  x\right)  +f\left(  y^{\prime}\right)  \right]  p\left(
s\right)  .
\]%
\[
\frac{1}{2}\left[  f\left(  \frac{x+2x^{\prime}-s}{2}\right)  +f\left(
\frac{s-x+2y}{2}\right)  \right]  p\left(  s\right)  \leq\frac{1}{2}\left[
f\left(  y\right)  +f\left(  x^{\prime}\right)  \right]  p\left(  s\right)  .
\]

Integrating the above inequalities over $s$ on $\left[  x,y\right]  $ and
using the above identities, we derive $\left(  \ref{2.28}\right)  .$

$\left(  3\right)  $ Using simple techniques of integration and $x^{\prime
}-y=y^{\prime}-x,$ we have the following identities%
\begin{multline*}
2Hp_{1}\left(  t\right)  =\int_{x}^{y}\left[  f\left(  ts+\left(  1-t\right)
y\right)  +f\left(  t\left(  x+y-s\right)  +\left(  1-t\right)  y\right)
\right. \\
+\left.  f\left(  t\left(  y+y^{\prime}-s\right)  +\left(  1-t\right)
y^{\prime}\right)  +f\left(  t\left(  y^{\prime}-x+s\right)  +\left(
1-t\right)  y^{\prime}\right)  \right]  p\left(  s\right)  ds
\end{multline*}
and%
\begin{multline*}
2Pp_{1}\left(  t\right)  =\int_{x}^{y}\left[  f\left(  tx+\left(  1-t\right)
s\right)  +f\left(  tx+\left(  1-t\right)  \left(  x+y-s\right)  \right)
\right. \\
+\left.  f\left(  tx^{\prime}+\left(  1-t\right)  \left(  x+x^{\prime
}-s\right)  \right)  +f\left(  tx^{\prime}+\left(  1-t\right)  \left(
y^{\prime}-x+s\right)  \right)  \right]  p\left(  s\right)  ds
\end{multline*}
for all $t\in\left[  0,1\right]  .$ Using Lemma \ref{l2} and $x^{\prime
}-y^{\prime}=y-x,$ the following inequalities hold for all $t\in\left[
0,1\right]  $ and $s\in\left[  x,y\right]  :$%
\begin{align*}
&  \left[  f\left(  ts+\left(  1-t\right)  y\right)  +f\left(  t\left(
x+y-s\right)  +\left(  1-t\right)  y\right)  \right]  p\left(  s\right) \\
&  \leq\left[  f\left(  tx+\left(  1-t\right)  y\right)  +f\left(  y\right)
\right]  p\left(  s\right)  ,
\end{align*}%
\begin{multline*}
\left[  f\left(  t\left(  y+y^{\prime}-s\right)  +\left(  1-t\right)
y^{\prime}\right)  +f\left(  t\left(  y^{\prime}-x+s\right)  +\left(
1-t\right)  y^{\prime}\right)  \right]  p\left(  s\right) \\
\leq\left[  f\left(  y^{\prime}\right)  +f\left(  tx^{\prime}+\left(
1-t\right)  y^{\prime}\right)  \right]  p\left(  s\right)  ,
\end{multline*}%
\begin{align*}
&  \left[  f\left(  tx+\left(  1-t\right)  s\right)  +f\left(  tx+\left(
1-t\right)  \left(  x+y-s\right)  \right)  \right]  p\left(  s\right) \\
&  \leq\left[  f\left(  x\right)  +f\left(  tx+\left(  1-t\right)  y\right)
\right]  p\left(  s\right)
\end{align*}
and%
\begin{multline*}
\left[  f\left(  tx^{\prime}+\left(  1-t\right)  \left(  x+x^{\prime
}-s\right)  \right)  +f\left(  tx^{\prime}+\left(  1-t\right)  \left(
y^{\prime}-x+s\right)  \right)  \right]  p\left(  s\right) \\
\leq\left[  f\left(  tx^{\prime}+\left(  1-t\right)  y^{\prime}\right)
+f\left(  x^{\prime}\right)  \right]  p\left(  s\right)  .
\end{multline*}

Integrating the above inequalities over $s$ on $\left[  x,y\right]  $ and
using the above identities, we obtain%
\begin{equation}
2Hp_{1}\left(  t\right)  \leq G_{1}\left(  t\right)  \int_{\Omega}p\left(
s\right)  ds+\frac{f\left(  y\right)  +f\left(  y^{\prime}\right)  }{2}%
\int_{\Omega}p\left(  s\right)  ds \label{2.31}%
\end{equation}
and
\begin{equation}
2Pp_{1}\left(  t\right)  \leq G_{1}\left(  t\right)  \int_{\Omega}p\left(
s\right)  ds+\frac{f\left(  x\right)  +f\left(  x^{\prime}\right)  }{2}%
\int_{\Omega}p\left(  s\right)  ds. \label{2.32}%
\end{equation}

Using $\left(  \ref{2.1}\right)  ,$ $\left(  \ref{2.26}\right)  ,$ $\left(
\ref{2.31}\right)  $ and $\left(  \ref{2.32}\right)  ,$ we derive $\left(
\ref{2.29}\right)  $ and $\left(  \ref{2.30}\right)  .$

This completes the proof.
\end{proof}

\begin{remark}
\label{r9}\textit{Using Remark \ref{r3}} we have the following results:
\end{remark}

\begin{enumerate}
\item \textit{The inequalities }$\left(  \ref{2.26}\right)  ,\left(
\ref{2.27}\right)  $\textit{\ and }$\left(  \ref{2.29}\right)  $%
\textit{\ reduce to the inequalities }$\left(  \ref{1.3}\right)  -\left(
\ref{1.5}\right)  $\textit{\ as }$x=a,$\textit{\ }$y=y^{\prime}=\frac{a+b}%
{2},$\textit{\ }$x^{\prime}=b$ \textit{and }$p\left(  s\right)  \equiv\frac
{1}{b-a}$ $\left(  s\in\left[  a,b\right]  \right)  .$

\item \textit{Theorem \ref{t6}} \textit{reduces to the conclusions} $\left(
3\right)  -\left(  5\right)  $\textit{\ of Theorem \ref{A18} as }$p\left(
s\right)  \equiv\frac{1}{2\left(  y-x\right)  }$\ $\left(  s\in\Omega\right)
. $

\item \textit{The inequalities }$\left(  \ref{2.26}\right)  ,\left(
\ref{2.27}\right)  $\textit{\ and }$\left(  \ref{2.29}\right)  $%
\textit{\ reduce to Theorem \ref{A24} as }$x=a,$ $y=y^{\prime}=\frac{a+b}%
{2},x^{\prime}=b$ \textit{and }$p\left(  s\right)  =g\left(  s\right)
$\ $\left(  s\in\left[  a,b\right]  \right)  .$
\end{enumerate}

\begin{theorem}
\label{t7}Let $x,y,y^{\prime},x^{\prime},\Omega,f,p,Lp_{1},Pp_{1},G_{1}%
,G_{2},Hp_{1},Hp_{2},Fp_{1}$\ be defined as above. Then, we have the following results:

\begin{enumerate}
\item $Lp_{1}$ is convex on $\left[  0,1\right]  $.

\item The following inequalities%
\begin{align}
&  \frac{G_{1}\left(  t\right)  +G_{2}\left(  t\right)  }{2}\int_{\Omega
}p\left(  s\right)  ds\text{
\ \ \ \ \ \ \ \ \ \ \ \ \ \ \ \ \ \ \ \ \ \ \ \ \ \ \ \ \ \ \ }\label{2.33}\\
&  \leq Lp_{1}\left(  t\right)  \leq Pp_{1}\left(  t\right)  \text{
\qquad\ \ \ \ \ \ \ \ \ \ \ \ \ \ \ \ \ \ }\left(  t\in\left[  0,1\right]
\right) \nonumber
\end{align}
and%
\begin{equation}
\sup\limits_{t\in\left[  0,1\right]  }Lp_{1}\left(  t\right)  =Lp_{1}\left(
1\right)  =\frac{f\left(  x\right)  +f\left(  x^{\prime}\right)  }{2}%
\int_{\Omega}p\left(  s\right)  ds \label{2.34}%
\end{equation}
hold.

\item The following inequalities%
\begin{align}
&  \frac{Hp_{1}\left(  1-t\right)  +Hp_{2}\left(  1-t\right)  }{2}\int%
_{\Omega}p\left(  s\right)  ds\text{ \ }\label{2.35}\\
&  \leq Fp_{1}\left(  t\right)  \leq Lp_{1}\left(  t\right)  \int_{\Omega
}p\left(  s\right)  ds,\nonumber
\end{align}%
\begin{equation}
\frac{Hp_{1}\left(  t\right)  +Hp_{2}\left(  t\right)  }{2}\int_{\Omega
}p\left(  s\right)  ds\leq Fp_{1}\left(  t\right)  \leq Lp_{1}\left(
t\right)  \int_{\Omega}p\left(  s\right)  ds \label{2.36}%
\end{equation}
and%
\begin{align}
&  \frac{Hp_{1}\left(  t\right)  +Hp_{2}\left(  t\right)  +Hp_{1}\left(
1-t\right)  +Hp_{2}\left(  1-t\right)  }{4}\int_{\Omega}p\left(  s\right)
ds\label{2.37}\\
&  \leq Fp_{1}\left(  t\right)  \leq Lp_{1}\left(  t\right)  \int_{\Omega
}p\left(  s\right)  ds\nonumber
\end{align}
hold for all $t\in\left[  0,1\right]  .$

\item If $p\left(  s\right)  =p\left(  x+x^{\prime}-s\right)  $ $\left(
s\in\left[  x,y\right]  \right)  $ and $p\left(  s\right)  =p\left(
x+y-s\right)  $ $\left(  s\in\left[  x,\frac{x+y}{2}\right]  \right)  ,$ then
the following inequality%
\begin{align}
0  &  \leq Fp_{1}\left(  t\right)  -\frac{Hp_{1}\left(  t\right)
+Hp_{2}\left(  t\right)  }{2}\int_{\Omega}p\left(  s\right)  ds\label{2.38}\\
&  \leq Lp_{1}\left(  1-t\right)  \int_{\Omega}p\left(  s\right)
ds-Fp_{1}\left(  t\right) \nonumber
\end{align}
holds for all $t\in\left[  0,1\right]  .$
\end{enumerate}
\end{theorem}

\begin{proof}
$\left(  1\right)  $ It is easily observed from the convexity of $f$ and the
hypothesis of $p$ that $Lp_{1}$ is convex on $\left[  0,1\right]  .$

$\left(  2\right)  $ Using simple techniques of integration, we have the
following identities%
\begin{align*}
&  \frac{G_{1}\left(  t\right)  +G_{2}\left(  t\right)  }{2}\int_{\Omega
}p\left(  s\right)  ds\\
&  =\frac{1}{2}\int_{x}^{y}\left[  f\left(  tx+\left(  1-t\right)  y\right)
+f\left(  tx+\left(  1-t\right)  y^{\prime}\right)  \right. \\
&  +\left.  f\left(  tx^{\prime}+\left(  1-t\right)  y\right)  +f\left(
tx^{\prime}+\left(  1-t\right)  y^{\prime}\right)  \right]  p\left(  s\right)
ds
\end{align*}
and%
\begin{align*}
Lp_{1}\left(  t\right)   &  =\frac{1}{2}\int_{x}^{y}\left[  f\left(
tx+\left(  1-t\right)  s\right)  +f\left(  tx+\left(  1-t\right)  \left(
x+x^{\prime}-s\right)  \right)  \right. \\
&  \qquad\qquad\left.  +f\left(  tx^{\prime}+\left(  1-t\right)  s\right)
+f\left(  tx^{\prime}+\left(  1-t\right)  \left(  x+x^{\prime}-s\right)
\right)  \right]  p\left(  s\right)  ds\\
&  =\frac{1}{2}Pp_{1}\left(  t\right)  +\frac{1}{2}\int_{x}^{y}\left[
f\left(  tx^{\prime}+\left(  1-t\right)  s\right)  \right. \\
&  \qquad\qquad+\left.  f\left(  tx+\left(  1-t\right)  \left(  x+x^{\prime
}-s\right)  \right)  \right]  p\left(  s\right)  ds
\end{align*}
for all $t\in\left[  0,1\right]  .$

By Lemma \ref{l2}, the following inequalities hold for all $t\in\left[
0,1\right]  $ and $s\in\left[  x,y\right]  :$%
\begin{align*}
&  \frac{1}{2}\left[  f\left(  tx+\left(  1-t\right)  y\right)  +f\left(
tx+\left(  1-t\right)  y^{\prime}\right)  \right]  p\left(  s\right) \\
&  \leq\frac{1}{2}\left[  f\left(  tx+\left(  1-t\right)  s\right)  +f\left(
tx+\left(  1-t\right)  \left(  x+x^{\prime}-s\right)  \right)  \right]
p\left(  s\right)  ,
\end{align*}%
\begin{align*}
&  \frac{1}{2}\left[  f\left(  tx^{\prime}+\left(  1-t\right)  y\right)
+f\left(  tx^{\prime}+\left(  1-t\right)  y^{\prime}\right)  \right]  p\left(
s\right) \\
&  \leq\frac{1}{2}\left[  f\left(  tx^{\prime}+\left(  1-t\right)  s\right)
+f\left(  tx^{\prime}+\left(  1-t\right)  \left(  x+x^{\prime}-s\right)
\right)  \right]  p\left(  s\right)
\end{align*}
and%
\begin{multline*}
\frac{1}{2}\left[  f\left(  tx^{\prime}+\left(  1-t\right)  s\right)
+f\left(  tx+\left(  1-t\right)  \left(  x+x^{\prime}-s\right)  \right)
\right]  p\left(  s\right) \\
\leq\frac{1}{2}\left[  f\left(  tx+\left(  1-t\right)  s\right)  +f\left(
tx^{\prime}+\left(  1-t\right)  \left(  x+x^{\prime}-s\right)  \right)
\right]  p\left(  s\right)  .
\end{multline*}

Integrating the above inequalities over $s$ on $\left[  x,y\right]  $ and
using the above identities and $\left(  \ref{2.12}\right)  $, we derive
$\left(  \ref{2.33}\right)  $ and $\left(  \ref{2.34}\right)  .$

$\left(  3\right)  $ Using simple techniques of integration, the symmetry of
$Fp_{1}$ and the hypothesis of $p$, we have the following identity%
\[
Fp_{1}\left(  t\right)  =\int_{\Omega}\int_{x}^{y}\left[  f\left(  tu+\left(
1-t\right)  s\right)  +f\left(  t\left(  x+x^{\prime}-u\right)  +\left(
1-t\right)  s\right)  \right]  p\left(  s\right)  p\left(  u\right)  duds
\]
for all $t\in\left[  0,1\right]  .$

By Lemma \ref{l2}, the following inequality holds for all $t\in\left[
0,1\right]  ,u\in\left[  x,y\right]  $ and $s\in\Omega.$%
\begin{align*}
&  \left[  f\left(  tu+\left(  1-t\right)  s\right)  +f\left(  t\left(
x+x^{\prime}-u\right)  +\left(  1-t\right)  s\right)  \right]  p\left(
s\right)  p\left(  u\right) \\
&  \leq\left[  f\left(  tx+\left(  1-t\right)  s\right)  +f\left(  tx^{\prime
}+\left(  1-t\right)  s\right)  \right]  p\left(  s\right)  p\left(  u\right)
.
\end{align*}

Integrating the above inequality over $u$ on $\left[  x,y\right]  $, over $s$
on $\Omega$ and using the above identity, simple techniques of integration and
the definition of $Lp_{1}$, we obtain%
\begin{equation}
Fp_{1}\left(  t\right)  \leq Lp_{1}\left(  t\right)  \int_{\Omega}p\left(
s\right)  ds \label{2.39}%
\end{equation}
for all $t\in\left[  0,1\right]  .$ Using $\left(  \ref{2.21}\right)  ,$
$\left(  \ref{2.39}\right)  $ and the symmetry of $Fp_{1}$, we derive $\left(
\ref{2.35}\right)  -\left(  \ref{2.37}\right)  .$

$\left(  4\right)  $ Using simple techniques of integration, $y^{\prime
}-x=x^{\prime}-y$ and the hypothesis of $p$, we have the following identity%
\begin{multline*}
Fp_{1}\left(  t\right)  =\int_{\Omega}\int_{x}^{\frac{x+y}{2}}\left[  f\left(
ts+\left(  1-t\right)  u\right)  +f\left(  ts+\left(  1-t\right)  \left(
x+y-u\right)  \right)  \right. \\
+\left.  f\left(  ts+\left(  1-t\right)  \left(  x^{\prime}-y+u\right)
\right)  +f\left(  ts+\left(  1-t\right)  \left(  y+y^{\prime}-u\right)
\right)  \right]  p\left(  s\right)  p\left(  u\right)  duds
\end{multline*}
for all $t\in\left[  0,1\right]  .$

By Lemma \ref{l2}, the following inequalities hold for all $t\in\left[
0,1\right]  ,$ $s\in\Omega$ and $u\in\left[  x,\frac{x+y}{2}\right]  :$%
\begin{align*}
&  \left[  f\left(  ts+\left(  1-t\right)  u\right)  +f\left(  ts+\left(
1-t\right)  \left(  x+y-u\right)  \right)  \right]  p\left(  s\right)
p\left(  u\right) \\
&  \leq\left[  f\left(  ts+\left(  1-t\right)  x\right)  +f\left(  ts+\left(
1-t\right)  y\right)  \right]  p\left(  s\right)  p\left(  u\right)
\end{align*}
and%
\begin{multline*}
\left[  f\left(  ts+\left(  1-t\right)  \left(  x^{\prime}-y+u\right)
\right)  +f\left(  ts+\left(  1-t\right)  \left(  y+y^{\prime}-u\right)
\right)  \right]  p\left(  s\right)  p\left(  u\right) \\
\leq\left[  f\left(  ts+\left(  1-t\right)  y^{\prime}\right)  +f\left(
ts+\left(  1-t\right)  x^{\prime}\right)  \right]  p\left(  s\right)  p\left(
u\right)  .
\end{multline*}

Integrating the above inequalities over $s$ on $\Omega,$ over $u$ on $\left[
x,\frac{x+y}{2}\right]  $ and using the above identity and the definitions of
$Hp_{1},Hp_{2}$ and $Lp_{1}$, we obtain%
\begin{equation}
Fp_{1}\left(  t\right)  \leq\frac{1}{2}\left[  Lp_{1}\left(  1-t\right)
+\frac{Hp_{1}\left(  t\right)  +Hp_{2}\left(  t\right)  }{2}\right]
\int_{\Omega}p\left(  s\right)  ds \label{2.40}%
\end{equation}
for all $t\in\left[  0,1\right]  .$ Using $\left(  \ref{2.36}\right)  $ and
$\left(  \ref{2.40}\right)  ,$ we derive $\left(  \ref{2.38}\right)  .$

This completes the proof.
\end{proof}

\begin{remark}
\label{r10}\textit{Using Remark \ref{r3}} we have the following results:
\end{remark}

\begin{enumerate}
\item \textit{Theorems \ref{t4}, \ref{t7}\ and the inequalities }$\left(
\ref{2.12}\right)  ,$ $\left(  \ref{2.30}\right)  $ and $\left(
\ref{2.33}\right)  $\textit{\ reduce to Theorem \ref{A9} as }$x=a,$%
\textit{\ }$y=y^{\prime}=\frac{a+b}{2},$\textit{\ }$x^{\prime}=b$ \textit{and
}$p\left(  s\right)  \equiv\frac{1}{b-a}$ $\left(  s\in\left[  a,b\right]
\right)  .$

\item \textit{The inequality }$\left(  \ref{2.38}\right)  $\textit{\ reduces
to Theorem \ref{A11} as }$x=a,$\textit{\ }$y=y^{\prime}=\frac{a+b}{2}%
,$\textit{\ }$x^{\prime}=b$ \textit{and }$p\left(  s\right)  \equiv\frac
{1}{b-a}$ $\left(  s\in\left[  a,b\right]  \right)  .$

\item \textit{Theorem \ref{t7}} \textit{reduces to \ref{A19} as }$p\left(
s\right)  \equiv\frac{1}{2\left(  y-x\right)  }$\ $\left(  s\in\Omega\right)
. $

\item \textit{The conclusions }$\left(  1\right)  -\left(  3\right)  $
\textit{of Theorem \ref{t7} and the inequalities }$\left(  \ref{2.12}\right)
$\textit{\ reduce to Theorem \ref{A25} as }$x=a,$ $y=y^{\prime}=\frac{a+b}%
{2},x^{\prime}=b$ \textit{and }$p\left(  s\right)  \equiv g\left(  s\right)
$\ $\left(  s\in\left[  a,b\right]  \right)  .$
\end{enumerate}

The following corollary is a natural consequence of Theorems \ref{t1} --
\ref{t7}.

\begin{corollary}
\label{c1}The following inequalities hold for all $t\in\left[  0,1\right]  $%
\begin{align*}
\frac{f\left(  y\right)  +f\left(  y^{\prime}\right)  }{2}\int_{\Omega
}p\left(  s\right)  ds  &  \leq Hp_{1}\left(  t\right)  \leq\int_{\Omega
}f\left(  s\right)  p\left(  s\right)  ds\ \\
&  \leq Pp_{1}\left(  t\right)  \leq\frac{f\left(  x\right)  +f\left(
x^{\prime}\right)  }{2}\int_{\Omega}p\left(  s\right)  ds;
\end{align*}%
\begin{align*}
&  \frac{1}{2}\left[  f\left(  \frac{y+y^{\prime}}{2}\right)  +\frac{f\left(
y\right)  +f\left(  y^{\prime}\right)  }{2}\right]  \left[  \int_{\Omega
}p\left(  s\right)  ds\right]  ^{2}\\
&  \leq\frac{Hp_{1}\left(  t\right)  +Hp_{2}\left(  t\right)  }{2}\int%
_{\Omega}p\left(  s\right)  ds\leq Fp_{1}\left(  t\right)  ;
\end{align*}%
\begin{align*}
Fp_{1}\left(  t\right)   &  \leq\int_{\Omega}f\left(  s\right)  p\left(
s\right)  ds\ \int_{\Omega}p\left(  s\right)  ds\\
&  \leq Pp_{1}\left(  t\right)  \int_{\Omega}p\left(  s\right)  ds\leq
\frac{f\left(  x\right)  +f\left(  x^{\prime}\right)  }{2}\left[  \int%
_{\Omega}p\left(  s\right)  ds\right]  ^{2};
\end{align*}%
\begin{align*}
&  \frac{1}{2}\left[  f\left(  \frac{y+y^{\prime}}{2}\right)  +\frac{f\left(
y\right)  +f\left(  y^{\prime}\right)  }{2}\right]  \left[  \int_{\Omega
}p\left(  s\right)  ds\right]  ^{2}\\
&  \leq\frac{Hp_{1}\left(  1-t\right)  +Hp_{2}\left(  1-t\right)  }{2}%
\int_{\Omega}p\left(  s\right)  ds\leq Fp_{1}\left(  t\right)  ;
\end{align*}%
\begin{align*}
Fp_{1}\left(  t\right)   &  \leq\frac{1}{2}\left[  \frac{Hp_{1}\left(
1-t\right)  +Hp_{2}\left(  1-t\right)  }{2}+Lp_{1}\left(  t\right)  \right]
\int_{\Omega}p\left(  s\right)  ds\\
&  \leq Lp_{1}\left(  t\right)  \int_{\Omega}p\left(  s\right)  ds\leq
Pp_{1}\left(  t\right)  \int_{\Omega}p\left(  s\right)  ds\\
&  \leq\frac{f\left(  x\right)  +f\left(  x^{\prime}\right)  }{2}\left[
\int_{\Omega}p\left(  s\right)  ds\right]  ^{2};
\end{align*}%
\begin{align*}
&  \frac{1}{2}\left[  f\left(  \frac{y+y^{\prime}}{2}\right)  +\frac{f\left(
y\right)  +f\left(  y^{\prime}\right)  }{2}\right]  \int_{\Omega}p\left(
s\right)  ds\\
&  \leq\frac{G_{1}\left(  t\right)  +G_{2}\left(  t\right)  }{2}\int_{\Omega
}p\left(  s\right)  ds\leq Lp_{1}\left(  t\right)  ;
\end{align*}%
\begin{align*}
\frac{f\left(  y\right)  +f\left(  y^{\prime}\right)  }{2}\int_{\Omega
}p\left(  s\right)  ds  &  \leq Hp_{1}\left(  t\right) \\
&  \leq\frac{1}{2}\left[  \frac{f\left(  y\right)  +f\left(  y^{\prime
}\right)  }{2}+G_{1}\left(  t\right)  \right]  \int_{\Omega}p\left(  s\right)
ds\\
&  \leq G_{1}\left(  t\right)  \int_{\Omega}p\left(  s\right)  ds
\end{align*}
and%
\begin{align*}
G_{1}\left(  t\right)  \int_{\Omega}p\left(  s\right)  ds  &  \leq
Pp_{1}\left(  t\right) \\
&  \leq\frac{1}{2}\left[  G_{1}\left(  t\right)  +\frac{f\left(  x\right)
+f\left(  x^{\prime}\right)  }{2}\right]  \int_{\Omega}p\left(  s\right)  ds\\
&  \leq\frac{f\left(  x\right)  +f\left(  x^{\prime}\right)  }{2}\int_{\Omega
}p\left(  s\right)  ds.
\end{align*}

\end{corollary}

\subsection{New Fejer-type Inequalities II}

Throughout this subsection, let $x,y,y^{\prime},x^{\prime},$ $\Omega,f,g,p,$
$H,H_{1},H_{2},Hg,Hp_{1},Hp_{2},P,P_{1},Pg,Pp_{1},$ $F,F_{1},Fg,Fp_{1},$
$L,L_{1},Lg,Lp_{1},$ $Q,$ $G,G_{1},G_{2},Ig,Jg,Mg,$ $Ng,Kg,Sg$\ be defined as
above and let $y=\left(  1-\alpha\right)  x+\alpha x^{\prime},$ $y^{\prime
}=\alpha x+\left(  1-\alpha\right)  x^{\prime}$ with $0<\alpha\leq\frac{1}{2}$
and define the following functions on $[0,1].$%
\[
Q_{1}\left(  t\right)  =\frac{1}{2}\left[  f\left(  tx+\left(  1-t\right)
x^{\prime}\right)  +f\left(  tx^{\prime}+\left(  1-t\right)  x\right)
\right]  ,
\]%
\begin{align*}
Ip_{1}\left(  t\right)   &  =%
{\displaystyle\int\nolimits_{x}^{x^{\prime}}}
\left[  f\left(  t\left(  \left(  1-\alpha\right)  x+\alpha s\right)  +\left(
1-t\right)  y\right)  \right. \\
&  \text{ \ \ \ \ \ \ \ \ }+\left.  f\left(  t\left(  \left(  1-\alpha\right)
x^{\prime}+\alpha s\right)  +\left(  1-t\right)  y^{\prime}\right)  \right]
p\left(  s\right)  ds,
\end{align*}%
\begin{align*}
Ip_{2}\left(  t\right)   &  =%
{\displaystyle\int\nolimits_{x}^{x^{\prime}}}
\left[  f\left(  t\left(  \left(  1-\alpha\right)  x+\alpha s\right)  +\left(
1-t\right)  y^{\prime}\right)  \right. \\
&  \text{ \ \ \ \ \ \ \ \ }+\left.  f\left(  t\left(  \left(  1-\alpha\right)
x^{\prime}+\alpha s\right)  +\left(  1-t\right)  y\right)  \right]  p\left(
s\right)  ds,
\end{align*}%
\begin{align*}
Jp\left(  t\right)   &  =%
{\displaystyle\int\nolimits_{x}^{x^{\prime}}}
\left[  f\left(  t\left(  \left(  1-\alpha\right)  x+\alpha s\right)  +\left(
1-t\right)  \frac{x+y}{2}\right)  \right. \\
&  \text{ \ \ \ \ \ \ }+\left.  f\left(  t\left(  \left(  1-\alpha\right)
x^{\prime}+\alpha s\right)  +\left(  1-t\right)  \frac{x^{\prime}+y^{\prime}%
}{2}\right)  \right]  p\left(  s\right)  ds,
\end{align*}%
\begin{align*}
Mp\left(  t\right)   &  =%
{\displaystyle\int\nolimits_{x}^{\frac{x+x^{\prime}}{2}}}
\left[  f\left(  tx+\left(  1-t\right)  \left(  \left(  1-\alpha\right)
x+\alpha s\right)  \right)  \right. \\
&  \text{ \ \ \ \ \ \ \ \ }+\left.  f\left(  ty^{\prime}+\left(  1-t\right)
\left(  \left(  1-\alpha\right)  x^{\prime}+\alpha s\right)  \right)  \right]
p\left(  s\right)  ds\\
&  +%
{\displaystyle\int\nolimits_{\frac{x+x^{\prime}}{2}}^{x^{\prime}}}
\left[  f\left(  ty+\left(  1-t\right)  \left(  \left(  1-\alpha\right)
x+\alpha s\right)  \right)  \right. \\
&  \text{ \ \ \ \ \ \ \ \ }+\left.  f\left(  tx^{\prime}+\left(  1-t\right)
\left(  \left(  1-\alpha\right)  x^{\prime}+\alpha s\right)  \right)  \right]
p\left(  s\right)  ds,
\end{align*}%
\begin{align*}
Np\left(  t\right)   &  =%
{\displaystyle\int\nolimits_{x}^{x^{\prime}}}
\left[  f\left(  tx+\left(  1-t\right)  \left(  \left(  1-\alpha\right)
x+\alpha s\right)  \right)  \right. \\
&  \text{ \ \ \ \ \ \ \ }+\left.  f\left(  tx^{\prime}+\left(  1-t\right)
\left(  \left(  1-\alpha\right)  x^{\prime}+\alpha s\right)  \right)  \right]
p\left(  s\right)  ds,
\end{align*}%
\begin{align*}
Sp\left(  t\right)   &  =%
{\displaystyle\int\nolimits_{x}^{x^{\prime}}}
\frac{1}{2}\left[  f\left(  tx+\left(  1-t\right)  \left(  \left(
1-\alpha\right)  x+\alpha s\right)  \right)  \right. \\
&  \text{ \ \ \ }+f\left(  tx+\left(  1-t\right)  \left(  \left(
1-\alpha\right)  x^{\prime}+\alpha s\right)  \right) \\
&  \text{ \ \ \ }+f\left(  tx^{\prime}+\left(  1-t\right)  \left(  \left(
1-\alpha\right)  x+\alpha s\right)  \right) \\
&  \text{ \ \ \ }+\left.  f\left(  tx^{\prime}+\left(  1-t\right)  \left(
\left(  1-\alpha\right)  x^{\prime}+\alpha s\right)  \right)  \right]
p\left(  s\right)  ds
\end{align*}
and%
\begin{align*}
&  Kp\left(  t\right) \\
&  =%
{\displaystyle\int\nolimits_{x}^{x^{\prime}}}
{\displaystyle\int\nolimits_{x}^{x^{\prime}}}
\left[  f\left(  t\left(  \left(  1-\alpha\right)  x+\alpha s\right)  +\left(
1-t\right)  \left(  \left(  1-\alpha\right)  x+\alpha u\right)  \right)
\right. \\
&  \text{ \ \ \ \ \ \ \ \ }+f\left(  t\left(  \left(  1-\alpha\right)
x+\alpha s\right)  +\left(  1-t\right)  \left(  \left(  1-\alpha\right)
x^{\prime}+\alpha u\right)  \right) \\
&  \text{ \ \ \ \ \ \ \ \ }+f\left(  t\left(  \left(  1-\alpha\right)
x^{\prime}+\alpha s\right)  +\left(  1-t\right)  \left(  \left(
1-\alpha\right)  x+\alpha u\right)  \right) \\
&  \text{ \ \ \ \ \ \ \ \ }+\left.  f\left(  t\left(  \left(  1-\alpha\right)
x^{\prime}+\alpha s\right)  +\left(  1-t\right)  \left(  \left(
1-\alpha\right)  x^{\prime}+\alpha u\right)  \right)  \right]  p\left(
s\right)  p\left(  u\right)  dsdu.
\end{align*}

\begin{remark}
\label{r11}\textit{We note that :}
\end{remark}

\begin{enumerate}
\item $Ip_{1}\left(  t\right)  =H_{1}\left(  t\right)  ,$ $Ip_{2}\left(
t\right)  =H_{2}\left(  t\right)  ,$ $Kp\left(  t\right)  =F_{1}\left(
t\right)  ,$ $Np\left(  t\right)  =P_{1}\left(  t\right)  ,$ $Sp\left(
t\right)  =L_{1}\left(  t\right)  $\textit{\ on }$\left[  0,1\right]
$\textit{\ as }$p\left(  s\right)  \equiv\frac{\alpha}{2\left(  y-x\right)  }$
\ $\left(  s\in\left[  x,x^{\prime}\right]  \right)  .$

\item $\Omega=\left[  a,b\right]  $\textit{\ and }$Q_{1}\left(  t\right)
=Q\left(  t\right)  ,$ $Ip_{1}\left(  t\right)  =Ip_{2}\left(  t\right)
=H\left(  t\right)  ,$ $Kp\left(  t\right)  =F\left(  t\right)  ,$ $Np\left(
t\right)  =P\left(  t\right)  ,$ $Sp\left(  t\right)  =L\left(  t\right)
$\textit{\ on }$\left[  0,1\right]  $\textit{\ as }$x=a,$ $y=y^{\prime}%
=\frac{a+b}{2},x^{\prime}=b$ \textit{and} $p\left(  s\right)  \equiv\frac
{1}{2\left(  b-a\right)  }$ $\left(  s\in\left[  a,b\right]  \right)  .$

\item $\Omega=\left[  a,b\right]  $\textit{\ and }$Ip_{1}\left(  t\right)
=Ip_{2}\left(  t\right)  =Ig\left(  t\right)  ,$ $Jp\left(  t\right)
=Jg\left(  t\right)  ,$ $Mp\left(  t\right)  =Mg\left(  t\right)  ,$
$Np\left(  t\right)  =Pg\left(  t\right)  ,$ $Sp\left(  t\right)  =Sg\left(
t\right)  ,$ $Kp\left(  t\right)  =Kg\left(  t\right)  $\textit{\ on }$\left[
0,1\right]  $\textit{\ as }$x=a,$ $y=y^{\prime}=\frac{a+b}{2},x^{\prime}=b$
\textit{and }$p\left(  s\right)  =\frac{g\left(  s\right)  }{2}$ $\left(
s\in\left[  a,b\right]  \right)  .$
\end{enumerate}

\begin{theorem}
\label{t8}\textit{Let }$x,y,y^{\prime},x^{\prime},f,p,Ip_{1}$ \textit{be
defined as above.\ }Then we have the following results:

\begin{enumerate}
\item $Ip_{1}$ is convex on $\left[  0,1\right]  .$

\item $Ip_{1}$\ is increasing on $\left[  0,1\right]  $\ and the following
\textit{Fej\'{e}r-type} inequalities%
\begin{align}
&  \left[  f\left(  y\right)  +f\left(  y^{\prime}\right)  \right]
\int\nolimits_{x}^{x^{\prime}}p\left(  s\right)  ds\label{2.41}\\
&  =Ip_{1}\left(  0\right)  \leq Ip_{1}\left(  t\right)  \leq Ip_{1}\left(
1\right) \nonumber\\
&  =\int\nolimits_{x}^{x^{\prime}}\left[  f\left(  \left(  1-\alpha\right)
x+\alpha s\right)  +f\left(  \left(  1-\alpha\right)  x^{\prime}+\alpha
s\right)  \right]  p\left(  s\right)  ds\nonumber
\end{align}
and%
\begin{align}
Ip_{1}\left(  t\right)   &  \leq\left(  1-t\right)  \left[  f\left(  y\right)
+f\left(  y^{\prime}\right)  \right]  \int\nolimits_{x}^{x^{\prime}}p\left(
s\right)  ds\label{2.42}\\
&  +t\int\nolimits_{x}^{x^{\prime}}\left[  f\left(  \left(  1-\alpha\right)
x+\alpha s\right)  +f\left(  \left(  1-\alpha\right)  x^{\prime}+\alpha
s\right)  \right]  p\left(  s\right)  ds\nonumber\\
&  \leq\int\nolimits_{x}^{x^{\prime}}\left[  f\left(  \left(  1-\alpha\right)
x+\alpha s\right)  +f\left(  \left(  1-\alpha\right)  x^{\prime}+\alpha
s\right)  \right]  p\left(  s\right)  ds\nonumber\\
&  \leq\left[  f\left(  x\right)  +f\left(  x^{\prime}\right)  \right]
\int\nolimits_{x}^{x^{\prime}}p\left(  s\right)  ds\nonumber
\end{align}
hold for all $t\in\left[  0,1\right]  .$
\end{enumerate}
\end{theorem}

\begin{proof}
$\left(  1\right)  $ It is easily observed from the convexity of $f$ and the
hypothesis of $p$ that $Ip_{1}$ is convex on $\left[  0,1\right]  .$

$\left(  2\right)  $ Using simple integration techniques and under the
hypothesis of $p$, the following identity holds on $\left[  0,1\right]  ,$%
\begin{align}
Ip_{1}\left(  t\right)   &  =%
{\displaystyle\int\nolimits_{x}^{x^{\prime}}}
\left[  f\left(  t\left(  \left(  1-\alpha\right)  x+\alpha s\right)  +\left(
1-t\right)  y\right)  p\left(  s\right)  \right. \label{2.43}\\
&  \left.  +f\left(  t\left(  \alpha x+x^{\prime}-\alpha s\right)  +\left(
1-t\right)  y^{\prime}\right)  p\left(  x+x^{\prime}-x\right)  \right]
ds\nonumber\\
&  =%
{\displaystyle\int\nolimits_{x}^{x^{\prime}}}
\left[  f\left(  t\left(  \left(  1-\alpha\right)  x+\alpha s\right)  +\left(
1-t\right)  y\right)  \right. \nonumber\\
&  \left.  +f\left(  t\left(  \alpha x+x^{\prime}-\alpha s\right)  +\left(
1-t\right)  y^{\prime}\right)  \right]  p\left(  s\right)  ds\nonumber\\
&  =%
{\displaystyle\int\nolimits_{x}^{y}}
\frac{1}{\alpha}\left[  f\left(  ts+\left(  1-t\right)  y\right)  \right.
\nonumber\\
&  +\left.  f\left(  t\left(  x+x^{\prime}-s\right)  +\left(  1-t\right)
y^{\prime}\right)  \right]  p\left(  \frac{1}{\alpha}\left(  s-x\right)
+x\right)  ds.\nonumber
\end{align}

Let $t_{1}<t_{2}$ in $\left[  0,1\right]  .$ By Lemma \ref{l2}, the following
inequality holds for all $s\in\left[  x,y\right]  $:%
\begin{align*}
&  \frac{1}{\alpha}\left[  f\left(  t_{1}s+\left(  1-t_{1}\right)  y\right)
+f\left(  t_{1}\left(  x+x^{\prime}-s\right)  +\left(  1-t_{1}\right)
y^{\prime}\right)  \right]  p\left(  \frac{1}{\alpha}\left(  s-x\right)
+x\right) \\
&  \leq\frac{1}{\alpha}\left[  f\left(  t_{2}s+\left(  1-t_{2}\right)
y\right)  +f\left(  t_{2}\left(  x+x^{\prime}-s\right)  +\left(
1-t_{2}\right)  y^{\prime}\right)  \right]  p\left(  \frac{1}{\alpha}\left(
s-x\right)  +x\right)  .
\end{align*}

Integrating the above inequality over $s$ on $\left[  x,y\right]  $ and using
the above identity, we derive $Ip_{1}\left(  t_{1}\right)  \leq Ip_{1}\left(
t_{2}\right)  .$ Thus $Ip_{1}$ is increasing on $\left[  0,1\right]  $ and
then the inequality $\left(  \ref{2.41}\right)  $ holds. Using the convexity
of $f,$ the inequality $\left(  \ref{2.41}\right)  $ and the substitution rule
for integration, we obtain the first and second inequalities of $\left(
\ref{2.42}\right)  .$ Using simple techniques of integration and the
hypothesis of $p$, we have the following identitiy%
\begin{align*}
&  \int\nolimits_{x}^{x^{\prime}}\left[  f\left(  \left(  1-\alpha\right)
x+\alpha s\right)  +f\left(  \left(  1-\alpha\right)  x^{\prime}+\alpha
s\right)  \right]  p\left(  s\right)  ds\\
&  =\int\nolimits_{x}^{x^{\prime}}\left[  f\left(  \left(  1-\alpha\right)
x+\alpha s\right)  +f\left(  \left(  1-\alpha\right)  x^{\prime}+\alpha\left(
x+x^{\prime}-s\right)  \right)  \right]  p\left(  s\right)  ds.
\end{align*}

By Lemma \ref{l2}, the inequality%
\begin{align*}
&  \left[  f\left(  \left(  1-\alpha\right)  x+\alpha s\right)  +f\left(
\left(  1-\alpha\right)  x^{\prime}+\alpha\left(  x+x^{\prime}-s\right)
\right)  \right]  p\left(  s\right) \\
&  \leq\left[  f\left(  x\right)  +f\left(  x^{\prime}\right)  \right]
p\left(  s\right)
\end{align*}
holds for all $s\in\left[  x,x^{\prime}\right]  $. Integrating the above
inequality over $s$ on $\left[  x,x^{\prime}\right]  $ and using the above
identity, we derive the last inequality of $\left(  \ref{2.42}\right)  .$

This completes the proof.
\end{proof}

\begin{remark}
\label{r12}\textit{Using Remark \ref{r11}} we have the following results:
\end{remark}

\begin{enumerate}
\item \textit{Theorem \ref{t8}} \textit{reduces to Theorem \ref{A3} as }$x=a,
$\textit{\ }$y=y^{\prime}=\frac{a+b}{2},$\textit{\ }$x^{\prime}=b$ \textit{and
}$p\left(  s\right)  \equiv\frac{1}{2\left(  b-a\right)  }$ $\left(
s\in\left[  a,b\right]  \right)  .$

\item \textit{Theorem \ref{t8}} \textit{reduces to Theorem \ref{A12} as
}$p\left(  s\right)  \equiv\frac{\alpha}{2\left(  y-x\right)  }$\ $\left(
s\in\Omega\right)  .$

\item \textit{Theorem \ref{t8}} \textit{reduces to Theorem \ref{A26} as
}$x=a,$ $y=y^{\prime}=\frac{a+b}{2},x^{\prime}=b$ \textit{and }$p\left(
s\right)  =\frac{g\left(  s\right)  }{2}$\ $\left(  s\in\left[  a,b\right]
\right)  .$
\end{enumerate}

\begin{theorem}
\label{t9}\textit{Let }$x,y,y^{\prime},x^{\prime},f,p,Ip_{1},Ip_{2}%
$\textit{\ be defined as above and let}%
\[
m=\frac{y^{\prime}-y}{2\left(  y^{\prime}-x\right)  }\text{ \ and }m^{\prime
}=\left\{
\begin{array}
[c]{cc}%
\frac{1}{2}, & y\neq y^{\prime}\\
0, & y=y^{\prime}%
\end{array}
\right.  .
\]
\textit{\ }Then we have the following results:

\begin{enumerate}
\item $Ip_{2}$\ is convex on $\left[  0,1\right]  .$

\item $Ip_{2}$\ is decreasing on $\left[  0,m\right]  $\ and increasing on
$\left[  m^{\prime},1\right]  $\ and\ the following \textit{Fej\'{e}r-type}
inequality%
\begin{align}
&  2f\left(  \frac{x+x^{\prime}}{2}\right)  \int\nolimits_{x}^{x^{\prime}%
}p\left(  s\right)  ds\label{2.44}\\
&  \leq Ip_{2}\left(  t\right)  \leq\left(  1-t\right)  \left[  f\left(
y\right)  +f\left(  y^{\prime}\right)  \right]  \int\nolimits_{x}^{x^{\prime}%
}p\left(  s\right)  ds\nonumber\\
&  +t\int\nolimits_{x}^{x^{\prime}}\left[  f\left(  \left(  1-\alpha\right)
x+\alpha s\right)  +f\left(  \left(  1-\alpha\right)  x^{\prime}+\alpha
s\right)  \right]  p\left(  s\right)  ds\nonumber\\
&  \leq\int\nolimits_{x}^{x^{\prime}}\left[  f\left(  \left(  1-\alpha\right)
x+\alpha s\right)  +f\left(  \left(  1-\alpha\right)  x^{\prime}+\alpha
s\right)  \right]  p\left(  s\right)  ds\nonumber\\
&  \leq\left[  f\left(  x\right)  +f\left(  x^{\prime}\right)  \right]
\int\nolimits_{x}^{x^{\prime}}p\left(  s\right)  ds\nonumber
\end{align}
holds for all $t\in\left[  0,1\right]  .$

\item The inequality%
\begin{equation}
Ip_{2}\left(  t\right)  \leq Ip_{1}\left(  t\right)  \label{2.45}%
\end{equation}
holds for all $t\in\left[  0,1\right]  .$
\end{enumerate}
\end{theorem}

\begin{proof}
$\left(  1\right)  $ It is easily observed from the convexity of $f$ and the
hypothesis of $p$ that $Ip_{2}$ is convex on $\left[  0,1\right]  .$

$\left(  2\right)  $ Using simple integration techniques and under the
hypothesis of $p$, the following identity holds on $\left[  0,1\right]  ,$%
\begin{align}
Ip_{2}\left(  t\right)   &  =%
{\displaystyle\int\nolimits_{x}^{x^{\prime}}}
\left[  f\left(  t\left(  \left(  1-\alpha\right)  x+\alpha s\right)  +\left(
1-t\right)  y^{\prime}\right)  p\left(  s\right)  \right. \label{2.46}\\
&  \left.  +f\left(  t\left(  \alpha x+x^{\prime}-\alpha s\right)  +\left(
1-t\right)  y\right)  p\left(  x+x^{\prime}-x\right)  \right]  ds\nonumber\\
&  =%
{\displaystyle\int\nolimits_{x}^{x^{\prime}}}
\left[  f\left(  t\left(  \left(  1-\alpha\right)  x+\alpha s\right)  +\left(
1-t\right)  y^{\prime}\right)  \right. \nonumber\\
&  \left.  +f\left(  t\left(  \alpha x+x^{\prime}-\alpha s\right)  +\left(
1-t\right)  y\right)  \right]  p\left(  s\right)  ds\nonumber\\
&  =%
{\displaystyle\int\nolimits_{x}^{y}}
\frac{1}{\alpha}\left[  f\left(  ts+\left(  1-t\right)  y^{\prime}\right)
\right. \nonumber\\
&  +\left.  f\left(  t\left(  x+x^{\prime}-s\right)  +\left(  1-t\right)
y\right)  \right]  p\left(  \frac{1}{\alpha}\left(  s-x\right)  +x\right)
ds.\nonumber
\end{align}

Let $0\leq t_{1}<t_{2}\leq m$ and $m^{\prime}\leq t_{3}<t_{4}\leq1.$ As a
simple calculation shows us that:%
\[
t_{2}\left(  x^{\prime}-y-x+y^{\prime}\right)  =2t_{2}\left(  y^{\prime
}-x\right)  \leq2m\left(  y^{\prime}-x\right)  =y^{\prime}-y
\]
and%
\begin{align*}
&  t_{3}\left(  x+x^{\prime}-2s+y^{\prime}-y\right) \\
&  =2t_{3}\left(  y^{\prime}-s\right)  \geq2t_{3}\left(  y^{\prime}-y\right)
\geq2m^{\prime}\left(  y^{\prime}-y\right)  =y^{\prime}-y.
\end{align*}
for all $s\in\left[  x,y\right]  .$ Then \ we have the following inequalities
for all $s\in\left[  x,y\right]  $:%
\begin{align*}
t_{1}\left(  x+x^{\prime}-s\right)  +\left(  1-t_{1}\right)  y  &  \leq
t_{2}\left(  x+x^{\prime}-s\right)  +\left(  1-t_{2}\right)  y\\
&  \leq t_{2}x^{\prime}+\left(  1-t_{2}\right)  y\leq t_{2}x+\left(
1-t_{2}\right)  y^{\prime}\\
&  \leq t_{2}s+\left(  1-t_{2}\right)  y^{\prime}\leq t_{1}s+\left(
1-t_{1}\right)  y^{\prime}%
\end{align*}
and%
\begin{align*}
t_{4}s+\left(  1-t_{4}\right)  y^{\prime}  &  \leq t_{3}s+\left(
1-t_{3}\right)  y^{\prime}\\
&  \leq t_{3}\left(  x+x^{\prime}-s\right)  +\left(  1-t_{3}\right)  y\leq
t_{4}\left(  x+x^{\prime}-s\right)  +\left(  1-t_{4}\right)  y.
\end{align*}

By Lemma \ref{l2} with the above inequalities, the following inequalities hold
for all $s\in\left[  x,y\right]  $:%
\begin{align*}
&  \frac{1}{\alpha}\left[  f\left(  t_{2}\left(  x+x^{\prime}-s\right)
+\left(  1-t_{2}\right)  y\right)  +f\left(  t_{2}s+\left(  1-t_{2}\right)
y^{\prime}\right)  \right]  p\left(  \frac{1}{\alpha}\left(  s-x\right)
+x\right) \\
&  \leq\frac{1}{\alpha}\left[  f\left(  t_{1}\left(  x+x^{\prime}-s\right)
+\left(  1-t_{1}\right)  y\right)  +f\left(  t_{1}s+\left(  1-t_{1}\right)
y^{\prime}\right)  \right]  p\left(  \frac{1}{\alpha}\left(  s-x\right)
+x\right)  .
\end{align*}%
\begin{align*}
&  \frac{1}{\alpha}\left[  f\left(  t_{3}s+\left(  1-t_{3}\right)  y^{\prime
}\right)  +f\left(  t_{3}\left(  x+x^{\prime}-s\right)  +\left(
1-t_{3}\right)  y\right)  \right]  p\left(  \frac{1}{\alpha}\left(
s-x\right)  +x\right) \\
&  \leq\frac{1}{\alpha}\left[  f\left(  t_{4}s+\left(  1-t_{4}\right)
y^{\prime}\right)  +f\left(  t_{4}\left(  x+x^{\prime}-s\right)  +\left(
1-t_{4}\right)  y\right)  \right]  p\left(  \frac{1}{\alpha}\left(
s-x\right)  +x\right)  .
\end{align*}

Integrating the above inequalities over $s$ on $\left[  x,y\right]  $ and
using the identity $\left(  \ref{2.46}\right)  $, we derive $Ip_{2}\left(
t_{2}\right)  \leq Ip_{2}\left(  t_{1}\right)  $ and $Ip_{2}\left(
t_{3}\right)  \leq Ip_{2}\left(  t_{4}\right)  .$ Thus $Ip_{2}$ is decreasing
on $\left[  0,m\right]  $\ and increasing on $\left[  m^{\prime},1\right]  .$
Using the convexity of $f,$ the identity $\left(  \ref{2.46}\right)  $ and the
inequality $\left(  \ref{2.42}\right)  $, we derive the inequality $\left(
\ref{2.44}\right)  .$

$\left(  3\right)  $ Again, using Lemma \ref{l2}, the inequality%
\begin{align*}
&  \left[  f\left(  t\left(  \left(  1-\alpha\right)  x+\alpha s\right)
+\left(  1-t\right)  y^{\prime}\right)  \right. \\
&  +\left.  f\left(  t\left(  \left(  1-\alpha\right)  x^{\prime}+\alpha
s\right)  +\left(  1-t\right)  y\right)  \right]  p\left(  s\right) \\
&  \leq\left[  f\left(  t\left(  \left(  1-\alpha\right)  x+\alpha s\right)
+\left(  1-t\right)  y\right)  \right. \\
&  +\left.  f\left(  t\left(  \left(  1-\alpha\right)  x^{\prime}+\alpha
s\right)  +\left(  1-t\right)  y^{\prime}\right)  \right]  p\left(  s\right)
\end{align*}
holds for all $t\in\left[  0,1\right]  $ and $s\in\left[  x,x^{\prime}\right]
$. Integrating the above inequality over $s$ on $\left[  x,y\right]  $ and
using the definitions of $Ip_{1}$ and $Ip_{2}$, we derive $\left(
\ref{2.45}\right)  .$

This completes the proof.
\end{proof}

\begin{remark}
\label{r13}\textit{Using Remark \ref{r11}} we have the following results:
\end{remark}

\begin{enumerate}
\item \textit{Theorem \ref{t9}} \textit{reduces to Theorem \ref{A3} as }$x=a,
$\textit{\ }$y=y^{\prime}=\frac{a+b}{2},$\textit{\ }$x^{\prime}=b$ \textit{and
}$p\left(  s\right)  \equiv\frac{1}{2\left(  b-a\right)  }$ $\left(
s\in\left[  a,b\right]  \right)  .$

\item \textit{Theorem \ref{t9}} \textit{reduces to Theorem \ref{A13} as
}$p\left(  s\right)  \equiv\frac{\alpha}{2\left(  y-x\right)  }$\ $\left(
s\in\Omega\right)  .$

\item \textit{Theorem \ref{t9}} \textit{reduces to Theorem \ref{A26} as
}$x=a,$ $y=y^{\prime}=\frac{a+b}{2},x^{\prime}=b$ \textit{and }$p\left(
s\right)  =\frac{g\left(  s\right)  }{2}$\ $\left(  s\in\left[  a,b\right]
\right)  .$
\end{enumerate}

\begin{theorem}
\label{t10}\textit{Let }$x,y,y^{\prime},x^{\prime},f,p,Jp$\textit{\ be defined
as above. Then }$Jp$\textit{\ is convex, increasing on }$\left[  0,1\right]
,$\textit{\ and for all }$t\in\left[  0,1\right]  $\textit{, we have the
following Fej\'{e}r-type inequality}%
\begin{align}
&  f\left(  \frac{x+y}{2}\right)  +f\left(  \frac{x^{\prime}+y^{\prime}}%
{2}\right)  \int\nolimits_{x}^{x^{\prime}}p\left(  s\right)  ds\label{2.47}\\
&  =Jp\left(  0\right)  \leq Jp\left(  t\right)  \leq Jp\left(  1\right)
\nonumber\\
&  =\int\nolimits_{x}^{x^{\prime}}\left[  f\left(  \left(  1-\alpha\right)
x+\alpha s\right)  +f\left(  \left(  1-\alpha\right)  x^{\prime}+\alpha
s\right)  \right]  p\left(  s\right)  ds.\nonumber
\end{align}

\end{theorem}

\begin{proof}
$\left(  1\right)  $ It is easily observed from the convexity of $f,$
$x^{\prime}-y=y^{\prime}-x$\ and the hypothesis of $p$ that $Jp$ is convex on
$\left[  0,1\right]  .$

$\left(  2\right)  $ Using simple integration techniques and under the
hypothesis of $p$, the following identity holds on $\left[  0,1\right]  :$%
\begin{align}
&  Jp\left(  t\right) \label{2.48}\\
&  =%
{\displaystyle\int\nolimits_{x}^{x^{\prime}}}
\left[  f\left(  t\left(  \left(  1-\alpha\right)  x+\alpha s\right)  +\left(
1-t\right)  \frac{x+y}{2}\right)  p\left(  s\right)  \right. \nonumber\\
&  \left.  +f\left(  t\left(  \alpha x+x^{\prime}-\alpha s\right)  +\left(
1-t\right)  \frac{x^{\prime}+y^{\prime}}{2}\right)  p\left(  x+x^{\prime
}-x\right)  \right]  ds\nonumber\\
&  =%
{\displaystyle\int\nolimits_{x}^{x^{\prime}}}
\left[  f\left(  t\left(  \left(  1-\alpha\right)  x+\alpha s\right)  +\left(
1-t\right)  \frac{x+y}{2}\right)  \right. \nonumber\\
&  \left.  +f\left(  t\left(  \alpha x+x^{\prime}-\alpha s\right)  +\left(
1-t\right)  \frac{x^{\prime}+y^{\prime}}{2}\right)  \right]  p\left(
s\right)  dx\nonumber\\
&  =%
{\displaystyle\int\nolimits_{x}^{y}}
\frac{1}{\alpha}\left[  f\left(  ts+\left(  1-t\right)  \frac{x+y}{2}\right)
\right. \nonumber\\
&  +\left.  f\left(  t\left(  x+x^{\prime}-s\right)  +\left(  1-t\right)
\frac{x^{\prime}+y^{\prime}}{2}\right)  \right]  p\left(  \frac{1}{\alpha
}\left(  s-x\right)  +x\right)  ds\nonumber\\
&  =%
{\displaystyle\int\nolimits_{x}^{\frac{x+y}{2}}}
\frac{1}{\alpha}\left[  f\left(  ts+\left(  1-t\right)  \frac{x+y}{2}\right)
p\left(  \frac{1}{\alpha}\left(  s-x\right)  +x\right)  \right. \nonumber\\
&  +f\left(  t\left(  x+y-s\right)  +\left(  1-t\right)  \frac{x+y}{2}\right)
p\left(  \frac{1}{\alpha}\left(  y-s\right)  +x\right) \nonumber\\
&  +f\left(  t\left(  x+x^{\prime}-s\right)  +\left(  1-t\right)
\frac{x^{\prime}+y^{\prime}}{2}\right)  p\left(  \frac{1}{\alpha}\left(
s-x\right)  +x\right) \nonumber\\
&  +\left.  f\left(  t\left(  x^{\prime}-y+s\right)  +\left(  1-t\right)
\frac{x^{\prime}+y^{\prime}}{2}\right)  p\left(  \frac{1}{\alpha}\left(
y-s\right)  +x\right)  \right]  ds\nonumber\\
&  =%
{\displaystyle\int\nolimits_{x}^{\frac{x+y}{2}}}
\frac{1}{\alpha}\left[  f\left(  ts+\left(  1-t\right)  \frac{x+y}{2}\right)
+f\left(  t\left(  x+y-s\right)  +\left(  1-t\right)  \frac{x+y}{2}\right)
\right. \nonumber\\
&  +f\left(  t\left(  x+x^{\prime}-s\right)  +\left(  1-t\right)
\frac{x^{\prime}+y^{\prime}}{2}\right) \nonumber\\
&  +\left.  f\left(  t\left(  y^{\prime}-x+s\right)  +\left(  1-t\right)
\frac{x^{\prime}+y^{\prime}}{2}\right)  \right]  p\left(  \frac{1}{\alpha
}\left(  s-x\right)  +x\right)  ds.\nonumber
\end{align}

Let $t_{1}<t_{2}$ in $\left[  0,1\right]  .$ By Lemma \ref{l2}, the following
inequalities hold for all $s\in\left[  x,\frac{x+y}{2}\right]  :$%
\begin{align*}
&  \frac{1}{\alpha}\left[  f\left(  t_{1}s+\left(  1-t_{1}\right)  \frac
{x+y}{2}\right)  \right. \\
&  \left.  +f\left(  t_{1}\left(  x+y-s\right)  +\left(  1-t_{1}\right)
\frac{x+y}{2}\right)  \right]  p\left(  \frac{1}{\alpha}\left(  s-x\right)
+x\right) \\
&  \leq\frac{1}{\alpha}\left[  f\left(  t_{2}s+\left(  1-t_{2}\right)
\frac{x+y}{2}\right)  \right. \\
&  \left.  +f\left(  t_{2}\left(  x+y-s\right)  +\left(  1-t_{2}\right)
\frac{x+y}{2}\right)  \right]  p\left(  \frac{1}{\alpha}\left(  s-x\right)
+x\right)  .
\end{align*}%
\begin{align*}
&  \frac{1}{\alpha}\left[  f\left(  t_{1}\left(  x+x^{\prime}-s\right)
+\left(  1-t_{1}\right)  \frac{x^{\prime}+y^{\prime}}{2}\right)  \right. \\
&  \left.  +f\left(  t_{1}\left(  y^{\prime}-x+s\right)  +\left(
1-t_{1}\right)  \frac{x^{\prime}+y^{\prime}}{2}\right)  \right]  p\left(
\frac{1}{\alpha}\left(  s-x\right)  +x\right) \\
&  \leq\frac{1}{\alpha}\left[  f\left(  t_{2}\left(  x+x^{\prime}-s\right)
+\left(  1-t_{2}\right)  \frac{x^{\prime}+y^{\prime}}{2}\right)  \right. \\
&  \left.  +f\left(  t_{2}\left(  y^{\prime}-x+s\right)  +\left(
1-t_{2}\right)  \frac{x^{\prime}+y^{\prime}}{2}\right)  \right]  p\left(
\frac{1}{\alpha}\left(  s-x\right)  +x\right)  .
\end{align*}

Integrating the above inequalities over $s$ on $\left[  x,\frac{x+y}%
{2}\right]  $ and using the above identity, we derive $Jp\left(  t_{1}\right)
\leq Jp\left(  t_{2}\right)  .$ Thus $Jp$ is increasing on $\left[
0,1\right]  $ and then the inequality $\left(  \ref{2.47}\right)  $ holds.

This completes the proof.
\end{proof}

\begin{remark}
\label{r14}\textit{Using Remark \ref{r11}, Theorem \ref{t10}} \textit{reduces
to Theorem \ref{A27} as }$x=a,$ $y=y^{\prime}=\frac{a+b}{2},x^{\prime}=b$
\textit{and }$p\left(  s\right)  \equiv\frac{g\left(  s\right)  }{2}%
$\ $\left(  s\in\left[  a,b\right]  \right)  .$
\end{remark}

\begin{theorem}
\label{t11}\textit{Let }$Ip_{1},Jp$\textit{\ be defined as above. Then
}$Ip_{1}\left(  t\right)  \leq Jp\left(  t\right)  $\textit{\ on }$\left[
0,1\right]  .$
\end{theorem}

\begin{proof}
By Lemma \ref{l2}, the following inequality holds for all $t\in\left[
0,1\right]  $ and $s\in\left[  x,y\right]  :$%
\begin{align*}
&  \frac{1}{\alpha}\left[  f\left(  ts+\left(  1-t\right)  y\right)  \right.
\\
&  +\left.  f\left(  t\left(  x+x^{\prime}-s\right)  +\left(  1-t\right)
y^{\prime}\right)  \right]  p\left(  \frac{1}{\alpha}\left(  s-x\right)
+x\right) \\
&  \leq\frac{1}{\alpha}\left[  f\left(  ts+\left(  1-t\right)  \frac{x+y}%
{2}\right)  \right. \\
&  +\left.  f\left(  t\left(  x+x^{\prime}-s\right)  +\left(  1-t\right)
\frac{x^{\prime}+y^{\prime}}{2}\right)  \right]  p\left(  \frac{1}{\alpha
}\left(  s-x\right)  +x\right)  .
\end{align*}

Integrating the above inequality over $s$ on $\left[  x,y\right]  $ and using
the identities $\left(  \ref{2.43}\right)  $\ and $\left(  \ref{2.48}\right)
$, we derive $Ip_{1}\left(  t\right)  \leq Jp\left(  t\right)  \mathit{\ }$on
$\left[  0,1\right]  .$

This completes the proof.
\end{proof}

\begin{remark}
\label{r15}\textit{Using Remark \ref{r11}, Theorem \ref{t11}} \textit{reduces
to Theorem \ref{A28} as }$x=a,$ $y=y^{\prime}=\frac{a+b}{2},x^{\prime}=b$
\textit{and }$p\left(  s\right)  \equiv\frac{g\left(  s\right)  }{2}%
$\ $\left(  s\in\left[  a,b\right]  \right)  .$
\end{remark}

\begin{theorem}
\label{t12}\textit{Let }$x,y,y^{\prime},x^{\prime},f,p,Mp$\textit{\ be defined
as above. Then }$Mp$\textit{\ is convex, increasing on }$\left[  0,1\right]
,$\textit{\ and for all }$t\in\left[  0,1\right]  $\textit{, we have the
following Fej\'{e}r-type inequalities}%
\begin{align}
&  \int\nolimits_{x}^{x^{\prime}}\left[  f\left(  \left(  1-\alpha\right)
x+\alpha s\right)  +f\left(  \left(  1-\alpha\right)  x^{\prime}+\alpha
s\right)  \right]  p\left(  s\right)  ds\label{2.49}\\
&  =Mp\left(  0\right)  \leq Mp\left(  t\right)  \leq Mp\left(  1\right)
=\frac{f\left(  x\right)  +f\left(  y\right)  +f\left(  y^{\prime}\right)
+f\left(  x^{\prime}\right)  }{2}\int\nolimits_{x}^{x^{\prime}}p\left(
s\right)  ds\nonumber
\end{align}
and%
\begin{align}
&  Mp\left(  t\right) \label{2.50}\\
&  \leq\left(  1-t\right)  \int\nolimits_{x}^{x^{\prime}}\left[  f\left(
\left(  1-\alpha\right)  x+\alpha s\right)  +f\left(  \left(  1-\alpha\right)
x^{\prime}+\alpha s\right)  \right]  p\left(  s\right)  ds\nonumber\\
&  +t\left[  \frac{f\left(  x\right)  +f\left(  y\right)  +f\left(  y^{\prime
}\right)  +f\left(  x^{\prime}\right)  }{2}\right]  \int\nolimits_{x}%
^{x^{\prime}}p\left(  s\right)  ds\nonumber\\
&  \leq\frac{f\left(  x\right)  +f\left(  y\right)  +f\left(  y^{\prime
}\right)  +f\left(  x^{\prime}\right)  }{2}\int\nolimits_{x}^{x^{\prime}%
}p\left(  s\right)  ds\nonumber\\
&  \leq\left[  f\left(  x\right)  +f\left(  x^{\prime}\right)  \right]
\int\nolimits_{x}^{x^{\prime}}p\left(  s\right)  ds.\nonumber
\end{align}

\end{theorem}

\begin{proof}
$\left(  1\right)  $ It is easily observed from the convexity of $f$ and the
hypothesis of $p$ that $Mp$ is convex on $\left[  0,1\right]  .$

$\left(  2\right)  $ Using simple integration techniques, $x^{\prime
}-y=y^{\prime}-x$ and under the hypothesis of $p$, the following identity
holds on $\left[  0,1\right]  :$%
\begin{align}
&  Mp\left(  t\right) \label{2.51}\\
&  =%
{\displaystyle\int\nolimits_{x}^{\frac{x+x^{\prime}}{2}}}
\left[  f\left(  tx+\left(  1-t\right)  \left(  \left(  1-\alpha\right)
x+\alpha s\right)  \right)  \right. \nonumber\\
&  +f\left(  ty^{\prime}+\left(  1-t\right)  \left(  \left(  1-\alpha\right)
x^{\prime}+\alpha s\right)  \right) \nonumber\\
&  +f\left(  ty+\left(  1-t\right)  \left(  \left(  1-\alpha\right)
x+\alpha\left(  x+x^{\prime}-s\right)  \right)  \right) \nonumber\\
&  \text{ \ \ \ \ \ \ \ \ }+\left.  f\left(  tx^{\prime}+\left(  1-t\right)
\left(  \left(  1-\alpha\right)  x^{\prime}+\alpha\left(  x+x^{\prime
}-s\right)  \right)  \right)  \right]  p\left(  s\right)  ds\nonumber\\
&  =%
{\displaystyle\int\nolimits_{x}^{\frac{x+y}{2}}}
\frac{1}{\alpha}\left[  f\left(  tx+\left(  1-t\right)  s\right)  +f\left(
ty^{\prime}+\left(  1-t\right)  \left(  s+y^{\prime}-x\right)  \right)
\right. \nonumber\\
&  \text{ \ \ \ \ }+f\left(  ty+\left(  1-t\right)  \left(  x+y-s\right)
\right) \nonumber\\
&  \text{ \ \ \ \ }\left.  +f\left(  tx^{\prime}+\left(  1-t\right)  \left(
x+x^{\prime}-s\right)  \right)  \right]  p\left(  \frac{1}{\alpha}\left(
s-x\right)  +x\right)  ds.\nonumber
\end{align}

Let $t_{1}<t_{2}$ in $\left[  0,1\right]  .$ By Lemma \ref{l2}, the following
inequalities hold for all $s\in\left[  x,\frac{x+y}{2}\right]  :$%
\begin{align*}
&  \frac{1}{\alpha}\left[  f\left(  t_{1}x+\left(  1-t_{1}\right)  s\right)
\right. \\
&  +\left.  f\left(  t_{1}y+\left(  1-t_{1}\right)  \left(  x+y-s\right)
\right)  \right]  p\left(  \frac{1}{\alpha}\left(  s-x\right)  +x\right) \\
&  \leq\frac{1}{\alpha}\left[  f\left(  t_{2}x+\left(  1-t_{2}\right)
s\right)  \right. \\
&  +\left.  f\left(  t_{2}y+\left(  1-t_{2}\right)  \left(  x+y-s\right)
\right)  \right]  p\left(  \frac{1}{\alpha}\left(  s-x\right)  +x\right)  .
\end{align*}%
\begin{align*}
&  \frac{1}{\alpha}\left[  f\left(  t_{1}y^{\prime}+\left(  1-t_{1}\right)
\left(  s+y^{\prime}-x\right)  \right)  \right. \\
&  +\left.  f\left(  t_{1}x^{\prime}+\left(  1-t_{1}\right)  \left(
x+x^{\prime}-s\right)  \right)  \right]  p\left(  \frac{1}{\alpha}\left(
s-x\right)  +x\right) \\
&  \leq\frac{1}{\alpha}\left[  f\left(  t_{2}y^{\prime}+\left(  1-t_{2}%
\right)  \left(  s+y^{\prime}-x\right)  \right)  \right. \\
&  +\left.  f\left(  t_{2}x^{\prime}+\left(  1-t_{2}\right)  \left(
x+x^{\prime}-s\right)  \right)  \right]  p\left(  \frac{1}{\alpha}\left(
s-x\right)  +x\right)  .
\end{align*}

Integrating the above inequalities over $s$ on $\left[  x,\frac{x+y}%
{2}\right]  $ and using the above identity, we derive $Mp\left(  t_{1}\right)
\leq Mp\left(  t_{2}\right)  .$ Thus $Mp$ is increasing on $\left[
0,1\right]  $ and then the inequality $\left(  \ref{2.49}\right)  $ holds.

Using the convexity of $f$ and the inequalities $\left(  \ref{2.41}\right)
,\left(  \ref{2.42}\right)  $ and$\ \left(  \ref{2.49}\right)  ,$ we obtain
the inequality $\left(  \ref{2.50}\right)  .$

This completes the proof.
\end{proof}

\begin{remark}
\label{r16}\textit{Using Remark \ref{r11}, Theorem \ref{t12}} \textit{reduces
to Theorem \ref{A29} as }$x=a,$ $y=y^{\prime}=\frac{a+b}{2},x^{\prime}=b$
\textit{and }$p\left(  s\right)  \equiv\frac{g\left(  s\right)  }{2}%
$\ $\left(  s\in\left[  a,b\right]  \right)  .$
\end{remark}

\begin{theorem}
\label{t13}\textit{Let }$x,y,y^{\prime},x^{\prime},f,p,Np$\textit{\ be defined
as above. Then }$Np$\textit{\ is convex, increasing on }$\left[  0,1\right]
,$\textit{\ and for all }$t\in\left[  0,1\right]  $\textit{, we have the
following Fej\'{e}r-type inequalities}%
\begin{align}
&  \int\nolimits_{x}^{x^{\prime}}\left[  f\left(  \left(  1-\alpha\right)
x+\alpha s\right)  +f\left(  \left(  1-\alpha\right)  x^{\prime}+\alpha
s\right)  \right]  p\left(  s\right)  ds\label{2.52}\\
&  =Np\left(  0\right)  \leq Np\left(  t\right)  \leq Np\left(  1\right)
=\left[  f\left(  x\right)  +f\left(  x^{\prime}\right)  \right]
\int\nolimits_{x}^{x^{\prime}}p\left(  s\right)  ds\nonumber
\end{align}
and%
\begin{align}
&  Np\left(  t\right) \label{2.53}\\
&  \leq\left(  1-t\right)  \int\nolimits_{x}^{x^{\prime}}\left[  f\left(
\left(  1-\alpha\right)  x+\alpha s\right)  +f\left(  \left(  1-\alpha\right)
x^{\prime}+\alpha s\right)  \right]  p\left(  s\right)  ds\nonumber\\
&  +t\left[  f\left(  x\right)  +f\left(  x^{\prime}\right)  \right]
\int\nolimits_{x}^{x^{\prime}}p\left(  s\right)  ds\nonumber\\
&  \leq\left[  f\left(  x\right)  +f\left(  x^{\prime}\right)  \right]
\int\nolimits_{x}^{x^{\prime}}p\left(  s\right)  ds.\nonumber
\end{align}

\end{theorem}

\begin{proof}
$\left(  1\right)  $ It is easily observed from the convexity of $f$ and the
hypothesis of $p$ that $Np$ is convex on $\left[  0,1\right]  .$

$\left(  2\right)  $ Using simple integration techniques and under the
hypothesis of $p$, the following identity holds on $\left[  0,1\right]  :$%
\begin{align}
&  Np\left(  t\right) \label{2.54}\\
&  =%
{\displaystyle\int\nolimits_{x}^{x^{\prime}}}
\left[  f\left(  tx+\left(  1-t\right)  \left(  \left(  1-\alpha\right)
x+\alpha s\right)  \right)  \right. \nonumber\\
&  \text{ \ \ \ \ \ \ \ \ }+\left.  f\left(  tx^{\prime}+\left(  1-t\right)
\left(  \left(  1-\alpha\right)  x^{\prime}+\alpha\left(  x+x^{\prime
}-s\right)  \right)  \right)  \right]  p\left(  s\right)  ds\nonumber\\
&  =%
{\displaystyle\int\nolimits_{x}^{y}}
\frac{1}{\alpha}\left[  f\left(  tx+\left(  1-t\right)  s\right)  \right.
\nonumber\\
&  \text{ \ \ \ \ }\left.  +f\left(  tx^{\prime}+\left(  1-t\right)  \left(
x+x^{\prime}-s\right)  \right)  \right]  p\left(  \frac{1}{\alpha}\left(
s-x\right)  +x\right)  ds.\nonumber
\end{align}

Let $t_{1}<t_{2}$ in $\left[  0,1\right]  .$ By Lemma \ref{l2}, the following
inequality hold for all $s\in\left[  x,y\right]  :$%
\begin{align*}
&  \frac{1}{\alpha}\left[  f\left(  t_{1}x+\left(  1-t_{1}\right)  s\right)
\right. \\
&  +\left.  f\left(  t_{1}x^{\prime}+\left(  1-t_{1}\right)  \left(
x+x^{\prime}-s\right)  \right)  \right]  p\left(  \frac{1}{\alpha}\left(
s-x\right)  +x\right) \\
&  \leq\frac{1}{\alpha}\left[  f\left(  t_{2}x+\left(  1-t_{2}\right)
s\right)  \right. \\
&  +\left.  f\left(  t_{2}x^{\prime}+\left(  1-t_{2}\right)  \left(
x+x^{\prime}-s\right)  \right)  \right]  p\left(  \frac{1}{\alpha}\left(
s-x\right)  +x\right)  .
\end{align*}
Integrating the above inequality over $s$ on $\left[  x,y\right]  $ and using
the identity $\left(  \ref{2.54}\right)  ,$ we derive $Np\left(  t_{1}\right)
\leq Np\left(  t_{2}\right)  .$ Thus $Np$ is increasing on $\left[
0,1\right]  $ and then the inequality $\left(  \ref{2.52}\right)  $ holds.

Using the convexity of $f$ and the inequality $\left(  \ref{2.52}\right)  ,$
we obtain the inequality $\left(  \ref{2.53}\right)  .$

This completes the proof.
\end{proof}

\begin{remark}
\label{r17}\textit{Using Remark \ref{r11}} we have the following results:
\end{remark}

\begin{enumerate}
\item \textit{Theorem \ref{t13}} \textit{reduces to Theorem \ref{A5} as
}$x=a,$\textit{\ }$y=y^{\prime}=\frac{a+b}{2},$\textit{\ }$x^{\prime}=b$
\textit{and }$p\left(  s\right)  \equiv\frac{1}{2\left(  b-a\right)  }$
$\left(  s\in\left[  a,b\right]  \right)  .$

\item \textit{Theorem \ref{t13}} \textit{reduces to Theorem \ref{A14} as
}$p\left(  s\right)  \equiv\frac{\alpha}{2\left(  y-x\right)  }$\ $\left(
s\in\Omega\right)  .$

\item \textit{Theorem \ref{t13}} \textit{reduces to Theorem \ref{A30} as
}$x=a,$ $y=y^{\prime}=\frac{a+b}{2},x^{\prime}=b$ \textit{and }$p\left(
s\right)  \equiv\frac{g\left(  s\right)  }{2}$\ $\left(  s\in\left[
a,b\right]  \right)  .$
\end{enumerate}

\begin{theorem}
\label{t14}\textit{Let }$Mp,Np$\textit{\ be defined as above. Then }$Mp\left(
t\right)  \leq Np\left(  t\right)  $\textit{\ on }$\left[  0,1\right]  .$
\end{theorem}

\begin{proof}
Using simple integration techniques, the identity $\left(  \ref{2.54}\right)
$ and under the hypothesis of $p$, the following identity holds on $\left[
0,1\right]  :$%
\begin{align}
Np\left(  t\right)   &  =%
{\displaystyle\int\nolimits_{x}^{\frac{x+y}{2}}}
\frac{1}{\alpha}\left[  f\left(  tx+\left(  1-t\right)  s\right)  p\left(
\frac{1}{\alpha}\left(  s-x\right)  +x\right)  \right. \label{2.55}\\
&  +f\left(  tx+\left(  1-t\right)  \left(  x+y-s\right)  \right)  p\left(
\frac{1}{\alpha}\left(  y-s\right)  +x\right) \nonumber\\
&  +f\left(  tx^{\prime}+\left(  1-t\right)  \left(  x+x^{\prime}-s\right)
\right)  p\left(  \frac{1}{\alpha}\left(  s-x\right)  +x\right) \nonumber\\
&  \left.  +f\left(  tx^{\prime}+\left(  1-t\right)  \left(  s+x^{\prime
}-y\right)  \right)  p\left(  \frac{1}{\alpha}\left(  y-s\right)  +x\right)
\right]  ds\nonumber\\
&  =%
{\displaystyle\int\nolimits_{x}^{\frac{x+y}{2}}}
\frac{1}{\alpha}\left[  f\left(  tx+\left(  1-t\right)  s\right)  +f\left(
tx+\left(  1-t\right)  \left(  x+y-s\right)  \right)  \right. \nonumber\\
&  +f\left(  tx^{\prime}+\left(  1-t\right)  \left(  x+x^{\prime}-s\right)
\right) \nonumber\\
&  \left.  +f\left(  tx^{\prime}+\left(  1-t\right)  \left(  s+x^{\prime
}-y\right)  \right)  \right]  p\left(  \frac{1}{\alpha}\left(  s-x\right)
+x\right)  ds.\nonumber
\end{align}

Using Lemma \ref{l2} and $y^{\prime}-x=x^{\prime}-y,$ the following inequality
holds for all $t\in\left[  0,1\right]  $ and $s\in\left[  x,\frac{x+y}%
{2}\right]  :$%
\begin{align*}
&  \frac{1}{\alpha}\left[  f\left(  ty+\left(  1-t\right)  \left(
x+y-s\right)  \right)  \right. \\
&  +\left.  f\left(  ty^{\prime}+\left(  1-t\right)  \left(  s+y^{\prime
}-x\right)  \right)  \right]  p\left(  \frac{1}{\alpha}\left(  s-x\right)
+x\right) \\
&  \leq\frac{1}{\alpha}\left[  \left(  tx+\left(  1-t\right)  \left(
x+y-s\right)  \right)  \right. \\
&  +\left.  f\left(  tx^{\prime}+\left(  1-t\right)  \left(  s+x^{\prime
}-y\right)  \right)  \right]  p\left(  \frac{1}{\alpha}\left(  s-x\right)
+x\right)  .
\end{align*}

Integrating the above inequality over $s$ on $\left[  x,\frac{x+y}{2}\right]
$ and using the identities $\left(  \ref{2.51}\right)  $\ and $\left(
\ref{2.55}\right)  $, we derive $Mp\left(  t\right)  \leq Np\left(  t\right)
\mathit{\ }$on $\left[  0,1\right]  .$

This completes the proof.
\end{proof}

\begin{remark}
\label{r18}\textit{Using Remark \ref{r11}, Theorem \ref{t14}} \textit{reduces
to Theorem \ref{A31} as }$x=a,$ $y=y^{\prime}=\frac{a+b}{2},x^{\prime}=b$
\textit{and }$p\left(  s\right)  \equiv\frac{g\left(  s\right)  }{2}%
$\ $\left(  s\in\left[  a,b\right]  \right)  .$
\end{remark}

The following corollary is a natural consequence of Theorems \ref{t8} --
\ref{t14}.

\begin{corollary}
\label{c2}\textit{Let }$x,y,y^{\prime},x^{\prime},f,p,Ip_{1},Ip_{2},Jp,Mp,Np$
\textit{be defined as above. }The inequalities%
\begin{align}
&  \left[  f\left(  y\right)  +f\left(  y^{\prime}\right)  \right]
\int\nolimits_{x}^{x^{\prime}}p\left(  s\right)  ds\label{2.56}\\
&  \leq\mathit{\ }Ip_{1}\left(  t\right)  \leq Jp\left(  t\right)  \leq
Mp\left(  t\right)  \leq Np\left(  t\right) \nonumber\\
&  \leq\left[  f\left(  x\right)  +f\left(  x^{\prime}\right)  \right]
\int\nolimits_{x}^{x^{\prime}}p\left(  s\right)  ds\nonumber
\end{align}%
\begin{align}
&  2f\left(  \frac{y+y^{\prime}}{2}\right)  \int\nolimits_{x}^{x^{\prime}%
}p\left(  s\right)  ds\label{2.57}\\
&  \leq\mathit{\ }Ip_{2}\left(  t\right)  \leq\mathit{\ }Ip_{1}\left(
t\right)  \leq Jp\left(  t\right)  \leq Mp\left(  t\right)  \leq Np\left(
t\right) \nonumber\\
&  \leq\left[  f\left(  x\right)  +f\left(  x^{\prime}\right)  \right]
\int\nolimits_{x}^{x^{\prime}}p\left(  s\right)  ds\nonumber
\end{align}
and%
\begin{align}
&  \left[  f\left(  y\right)  +f\left(  y^{\prime}\right)  \right]
\int\nolimits_{x}^{x^{\prime}}p\left(  s\right)  ds\label{2.58}\\
&  \leq\left[  f\left(  \frac{x+y}{2}\right)  +f\left(  \frac{x^{\prime
}+y^{\prime}}{2}\right)  \right]  \int\nolimits_{x}^{x^{\prime}}p\left(
s\right)  ds\nonumber\\
&  \leq\int\nolimits_{x}^{x^{\prime}}\left[  f\left(  \left(  1-\alpha\right)
x+\alpha s\right)  +f\left(  \left(  1-\alpha\right)  x^{\prime}+\alpha
s\right)  \right]  p\left(  s\right)  ds\nonumber\\
&  \leq\frac{f\left(  x\right)  +f\left(  y\right)  +f\left(  y^{\prime
}\right)  +f\left(  x^{\prime}\right)  }{2}\int\nolimits_{x}^{x^{\prime}%
}p\left(  s\right)  ds\nonumber\\
&  \leq\left[  f\left(  x\right)  +f\left(  x^{\prime}\right)  \right]
\int\nolimits_{x}^{x^{\prime}}p\left(  s\right)  ds\nonumber
\end{align}
hold for all $t\in\left[  0,1\right]  .$
\end{corollary}

\begin{remark}
\label{r19}In \textit{Corollary \ref{c2}}, let\textit{\ }$x=a,$ $y=y^{\prime
}=\frac{a+b}{2},x^{\prime}=b$ \textit{and }$p\left(  s\right)  =\frac{g\left(
s\right)  }{2}$\ $\left(  s\in\left[  a,b\right]  \right)  .$ Then we have
the\textit{\ Bullen-type inequality }\cite{13b}%
\begin{align*}
&  f\left(  \frac{a+b}{2}\right)  \int\nolimits_{a}^{b}g\left(  s\right)  ds\\
&  \leq\frac{1}{2}\left[  f\left(  \frac{3a+b}{4}\right)  +f\left(
\frac{a+3b}{4}\right)  \right]  \int\nolimits_{a}^{b}g\left(  s\right)  ds\\
&  \leq\int\nolimits_{a}^{b}\left[  f\left(  \frac{x+a}{2}\right)  +f\left(
\frac{x+b}{2}\right)  \right]  g\left(  s\right)  ds\\
&  \leq\frac{1}{2}\left[  f\left(  \frac{a+b}{2}\right)  +\frac{f\left(
a\right)  +f\left(  b\right)  }{2}\right]  \int\nolimits_{a}^{b}g\left(
s\right)  ds\\
&  \leq\frac{f\left(  a\right)  +f\left(  b\right)  }{2}\int\nolimits_{a}%
^{b}g\left(  s\right)  ds
\end{align*}
which refines Fej\'{e}r inequality $\left(  \ref{1.2}\right)  .$
\end{remark}

\begin{theorem}
\label{t15}Let $x,y,y^{\prime},x^{\prime},\Omega,f,p,Ip_{1}$\ be defined as
above. Then:

\begin{enumerate}
\item The inequality
\begin{align}
&  \left[  f\left(  y\right)  +f\left(  y^{\prime}\right)  \right]
\int\nolimits_{x}^{x^{\prime}}p\left(  s\right)  ds\label{2.59}\\
&  \leq\frac{2}{\alpha}\left[  \int_{\frac{x+y}{2}}^{y}f\left(  s\right)
p\left(  \frac{1}{\alpha}\left(  2s-x-y\right)  +x\right)  ds\right.
\nonumber\\
&  \text{ \ \ \ \ \ \ \ \ \ }+\left.  \int_{y^{\prime}}^{\frac{x^{\prime
}+y^{\prime}}{2}}f\left(  s\right)  p\left(  \frac{1}{\alpha}\left(
2s-y^{\prime}-x^{\prime}\right)  +x^{\prime}\right)  ds\right] \nonumber\\
&  \leq\int_{0}^{1}Ip_{1}\left(  t\right)  dt\nonumber\\
&  \leq\frac{1}{2}\left[  \left(  f\left(  y\right)  +f\left(  y^{\prime
}\right)  \right)  \int\nolimits_{x}^{x^{\prime}}p\left(  s\right)  ds\right.
\nonumber\\
&  \text{ \ \ \ \ \ \ \ \ }+\left.  \int\nolimits_{x}^{x^{\prime}}\left[
f\left(  \left(  1-\alpha\right)  x+\alpha s\right)  +f\left(  \left(
1-\alpha\right)  x^{\prime}+\alpha s\right)  \right]  p\left(  s\right)
ds\right] \nonumber
\end{align}
holds.

\item If $f$\ is differentiable on $\left[  x,x^{\prime}\right]  $ and $p$ is
bounded on $\left[  x,x^{\prime}\right]  ,$ then the inequality%
\begin{align}
0  &  \leq\int\nolimits_{x}^{x^{\prime}}\left[  f\left(  \left(
1-\alpha\right)  x+\alpha s\right)  +f\left(  \left(  1-\alpha\right)
x^{\prime}+\alpha s\right)  \right]  p\left(  s\right)  ds-Ip_{1}\left(
t\right) \label{2.60}\\
&  \leq\frac{1-t}{\alpha}\left[  \left(  f\left(  x\right)  +f\left(
x^{\prime}\right)  \right)  \left(  y-x\right)  -\int_{\Omega}f\left(
s\right)  ds\right]  \left\Vert p\right\Vert _{\infty}\nonumber
\end{align}
holds for all $t\in\left[  0,1\right]  .$

\item If $f$\ is differentiable on $\left[  x,x^{\prime}\right]  ,$ then the
inequalities%
\begin{align}
0  &  \leq\left[  f\left(  x\right)  +f\left(  x^{\prime}\right)  \right]
\int\nolimits_{x}^{x^{\prime}}p\left(  s\right)  ds-Ip_{1}\left(  t\right)
\label{2.61}\\
&  \leq\left(  y-x\right)  \left(  f^{\prime}\left(  x^{\prime}\right)
-f^{\prime}\left(  x\right)  \right)  \int\nolimits_{x}^{x^{\prime}}p\left(
s\right)  ds\nonumber
\end{align}
and%
\begin{align}
0  &  \leq Ip_{1}\left(  t\right)  -\left[  f\left(  y\right)  +f\left(
y^{\prime}\right)  \right]  \int\nolimits_{x}^{x^{\prime}}p\left(  s\right)
ds\label{2.62}\\
&  \leq\left(  y-x\right)  \left(  f^{\prime}\left(  x^{\prime}\right)
-f^{\prime}\left(  x\right)  \right)  \int\nolimits_{x}^{x^{\prime}}p\left(
s\right)  ds\nonumber
\end{align}
hold for all $t\in\left[  0,1\right]  .$
\end{enumerate}
\end{theorem}

\begin{proof}
$\left(  1\right)  $ Using simple techniques of integration, $x+x^{\prime
}=y+y^{\prime},$ $2x+x^{\prime}-y^{\prime}=x+y$ and the hypothesis of $p$%
\begin{align*}
&  p\left(  \frac{1}{\alpha}\left(  2s-x-y\right)  +x\right) \\
&  =p\left(  \frac{1}{\alpha}\left(  x+y-2s\right)  +x^{\prime}\right)  \text{
\ }\left(  s\in\left[  \frac{x+y}{2},y\right]  \right)  ,
\end{align*}
we have the following identities
\begin{align*}
&  \left[  f\left(  y\right)  +f\left(  y^{\prime}\right)  \right]
\int\nolimits_{x}^{x^{\prime}}p\left(  s\right)  ds\\
&  =\int_{x}^{y}\int_{0}^{\frac{1}{2}}\frac{2}{\alpha}\left[  f\left(
y\right)  +f\left(  y^{\prime}\right)  \right]  p\left(  \frac{1}{\alpha
}\left(  s-x\right)  +x\right)  dtds,
\end{align*}%
\begin{align*}
&  \frac{2}{\alpha}\left[  \int_{\frac{x+y}{2}}^{y}f\left(  s\right)  p\left(
\frac{1}{\alpha}\left(  2s-x-y\right)  +x\right)  ds\right. \\
&  \left.  +\int_{y^{\prime}}^{\frac{x^{\prime}+y^{\prime}}{2}}f\left(
s\right)  p\left(  \frac{1}{\alpha}\left(  2s-y^{\prime}-x^{\prime}\right)
+x^{\prime}\right)  ds\right] \\
&  =\frac{2}{\alpha}\left[  \int_{\frac{x+y}{2}}^{y}f\left(  s\right)
p\left(  \frac{1}{\alpha}\left(  2s-x-y\right)  +x\right)  ds\right. \\
&  \left.  +\int_{\frac{x+y}{2}}^{y}f\left(  x+x^{\prime}-s\right)  p\left(
\frac{1}{\alpha}\left(  x+y-2s\right)  +x^{\prime}\right)  ds\right] \\
&  =\frac{2}{\alpha}\int_{\frac{x+y}{2}}^{y}\left[  f\left(  s\right)
+f\left(  x+x^{\prime}-s\right)  \right]  p\left(  \frac{1}{\alpha}\left(
2s-x-y\right)  +x\right)  ds\\
&  =\int_{x}^{y}\int_{0}^{\frac{1}{2}}\frac{2}{\alpha}\left[  f\left(
\frac{s+y}{2}\right)  +f\left(  \frac{y+2y^{\prime}-s}{2}\right)  \right]
p\left(  \frac{1}{\alpha}\left(  s-x\right)  +x\right)  dtds,
\end{align*}%
\begin{align*}
&  \int_{0}^{1}Ip_{1}\left(  t\right)  dt\\
&  =\int_{x}^{x^{\prime}}\int_{0}^{\frac{1}{2}}\left[  f\left(  t\left(
\left(  1-\alpha\right)  x+\alpha s\right)  +\left(  1-t\right)  y\right)
\right. \\
&  +\left.  f\left(  \left(  1-t\right)  \left(  \left(  1-\alpha\right)
x+\alpha s\right)  +ty\right)  \right]  p\left(  s\right)  dtds\\
&  +\int_{x}^{x^{\prime}}\int_{0}^{\frac{1}{2}}\left[  f\left(  t\left(
\left(  1-\alpha\right)  x^{\prime}+\alpha s\right)  +\left(  1-t\right)
y^{\prime}\right)  \right. \\
&  +\left.  f\left(  \left(  1-t\right)  \left(  \left(  1-\alpha\right)
x^{\prime}+\alpha s\right)  +ty^{\prime}\right)  \right]  p\left(  s\right)
dtds\text{ \ }\\
&  =\int_{x}^{x^{\prime}}\int_{0}^{\frac{1}{2}}\left[  f\left(  t\left(
\left(  1-\alpha\right)  x+\alpha s\right)  +\left(  1-t\right)  y\right)
\right. \\
&  +\left.  f\left(  \left(  1-t\right)  \left(  \left(  1-\alpha\right)
x+\alpha s\right)  +ty\right)  \right]  p\left(  s\right)  dtds\\
&  +\int_{x}^{x^{\prime}}\int_{0}^{\frac{1}{2}}\left[  f\left(  t\left(
x^{\prime}+\alpha x-\alpha s\right)  +\left(  1-t\right)  y^{\prime}\right)
\right. \\
&  \text{\ }+\left.  f\left(  \left(  1-t\right)  \left(  x^{\prime}+\alpha
x-\alpha s\right)  +ty^{\prime}\right)  \right]  p\left(  x+x^{\prime
}-s\right)  dtds
\end{align*}%
\begin{align*}
&  =\int_{x}^{x^{\prime}}\int_{0}^{\frac{1}{2}}\left[  f\left(  t\left(
\left(  1-\alpha\right)  x+\alpha s\right)  +\left(  1-t\right)  y\right)
\right. \\
&  +\left.  f\left(  \left(  1-t\right)  \left(  \left(  1-\alpha\right)
x+\alpha s\right)  +ty\right)  \right]  p\left(  s\right)  dtds\text{
\ \ \ \ \ \ \ \ \ \ \ \ \ \ \ }\\
&  +\int_{x}^{x^{\prime}}\int_{0}^{\frac{1}{2}}\left[  f\left(  t\left(
x^{\prime}+\alpha x-\alpha s\right)  +\left(  1-t\right)  y^{\prime}\right)
\right. \\
&  +\left.  f\left(  \left(  1-t\right)  \left(  x^{\prime}+\alpha x-\alpha
s\right)  +ty^{\prime}\right)  \right]  p\left(  s\right)  dtds
\end{align*}%
\begin{align*}
&  =\int_{x}^{y}\int_{0}^{\frac{1}{2}}\frac{1}{\alpha}\left[  f\left(
ts+\left(  1-t\right)  y\right)  +f\left(  \left(  1-t\right)  s+ty\right)
\right. \\
&  +f\left(  t\left(  x+x^{\prime}-s\right)  +\left(  1-t\right)  y^{\prime
}\right) \\
&  \text{\ }+\left.  f\left(  \left(  1-t\right)  \left(  x+x^{\prime
}-s\right)  +ty^{\prime}\right)  \right]  p\left(  \frac{1}{\alpha}\left(
s-x\right)  +x\right)  dtds
\end{align*}
and%
\begin{align*}
&  \frac{1}{2}\left[  \left(  f\left(  y\right)  +f\left(  y^{\prime}\right)
\right)  \int\nolimits_{x}^{x^{\prime}}p\left(  s\right)  ds\right. \\
&  +\left.  \int\nolimits_{x}^{x^{\prime}}\left[  f\left(  \left(
1-\alpha\right)  x+\alpha s\right)  +f\left(  \left(  1-\alpha\right)
x^{\prime}+\alpha s\right)  \right]  p\left(  s\right)  ds\right] \\
&  =\frac{1}{2}\left[  \left(  f\left(  y\right)  +f\left(  y^{\prime}\right)
\right)  \int\nolimits_{x}^{x^{\prime}}p\left(  s\right)  ds\right. \\
&  +\left.  \int\nolimits_{x}^{x^{\prime}}\left[  f\left(  \left(
1-\alpha\right)  x+\alpha s\right)  +f\left(  x^{\prime}+\alpha x-\alpha
s\right)  \right]  p\left(  s\right)  ds\right] \\
&  =\int_{x}^{y}\int_{0}^{\frac{1}{2}}\frac{1}{\alpha}\left[  f\left(
s\right)  +f\left(  y\right)  \right]  p\left(  \frac{1}{\alpha}\left(
s-x\right)  +x\right)  dtds\\
&  +\int_{x}^{y}\int_{0}^{\frac{1}{2}}\frac{1}{\alpha}\left[  f\left(
y^{\prime}\right)  +f\left(  x+x^{\prime}-s\right)  \right]  p\left(  \frac
{1}{\alpha}\left(  s-x\right)  +x\right)  dtds.
\end{align*}

By Lemma \ref{l2}, the following inequalities hold for all $t\in\left[
0,\frac{1}{2}\right]  $ and $s\in\left[  x,y\right]  :$%
\begin{align*}
&  \frac{2}{\alpha}\left[  f\left(  y\right)  +f\left(  y^{\prime}\right)
\right]  p\left(  \frac{1}{\alpha}\left(  s-x\right)  +x\right) \\
&  \leq\frac{2}{\alpha}\left[  f\left(  \frac{s+y}{2}\right)  +f\left(
\frac{y+2y^{\prime}-s}{2}\right)  \right]  p\left(  \frac{1}{\alpha}\left(
s-x\right)  +x\right)  ,
\end{align*}%
\begin{align*}
&  \frac{2}{\alpha}f\left(  \frac{s+y}{2}\right)  p\left(  \frac{1}{\alpha
}\left(  s-x\right)  +x\right) \\
&  \leq\frac{1}{\alpha}\left[  f\left(  \left(  1-t\right)  s+ty\right)
+f\left(  ts+\left(  1-t\right)  y\right)  \right]  p\left(  \frac{1}{\alpha
}\left(  s-x\right)  +x\right)  ,
\end{align*}%
\begin{align*}
&  \frac{2}{\alpha}f\left(  \frac{y+2y^{\prime}-s}{2}\right)  p\left(
\frac{1}{\alpha}\left(  s-x\right)  +x\right) \\
&  \leq\frac{1}{\alpha}\left[  f\left(  t\left(  y+\text{ }y^{\prime
}-s\right)  +\left(  1-t\right)  y^{\prime}\right)  \right. \\
&  +\left.  f\left(  \left(  1-t\right)  \left(  y+y^{\prime}-s\right)
+ty^{\prime}\right)  \right]  p\left(  \frac{1}{\alpha}\left(  s-x\right)
+x\right)  ,
\end{align*}%
\begin{align*}
&  \frac{1}{\alpha}\left[  f\left(  ty+\left(  1-t\right)  s\right)  +f\left(
ts+\left(  1-t\right)  y\right)  \right]  p\left(  \frac{1}{\alpha}\left(
s-x\right)  +x\right) \\
&  \leq\frac{1}{\alpha}\left[  f\left(  s\right)  +f\left(  y\right)  \right]
p\left(  \frac{1}{\alpha}\left(  s-x\right)  +x\right)
\end{align*}
and%
\begin{align*}
&  \frac{1}{\alpha}\left[  f\left(  t\left(  y+y^{\prime}-s\right)  +\left(
1-t\right)  y^{\prime}\right)  \right. \\
&  +\left.  f\left(  \left(  1-t\right)  \left(  y+y^{\prime}-s\right)
+ty^{\prime}\right)  \right]  p\left(  \frac{1}{\alpha}\left(  s-x\right)
+x\right) \\
&  \leq\frac{1}{\alpha}\left[  f\left(  y^{\prime}\right)  +f\left(
y+y^{\prime}-s\right)  \right]  p\left(  \frac{1}{\alpha}\left(  s-x\right)
+x\right)  .
\end{align*}

Integrating the above inequalities over $t$ on $\left[  0,\frac{1}{2}\right]
, $ over $s$ on $\left[  x,y\right]  $ and using the above identities, we
derive $\left(  \ref{2.59}\right)  .$

$\left(  2\right)  $ By integration by parts and $\Omega=\left[  x,y\right]
\cup\left[  y^{\prime},x^{\prime}\right]  ,$ we have the following identity%
\begin{align*}
&  \int_{x}^{y}\left[  \left(  s-y\right)  f^{\prime}\left(  s\right)
+\left(  y-s\right)  f^{\prime}\left(  x+x^{\prime}-s\right)  \right]  ds\\
&  =\left(  f\left(  x\right)  +f\left(  x^{\prime}\right)  \right)  \left(
y-x\right)  -\int_{\Omega}f\left(  s\right)  ds.
\end{align*}

Now, using the convexity of $f$ and $x+x^{\prime}=y+y^{\prime},$ the
inequalities
\[
f\left(  s\right)  -f\left(  ts+\left(  1-t\right)  y\right)  \leq\left(
1-t\right)  \left(  s-y\right)  f^{\prime}\left(  s\right)
\]
and%
\[
f\left(  x+x^{\prime}-s\right)  -f\left(  t\left(  x+x^{\prime}-s\right)
+\left(  1-t\right)  y^{\prime}\right)  \leq\left(  1-t\right)  \left(
y-s\right)  f^{\prime}\left(  x+x^{\prime}-s\right)
\]
hold for all $t\in\left[  0,1\right]  $ and $s\in\left[  x,y\right]  $. Using
simple techniques of integration, the identity $\left(  \ref{2.43}\right)  ,$
the hypothesis of $p,$ the above inequalities, the convexity of $f$ and the
above identity, we have%
\begin{align*}
&  \int\nolimits_{x}^{x^{\prime}}\left[  f\left(  \left(  1-\alpha\right)
x+\alpha s\right)  +f\left(  \left(  1-\alpha\right)  x^{\prime}+\alpha
s\right)  \right]  p\left(  s\right)  ds-Ip_{1}\left(  t\right) \\
&  =\int_{x}^{y}\frac{1}{\alpha}\left[  \left(  f\left(  s\right)  -f\left(
ts+\left(  1-t\right)  y\right)  \right)  \right. \\
&  +\left.  \left(  f\left(  x+x^{\prime}-s\right)  -f\left(  t\left(
x+x^{\prime}-s\right)  +\left(  1-t\right)  y^{\prime}\right)  \right)
\right]  p\left(  \frac{1}{\alpha}\left(  s-x\right)  +x\right)  ds\\
&  \leq\int_{x}^{y}\frac{1}{\alpha}\left[  \left(  1-t\right)  \left(
s-y\right)  f^{\prime}\left(  s\right)  \right. \\
&  \text{ \ \ \ \ \ \ \ \ \ \ \ \ \ \ \ \ \ \ \ }+\left.  \left(  1-t\right)
\left(  y-s\right)  f^{\prime}\left(  x+x^{\prime}-s\right)  p\left(  \frac
{1}{\alpha}\left(  s-x\right)  +x\right)  \right]  ds\\
&  =\frac{1-t}{\alpha}\int_{x}^{y}\left(  y-s\right)  \left(  f^{\prime
}\left(  x+x^{\prime}-s\right)  -f^{\prime}\left(  s\right)  \right)  p\left(
\frac{1}{\alpha}\left(  s-x\right)  +x\right)  ds\\
&  \leq\frac{1-t}{\alpha}\int_{x}^{y}\left(  y-s\right)  \left(  f^{\prime
}\left(  x+x^{\prime}-s\right)  -f^{\prime}\left(  s\right)  \right)
ds\left\Vert p\right\Vert _{\infty}\\
&  =\frac{1-t}{\alpha}\left[  \left(  f\left(  x\right)  +f\left(  x^{\prime
}\right)  \right)  \left(  y-x\right)  -\int_{\Omega}f\left(  s\right)
ds\right]  \left\Vert p\right\Vert _{\infty}.
\end{align*}
\ 

Using the above inequality and $\left(  \ref{2.41}\right)  $, we derive
$\left(  \ref{2.60}\right)  .$

$\left(  3\right)  $ Using the convexity of $f$ and the hypothesis of $p$, we
have the inequality
\begin{align}
&  \left(  f\left(  x\right)  +f\left(  x^{\prime}\right)  \right)  -\left(
f\left(  y\right)  +f\left(  y^{\prime}\right)  \right) \label{2.63}\\
&  =\left(  f\left(  x\right)  -f\left(  y\right)  \right)  +\left(  f\left(
x^{\prime}\right)  -f\left(  y^{\prime}\right)  \right) \nonumber\\
&  \leq\left(  x-y\right)  f^{\prime}\left(  x\right)  +\left(  x^{\prime
}-y^{\prime}\right)  f^{\prime}\left(  x^{\prime}\right) \nonumber\\
&  =\left(  y-x\right)  \left(  f^{\prime}\left(  x^{\prime}\right)
-f^{\prime}\left(  x\right)  \right)  .\nonumber
\end{align}

The inequalities $\left(  \ref{2.61}\right)  $ and $\left(  \ref{2.62}\right)
$ follow from the inequalities $\left(  \ref{2.56}\right)  $ and $\left(
\ref{2.63}\right)  .$

This completes the proof.
\end{proof}

\begin{remark}
\label{r20}\textit{Using Remark \ref{r11}} we have the following results:
\end{remark}

\begin{enumerate}
\item \textit{Theorem \ref{t15}} \textit{reduces to Theorem \ref{A6} as
}$x=a,$\textit{\ }$y=y^{\prime}=\frac{a+b}{2},$\textit{\ }$x^{\prime}=b$
\textit{and }$p\left(  s\right)  \equiv\frac{1}{2\left(  b-a\right)  }$
$\left(  s\in\left[  a,b\right]  \right)  .$

\item \textit{Theorem \ref{t15}} \textit{reduces to Theorem \ref{A32} as
}$x=a,$ $y=y^{\prime}=\frac{a+b}{2},x^{\prime}=b$ \textit{and }$p\left(
s\right)  =\frac{g\left(  s\right)  }{2}$\ $\left(  s\in\left[  a,b\right]
\right)  .$
\end{enumerate}

\begin{theorem}
\label{t16}Let $x,y,y^{\prime},x^{\prime},\Omega,f,p,Ip_{2}$\ be defined as
above. Then:

\begin{enumerate}
\item The inequality%
\begin{align}
&  2f\left(  \frac{y+y^{\prime}}{2}\right)  \int\nolimits_{x}^{x^{\prime}%
}p\left(  s\right)  ds\label{2.64}\\
&  \leq\frac{2}{\alpha}\left[  \int_{\frac{x+y^{\prime}}{2}}^{\frac
{y+y^{\prime}}{2}}f\left(  s\right)  p\left(  \frac{1}{\alpha}\left(
2s-x-y^{\prime}\right)  +x\right)  ds\right. \nonumber\\
&  \text{ \ \ \ \ \ \ \ \ \ }+\left.  \int_{\frac{y+y^{\prime}}{2}}%
^{\frac{x^{\prime}+y}{2}}f\left(  s\right)  p\left(  \frac{1}{\alpha}\left(
2s-y-x^{\prime}\right)  +x^{\prime}\right)  ds\right] \nonumber\\
&  \leq\int_{0}^{1}Ip_{2}\left(  t\right)  dt\nonumber\\
&  \leq\frac{1}{2}\left[  \left(  f\left(  y\right)  +f\left(  y^{\prime
}\right)  \right)  \int\nolimits_{x}^{x^{\prime}}p\left(  s\right)  ds\right.
\nonumber\\
&  \text{ \ \ \ \ \ \ \ \ }+\left.  \int\nolimits_{x}^{x^{\prime}}\left[
f\left(  \left(  1-\alpha\right)  x+\alpha s\right)  +f\left(  \left(
1-\alpha\right)  x^{\prime}+\alpha s\right)  \right]  p\left(  s\right)
ds\right] \nonumber
\end{align}
holds.

\item If $f$\ is differentiable on $\left[  x,x^{\prime}\right]  $ and $p$ is
bounded on $\left[  x,x^{\prime}\right]  ,$ then the inequality%
\begin{align}
0  &  \leq\int\nolimits_{x}^{x^{\prime}}\left[  f\left(  \left(
1-\alpha\right)  x+\alpha s\right)  +f\left(  \left(  1-\alpha\right)
x^{\prime}+\alpha s\right)  \right]  p\left(  s\right)  ds-Ip_{2}\left(
t\right) \label{2.65}\\
&  \leq\frac{1-t}{\alpha}\left[  \left(  f\left(  x\right)  +f\left(
x^{\prime}\right)  \right)  \left(  y^{\prime}-x\right)  -\left(  f\left(
y\right)  +f\left(  y^{\prime}\right)  \right)  \left(  y^{\prime}-y\right)
\right. \nonumber\\
&  \text{
\ \ \ \ \ \ \ \ \ \ \ \ \ \ \ \ \ \ \ \ \ \ \ \ \ \ \ \ \ \ \ \ \ \ \ \ \ \ \ \ \ \ \ \ \ \ \ \ \ \ }%
-\left.  \int_{\Omega}f\left(  s\right)  ds\right]  \left\Vert p\right\Vert
_{\infty}\nonumber
\end{align}
holds for all $t\in\left[  0,1\right]  .$

\item If $f$\ is differentiable on $\left[  x,x^{\prime}\right]  ,$ then the
inequalities%
\begin{align}
0  &  \leq Ip_{2}\left(  t\right)  -2f\left(  \frac{y+y^{\prime}}{2}\right)
\int\nolimits_{x}^{x^{\prime}}p\left(  s\right)  ds\label{2.66}\\
&  \leq\frac{2x^{\prime}-y-y^{\prime}}{2}\left(  f^{\prime}\left(  x^{\prime
}\right)  -f^{\prime}\left(  x\right)  \right)  \int\nolimits_{x}^{x^{\prime}%
}p\left(  s\right)  ds\nonumber
\end{align}
and%
\begin{align}
0  &  \leq\left[  f\left(  x\right)  +f\left(  x^{\prime}\right)  \right]
\int\nolimits_{x}^{x^{\prime}}p\left(  s\right)  ds-Ip_{2}\left(  t\right)
\label{2.67}\\
&  \leq\frac{2x^{\prime}-y-y^{\prime}}{2}\left(  f^{\prime}\left(  x^{\prime
}\right)  -f^{\prime}\left(  x\right)  \right)  \int\nolimits_{x}^{x^{\prime}%
}p\left(  s\right)  ds\nonumber
\end{align}
hold for all $t\in\left[  0,1\right]  .$
\end{enumerate}
\end{theorem}

\begin{proof}
$\left(  1\right)  $ Using simple techniques of integration, $2x+x^{\prime
}-y=x+y^{\prime},$ $x+x^{\prime}+y=2y+y^{\prime},$ $2x+2x^{\prime}-y^{\prime
}=2y+y^{\prime}$\ and the hypothesis of $p$%
\begin{align*}
&  p\left(  \frac{1}{\alpha}\left(  2s-x-y^{\prime}\right)  +x\right) \\
&  =p\left(  \frac{1}{\alpha}\left(  x+y^{\prime}-2s\right)  +x^{\prime
}\right)  \text{ \ }\left(  s\in\left[  \frac{x+y^{\prime}}{2},\frac
{y+y^{\prime}}{2}\right]  \right)  ,
\end{align*}
we have the following identities
\begin{align*}
&  2f\left(  \frac{y+y^{\prime}}{2}\right)  \int\nolimits_{x}^{x^{\prime}%
}p\left(  s\right)  ds\\
&  =\int_{x}^{y}\int_{0}^{\frac{1}{2}}\frac{4}{\alpha}f\left(  \frac
{y+y^{\prime}}{2}\right)  p\left(  \frac{1}{\alpha}\left(  s-x\right)
+x\right)  dtds,
\end{align*}%
\begin{align*}
&  \frac{2}{\alpha}\left[  \int_{\frac{x+y^{\prime}}{2}}^{\frac{y+y^{\prime}%
}{2}}f\left(  s\right)  p\left(  \frac{1}{\alpha}\left(  2s-x-y^{\prime
}\right)  +x\right)  ds\right. \\
&  +\left.  \int_{\frac{y+y^{\prime}}{2}}^{\frac{x^{\prime}+y}{2}}f\left(
s\right)  p\left(  \frac{1}{\alpha}\left(  2s-y-x^{\prime}\right)  +x^{\prime
}\right)  ds\right] \\
&  =\frac{2}{\alpha}\left[  \int_{\frac{x+y^{\prime}}{2}}^{\frac{y+y^{\prime}%
}{2}}f\left(  s\right)  p\left(  \frac{1}{\alpha}\left(  2s-x-y^{\prime
}\right)  +x\right)  ds\right. \\
&  +\left.  \int_{\frac{x+y^{\prime}}{2}}^{\frac{y+y^{\prime}}{2}}f\left(
x+x^{\prime}-s\right)  p\left(  \frac{1}{\alpha}\left(  x+y^{\prime
}-2s\right)  +x^{\prime}\right)  ds\right] \\
&  =\frac{2}{\alpha}\int_{\frac{x+y^{\prime}}{2}}^{\frac{y+y^{\prime}}{2}%
}\left[  f\left(  s\right)  +f\left(  x+x^{\prime}-s\right)  \right]  p\left(
\frac{1}{\alpha}\left(  2s-x-y^{\prime}\right)  +x\right)  ds\\
&  =\int_{x}^{y}\int_{0}^{\frac{1}{2}}\frac{2}{\alpha}\left[  f\left(
\frac{s+y^{\prime}}{2}\right)  +f\left(  \frac{2y+y^{\prime}-s}{2}\right)
\right]  p\left(  \frac{1}{\alpha}\left(  s-x\right)  +x\right)  dtds
\end{align*}%
\begin{align*}
&  \int_{0}^{1}Ip_{2}\left(  t\right)  dt\\
&  =\int_{x}^{x^{\prime}}\int_{0}^{\frac{1}{2}}\left[  f\left(  t\left(
\left(  1-\alpha\right)  x+\alpha s\right)  +\left(  1-t\right)  y^{\prime
}\right)  \right. \\
&  +\left.  f\left(  \left(  1-t\right)  \left(  \left(  1-\alpha\right)
x+\alpha s\right)  +ty^{\prime}\right)  \right]  p\left(  s\right)  dtds\\
&  +\int_{x}^{x^{\prime}}\int_{0}^{\frac{1}{2}}\left[  f\left(  t\left(
\left(  1-\alpha\right)  x^{\prime}+\alpha s\right)  +\left(  1-t\right)
y\right)  \right. \\
&  +\left.  f\left(  \left(  1-t\right)  \left(  \left(  1-\alpha\right)
x^{\prime}+\alpha s\right)  +ty\right)  \right]  p\left(  s\right)  dtds\\
&  =\int_{x}^{x^{\prime}}\int_{0}^{\frac{1}{2}}\left[  f\left(  t\left(
\left(  1-\alpha\right)  x+\alpha s\right)  +\left(  1-t\right)  y^{\prime
}\right)  \right. \\
&  +\left.  f\left(  \left(  1-t\right)  \left(  \left(  1-\alpha\right)
x+\alpha s\right)  +ty^{\prime}\right)  \right]  p\left(  s\right)  dtds\\
&  +\int_{x}^{x^{\prime}}\int_{0}^{\frac{1}{2}}\left[  f\left(  t\left(
x^{\prime}+\alpha x-\alpha s\right)  +\left(  1-t\right)  y\right)  \right. \\
&  +\left.  f\left(  \left(  1-t\right)  \left(  x^{\prime}+\alpha x-\alpha
s\right)  +ty\right)  \right]  p\left(  x+x^{\prime}-s\right)  dtds\text{
\ \ }%
\end{align*}%
\begin{align*}
&  =\int_{x}^{x^{\prime}}\int_{0}^{\frac{1}{2}}\left[  f\left(  t\left(
\left(  1-\alpha\right)  x+\alpha s\right)  +\left(  1-t\right)  y^{\prime
}\right)  \right. \\
&  +\left.  f\left(  \left(  1-t\right)  \left(  \left(  1-\alpha\right)
x+\alpha s\right)  +ty^{\prime}\right)  \right] \\
&  +f\left(  t\left(  x^{\prime}+\alpha x-\alpha s\right)  +\left(
1-t\right)  y\right) \\
&  +\left.  f\left(  \left(  1-t\right)  \left(  x^{\prime}+\alpha x-\alpha
s\right)  +ty\right)  \right]  p\left(  s\right)  dtds\text{
\ \ \ \ \ \ \ \ \ \ \ \ \ \ }%
\end{align*}%
\begin{align*}
&  =\int_{x}^{y}\int_{0}^{\frac{1}{2}}\frac{1}{\alpha}\left[  f\left(
ts+\left(  1-t\right)  y^{\prime}\right)  +f\left(  \left(  1-t\right)
s+ty^{\prime}\right)  \right. \\
&  +f\left(  t\left(  x+x^{\prime}-s\right)  +\left(  1-t\right)  y\right) \\
&  +\left.  f\left(  \left(  1-t\right)  \left(  x+x^{\prime}-s\right)
+ty\right)  \right]  p\left(  \frac{1}{\alpha}\left(  s-x\right)  +x\right)
dtds
\end{align*}
and%
\begin{align*}
&  \frac{1}{2}\left[  \left(  f\left(  y\right)  +f\left(  y^{\prime}\right)
\right)  \int\nolimits_{x}^{x^{\prime}}p\left(  s\right)  ds\right. \\
&  +\left.  \int\nolimits_{x}^{x^{\prime}}\left[  f\left(  \left(
1-\alpha\right)  x+\alpha s\right)  +f\left(  \left(  1-\alpha\right)
x^{\prime}+\alpha s\right)  \right]  p\left(  s\right)  ds\right] \\
&  =\frac{1}{2}\left[  \left(  f\left(  y\right)  +f\left(  y^{\prime}\right)
\right)  \int\nolimits_{x}^{x^{\prime}}p\left(  s\right)  ds\right. \\
&  +\left.  \int\nolimits_{x}^{x^{\prime}}\left[  f\left(  \left(
1-\alpha\right)  x+\alpha s\right)  +f\left(  x^{\prime}+\alpha x-\alpha
s\right)  \right]  p\left(  s\right)  ds\right] \\
&  =\int_{x}^{y}\int_{0}^{\frac{1}{2}}\frac{1}{\alpha}\left[  f\left(
s\right)  +f\left(  y^{\prime}\right)  \right]  p\left(  \frac{1}{\alpha
}\left(  s-x\right)  +x\right)  dtds\\
&  +\int_{x}^{y}\int_{0}^{\frac{1}{2}}\frac{1}{\alpha}\left[  f\left(
y\right)  +f\left(  x+x^{\prime}-s\right)  \right]  p\left(  \frac{1}{\alpha
}\left(  s-x\right)  +x\right)  dtds.
\end{align*}

By Lemma \ref{l2}, the following inequalities hold for all $t\in\left[
0,\frac{1}{2}\right]  $ and $s\in\left[  x,y\right]  :$%
\begin{align*}
&  \frac{4}{\alpha}f\left(  \frac{y+y^{\prime}}{2}\right)  p\left(  \frac
{1}{\alpha}\left(  s-x\right)  +x\right) \\
&  \leq\frac{2}{\alpha}\left[  f\left(  \frac{s+y^{\prime}}{2}\right)
+f\left(  \frac{2y+y^{\prime}-s}{2}\right)  \right]  p\left(  \frac{1}{\alpha
}\left(  s-x\right)  +x\right)  ,
\end{align*}%
\begin{align*}
&  \frac{2}{\alpha}f\left(  \frac{s+y^{\prime}}{2}\right)  p\left(  \frac
{1}{\alpha}\left(  s-x\right)  +x\right) \\
&  \leq\frac{1}{\alpha}\left[  f\left(  ty^{\prime}+\left(  1-t\right)
s\right)  +f\left(  ts+\left(  1-t\right)  y^{\prime}\right)  \right]
p\left(  \frac{1}{\alpha}\left(  s-x\right)  +x\right)  ,
\end{align*}%
\begin{align*}
&  \frac{2}{\alpha}f\left(  \frac{2y+y^{\prime}-s}{2}\right)  p\left(
\frac{1}{\alpha}\left(  s-x\right)  +x\right) \\
&  \leq\frac{1}{\alpha}\left[  f\left(  t\left(  x+\text{ }x^{\prime
}-s\right)  +\left(  1-t\right)  y\right)  \right. \\
&  \left.  +f\left(  \left(  1-t\right)  \left(  x+x^{\prime}-s\right)
+ty\right)  \right]  p\left(  \frac{1}{\alpha}\left(  s-x\right)  +x\right)  ,
\end{align*}%
\begin{align*}
&  \frac{1}{\alpha}\left[  f\left(  ty^{\prime}+\left(  1-t\right)  s\right)
+f\left(  ts+\left(  1-t\right)  y^{\prime}\right)  \right]  p\left(  \frac
{1}{\alpha}\left(  s-x\right)  +x\right) \\
&  \leq\frac{1}{\alpha}\left[  f\left(  s\right)  +f\left(  y^{\prime}\right)
\right]  p\left(  \frac{1}{\alpha}\left(  s-x\right)  +x\right)
\end{align*}
and%
\begin{align*}
&  \frac{1}{\alpha}\left[  f\left(  t\left(  x+x^{\prime}-s\right)  +\left(
1-t\right)  y\right)  \right. \\
&  \left.  +f\left(  \left(  1-t\right)  \left(  x+x^{\prime}-s\right)
+ty\right)  \right]  p\left(  \frac{1}{\alpha}\left(  s-x\right)  +x\right) \\
&  \leq\frac{1}{\alpha}\left[  f\left(  y\right)  +f\left(  x+x^{\prime
}-s\right)  \right]  p\left(  \frac{1}{\alpha}\left(  s-x\right)  +x\right)  .
\end{align*}

Integrating the above inequalities over $t$ on $\left[  0,\frac{1}{2}\right]
, $ over $s$ on $\left[  x,y\right]  $ and using the above identities, we
derive $\left(  \ref{2.64}\right)  .$

$\left(  2\right)  $ By integration by parts and $\Omega=\left[  x,y\right]
\cup\left[  y^{\prime},x^{\prime}\right]  ,$ we have the following identity%
\begin{align*}
&  \int_{x}^{y}\left[  \left(  s-y^{\prime}\right)  f^{\prime}\left(
s\right)  +\left(  y^{\prime}-s\right)  f^{\prime}\left(  x+x^{\prime
}-s\right)  \right]  ds\\
&  =\left(  f\left(  x\right)  +f\left(  x^{\prime}\right)  \right)  \left(
y^{\prime}-x\right)  -\left(  f\left(  y\right)  +f\left(  y^{\prime}\right)
\right)  \left(  y^{\prime}-y\right)  -\int_{\Omega}f\left(  s\right)  ds.
\end{align*}

Now, using the convexity of $f$ and $x+x^{\prime}=y+y^{\prime},$ the
inequalities
\[
f\left(  s\right)  -f\left(  ts+\left(  1-t\right)  y^{\prime}\right)
\leq\left(  1-t\right)  \left(  s-y^{\prime}\right)  f^{\prime}\left(
s\right)
\]
and%
\[
f\left(  x+x^{\prime}-s\right)  -f\left(  t\left(  x+x^{\prime}-s\right)
+\left(  1-t\right)  y\right)  \leq\left(  1-t\right)  \left(  y^{\prime
}-s\right)  f^{\prime}\left(  x+x^{\prime}-s\right)
\]
hold for all $t\in\left[  0,1\right]  $ and $s\in\left[  x,y\right]  $. Using
the identity $\left(  \ref{2.46}\right)  ,$ the hypothesis of $p,$ the above
inequalities, the convexity of $f$ and the above identity, we have%
\begin{align*}
&  \int\nolimits_{x}^{x^{\prime}}\left[  f\left(  \left(  1-\alpha\right)
x+\alpha s\right)  +f\left(  \left(  1-\alpha\right)  x^{\prime}+\alpha
s\right)  \right]  p\left(  s\right)  ds-Ip_{2}\left(  t\right) \\
&  =\int_{x}^{y}\frac{1}{\alpha}\left[  \left(  f\left(  s\right)  -f\left(
ts+\left(  1-t\right)  y^{\prime}\right)  \right)  \right. \\
&  +\left.  \left(  f\left(  x+x^{\prime}-s\right)  -f\left(  t\left(
x+x^{\prime}-s\right)  +\left(  1-t\right)  y\right)  \right)  \right]
p\left(  \frac{1}{\alpha}\left(  s-x\right)  +x\right)  ds\\
&  \leq\int_{x}^{y}\frac{1}{\alpha}\left[  \left(  1-t\right)  \left(
s-y^{\prime}\right)  f^{\prime}\left(  s\right)  \right. \\
&  \text{ \ \ \ \ \ \ \ \ \ \ \ \ \ \ \ \ \ \ \ \ \ }+\left.  \left(
1-t\right)  \left(  y^{\prime}-s\right)  f^{\prime}\left(  x+x^{\prime
}-s\right)  \right]  p\left(  \frac{1}{\alpha}\left(  s-x\right)  +x\right)
ds\\
&  =\frac{1-t}{\alpha}\int_{x}^{y}\left(  y^{\prime}-s\right)  \left(
f^{\prime}\left(  x+x^{\prime}-s\right)  -f^{\prime}\left(  s\right)  \right)
p\left(  \frac{1}{\alpha}\left(  s-x\right)  +x\right)  ds\\
&  \leq\frac{1-t}{\alpha}\int_{x}^{y}\left(  y^{\prime}-s\right)  \left(
f^{\prime}\left(  x+x^{\prime}-s\right)  -f^{\prime}\left(  s\right)  \right)
ds\left\Vert p\right\Vert _{\infty}\\
&  =\frac{1-t}{\alpha}\left[  \left(  f\left(  x\right)  +f\left(  x^{\prime
}\right)  \right)  \left(  y^{\prime}-x\right)  -\left(  f\left(  y\right)
+f\left(  y^{\prime}\right)  \right)  \left(  y^{\prime}-y\right)  \right. \\
&  \text{
\ \ \ \ \ \ \ \ \ \ \ \ \ \ \ \ \ \ \ \ \ \ \ \ \ \ \ \ \ \ \ \ \ \ \ \ \ \ \ \ \ \ \ \ \ \ \ \ \ \ }%
-\left.  \int_{\Omega}f\left(  s\right)  ds\right]  \left\Vert p\right\Vert
_{\infty}.
\end{align*}
\ 

Using the above inequality and $\left(  \ref{2.44}\right)  $, we have derive
$\left(  \ref{2.65}\right)  .$

$\left(  3\right)  $ Using the convexity of $f,$ $y+y^{\prime}-2x=2x^{\prime
}-y-y^{\prime}$\ and the hypothesis of $p$, we have the inequality
\begin{align}
&  \left[  f\left(  x\right)  +f\left(  x^{\prime}\right)  \right]  -2f\left(
\frac{y+y^{\prime}}{2}\right) \label{2.68}\\
&  =\left[  f\left(  x\right)  -f\left(  \frac{y+y^{\prime}}{2}\right)
\right]  +\left[  f\left(  x^{\prime}\right)  -f\left(  \frac{y+y^{\prime}}%
{2}\right)  \right] \nonumber\\
&  \leq\frac{2x-y-y^{\prime}}{2}f^{\prime}\left(  x\right)  +\frac{2x^{\prime
}-y-y^{\prime}}{2}f^{\prime}\left(  x^{\prime}\right) \nonumber\\
&  =\frac{2x^{\prime}-y-y^{\prime}}{2}\left[  f^{\prime}\left(  x^{\prime
}\right)  -f^{\prime}\left(  x\right)  \right]  .\nonumber
\end{align}

The inequalities $\left(  \ref{2.66}\right)  $ and $\left(  \ref{2.67}\right)
$ follow from the inequalities $\left(  \ref{2.57}\right)  ,$ and $\left(
\ref{2.68}\right)  .$

This completes the proof.
\end{proof}

\begin{remark}
\label{r21}\textit{Using Remark \ref{r11}} we have the following results:
\end{remark}

\begin{enumerate}
\item \textit{Theorem \ref{t16}} \textit{reduces to Theorem \ref{A6} as
}$x=a,$\textit{\ }$y=y^{\prime}=\frac{a+b}{2},$\textit{\ }$x^{\prime}=b$
\textit{and }$p\left(  s\right)  \equiv\frac{1}{2\left(  b-a\right)  }$
$\left(  s\in\left[  a,b\right]  \right)  .$

\item \textit{Theorem \ref{t16}} \textit{reduces to Theorem \ref{A32} as
}$x=a,$ $y=y^{\prime}=\frac{a+b}{2},x^{\prime}=b$ \textit{and }$p\left(
s\right)  =\frac{g\left(  s\right)  }{2}$\ $\left(  s\in\left[  a,b\right]
\right)  .$
\end{enumerate}

\begin{remark}
\label{r22}\textit{Using Remark \ref{r11}, Theorems \ref{t15}} and
\textit{\ref{t16}}\ \textit{reduce to Theorem \ref{A16} as }$p\left(
s\right)  \equiv\frac{\alpha}{2\left(  y-x\right)  }$\ $\left(  s\in
\Omega\right)  .$
\end{remark}

The following corollary is a natural consequence of the inequalities $\left(
\ref{2.56}\right)  ,$ $\left(  \ref{2.57}\right)  ,$ $\left(  \ref{2.63}%
\right)  $ and $\left(  \ref{2.68}\right)  .$

\begin{corollary}
\label{c3}\textit{Let }$x,y,y^{\prime},x^{\prime},f,p,Ip_{1},Ip_{2},Jp,Mp,Np$
\textit{be defined as above. }If $f$\ is differentiable on $\left[
x,x^{\prime}\right]  ,$ then we have the following inequalities%
\begin{equation}
0\leq\Gamma_{1}-\Theta_{1}\leq\frac{2x^{\prime}-y-y^{\prime}}{2}\left(
f^{\prime}\left(  x^{\prime}\right)  -f^{\prime}\left(  x\right)  \right)
\int\nolimits_{x}^{x^{\prime}}p\left(  s\right)  ds\nonumber
\end{equation}
and%
\begin{equation}
0\leq\Gamma_{2}-\Theta_{2}\leq\left(  y-x\right)  \left(  f^{\prime}\left(
x^{\prime}\right)  -f^{\prime}\left(  x\right)  \right)  \int\nolimits_{x}%
^{x^{\prime}}p\left(  s\right)  ds\nonumber
\end{equation}
where
\[
\Gamma_{1},\Theta_{1}\in\left\{  Ip_{1},Ip_{2},Jp,Mp,Np,2f\left(
\frac{y+y^{\prime}}{2}\right)  ,\left[  f\left(  x\right)  +f\left(
x^{\prime}\right)  \right]  \right\}
\]
and%
\[
\Gamma_{2},\Theta_{2}\in\left\{  Ip_{1},Jp,Mp,Np,\left[  f\left(  y\right)
+f\left(  y^{\prime}\right)  \right]  ,\left[  f\left(  x\right)  +f\left(
x^{\prime}\right)  \right]  \right\}
\]
with $\Theta_{1}\leq\Gamma_{1}$ and $\Theta_{2}\leq\Gamma_{2}.$
\end{corollary}

\begin{theorem}
\label{t17}\textit{Let }$x,y,y^{\prime},x^{\prime},f,p,H_{1},G_{1},Ip_{1}%
$\textit{\ be defined as above. Then:}
\end{theorem}

\begin{enumerate}
\item \textit{The following inequality holds for all }$t\in\left[  0,1\right]
$:%
\begin{equation}
Ip_{1}\left(  t\right)  \leq2G_{1}\left(  t\right)  \int\nolimits_{x}%
^{x^{\prime}}p\left(  s\right)  ds. \label{2.69}%
\end{equation}

\item \textit{If }$f$ \textit{is differentiable on }$\Omega$ \textit{and }$p
$\textit{\ is bounded on }$\left[  x,x^{\prime}\right]  ,$\textit{\ then, for
all }$t\in\left[  0,1\right]  ,$ \textit{we have the inequality}
\begin{align}
0  &  \leq Ip_{1}\left(  t\right)  -\left[  f\left(  y\right)  +f\left(
y^{\prime}\right)  \right]  \int\nolimits_{x}^{x^{\prime}}p\left(  s\right)
ds\label{2.70}\\
&  \leq2\left(  x^{\prime}-x\right)  \left[  G_{1}\left(  t\right)
-H_{1}\left(  t\right)  \right]  \left\Vert p\right\Vert _{\infty}.\nonumber
\end{align}

\end{enumerate}

\begin{proof}
$\left(  1\right)  $ Using simple integration techniques, the following
identities hold:%
\begin{equation}
\int\nolimits_{x}^{x^{\prime}}p\left(  s\right)  ds=\int_{x}^{y}\frac
{1}{\alpha}p\left(  \frac{1}{\alpha}\left(  s-x\right)  +x\right)  ds.
\label{2.71}%
\end{equation}%
\begin{align*}
&  2G_{1}\left(  t\right)  \int\nolimits_{x}^{x^{\prime}}p\left(  s\right)
ds\\
&  =\int_{x}^{y}\frac{1}{\alpha}\left[  f\left(  tx+\left(  1-t\right)
y\right)  \right. \\
&  +\left.  f\left(  tx^{\prime}+\left(  1-t\right)  y^{\prime}\right)
\right]  p\left(  \frac{1}{\alpha}\left(  s-x\right)  +x\right)  ds\text{
\ }\left(  t\in\left[  0,1\right]  \right)  .
\end{align*}

By Lemma \ref{l2}, the following inequalities hold for all $t\in\left[
0,1\right]  $ and $s\in\left[  x,y\right]  :$%
\begin{align*}
&  \frac{1}{\alpha}\left[  f\left(  ts+\left(  1-t\right)  y\right)  +f\left(
t\left(  x+x^{\prime}-s\right)  +\left(  1-t\right)  y^{\prime}\right)
\right]  p\left(  \frac{1}{\alpha}\left(  s-x\right)  +x\right) \\
&  \leq\frac{1}{\alpha}\left[  f\left(  tx+\left(  1-t\right)  y\right)
+f\left(  tx^{\prime}+\left(  1-t\right)  y^{\prime}\right)  \right]  p\left(
\frac{1}{\alpha}\left(  s-x\right)  +x\right)
\end{align*}

Integrating the above inequalities over $s$ on $\left[  x,y\right]  $ and
using the above identities and $\left(  \ref{2.43}\right)  ,$ we derive
$\left(  \ref{2.69}\right)  .\medskip$

\noindent$\left(  2\right)  $ By integration by parts and $x+x^{\prime
}=y+y^{\prime},$ the following identity%
\begin{align*}
&  \frac{t}{\alpha}\int\nolimits_{x}^{y}\left(  y-s\right)  \left[  f^{\prime
}\left(  t\left(  x+x^{\prime}-s\right)  +\left(  1-t\right)  y^{\prime
}\right)  -f^{\prime}\left(  ts+\left(  1-t\right)  y\right)  \right]  ds\\
&  =\frac{t}{\alpha}\int\nolimits_{x}^{y}\left(  s-y\right)  f^{\prime}\left(
ts+\left(  1-t\right)  y\right)  ds\\
&  \text{ \ }+\frac{t}{\alpha}\int\nolimits_{x}^{y}\left(  y-s\right)
f^{\prime}\left(  t\left(  x+x^{\prime}-s\right)  +\left(  1-t\right)
y^{\prime}\right)  ds\\
&  =\frac{y-x}{\alpha}\left[  f\left(  tx+\left(  1-t\right)  y\right)
+f\left(  tx^{\prime}+\left(  1-t\right)  y^{\prime}\right)  \right] \\
&  \text{ \ \ \ \ }-\frac{1}{\alpha}%
{\displaystyle\int\nolimits_{x}^{y}}
\left[  f\left(  ts+\left(  1-t\right)  y\right)  +f\left(  t\left(
x+x^{\prime}-s\right)  +\left(  1-t\right)  y^{\prime}\right)  \right]  ds\\
&  =\frac{2\left(  y-x\right)  }{\alpha}\left[  G_{1}\left(  t\right)
-H_{1}\left(  t\right)  \right]  ,
\end{align*}%
\[
f^{\prime}\left(  t\left(  x+x^{\prime}-s\right)  +\left(  1-t\right)
y^{\prime}\right)  -f^{\prime}\left(  ts+\left(  1-t\right)  y\right)  \geq0
\]
hold for all $t\in\left[  0,1\right]  .$ Now, using the convexity of $f,$ the
inequalities
\[
\left[  f\left(  ts+\left(  1-t\right)  y\right)  -f\left(  y\right)  \right]
\leq t\left(  s-y\right)  f^{\prime}\left(  ts+\left(  1-t\right)  y\right)
\]
and%
\[
\left[  f\left(  t\left(  x+x^{\prime}-s\right)  +\left(  1-t\right)
y^{\prime}\right)  -f\left(  y^{\prime}\right)  \right]  \leq t\left(
y-s\right)  f^{\prime}\left(  t\left(  x+x^{\prime}-s\right)  +\left(
1-t\right)  y^{\prime}\right)
\]
hold for all $t\in\left[  0,1\right]  $ and $s\in\left[  x,y\right]  $. Using
the convexity of $f,$ $\left(  \ref{2.43}\right)  ,$ $\left(  \ref{2.71}%
\right)  ,$ the above inequalities and the above identity, we have%
\begin{align*}
&  Ip_{1}\left(  t\right)  -\left[  f\left(  y\right)  +f\left(  y^{\prime
}\right)  \right]  \int\nolimits_{x}^{x^{\prime}}p\left(  s\right)  ds\\
&  =\int_{x}^{y}\frac{1}{\alpha}\left[  f\left(  ts+\left(  1-t\right)
y\right)  -f\left(  y\right)  \right. \\
&  +\left.  f\left(  t\left(  x+x^{\prime}-s\right)  +\left(  1-t\right)
y^{\prime}\right)  -f\left(  y^{\prime}\right)  \right]  p\left(  \frac
{1}{\alpha}\left(  s-x\right)  +x\right)  ds\\
&  \leq\int_{x}^{y}\frac{1}{\alpha}\left[  t\left(  s-y\right)  f^{\prime
}\left(  ts+\left(  1-t\right)  y\right)  \right. \\
&  +\left.  t\left(  y-s\right)  f^{\prime}\left(  t\left(  x+x^{\prime
}-s\right)  +\left(  1-t\right)  y^{\prime}\right)  p\left(  \frac{1}{\alpha
}\left(  s-x\right)  +x\right)  \right]  ds\\
&  =\frac{t}{\alpha}\int_{x}^{y}\left(  y-s\right)  \left[  f^{\prime}\left(
t\left(  x+x^{\prime}-s\right)  +\left(  1-t\right)  y^{\prime}\right)
\right. \\
&  \text{ \ \ \ \ \ \ \ \ \ \ \ \ \ \ \ \ \ }-\left.  f^{\prime}\left(
ts+\left(  1-t\right)  y\right)  \right]  p\left(  \frac{1}{\alpha}\left(
s-x\right)  +x\right)  ds\\
&  \leq\frac{t}{\alpha}\int_{x}^{y}\left(  y-s\right)  \left[  f^{\prime
}\left(  t\left(  x+x^{\prime}-s\right)  +\left(  1-t\right)  y^{\prime
}\right)  \right. \\
&  \text{\ \ \ \ \ \ \ \ \ \ \ \ \ \ \ \ }-\left.  f^{\prime}\left(
ts+\left(  1-t\right)  y\right)  \right]  ds\left\Vert p\right\Vert _{\infty
}\\
&  =\frac{2\left(  y-x\right)  }{\alpha}\left[  G_{1}\left(  t\right)
-H_{1}\left(  t\right)  \right]  \left\Vert p\right\Vert _{\infty}\\
&  =2\left(  x^{\prime}-x\right)  \left[  G_{1}\left(  t\right)  -H_{1}\left(
t\right)  \right]  \left\Vert p\right\Vert _{\infty}.
\end{align*}
\ 

Using the above inequality and $\left(  \ref{2.41}\right)  $, we derive
$\left(  \ref{2.70}\right)  .$

This completes the proof.
\end{proof}

\begin{remark}
\label{r23}\textit{Using Remark \ref{r11}} we have the following results:
\end{remark}

\begin{enumerate}
\item \textit{Theorem \ref{t17}} \textit{reduces to the inequalities }
$\left(  \ref{1.3}\right)  $\textit{\ and } $\left(  \ref{1.5}\right)
$\textit{\ as }$x=a,$\textit{\ }$y=y^{\prime}=\frac{a+b}{2},$\textit{\ }%
$x^{\prime}=b$ \textit{and }$p\left(  s\right)  \equiv\frac{1}{2\left(
b-a\right)  }$ $\left(  s\in\left[  a,b\right]  \right)  .$

\item \textit{Theorem \ref{t17}} \textit{reduces to Theorem \ref{A33} as
}$x=a,$ $y=y^{\prime}=\frac{a+b}{2},x^{\prime}=b$ \textit{and }$p\left(
s\right)  =\frac{g\left(  s\right)  }{2}$\ $\left(  s\in\left[  a,b\right]
\right)  .$
\end{enumerate}

\begin{theorem}
\label{t18}\textit{Let }$m,$ $m^{\prime}$\textit{\ be defined as in Theorem
\ref{t9} and let }$x,y,y^{\prime},x^{\prime},f,p,H_{2},G_{2},Ip_{2}%
$\textit{\ be defined as above. Then:}
\end{theorem}

\begin{enumerate}
\item \textit{The following inequality holds for all }$t\in\left[  0,m\right]
\cup\left[  m^{\prime},1\right]  $:%
\begin{equation}
Ip_{2}\left(  t\right)  \leq\left(  \geq\right)  2G_{2}\left(  t\right)
\int\nolimits_{x}^{x^{\prime}}p\left(  s\right)  ds\text{ \ \ \textit{as }%
}t\in\left[  m^{\prime},1\right]  \text{ \ }\left(  t\in\left[  0,m\right]
\right)  . \label{2.72}%
\end{equation}

\item \textit{If }$f$ \textit{is differentiable on }$\Omega$ \textit{and }$p
$\textit{\ is bounded on }$\left[  x,x^{\prime}\right]  ,$\textit{\ then, for
all }$t\in\left[  m^{\prime},1\right]  ,$ \textit{we have the inequality}
\begin{align}
0  &  \leq Ip_{2}\left(  t\right)  -2f\left(  \frac{y+y^{\prime}}{2}\right)
\int\nolimits_{x}^{x^{\prime}}p\left(  s\right)  ds\label{2.73}\\
&  \leq\frac{2}{\alpha}\left[  \left(  y^{\prime}-x\right)  G_{2}\left(
t\right)  -2\left(  y^{\prime}-y\right)  G_{3}\left(  t\right)  -\left(
y-x\right)  H_{2}\left(  t\right)  \right]  \left\Vert p\right\Vert _{\infty
},\nonumber
\end{align}
\textit{where }%
\[
G_{3}\left(  t\right)  =\frac{1}{2}\left[  f\left(  ty+\left(  1-t\right)
y^{\prime}\right)  +f\left(  ty^{\prime}+\left(  1-t\right)  y\right)
\right]  \text{ \ \ }\left(  t\in\left[  0,1\right]  \right)  .
\]

\end{enumerate}

\begin{proof}
$\left(  1\right)  $ Using the identity $\left(  \ref{2.71}\right)  ,$ the
following identity holds on $\left[  0,m\right]  \cup\left[  m^{\prime
},1\right]  :$%
\begin{align*}
2G_{2}\left(  t\right)  \int\nolimits_{x}^{x^{\prime}}p\left(  s\right)  ds
&  =\int_{x}^{y}\frac{1}{\alpha}\left[  f\left(  tx+\left(  1-t\right)
y^{\prime}\right)  \right. \\
&  +\left.  f\left(  tx^{\prime}+\left(  1-t\right)  y\right)  \right]
p\left(  \frac{1}{\alpha}\left(  s-x\right)  +x\right)  ds.
\end{align*}

As a simple calculation shows us that:%
\[
t\left(  x^{\prime}-y-x+y^{\prime}\right)  =2t\left(  y^{\prime}-x\right)
\leq2m\left(  y^{\prime}-x\right)  =y^{\prime}-y\text{ \ as }t\in\left[
0,m\right]  .
\]%
\begin{align*}
&  t\left(  x+x^{\prime}-2s+y^{\prime}-y\right) \\
&  =2t\left(  y^{\prime}-s\right)  \geq2t\left(  y^{\prime}-y\right)
\geq2m^{\prime}\left(  y^{\prime}-y\right)  =y^{\prime}-y\text{ \ as }%
t\in\left[  m^{\prime},1\right]  .
\end{align*}
for all $s\in\left[  x,y\right]  .$ Then we have the following inequalities
for all $s\in\left[  x,y\right]  $:%
\begin{align*}
t\left(  x+x^{\prime}-s\right)  +\left(  1-t\right)  y  &  \leq tx^{\prime
}+\left(  1-t\right)  y\\
&  \leq tx+\left(  1-t\right)  y^{\prime}\\
&  \leq ts+\left(  1-t\right)  y^{\prime}\text{ \ as }t\in\left[  0,m\right]
\end{align*}
and%
\begin{align}
tx+\left(  1-t\right)  y^{\prime}  &  \leq ts+\left(  1-t\right)  y^{\prime
}\label{2.74}\\
&  \leq t\left(  x+x^{\prime}-s\right)  +\left(  1-t\right)  y\nonumber\\
&  \leq tx^{\prime}+\left(  1-t\right)  y\text{ \ as }t\in\left[  m^{\prime
},1\right]  .\nonumber
\end{align}

Now, using Lemma \ref{l2} under the above inequalities, the following
inequalities hold for all $t\in\left[  0,m\right]  \cup\left[  m^{\prime
},1\right]  $ and $s\in\left[  x,y\right]  :$
\begin{align*}
&  \frac{1}{\alpha}\left[  f\left(  ts+\left(  1-t\right)  y^{\prime}\right)
+f\left(  t\left(  x+x^{\prime}-s\right)  +\left(  1-t\right)  y\right)
\right]  p\left(  \frac{1}{\alpha}\left(  s-x\right)  +x\right) \\
&  \leq\left(  \geq\right)  \frac{1}{\alpha}\left[  f\left(  tx+\left(
1-t\right)  y^{\prime}\right)  +f\left(  tx^{\prime}+\left(  1-t\right)
y\right)  \right]  p\left(  \frac{1}{\alpha}\left(  s-x\right)  +x\right)
\end{align*}
\ as $t\in\left[  m^{\prime},1\right]  $ \ $\left(  \left[  0,m\right]
\right)  .$ Integrating the above inequalities over $s$ on $\left[
x,y\right]  $ and using the above identity and $\left(  \ref{2.46}\right)  ,$
we derive $\left(  \ref{2.72}\right)  .\medskip$

\noindent$\left(  2\right)  $ By integration by parts and $x+x^{\prime
}=y+y^{\prime},$ the following identity%
\begin{align*}
&  \frac{t}{\alpha}\int\nolimits_{x}^{y}\left(  y^{\prime}-s\right)  \left[
f^{\prime}\left(  t\left(  x+x^{\prime}-s\right)  +\left(  1-t\right)
y\right)  -f^{\prime}\left(  ts+\left(  1-t\right)  y^{\prime}\right)
\right]  ds\\
&  =\frac{t}{\alpha}\int\nolimits_{x}^{y}\left(  s-y^{\prime}\right)
f^{\prime}\left(  ts+\left(  1-t\right)  y^{\prime}\right)  ds\\
&  \text{ \ }+\frac{t}{\alpha}\int\nolimits_{x}^{y}\left(  y^{\prime
}-s\right)  f^{\prime}\left(  t\left(  x+x^{\prime}-s\right)  +\left(
1-t\right)  y\right)  ds\\
&  =\frac{y^{\prime}-x}{\alpha}\left[  f\left(  tx+\left(  1-t\right)
y^{\prime}\right)  +f\left(  tx^{\prime}+\left(  1-t\right)  y\right)  \right]
\\
&  \text{ \ \ \ \ \ }-\frac{y^{\prime}-y}{\alpha}\left[  f\left(  ty+\left(
1-t\right)  y^{\prime}\right)  +f\left(  ty^{\prime}+\left(  1-t\right)
y\right)  \right] \\
&  \text{ \ \ \ \ \ \ }-\frac{1}{\alpha}%
{\displaystyle\int\nolimits_{x}^{y}}
\left[  f\left(  ts+\left(  1-t\right)  y^{\prime}\right)  +f\left(  t\left(
x+x^{\prime}-s\right)  +\left(  1-t\right)  y\right)  \right]  ds\\
&  =\frac{2}{\alpha}\left[  \left(  y^{\prime}-x\right)  G_{2}\left(
t\right)  -\left(  y^{\prime}-y\right)  G_{3}\left(  t\right)  -\left(
y-x\right)  H_{2}\left(  t\right)  \right]
\end{align*}
holds for all $t\in\left[  m^{\prime},1\right]  .$ Now, using the convexity of
$f$ and $\left(  \ref{2.74}\right)  ,$ the inequalities%
\[
f^{\prime}\left(  t\left(  x+x^{\prime}-s\right)  +\left(  1-t\right)
y\right)  -f^{\prime}\left(  ts+\left(  1-t\right)  y^{\prime}\right)  \geq0,
\]%
\begin{align*}
&  f\left(  ts+\left(  1-t\right)  y^{\prime}\right)  -f\left(  \frac
{y+y^{\prime}}{2}\right) \\
&  \leq\left[  t\left(  s-y^{\prime}\right)  +\frac{y^{\prime}-y}{2}\right]
f^{\prime}\left(  ts+\left(  1-t\right)  y^{\prime}\right)
\end{align*}
and%
\begin{align*}
&  f\left(  t\left(  x+x^{\prime}-s\right)  +\left(  1-t\right)  y\right)
-f\left(  \frac{y+y^{\prime}}{2}\right) \\
&  \leq\left[  t\left(  y^{\prime}-s\right)  -\frac{y^{\prime}-y}{2}\right]
f^{\prime}\left(  t\left(  x+x^{\prime}-s\right)  +\left(  1-t\right)
y^{\prime}\right)
\end{align*}
hold for all $t\in\left[  0,1\right]  $ and $s\in\left[  x,y\right]  $. Using
$\left(  \ref{2.46}\right)  ,$ $\left(  \ref{2.71}\right)  ,$ the above
inequalities and the above identity, we have%
\begin{align*}
&  Ip_{2}\left(  t\right)  -2f\left(  \frac{y+y^{\prime}}{2}\right)
\int\nolimits_{x}^{x^{\prime}}p\left(  s\right)  ds\\
&  =\int_{x}^{y}\frac{1}{\alpha}\left[  \left(  f\left(  ts+\left(
1-t\right)  y^{\prime}\right)  -f\left(  \frac{y+y^{\prime}}{2}\right)
\right)  \right. \\
&  +\left.  f\left(  t\left(  x+x^{\prime}-s\right)  +\left(  1-t\right)
y\right)  -f\left(  \frac{y+y^{\prime}}{2}\right)  \right]  p\left(  \frac
{1}{\alpha}\left(  s-x\right)  +x\right)  ds\\
&  \leq\frac{t}{\alpha}\int_{x}^{y}\left(  y^{\prime}-s\right)  \left[
f^{\prime}\left(  t\left(  x+x^{\prime}-s\right)  +\left(  1-t\right)
y\right)  \right. \\
&  \text{ \ \ \ \ \ \ \ \ \ \ \ \ \ \ \ \ \ }-\left.  f^{\prime}\left(
ts+\left(  1-t\right)  y\right)  \right]  p\left(  \frac{1}{\alpha}\left(
s-x\right)  +x\right)  ds\\
&  -\frac{y^{\prime}-y}{2\alpha}\int_{x}^{y}\left[  f^{\prime}\left(  t\left(
x+x^{\prime}-s\right)  +\left(  1-t\right)  y\right)  \right. \\
&  \text{ \ \ \ \ \ \ \ \ \ \ \ }-\left.  f^{\prime}\left(  ts+\left(
1-t\right)  y^{\prime}\right)  \right]  p\left(  \frac{1}{\alpha}\left(
s-x\right)  +x\right)  ds\\
&  \leq\frac{t}{\alpha}\int_{x}^{y}\left(  y^{\prime}-s\right)  \left[
f^{\prime}\left(  t\left(  x+x^{\prime}-s\right)  +\left(  1-t\right)
y^{\prime}\right)  \right. \\
&  \text{\ \ \ \ \ \ \ \ \ \ \ \ \ \ \ \ }-\left.  f^{\prime}\left(
ts+\left(  1-t\right)  y\right)  \right]  ds\left\Vert p\right\Vert _{\infty
}\\
&  =\frac{2}{\alpha}\left[  \left(  y^{\prime}-x\right)  G_{2}\left(
t\right)  -2\left(  y^{\prime}-y\right)  G_{3}\left(  t\right)  -\left(
y-x\right)  H_{2}\left(  t\right)  \right]  \left\Vert p\right\Vert _{\infty}.
\end{align*}
\ 

Using the above inequality and $\left(  \ref{2.44}\right)  $, we derive
$\left(  \ref{2.73}\right)  .$

This completes the proof.
\end{proof}

The following corollary is a natural consequence of \textit{Theorem \ref{t18}
as }$p\left(  s\right)  \equiv\frac{\alpha}{2\left(  y-x\right)  }$\ $\left(
s\in\Omega\right)  .$

\begin{corollary}
\label{c4}\textit{Let }$x,y,y^{\prime},x^{\prime},\Omega,m,m^{\prime}%
,f,H_{2},G_{2},G_{3}$\textit{\ be defined as in Theorem \ref{t9}. Then:}
\end{corollary}

\begin{enumerate}
\item \textit{The following inequality holds for all }$t\in\left[  0,m\right]
\cup\left[  m^{\prime},1\right]  :$%
\[
H_{2}\left(  t\right)  \leq\left(  \geq\right)  G_{2}\left(  t\right)
\text{\ \ \textit{as }}t\in\left[  m^{\prime},1\right]  \text{ \ }\left(
t\in\left[  0,m\right]  \right)  .
\]

\item \textit{If }$f$ \textit{is differentiable on }$\Omega,$\textit{\ then,
for all }$t\in\left[  m^{\prime},1\right]  ,$ \textit{we have the inequality}
\begin{align*}
0  &  \leq H_{2}\left(  t\right)  -f\left(  \frac{y+y^{\prime}}{2}\right) \\
&  \leq\frac{y^{\prime}-x}{y-x}G_{2}\left(  t\right)  -\frac{y^{\prime}%
-y}{y-x}G_{3}\left(  t\right)  -H_{2}\left(  t\right)  .
\end{align*}

\end{enumerate}

\begin{remark}
\label{r24}\textit{Using Remark \ref{r11}} we have the following results:
\end{remark}

\begin{enumerate}
\item \textit{Theorem \ref{t18}} \textit{reduces to the inequalities }
$\left(  \ref{1.3}\right)  $\textit{\ and } $\left(  \ref{1.5}\right)
$\textit{\ as }$x=a,$\textit{\ }$y=y^{\prime}=\frac{a+b}{2},$\textit{\ }%
$x^{\prime}=b$ \textit{and }$p\left(  s\right)  \equiv\frac{1}{2\left(
b-a\right)  }$ $\left(  s\in\left[  a,b\right]  \right)  .$

\item \textit{Theorem \ref{t18}} \textit{reduces to Theorem \ref{A33} as
}$x=a,$ $y=y^{\prime}=\frac{a+b}{2},x^{\prime}=b$ \textit{and }$p\left(
s\right)  =\frac{g\left(  s\right)  }{2}$\ $\left(  s\in\left[  a,b\right]
\right)  .$
\end{enumerate}

\begin{theorem}
\label{t19}\textit{Let }$x,y,y^{\prime},x^{\prime},f,p,G_{1},G_{2}%
,Ip_{1},Ip_{2},Sp$\textit{\ be defined as above. Then} \textit{we have the
following results:}
\end{theorem}

\begin{enumerate}
\item $Sp$\textit{\ is convex on }$\left[  0,1\right]  .$

\item \textit{The following inequalities hold for all }$t\in\left[
0,1\right]  :$%
\begin{align}
&  \left[  G_{1}\left(  t\right)  +G_{2}\left(  t\right)  \right]
\int\nolimits_{x}^{x^{\prime}}p\left(  s\right)  ds\label{2.75}\\
&  \leq Sp\left(  t\right) \nonumber\\
&  \leq\left(  1-t\right)  \int\nolimits_{x}^{x^{\prime}}\left[  f\left(
\left(  1-\alpha\right)  x+\alpha s\right)  +f\left(  \left(  1-\alpha\right)
x^{\prime}+\alpha s\right)  \right]  p\left(  s\right)  ds\nonumber\\
&  \qquad\qquad+t\left[  f\left(  x\right)  +f\left(  x^{\prime}\right)
\right]  \int\nolimits_{x}^{x^{\prime}}p\left(  s\right)  ds\nonumber\\
&  \leq\left[  f\left(  x\right)  +f\left(  x^{\prime}\right)  \right]
\int\nolimits_{x}^{x^{\prime}}p\left(  s\right)  ds.\nonumber
\end{align}%
\begin{equation}
\frac{Ip_{1}\left(  1-t\right)  +Ip_{2}\left(  1-t\right)  }{2}\leq Sp\left(
t\right)  . \label{2.76}%
\end{equation}

\item \textit{The following inequality and the identity hold:}%
\begin{equation}
\frac{Ip_{1}\left(  t\right)  +Ip_{2}\left(  t\right)  +Ip_{1}\left(
1-t\right)  +Ip_{2}\left(  1-t\right)  }{4}\leq Sp\left(  t\right)  \text{
\ \ }\left(  t\in\left[  m^{\prime},1\right]  \right)  \label{2.77}%
\end{equation}
\textit{where }$m^{\prime}$\textit{\ be defined as in Theorem \ref{t9}.}%
\begin{equation}
\sup\limits_{t\in\left[  0,1\right]  }Sp\left(  t\right)  =\left[  f\left(
x\right)  +f\left(  x^{\prime}\right)  \right]  \int\nolimits_{x}^{x^{\prime}%
}p\left(  s\right)  ds. \label{2.78}%
\end{equation}

\end{enumerate}

\begin{proof}
$\left(  1\right)  $ It is easily observed from the convexity of $f$ and the
hypothesis of $p$ that $Sp$ is convex on $\left[  0,1\right]  .$

$\left(  2\right)  $ Using simple integration techniques, $\left(
\ref{2.71}\right)  $ and under the hypothesis of $p$, the following identities
hold on $\left[  0,1\right]  :$%
\begin{align*}
&  \left[  G_{1}\left(  t\right)  +G_{2}\left(  t\right)  \right]
\int\nolimits_{x}^{x^{\prime}}p\left(  s\right)  ds\\
&  =%
{\displaystyle\int\nolimits_{x}^{y}}
\frac{1}{2\alpha}\left[  f\left(  tx+\left(  1-t\right)  y\right)  +f\left(
tx^{\prime}+\left(  1-t\right)  y^{\prime}\right)  \right. \\
&  +\left.  f\left(  tx+\left(  1-t\right)  y^{\prime}\right)  +f\left(
tx^{\prime}+\left(  1-t\right)  y\right)  \right]  p\left(  \frac{1}{\alpha
}\left(  s-x\right)  +x\right)  ds.
\end{align*}%
\begin{align}
&  Sp\left(  t\right) \label{2.79}\\
&  =%
{\displaystyle\int\nolimits_{x}^{x^{\prime}}}
\frac{1}{2}\left[  f\left(  tx+\left(  1-t\right)  \left(  \left(
1-\alpha\right)  x+\alpha s\right)  \right)  p\left(  s\right)  \right.
\nonumber\\
&  +f\left(  tx+\left(  1-t\right)  \left(  x^{\prime}+\alpha x-\alpha
s\right)  \right)  p\left(  x+x^{\prime}-s\right) \nonumber\\
&  +f\left(  tx^{\prime}+\left(  1-t\right)  \left(  \left(  1-\alpha\right)
x+\alpha s\right)  \right)  p\left(  s\right) \nonumber\\
&  +\left.  f\left(  tx^{\prime}+\left(  1-t\right)  \left(  x^{\prime}+\alpha
x-\alpha s\right)  \right)  p\left(  x+x^{\prime}-s\right)  \right]  ds\text{
\ \ \ \ }\nonumber
\end{align}%
\begin{align*}
&  =%
{\displaystyle\int\nolimits_{x}^{x^{\prime}}}
\frac{1}{2}\left[  f\left(  tx+\left(  1-t\right)  \left(  \left(
1-\alpha\right)  x+\alpha s\right)  \right)  \right. \\
&  +f\left(  tx+\left(  1-t\right)  \left(  x^{\prime}+\alpha x-\alpha
s\right)  \right) \\
&  +f\left(  tx^{\prime}+\left(  1-t\right)  \left(  \left(  1-\alpha\right)
x+\alpha s\right)  \right) \\
&  +\left.  f\left(  tx^{\prime}+\left(  1-t\right)  \left(  x^{\prime}+\alpha
x-\alpha s\right)  \right)  \right]  p\left(  s\right)  ds\text{
\ \ \ \ \ \ \ \ \ \ \ \ \ \ \ \ \ }%
\end{align*}%
\begin{align*}
&  =%
{\displaystyle\int\nolimits_{x}^{y}}
\frac{1}{2\alpha}\left[  f\left(  tx+\left(  1-t\right)  s\right)  \right. \\
&  +f\left(  tx+\left(  1-t\right)  \left(  x+x^{\prime}-s\right)  \right) \\
&  +f\left(  tx^{\prime}+\left(  1-t\right)  s\right) \\
&  +\left.  f\left(  tx^{\prime}+\left(  1-t\right)  \left(  x+x^{\prime
}-s\right)  \right)  \right]  p\left(  \frac{1}{\alpha}\left(  s-x\right)
+x\right)  ds.
\end{align*}

By Lemma \ref{l2}, the following inequalities hold for all $t\in\left[
0,1\right]  $ and $s\in\left[  x,y\right]  :$%
\begin{align*}
&  \frac{1}{2\alpha}\left[  f\left(  tx+\left(  1-t\right)  y\right)
+f\left(  tx+\left(  1-t\right)  y^{\prime}\right)  \right]  p\left(  \frac
{1}{\alpha}\left(  s-x\right)  +x\right) \\
&  \leq\frac{1}{2\alpha}\left[  f\left(  tx+\left(  1-t\right)  s\right)
\right. \\
&  \text{ \ \ \ }+\left.  f\left(  tx+\left(  1-t\right)  \left(  x+x^{\prime
}-s\right)  \right)  \right]  p\left(  \frac{1}{\alpha}\left(  s-x\right)
+x\right)  .
\end{align*}%
\begin{align*}
&  \frac{1}{2\alpha}\left[  f\left(  tx^{\prime}+\left(  1-t\right)  y\right)
+f\left(  tx^{\prime}+\left(  1-t\right)  y^{\prime}\right)  \right]  p\left(
\frac{1}{\alpha}\left(  s-x\right)  +x\right) \\
&  \leq\frac{1}{2\alpha}\left[  f\left(  tx^{\prime}+\left(  1-t\right)
s\right)  \right. \\
&  \text{ \ \ \ }+\left.  f\left(  tx^{\prime}+\left(  1-t\right)  \left(
x+x^{\prime}-s\right)  \right)  \right]  p\left(  \frac{1}{\alpha}\left(
s-x\right)  +x\right)  .
\end{align*}

Integrating the above inequalities over $s$ on $\left[  x,y\right]  $ and
using the above identities, we obtain the first inequality of $\left(
\ref{2.75}\right)  .$ Using the convexity of $f$ and the inequality $\left(
\ref{2.52}\right)  ,$ we obtain the second and third inequalities of $\left(
\ref{2.75}\right)  .$ Next, by the convexity of $f,$ $y=\left(  1-\alpha
\right)  x+\alpha x^{\prime},$ $y^{\prime}=\alpha x+\left(  1-\alpha\right)
x^{\prime}$ and the identities $\left(  \ref{2.43}\right)  ,$ $\left(
\ref{2.46}\right)  ,$ $\left(  \ref{2.79}\right)  ,$\ we get%
\begin{align*}
&  \frac{Ip_{1}\left(  1-t\right)  +Ip_{2}\left(  1-t\right)  }{2}\\
&  =%
{\displaystyle\int\nolimits_{x}^{y}}
\frac{1}{2\alpha}\left[  f\left(  \left(  1-t\right)  s+ty\right)  \right. \\
&  +f\left(  \left(  1-t\right)  \left(  x+x^{\prime}-s\right)  +ty^{\prime
}\right)  +f\left(  \left(  1-t\right)  s+ty^{\prime}\right) \\
&  +\left.  f\left(  \left(  1-t\right)  \left(  x+x^{\prime}-s\right)
+ty\right)  \right]  p\left(  \frac{1}{\alpha}\left(  s-x\right)  +x\right)
ds\text{ \ \ \ \ \ \ \ \ \ \ \ \ \ \ \ \ \ \ }%
\end{align*}%
\begin{align*}
&  =%
{\displaystyle\int\nolimits_{x}^{y}}
\frac{1}{2\alpha}\left\{  f\left(  \left(  1-\alpha\right)  \left[  \left(
1-t\right)  s+tx\right]  +\alpha\left[  \left(  1-t\right)  s+tx^{\prime
}\right]  \right)  \right. \\
&  +f\left(  \left(  1-\alpha\right)  \left[  \left(  1-t\right)  \left(
x+x^{\prime}-s\right)  +tx\right]  \right. \\
&  \text{ \ \ \ \ \ \ }\left.  +\alpha\left[  \left(  1-t\right)  \left(
x+x^{\prime}-s\right)  +tx^{\prime}\right]  \right) \\
&  +f\left(  \alpha\left[  \left(  1-t\right)  s+tx\right]  +\left(
1-\alpha\right)  \left[  \left(  1-t\right)  s+tx^{\prime}\right]  \right) \\
&  +f\left(  \left[  \alpha\left(  1-t\right)  \left(  x+x^{\prime}-s\right)
+tx\right]  \right. \\
&  \text{\ \ }\left.  \text{\ \ \ \ }\left.  +\left(  1-\alpha\right)  \left[
\left(  1-t\right)  \left(  x+x^{\prime}-s\right)  +tx^{\prime}\right]
\right)  \right\}  p\left(  \frac{1}{\alpha}\left(  s-x\right)  +x\right)
ds\\
&  \leq\left(  1-\alpha\right)  Sp\left(  t\right)  +\alpha Sp\left(
t\right)  =Sp\left(  t\right)
\end{align*}
and the inequality $\left(  \ref{2.76}\right)  $ is proved.

$\left(  3\right)  $ $\left(  \ref{2.77}\right)  $ and $\left(  \ref{2.78}%
\right)  $ follow from the inequalities $\left(  \ref{2.69}\right)  ,$
$\left(  \ref{2.72}\right)  ,$ $\left(  \ref{2.75}\right)  $ and $\left(
\ref{2.76}\right)  .$

This completes the proof.
\end{proof}

\begin{remark}
\label{r25}\textit{Using Remark \ref{r11}} we have the following results:
\end{remark}

\begin{enumerate}
\item \textit{Theorem \ref{t19}} \textit{reduces to Theorem \ref{A8} as
}$x=a,$\textit{\ }$y=y^{\prime}=\frac{a+b}{2},$\textit{\ }$x^{\prime}=b$
\textit{and }$p\left(  s\right)  \equiv\frac{1}{2\left(  b-a\right)  }$
$\left(  s\in\left[  a,b\right]  \right)  .$

\item \textit{Theorem \ref{t19}} \textit{reduces to Theorem \ref{A25} as
}$p\left(  s\right)  \equiv\frac{\alpha}{2\left(  y-x\right)  }$\ $\left(
s\in\Omega\right)  .$

\item \textit{Theorem \ref{t19}} \textit{reduces to Theorem \ref{A34} as
}$x=a,$ $y=y^{\prime}=\frac{a+b}{2},x^{\prime}=b$ \textit{and }$p\left(
s\right)  \equiv\frac{g\left(  s\right)  }{2}$\ $\left(  s\in\left[
a,b\right]  \right)  .$
\end{enumerate}

\begin{theorem}
\label{t20}Let $x,y,y^{\prime},x^{\prime},f,G_{1},Q_{1}$\ be defined as above. Then
\end{theorem}

\begin{enumerate}
\item $Q_{1}$\textit{\ is symmetric about }$\frac{1}{2},$\textit{\ }$Q_{1}%
$\textit{\ is decreasing on }$\left[  0,\frac{1}{2}\right]  $\textit{\ and
increasing on }$\left[  \frac{1}{2},1\right]  .$

\item \textit{The following inequalities hold:}%
\begin{equation}
G_{1}\left(  \frac{t}{\alpha}\right)  \leq Q_{1}\left(  t\right)  \text{
\quad}\left(  t\in\left[  0,\frac{\alpha}{2}\right]  \right)  . \label{2.80}%
\end{equation}%
\begin{equation}
G_{1}\left(  \frac{t}{\alpha}\right)  \geq Q_{1}\left(  t\right)  \quad\left(
t\in\left[  \frac{\alpha}{2},\alpha\right]  \right)  . \label{2.81}%
\end{equation}%
\begin{equation}
G_{1}\left(  \frac{1-t}{\alpha}\right)  \geq Q_{1}\left(  t\right)  \text{
\quad}\left(  t\in\left[  1-\alpha,1-\frac{\alpha}{2}\right]  \right)  .
\label{2.82}%
\end{equation}%
\begin{equation}
G_{1}\left(  \frac{1-t}{\alpha}\right)  \leq Q_{1}\left(  t\right)
\quad\left(  t\in\left[  1-\frac{\alpha}{2},1\right]  \right)  . \label{2.83}%
\end{equation}

\end{enumerate}

\begin{proof}
$\left(  1\right)  $ It is obvious that $Q_{1}$ is symmetric about $\frac
{1}{2}.$

By Lemma \ref{l2}, the following inequalities hold for all $0\leq s<t\leq
\frac{1}{2}\leq u<v\leq1.$%
\begin{align*}
&  \frac{1}{2}\left[  f\left(  tx^{\prime}+\left(  1-t\right)  x\right)
+f\left(  tx+\left(  1-t\right)  x^{\prime}\right)  \right] \\
&  \leq\frac{1}{2}\left[  f\left(  sx^{\prime}+\left(  1-s\right)  x\right)
+f\left(  sx+\left(  1-s\right)  x^{\prime}\right)  \right]  .
\end{align*}%
\begin{align*}
&  \frac{1}{2}\left[  f\left(  ux+\left(  1-u\right)  x^{\prime}\right)
+f\left(  ux^{\prime}+\left(  1-u\right)  x\right)  \right] \\
&  \leq\frac{1}{2}\left[  f\left(  vx+\left(  1-v\right)  x^{\prime}\right)
+f\left(  vx^{\prime}+\left(  1-v\right)  x\right)  \right]  .
\end{align*}

Thus, $Q_{1}$ is decreasing on\textit{\ }$\left[  0,\frac{1}{2}\right]  $
and\textit{\ }increasing on\textit{\ }$\left[  \frac{1}{2},1\right]  .$

$\left(  2\right)  $ Let $t_{1}\in\left[  0,\frac{\alpha}{2}\right]  ,$
$t_{2}\in\left[  \frac{\alpha}{2},\alpha\right]  ,$ $t_{3}\in\left[
1-\alpha,1-\frac{\alpha}{2}\right]  $ and $t_{4}\in\left[  1-\frac{\alpha}%
{2},1\right]  .$ As a simple calculation shows us that:%
\begin{align*}
\left(  1-t_{1}\right)  x+t_{1}x^{\prime}  &  \leq\frac{t_{1}}{\alpha
}x+\left(  1-\frac{t_{1}}{\alpha}\right)  y\\
&  \leq\frac{t_{1}}{\alpha}x^{\prime}+\left(  1-\frac{t_{1}}{\alpha}\right)
y^{\prime}\\
&  \leq t_{1}x+\left(  1-t_{1}\right)  x^{\prime}.
\end{align*}%
\begin{align*}
\frac{t_{2}}{\alpha}x+\left(  1-\frac{t_{2}}{\alpha}\right)  y  &  \leq\left(
1-t_{2}\right)  x+t_{2}x^{\prime}\\
&  \leq t_{2}x+\left(  1-t_{2}\right)  x^{\prime}\\
&  \leq\frac{t_{2}}{\alpha}x^{\prime}+\left(  1-\frac{t_{2}}{\alpha}\right)
y^{\prime}.
\end{align*}%
\begin{align*}
\frac{1-t_{3}}{\alpha}x+\left(  1-\frac{1-t_{3}}{\alpha}\right)  y  &  \leq
t_{3}x+\left(  1-t_{3}\right)  x^{\prime}\\
&  \leq\left(  1-t_{3}\right)  x+t_{3}x^{\prime}\\
&  \leq\frac{1-t_{3}}{\alpha}x^{\prime}+\left(  1-\frac{1-t_{3}}{\alpha
}\right)  y^{\prime}.
\end{align*}%
\begin{align*}
t_{4}x+\left(  1-t_{4}\right)  x^{\prime}  &  \leq\frac{1-t_{4}}{\alpha
}x+\left(  1-\frac{1-t_{4}}{\alpha}\right)  y\\
&  \leq\frac{1-t_{4}}{\alpha}x^{\prime}+\left(  1-\frac{1-t_{4}}{\alpha
}\right)  y^{\prime}\\
&  \leq\left(  1-t_{4}\right)  x+t_{4}x^{\prime}.
\end{align*}

Now, using Lemma \ref{l2} under the above inequalities, we derive $\left(
\ref{2.80}\right)  -\left(  \ref{2.83}\right)  .$

This completes the proof.
\end{proof}

\begin{remark}
\label{r26}\textit{Using Remark \ref{r11}, Theorem \ref{t20}} \textit{reduces
to Theorem \ref{A10} as }$x=a,$ $y=y^{\prime}=\frac{a+b}{2},x^{\prime}=b$ and
$\alpha=\frac{1}{2}.$
\end{remark}

\begin{theorem}
\label{t21}\textit{Let }$x,y,y^{\prime},x^{\prime},\Omega,m,m^{\prime
},f,p,G_{1},G_{2},Hp_{1},Hp_{2},Pp_{1},Ip_{1},Ip_{2},Np,Sp$\textit{\ be
defined as above. }
\end{theorem}

\begin{enumerate}
\item \textit{The inequality}%
\begin{align}
0  &  \leq Np\left(  t\right)  -2G_{1}\left(  t\right)  \int\nolimits_{x}%
^{x^{\prime}}p\left(  s\right)  ds\label{2.84}\\
&  \leq\left[  f\left(  x\right)  +f\left(  x^{\prime}\right)  \right]
\int\nolimits_{x}^{x^{\prime}}p\left(  s\right)  ds-Np\left(  t\right)
\nonumber
\end{align}
\textit{holds for all }$t\in\left[  0,1\right]  .$

\item \textit{If }$f$\textit{\ is differentiable on }$\Omega,$\textit{\ then
the inequalities}%
\begin{align}
0  &  \leq Lp_{1}\left(  t\right)  -\frac{Hp_{1}\left(  t\right)
+Hp_{2}\left(  t\right)  }{2}\label{2.85}\\
&  \leq\frac{4x^{\prime}-y-3y^{\prime}}{8}\left(  f^{\prime}\left(  x^{\prime
}\right)  -f^{\prime}\left(  x\right)  \right)  \int_{\Omega}p\left(
s\right)  ds\nonumber
\end{align}
\textit{and}%
\begin{align}
0  &  \leq Pp_{1}-Lp_{1}\left(  t\right) \label{2.86}\\
&  \leq\frac{4x^{\prime}-y-3y^{\prime}}{8}\left(  f^{\prime}\left(  x^{\prime
}\right)  -f^{\prime}\left(  x\right)  \right)  \int_{\Omega}p\left(
s\right)  ds\nonumber
\end{align}
\textit{hold for all }$t\in\left[  0,1\right]  .$

\item \textit{If }$f$\textit{\ is differentiable on }$\left[  x,x^{\prime
}\right]  ,$\textit{\ then the inequalities}%
\begin{align}
0  &  \leq Np\left(  t\right)  -Ip_{1}\left(  t\right) \label{2.87}\\
&  \leq\left(  y-x\right)  \left(  f^{\prime}\left(  x^{\prime}\right)
-f^{\prime}\left(  x\right)  \right)  \int_{x}^{x^{\prime}}p\left(  s\right)
ds\text{ \ }\left(  t\in\left[  0,1\right]  \right) \nonumber
\end{align}
\textit{and}%
\begin{align}
0  &  \leq Sp\left(  t\right)  -\frac{Ip_{1}\left(  t\right)  +Ip_{2}\left(
t\right)  }{2}\label{2.88}\\
&  \leq\frac{4x^{\prime}-y-3y^{\prime}}{4}\left(  f^{\prime}\left(  x^{\prime
}\right)  -f^{\prime}\left(  x\right)  \right)  \int_{x}^{x^{\prime}}p\left(
s\right)  ds\text{ \ }\left(  t\in\left[  m^{\prime},1\right]  \right)
\nonumber
\end{align}

\end{enumerate}

\begin{proof}
$\left(  1\right)  $ By Lemma \ref{l2}, the following inequality hold for all
$t\in\left[  0,1\right]  $ and $s\in\left[  x,y\right]  :$%
\begin{align*}
&  \frac{1}{\alpha}\left[  f\left(  tx+\left(  1-t\right)  y\right)  +f\left(
tx^{\prime}+\left(  1-t\right)  y^{\prime}\right)  \right]  p\left(  \frac
{1}{\alpha}\left(  s-x\right)  +x\right) \\
&  \leq\frac{1}{\alpha}\left[  f\left(  tx+\left(  1-t\right)  s\right)
\right. \\
&  \left.  +f\left(  tx^{\prime}+\left(  1-t\right)  \left(  x+x^{\prime
}-s\right)  \right)  \right]  p\left(  \frac{1}{\alpha}\left(  s-x\right)
+x\right)  .
\end{align*}

Integrating the above inequality over $s$ on $\left[  x,y\right]  $ and using
the above identities $\left(  \ref{2.55}\right)  $ and $\left(  \ref{2.71}%
\right)  ,$ we obtain the inequality%
\[
2G_{1}\left(  t\right)  \int\nolimits_{x}^{x^{\prime}}p\left(  s\right)
ds\leq Np\left(  t\right)  \ \left(  \text{ }t\in\left[  0,1\right]  \right)
\]
and from which the first inequality of $\left(  \ref{2.84}\right)  $ holds.

Again, using Lemma \ref{l2}, the following inequality hold for all
$t\in\left[  0,1\right]  $ and $s\in\left[  x,\frac{x+y}{2}\right]  :$%
\begin{align*}
&  \frac{1}{\alpha}\left[  f\left(  tx+\left(  1-t\right)  s\right)  \right.
\\
&  \text{\ \ }\left.  +f\left(  tx+\left(  1-t\right)  \left(  x+y-s\right)
\right)  \right]  p\left(  \frac{1}{\alpha}\left(  s-x\right)  +x\right) \\
&  \leq\frac{1}{\alpha}\left[  f\left(  x\right)  +f\left(  tx+\left(
1-t\right)  y\right)  \right]  p\left(  \frac{1}{\alpha}\left(  s-x\right)
+x\right)  .
\end{align*}%
\begin{align*}
&  \frac{1}{\alpha}\left[  f\left(  tx^{\prime}+\left(  1-t\right)  \left(
s+x^{\prime}-y\right)  \right)  \right. \\
&  \text{\ \ }\left.  +f\left(  tx^{\prime}+\left(  1-t\right)  \left(
x+x^{\prime}-s\right)  \right)  \right]  p\left(  \frac{1}{\alpha}\left(
s-x\right)  +x\right) \\
&  \leq\frac{1}{\alpha}\left[  f\left(  tx^{\prime}+\left(  1-t\right)
y^{\prime}\right)  +f\left(  x^{\prime}\right)  \right]  p\left(  \frac
{1}{\alpha}\left(  s-x\right)  +x\right)  .
\end{align*}

Integrating the above inequalities over $s$ on $\left[  x,\frac{x+y}%
{2}\right]  $ and using the identity $\left(  \ref{2.55}\right)  $ and the
identity%
\begin{align}
&  \int_{x}^{\frac{x+y}{2}}\frac{1}{\alpha}p\left(  \frac{1}{\alpha}\left(
s-x\right)  +x\right)  ds\label{2.89}\\
&  =\int_{x}^{\frac{x+y}{2}}\frac{1}{2\alpha}p\left(  \frac{1}{\alpha}\left(
s-x\right)  +x\right)  ds+\int_{x}^{\frac{x+y}{2}}\frac{1}{2\alpha}p\left(
\frac{1}{\alpha}\left(  y-s\right)  +x\right)  ds\nonumber\\
&  =\int_{x}^{y}\frac{1}{2\alpha}p\left(  \frac{1}{\alpha}\left(  s-x\right)
+x\right)  ds\nonumber\\
&  =\frac{1}{2}\int\nolimits_{x}^{x^{\prime}}p\left(  s\right)  ds,\nonumber
\end{align}
we obtain the inequality
\[
Np\left(  t\right)  \leq G_{1}\left(  t\right)  \int\nolimits_{x}^{x^{\prime}%
}p\left(  s\right)  ds+\frac{f\left(  x\right)  +f\left(  x^{\prime}\right)
}{2}\int\nolimits_{x}^{x^{\prime}}p\left(  s\right)  ds\ \left(  \text{ }%
t\in\left[  0,1\right]  \right)
\]
\ and from which the second inequality of $\left(  \ref{2.84}\right)  $ holds.
This proves the inequality $\left(  \ref{2.84}\right)  .$

$\left(  2\right)  $ From $\left(  \ref{2.1}\right)  ,$ $\left(
\ref{2.3}\right)  ,$ $\left(  \ref{2.11}\right)  ,$ $\left(  \ref{2.33}%
\right)  $\ and $\left(  \ref{2.36}\right)  $, the following inequality holds
on $\left[  0,1\right]  :$
\begin{align}
&  \frac{1}{2}\left[  f\left(  \frac{y+y^{\prime}}{2}\right)  +\frac{f\left(
y\right)  +f\left(  y^{\prime}\right)  }{2}\right]  \int_{\Omega}p\left(
s\right)  ds\label{2.90}\\
&  \leq\frac{Hp_{1}\left(  t\right)  +Hp_{2}\left(  t\right)  }{2}\leq
Lp_{1}\left(  t\right)  \leq Pp_{1}\left(  t\right)  \leq\frac{f\left(
x\right)  +f\left(  x^{\prime}\right)  }{2}\int_{\Omega}p\left(  s\right)
ds.\nonumber
\end{align}

By the convexity of $f$ and $x+x^{\prime}=y+y^{\prime},$\ we have the
inequality%
\begin{align}
&  \left[  f\left(  x\right)  +f\left(  x^{\prime}\right)  \right]  -\left[
f\left(  \frac{y+y^{\prime}}{2}\right)  +\frac{f\left(  y\right)  +f\left(
y^{\prime}\right)  }{2}\right] \label{2.91}\\
&  =\frac{1}{2}\left[  \left(  f\left(  x\right)  -f\left(  \frac{y+y^{\prime
}}{2}\right)  \right)  +\left(  f\left(  x\right)  -f\left(  y\right)
\right)  \right. \nonumber\\
&  \text{ \ \ \ \ \ \ \ \ }+\left.  \left(  f\left(  x^{\prime}\right)
-f\left(  \frac{y+y^{\prime}}{2}\right)  \right)  +\left(  f\left(  x^{\prime
}\right)  -f\left(  y^{\prime}\right)  \right)  \right] \nonumber\\
&  \leq\frac{1}{2}\left[  \frac{2x-y-y^{\prime}}{2}f^{\prime}\left(  x\right)
+\left(  x-y\right)  f^{\prime}\left(  x\right)  \right. \nonumber\\
&  \text{ \ \ \ \ \ \ \ }+\left.  \frac{2x^{\prime}-y-y^{\prime}}{2}f^{\prime
}\left(  x^{\prime}\right)  +\left(  x^{\prime}-y^{\prime}\right)  f^{\prime
}\left(  x^{\prime}\right)  \right] \nonumber\\
&  =\frac{1}{2}\left[  \frac{2x^{\prime}-y-y^{\prime}}{2}\left(  f^{\prime
}\left(  x^{\prime}\right)  -f^{\prime}\left(  x\right)  \right)  +\left(
x^{\prime}-y^{\prime}\right)  \left(  f^{\prime}\left(  x^{\prime}\right)
-f^{\prime}\left(  x\right)  \right)  \right] \\
&  =\frac{4x^{\prime}-y-3y^{\prime}}{4}\left(  f^{\prime}\left(  x^{\prime
}\right)  -f^{\prime}\left(  x\right)  \right)  .\nonumber
\end{align}

Using the inequalities $\left(  \ref{2.90}\right)  -\left(  \ref{2.91}\right)
,$ we derive the inequalities $\left(  \ref{2.85}\right)  -\left(
\ref{2.86}\right)  .$

$\left(  3\right)  $ The inequality $\left(  \ref{2.87}\right)  $ follows from
the inequalities $\left(  \ref{2.56}\right)  $ and$\ \left(  \ref{2.63}%
\right)  . $\ From the inequalities $\left(  \ref{2.41}\right)  ,$ $\left(
\ref{2.44}\right)  ,$ $\left(  \ref{2.69}\right)  ,$ $\left(  \ref{2.72}%
\right)  $\ and $\left(  \ref{2.75}\right)  $, the following inequality holds
on $\left[  m^{\prime},1\right]  :$%
\begin{align*}
&  \frac{1}{2}\left[  f\left(  \frac{y+y^{\prime}}{2}\right)  +\frac{f\left(
y\right)  +f\left(  y^{\prime}\right)  }{2}\right]  \int_{x}^{x^{\prime}%
}p\left(  s\right)  ds\\
&  \leq\frac{Ip_{1}\left(  t\right)  +Ip_{2}\left(  t\right)  }{2}\leq\left(
G_{1}\left(  t\right)  +G_{2}\left(  t\right)  \right)  \int_{x}^{x^{\prime}%
}p\left(  s\right)  ds\\
&  \leq Sp\left(  t\right)  \leq\frac{f\left(  x\right)  +f\left(  x^{\prime
}\right)  }{2}\int_{x}^{x^{\prime}}p\left(  s\right)  ds.
\end{align*}

Using the above inequality and $\left(  \ref{2.91}\right)  ,$ we derive the
inequality $\left(  \ref{2.88}\right)  .$ This completes the proof.
\end{proof}

\begin{remark}
\label{r27}\textit{Using Remark \ref{r11}} we have the following results:
\end{remark}

\begin{enumerate}
\item \textit{As }$x=a,$\textit{\ }$y=y^{\prime}=\frac{a+b}{2},$%
\textit{\ }$x^{\prime}=b$ \textit{and }$p\left(  s\right)  \equiv\frac
{1}{2\left(  b-a\right)  }$ \textit{on} $\left[  a,b\right]  ,$\textit{Theorem
\ref{A9}} \textit{follows from a combination of the inequalities }$\left(
\ref{2.12}\right)  ,$ $\left(  \ref{2.14}\right)  ,$ $\left(  \ref{2.33}%
\right)  $ \textit{and }$\left(  \ref{2.84}\right)  .$

\item \textit{As }$p\left(  s\right)  \equiv\frac{\alpha}{2\left(  y-x\right)
}$\ $\left(  s\in\Omega\right)  ,$ \textit{Theorem \ref{A17}} \textit{follows
from a combination of the inequalities }$\left(  \ref{2.14}\right)  -\left(
\ref{2.17}\right)  $ \textit{and }$\left(  \ref{2.85}\right)  .$

\item \textit{As }$x=a,$ $y=y^{\prime}=\frac{a+b}{2},x^{\prime}=b$ \textit{and
}$p\left(  s\right)  \equiv\frac{g\left(  s\right)  }{2}$\ \textit{on}%
$\ \left[  a,b\right]  ,$\textit{Theorem \ref{A35}} \textit{follows from a
combination of the inequalities }$\left(  \ref{2.12}\right)  ,$ $\left(
\ref{2.14}\right)  -\left(  \ref{2.16}\right)  ,$ $\left(  \ref{2.33}\right)
$ \textit{and }$\left(  \ref{2.84}\right)  -\left(  \ref{2.88}\right)  .$
\end{enumerate}

\begin{theorem}
\label{t22}\textit{Let }$x,y,y^{\prime},x^{\prime},\Omega,f,p,G_{1}%
,G_{2},Q_{1},Hp_{1},Pp_{1},Sp$\textit{\ be defined as above.} Then:
\end{theorem}

\begin{enumerate}
\item \textit{The inequalities}%
\begin{align}
Hp_{1}\left(  t\right)   &  \leq Q_{1}\left(  t\right)  \int_{\Omega}p\left(
s\right)  ds\label{2.92}\\
&  \leq\frac{f\left(  x\right)  +f\left(  x^{\prime}\right)  }{2}\text{ }%
\int_{\Omega}p\left(  s\right)  ds\qquad\left(  t\in\left[  0,\frac{\alpha
}{1+\alpha}\right]  \right) \nonumber
\end{align}
\textit{and}%
\begin{align}
f\left(  \frac{x+x^{\prime}}{2}\right)  \int_{\Omega}p\left(  s\right)  dx  &
\leq Q_{1}\left(  t\right)  \int_{\Omega}p\left(  s\right)  ds\label{2.93}\\
&  \leq Pp_{1}\left(  t\right)  \text{ \qquad}\left(  t\in\left[  \frac
{\alpha}{1+\alpha},1\right]  \right) \nonumber
\end{align}
\textit{hold.}

\item \textit{The inequality}
\begin{align}
0  &  \leq Sp\left(  t\right)  -\left(  G_{1}\left(  t\right)  +G_{2}\left(
t\right)  \right)  \int_{x}^{x^{\prime}}p\left(  s\right)  ds\label{2.94}\\
&  \leq\left[  \frac{f\left(  x\right)  +f\left(  x^{\prime}\right)  }%
{2}+Q_{1}\left(  t\right)  \right]  \int_{x}^{x^{\prime}}p\left(  s\right)
ds-Sp\left(  t\right) \nonumber
\end{align}
\textit{holds for all }$t\in\left[  0,1\right]  .$
\end{enumerate}

\begin{proof}
$\left(  1\right)  $ We discuss the following two cases.

\textbf{Case 1.} $t\in\left[  0,\frac{\alpha}{1+\alpha}\right]  .$

By Lemma \ref{l2}, the following inequality holds for all $t\in\left[
0,\frac{\alpha}{1+\alpha}\right]  $ and $s\in\left[  x,y\right]  :$%
\begin{align*}
&  \left[  f\left(  ts+\left(  1-t\right)  y\right)  +f\left(  t\left(
y+y^{\prime}-s\right)  +\left(  1-t\right)  y^{\prime}\right)  \right]
p\left(  s\right) \\
&  \leq\left[  f\left(  tx^{\prime}+\left(  1-t\right)  x\right)  +f\left(
tx+\left(  1-t\right)  x^{\prime}\right)  \right]  p\left(  s\right)  .
\end{align*}
Integrating the above inequality over $s$ on $\left[  x,y\right]  $ and using
the identity%
\begin{equation}
\int_{x}^{y}p\left(  s\right)  ds=\frac{1}{2}\int_{\Omega}p\left(  s\right)
ds, \label{2.95}%
\end{equation}
we obtain the first inequality of $\left(  \ref{2.92}\right)  .$ From Theorem
\ref{t20}, we have%
\[
\sup\limits_{t\in\left[  0,\frac{\alpha}{1+\alpha}\right]  }Q_{1}\left(
t\right)  =Q_{1}\left(  0\right)  =\frac{f\left(  x\right)  +f\left(
x^{\prime}\right)  }{2}%
\]
and from which the second inequality of $\left(  \ref{2.92}\right)  $ holds.
This proves the inequality $\left(  \ref{2.92}\right)  .$

\textbf{Case 2.} $t\in\left[  \frac{\alpha}{1+\alpha},1\right]  .$

By Lemma \ref{l2}, the following inequality holds for all $t\in\left[
\frac{\alpha}{1+\alpha},1\right]  $ and $s\in\left[  x,y\right]  :$%
\begin{align*}
&  \left[  f\left(  tx+\left(  1-t\right)  x^{\prime}\right)  +f\left(
tx^{\prime}+\left(  1-t\right)  x\right)  \right]  p\left(  s\right) \\
&  \leq\left[  f\left(  tx+\left(  1-t\right)  s\right)  +f\left(  tx^{\prime
}+\left(  1-t\right)  \left(  x+x^{\prime}-s\right)  \right)  \right]
p\left(  s\right)  .
\end{align*}

Integrating the above inequality over $s$ on $\left[  x,y\right]  $ and using
the identity $\left(  \ref{2.95}\right)  ,$ we obtain the second inequality of
$\left(  \ref{2.93}\right)  .$ From Theorem \ref{t20}, we have%
\[
\inf\limits_{t\in\left[  \frac{\alpha}{1+\alpha},1\right]  }Q_{1}\left(
t\right)  =Q_{1}\left(  \frac{1}{2}\right)  =f\left(  \frac{x+x^{\prime}}%
{2}\right)
\]
and from which the first inequality of $\left(  \ref{2.93}\right)  $ holds.
This proves the inequality $\left(  \ref{2.93}\right)  .$

$\left(  2\right)  $ Using simple integration techniques, $\left(
\ref{2.79}\right)  $ and under the hypothesis of $p$%
\begin{equation}
p\left(  \frac{1}{\alpha}\left(  s-x\right)  +x\right)  =p\left(  \frac
{1}{\alpha}\left(  y-s\right)  +x\right)  , \label{2.96}%
\end{equation}
the following identity holds on $\left[  0,1\right]  :$%
\begin{align*}
2Sp\left(  t\right)   &  =%
{\displaystyle\int\nolimits_{x}^{y}}
\frac{1}{\alpha}\left[  f\left(  tx+\left(  1-t\right)  s\right)  \right. \\
&  +f\left(  tx+\left(  1-t\right)  \left(  x+x^{\prime}-s\right)  \right) \\
&  +f\left(  tx^{\prime}+\left(  1-t\right)  s\right) \\
&  +\left.  f\left(  tx^{\prime}+\left(  1-t\right)  \left(  x+x^{\prime
}-s\right)  \right)  \right]  p\left(  \frac{1}{\alpha}\left(  s-x\right)
+x\right)  ds\\
&  =%
{\displaystyle\int\nolimits_{x}^{\frac{x+y}{2}}}
\frac{1}{\alpha}\left[  f\left(  tx+\left(  1-t\right)  s\right)  +f\left(
tx+\left(  1-t\right)  \left(  x+y-s\right)  \right)  \right. \\
&  +f\left(  tx+\left(  1-t\right)  \left(  x+x^{\prime}-s\right)  \right)
+f\left(  tx+\left(  1-t\right)  \left(  x^{\prime}-y+s\right)  \right) \\
&  +f\left(  tx^{\prime}+\left(  1-t\right)  s\right)  ++f\left(  tx^{\prime
}+\left(  1-t\right)  \left(  x+y-s\right)  \right) \\
&  +f\left(  tx^{\prime}+\left(  1-t\right)  \left(  x+x^{\prime}-s\right)
\right) \\
&  +\left.  f\left(  tx^{\prime}+\left(  1-t\right)  \left(  x^{\prime
}-y+s\right)  \right)  \right]  p\left(  \frac{1}{\alpha}\left(  s-x\right)
+x\right)  ds.
\end{align*}

By Lemma \ref{l2}, the following inequalities hold for all $t\in\left[
0,1\right]  $ and $s\in\left[  x,\frac{x+y}{2}\right]  :$%
\begin{align*}
&  \frac{1}{\alpha}\left[  f\left(  tx+\left(  1-t\right)  s\right)  \right.
\\
&  +\left.  f\left(  tx+\left(  1-t\right)  \left(  x+y-s\right)  \right)
\right]  p\left(  \frac{1}{\alpha}\left(  s-x\right)  +x\right) \\
&  \leq\frac{1}{\alpha}\left[  f\left(  x\right)  +f\left(  tx+\left(
1-t\right)  y\right)  \right]  p\left(  \frac{1}{\alpha}\left(  s-x\right)
+x\right)  .
\end{align*}%
\begin{align*}
&  \frac{1}{\alpha}\left[  f\left(  tx+\left(  1-t\right)  \left(  x^{\prime
}-y+s\right)  \right)  \right. \\
&  +\left.  f\left(  tx+\left(  1-t\right)  \left(  x+x^{\prime}-s\right)
\right)  \right]  p\left(  \frac{1}{\alpha}\left(  s-x\right)  +x\right) \\
&  \leq\frac{1}{\alpha}\left[  f\left(  tx+\left(  1-t\right)  y^{\prime
}\right)  +f\left(  tx+\left(  1-t\right)  x^{\prime}\right)  \right]
p\left(  \frac{1}{\alpha}\left(  s-x\right)  +x\right)  .
\end{align*}%
\begin{align*}
&  \frac{1}{\alpha}\left[  f\left(  tx^{\prime}+\left(  1-t\right)  s\right)
\right. \\
&  +\left.  f\left(  tx^{\prime}+\left(  1-t\right)  \left(  x+y-s\right)
\right)  \right]  p\left(  \frac{1}{\alpha}\left(  s-x\right)  +x\right) \\
&  \leq\frac{1}{\alpha}\left[  f\left(  tx^{\prime}+\left(  1-t\right)
x\right)  +f\left(  tx^{\prime}+\left(  1-t\right)  y\right)  \right]
p\left(  \frac{1}{\alpha}\left(  s-x\right)  +x\right)  .
\end{align*}%
\begin{align*}
&  \frac{1}{\alpha}\left[  f\left(  tx^{\prime}+\left(  1-t\right)  \left(
x^{\prime}-y+s\right)  \right)  \right. \\
&  +\left.  f\left(  tx^{\prime}+\left(  1-t\right)  \left(  x+x^{\prime
}-s\right)  \right)  \right]  p\left(  \frac{1}{\alpha}\left(  s-x\right)
+x\right) \\
&  \leq\frac{1}{\alpha}\left[  f\left(  tx^{\prime}+\left(  1-t\right)
y^{\prime}\right)  +f\left(  x^{\prime}\right)  \right]  p\left(  \frac
{1}{\alpha}\left(  s-x\right)  +x\right)  .
\end{align*}

Integrating the above inequalities over $s$ on $\left[  x,\frac{x+y}%
{2}\right]  $ and using the above identity and $\left(  \ref{2.89}\right)  $,
we obtain the inequality%
\begin{align*}
2Sp\left(  t\right)   &  \leq\left(  G_{1}\left(  t\right)  +G_{2}\left(
t\right)  \right)  \int_{x}^{x^{\prime}}p\left(  s\right)  ds\\
&  +\left[  \frac{f\left(  x\right)  +f\left(  x^{\prime}\right)  }{2}%
+Q_{1}\left(  t\right)  \right]  \int_{x}^{x^{\prime}}p\left(  s\right)
ds\ \left(  \text{ }t\in\left[  0,1\right]  \right)  .
\end{align*}

Using the above inequality and $\left(  \ref{2.75}\right)  ,$ we derive the
inequality $\left(  \ref{2.94}\right)  .$ This completes the proof.
\end{proof}

\begin{remark}
\label{r28}\textit{Using Remark \ref{r11}, Theorem \ref{t22}} \textit{reduces
to Theorem \ref{A36} as }$x=a,$ $y=y^{\prime}=\frac{a+b}{2},x^{\prime}=b,$
$\alpha=\frac{1}{2}$\ \textit{and }$p\left(  s\right)  =\frac{g\left(
s\right)  }{2}$ $\left(  s\in\left[  a,b\right]  \right)  .$
\end{remark}

\begin{theorem}
\label{t23}\textit{Let }$x,y,y^{\prime},x^{\prime},\Omega,f,p,G_{1}%
,G_{2},Q_{1},Ip_{1},Ip_{2},Kp$\textit{\ be defined as above.} Then:

\begin{enumerate}
\item $Kp$\ is convex on $\left[  0,1\right]  $\ and symmetric about $\frac
{1}{2}.$

\item $Kp$ is decreasing on $\left[  0,\frac{1}{2}\right]  $ and increasing on
$\left[  \frac{1}{2},1\right]  ,$%
\begin{align}
&  \sup\limits_{t\in\left[  0,1\right]  }Kp\left(  t\right) \label{2.97}\\
&  =Kp\left(  0\right)  =Kp\left(  1\right) \nonumber\\
&  =%
{\displaystyle\int\nolimits_{x}^{x^{\prime}}}
2\left[  f\left(  \left(  \left(  1-\alpha\right)  x+\alpha s\right)  \right)
+f\left(  \left(  \left(  1-\alpha\right)  x^{\prime}+\alpha s\right)
\right)  \right]  p\left(  s\right)  ds\ \int_{x}^{x^{\prime}}p\left(
s\right)  ds\ \nonumber
\end{align}
and
\begin{align}
\inf\limits_{t\in\left[  0,1\right]  }Kp\left(  t\right)   &  =Kp\left(
\frac{1}{2}\right) \label{2.98}\\
&  =%
{\displaystyle\int\nolimits_{x}^{x^{\prime}}}
{\displaystyle\int\nolimits_{x}^{x^{\prime}}}
\left[  f\left(  \left(  1-\alpha\right)  x+\alpha\frac{s+u}{2}\right)
\right. \nonumber\\
&  +2f\left(  \left(  1-\alpha\right)  \frac{x+x^{\prime}}{2}+\alpha\frac
{s+u}{2}\right) \nonumber\\
&  +\left.  f\left(  \left(  1-\alpha\right)  x^{\prime}+\alpha\frac{s+u}%
{2}\right)  \right]  p\left(  s\right)  p\left(  u\right)  dsdu.\nonumber
\end{align}

\item We have:%
\begin{equation}
\left[  Ip_{1}\left(  t\right)  +Ip_{2}\left(  t\right)  \right]  \int%
_{x}^{x^{\prime}}p\left(  s\right)  ds\leq Kp\left(  t\right)  \text{
\ }\left(  t\in\left[  0,1\right]  \right)  \label{2.99}%
\end{equation}
and%
\begin{equation}
\left[  f\left(  y\right)  +2f\left(  \frac{x+x^{\prime}}{2}\right)  +f\left(
y^{\prime}\right)  \right]  \left[  \int_{x}^{x^{\prime}}p\left(  s\right)
ds\right]  ^{2}\leq Kp\left(  \frac{1}{2}\right)  . \label{2.100}%
\end{equation}

\end{enumerate}
\end{theorem}

\begin{proof}
$\left(  1\right)  $ It is easily observed from the convexity of $f$ and the
hypothesis of $p$ that $Kp$ is convex on $\left[  0,1\right]  .$

By changing variables, we have
\[
Kp\left(  t\right)  =Kp\left(  1-t\right)  ,\text{ \qquad\ }t\in\left[
0,1\right]
\]
and from which we get that $Kp$ is symmetric about $\frac{1}{2}.$

$\left(  2\right)  $ Let $t_{1}<t_{2}$ in $\left[  0,\frac{1}{2}\right]  .$
Using the symmetry of $F_{1}$, we have%
\begin{equation}
Kp\left(  t_{1}\right)  =\frac{1}{2}\left[  Kp\left(  t_{1}\right)  +Kp\left(
1-t_{1}\right)  \right]  , \label{2.101}%
\end{equation}%
\begin{equation}
Kp\left(  t_{2}\right)  =\frac{1}{2}\left[  Kp\left(  t_{2}\right)  +Kp\left(
1-t_{2}\right)  \right]  \label{2.102}%
\end{equation}
and, by Lemma \ref{l2}, we obtain%
\begin{equation}
\frac{1}{2}\left[  Kp\left(  t_{2}\right)  +Kp\left(  1-t_{2}\right)  \right]
\leq\frac{1}{2}\left[  Kp\left(  t_{1}\right)  +Kp\left(  1-t_{1}\right)
\right]  . \label{2.103}%
\end{equation}

From $\left(  \ref{2.101}\right)  -\left(  \ref{2.103}\right)  ,$ we obtain
that $Kp$ is decreasing on $\left[  0,\frac{1}{2}\right]  .$ Since $Kp$ is
symmetric about $\frac{1}{2}$ and $Kp$ is decreasing on $\left[  0,\frac{1}%
{2}\right]  $, we get that $Kp$ is increasing on $\left[  \frac{1}%
{2},1\right]  .$ Using the symmetry and monotonicity of $Kp,$ we derive the
inequalities $\left(  \ref{2.97}\right)  $ and $\left(  \ref{2.98}\right)  .$

$\left(  3\right)  $ Using simple integration techniques and under the
hypothesis of $p,$ the following identity holds on $\left[  0,1\right]  :$%
\begin{align}
&  Kp\left(  t\right) \label{2.104}\\
&  =%
{\displaystyle\int\nolimits_{x}^{x^{\prime}}}
{\displaystyle\int\nolimits_{x}^{x^{\prime}}}
\left[  f\left(  t\left(  \left(  1-\alpha\right)  x+\alpha s\right)  +\left(
1-t\right)  \left(  \left(  1-\alpha\right)  x+\alpha u\right)  \right)
\right. \nonumber\\
&  +f\left(  t\left(  \left(  1-\alpha\right)  x+\alpha s\right)  +\left(
1-t\right)  \left(  \left(  1-\alpha\right)  x^{\prime}+\alpha\left(
x+x^{\prime}-u\right)  \right)  \right) \nonumber\\
&  +f\left(  t\left(  \left(  1-\alpha\right)  x^{\prime}+\alpha s\right)
+\left(  1-t\right)  \left(  \left(  1-\alpha\right)  x+\alpha u\right)
\right) \nonumber\\
&  +\left.  f\left(  t\left(  \left(  1-\alpha\right)  x^{\prime}+\alpha
s\right)  +\left(  1-t\right)  \left(  \left(  1-\alpha\right)  x^{\prime
}+\alpha\left(  x+x^{\prime}-u\right)  \right)  \right)  \right] \nonumber\\
&  \text{
\ \ \ \ \ \ \ \ \ \ \ \ \ \ \ \ \ \ \ \ \ \ \ \ \ \ \ \ \ \ \ \ \ \ \ \ \ \ \ \ \ \ \ \ \ \ \ \ \ \ \ \ \ \ \ \ \ \ \ \ \ \ \ \ \ \ \ \ \ \ \ \ \ \ }%
p\left(  s\right)  p\left(  u\right)  dsdu.\nonumber
\end{align}

By Lemma \ref{l2} and $x+x^{\prime}=y+y^{\prime},$ the following inequalities
hold for all $t\in\left[  0,1\right]  ,$ $u\in\left[  x,x^{\prime}\right]
$\ and $s\in\left[  x,x^{\prime}\right]  :$%
\begin{align*}
&  \left[  f\left(  t\left(  \left(  1-\alpha\right)  x+\alpha s\right)
+\left(  1-t\right)  y\right)  \right. \\
&  +\left.  f\left(  t\left(  \left(  1-\alpha\right)  x+\alpha s\right)
+\left(  1-t\right)  y^{\prime}\right)  \right]  p\left(  s\right)  p\left(
u\right) \\
&  \leq\left[  f\left(  t\left(  \left(  1-\alpha\right)  x+\alpha s\right)
+\left(  1-t\right)  \left(  \left(  1-\alpha\right)  x+\alpha u\right)
\right)  \right. \\
&  +\left.  f\left(  t\left(  \left(  1-\alpha\right)  x+\alpha s\right)
+\left(  1-t\right)  \left(  \left(  1-\alpha\right)  x^{\prime}+\alpha\left(
x+x^{\prime}-u\right)  \right)  \right)  \right]  p\left(  s\right)  p\left(
u\right)  .
\end{align*}%
\begin{align*}
&  \left[  f\left(  t\left(  \left(  1-\alpha\right)  x^{\prime}+\alpha
s\right)  +\left(  1-t\right)  y\right)  \right. \\
&  +\left.  f\left(  t\left(  \left(  1-\alpha\right)  x^{\prime}+\alpha
s\right)  +\left(  1-t\right)  y^{\prime}\right)  \right]  p\left(  s\right)
p\left(  u\right) \\
&  \leq\left[  f\left(  t\left(  \left(  1-\alpha\right)  x^{\prime}+\alpha
s\right)  +\left(  1-t\right)  \left(  \left(  1-\alpha\right)  x+\alpha
u\right)  \right)  \right. \\
&  +\left.  f\left(  t\left(  \left(  1-\alpha\right)  x^{\prime}+\alpha
s\right)  +\left(  1-t\right)  \left(  \left(  1-\alpha\right)  x^{\prime
}+\alpha\left(  x+x^{\prime}-u\right)  \right)  \right)  \right]  p\left(
s\right)  p\left(  u\right)  .
\end{align*}

Integrating the above inequalities over $u$ on $\left[  x,x^{\prime}\right]
$, over $s$ on $\left[  x,x^{\prime}\right]  $ and using the above identity,
we derive the inequality $\left(  \ref{2.99}\right)  .$

From the inequalities $\left(  \ref{2.41}\right)  ,$ $\left(  \ref{2.44}%
\right)  $ and $\left(  \ref{2.99}\right)  $, we have%
\begin{align*}
&  \left[  f\left(  y\right)  +2f\left(  \frac{x+x^{\prime}}{2}\right)
+f\left(  y^{\prime}\right)  \right]  \left[
{\displaystyle\int\nolimits_{x}^{x^{\prime}}}
p\left(  s\right)  ds\right]  ^{2}\\
&  \leq\left[  Ip_{1}\left(  \frac{1}{2}\right)  +Ip_{2}\left(  \frac{1}%
{2}\right)  \right]
{\displaystyle\int\nolimits_{x}^{x^{\prime}}}
p\left(  s\right)  ds\\
&  \leq Kp\left(  \frac{1}{2}\right)
\end{align*}
and from which we derive the inequality $\left(  \ref{2.100}\right)  .$
\end{proof}

\begin{remark}
\label{r29}\textit{Using Remark \ref{r11}, Theorem \ref{t23}} \textit{reduces
to Theorem \ref{A37} as }$x=a,$ $y=y^{\prime}=\frac{a+b}{2},x^{\prime}=b,$
$\alpha=\frac{1}{2}$\ \textit{and }$p\left(  s\right)  =\frac{g\left(
s\right)  }{2}$ $\left(  s\in\left[  a,b\right]  \right)  .$
\end{remark}

\begin{theorem}
\label{t24}\textit{Let }$x,y,y^{\prime},x^{\prime},\Omega,f,p,G_{1}%
,G_{2},Q_{1},Ip_{1},Ip_{2},Kp,Sp$\textit{\ be defined as above.} Then we have
the inequality%
\begin{align}
0  &  \leq Kp\left(  t\right)  -\left[  Ip_{1}\left(  t\right)  +Ip_{2}\left(
t\right)  \right]  \int_{x}^{x^{\prime}}p\left(  s\right)  ds\label{2.105}\\
&  \leq2Sp\left(  1-t\right)  \int_{x}^{x^{\prime}}p\left(  s\right)
ds-Kp\left(  t\right) \nonumber
\end{align}
for all $t\in\left[  0,1\right]  .$
\end{theorem}

\begin{proof}
Using simple integration techniques, $\left(  \ref{2.104}\right)  $ and
$\left(  \ref{2.96}\right)  ,$ the following identity holds on $\left[
0,1\right]  :$%
\begin{align*}
&  2Kp\left(  t\right) \\
&  =%
{\displaystyle\int\nolimits_{x}^{y}}
{\displaystyle\int\nolimits_{x}^{x^{\prime}}}
\frac{2}{\alpha}\left[  f\left(  t\left(  \left(  1-\alpha\right)  x+\alpha
s\right)  +\left(  1-t\right)  u\right)  \right. \\
&  +f\left(  t\left(  \left(  1-\alpha\right)  x+\alpha s\right)  +\left(
1-t\right)  \left(  x+x^{\prime}-u\right)  \right) \\
&  +f\left(  t\left(  \left(  1-\alpha\right)  x^{\prime}+\alpha s\right)
+\left(  1-t\right)  u\right) \\
&  +\left.  f\left(  t\left(  \left(  1-\alpha\right)  x^{\prime}+\alpha
s\right)  +\left(  1-t\right)  \left(  x+x^{\prime}-u\right)  \right)
\right]  p\left(  s\right)  p\left(  \frac{1}{\alpha}\left(  u-x\right)
+x\right)  dsdu\\
&  =%
{\displaystyle\int\nolimits_{x}^{\frac{x+y}{2}}}
{\displaystyle\int\nolimits_{x}^{x^{\prime}}}
\frac{2}{\alpha}\left[  f\left(  t\left(  \left(  1-\alpha\right)  x+\alpha
s\right)  +\left(  1-t\right)  u\right)  \right. \\
&  +f\left(  t\left(  \left(  1-\alpha\right)  x+\alpha s\right)  +\left(
1-t\right)  \left(  x+y-u\right)  \right)
\end{align*}%
\begin{align*}
&  \text{ \ \ \ }+f\left(  t\left(  \left(  1-\alpha\right)  x+\alpha
s\right)  +\left(  1-t\right)  \left(  x+x^{\prime}-u\right)  \right) \\
&  \text{ \ \ \ }+f\left(  t\left(  \left(  1-\alpha\right)  x+\alpha
s\right)  +\left(  1-t\right)  \left(  u+x^{\prime}-y\right)  \right) \\
&  \text{ \ \ \ }+f\left(  t\left(  \left(  1-\alpha\right)  x^{\prime}+\alpha
s\right)  +\left(  1-t\right)  u\right) \\
&  \text{ \ \ \ }+f\left(  t\left(  \left(  1-\alpha\right)  x^{\prime}+\alpha
s\right)  +\left(  1-t\right)  \left(  x+y-u\right)  \right) \\
&  \text{ \ \ \ }+f\left(  t\left(  \left(  1-\alpha\right)  x^{\prime}+\alpha
s\right)  +\left(  1-t\right)  \left(  x+x^{\prime}-u\right)  \right) \\
&  \text{ \ \ \ }+\left.  f\left(  t\left(  \left(  1-\alpha\right)
x^{\prime}+\alpha s\right)  +\left(  1-t\right)  \left(  u+x^{\prime
}-y\right)  \right)  \right]  p\left(  s\right)  p\left(  \frac{1}{\alpha
}\left(  u-x\right)  +x\right)  dsdu.
\end{align*}

By Lemma \ref{l2} and $x+x^{\prime}=y+y^{\prime},$ the following inequalities
hold for all $t\in\left[  0,1\right]  ,$ $u\in\left[  x,\frac{x+y}{2}\right]
$\ and $s\in\left[  x,x^{\prime}\right]  :$%
\begin{align*}
&  \frac{2}{\alpha}\left[  f\left(  t\left(  \left(  1-\alpha\right)  x+\alpha
s\right)  +\left(  1-t\right)  u\right)  \right. \\
&  +\left.  f\left(  t\left(  \left(  1-\alpha\right)  x+\alpha s\right)
+\left(  1-t\right)  \left(  x+y-u\right)  \right)  \right]  p\left(
s\right)  p\left(  \frac{1}{\alpha}\left(  u-x\right)  +x\right) \\
&  \leq\frac{2}{\alpha}\left[  f\left(  t\left(  \left(  1-\alpha\right)
x+\alpha s\right)  +\left(  1-t\right)  x\right)  \right. \\
&  +\left.  f\left(  t\left(  \left(  1-\alpha\right)  x+\alpha s\right)
+\left(  1-t\right)  y\right)  \right]  p\left(  s\right)  p\left(  \frac
{1}{\alpha}\left(  u-x\right)  +x\right)  .
\end{align*}%
\begin{align*}
&  \frac{2}{\alpha}\left[  f\left(  t\left(  \left(  1-\alpha\right)  x+\alpha
s\right)  +\left(  1-t\right)  \left(  u+x^{\prime}-y\right)  \right)  \right.
\\
&  +\left.  f\left(  t\left(  \left(  1-\alpha\right)  x+\alpha s\right)
+\left(  1-t\right)  \left(  x+x^{\prime}-u\right)  \right)  \right]  p\left(
s\right)  p\left(  \frac{1}{\alpha}\left(  u-x\right)  +x\right) \\
&  \leq\frac{2}{\alpha}\left[  f\left(  t\left(  \left(  1-\alpha\right)
x+\alpha s\right)  +\left(  1-t\right)  y^{\prime}\right)  \right. \\
&  +\left.  f\left(  t\left(  \left(  1-\alpha\right)  x+\alpha s\right)
+\left(  1-t\right)  x^{\prime}\right)  \right]  p\left(  s\right)  p\left(
\frac{1}{\alpha}\left(  u-x\right)  +x\right)  .
\end{align*}%
\begin{align*}
&  \frac{2}{\alpha}\left[  f\left(  t\left(  \left(  1-\alpha\right)
x^{\prime}+\alpha s\right)  +\left(  1-t\right)  u\right)  \right. \\
&  +\left.  f\left(  t\left(  \left(  1-\alpha\right)  x^{\prime}+\alpha
s\right)  +\left(  1-t\right)  \left(  x+y-u\right)  \right)  \right]
p\left(  s\right)  p\left(  \frac{1}{\alpha}\left(  u-x\right)  +x\right) \\
&  \leq\frac{2}{\alpha}\left[  f\left(  t\left(  \left(  1-\alpha\right)
x^{\prime}+\alpha s\right)  +\left(  1-t\right)  x\right)  \right. \\
&  +\left.  f\left(  t\left(  \left(  1-\alpha\right)  x^{\prime}+\alpha
s\right)  +\left(  1-t\right)  y\right)  \right]  p\left(  s\right)  p\left(
\frac{1}{\alpha}\left(  u-x\right)  +x\right)  .
\end{align*}%
\begin{align*}
&  \frac{2}{\alpha}\left[  f\left(  t\left(  \left(  1-\alpha\right)
x^{\prime}+\alpha s\right)  +\left(  1-t\right)  \left(  u+x^{\prime
}-y\right)  \right)  \right. \\
&  +\left.  f\left(  t\left(  \left(  1-\alpha\right)  x^{\prime}+\alpha
s\right)  +\left(  1-t\right)  \left(  x+x^{\prime}-u\right)  \right)
\right]  p\left(  s\right)  p\left(  \frac{1}{\alpha}\left(  u-x\right)
+x\right) \\
&  \leq\frac{2}{\alpha}\left[  f\left(  t\left(  \left(  1-\alpha\right)
x^{\prime}+\alpha s\right)  +\left(  1-t\right)  y^{\prime}\right)  \right. \\
&  +\left.  f\left(  t\left(  \left(  1-\alpha\right)  x^{\prime}+\alpha
s\right)  +\left(  1-t\right)  x^{\prime}\right)  \right]  p\left(  s\right)
p\left(  \frac{1}{\alpha}\left(  u-x\right)  +x\right)  .
\end{align*}

Integrating the above inequalities over $u$ on $\left[  x,\frac{x+y}%
{2}\right]  ,$ over $s$ on $\left[  x,x^{\prime}\right]  $ and using the above
identity and $\left(  \ref{2.89}\right)  ,$ we obtain the inequality%
\begin{align*}
2Kp\left(  t\right)   &  \leq\left[  Ip_{1}\left(  t\right)  +Ip_{2}\left(
t\right)  \right]  \int_{x}^{x^{\prime}}p\left(  s\right)  ds\\
&  +2Sp\left(  1-t\right)  \int_{x}^{x^{\prime}}p\left(  s\right)  ds\ \left(
\text{ }t\in\left[  0,1\right]  \right)  .
\end{align*}

Using the above inequality and $\left(  \ref{2.99}\right)  ,$ we derive the
inequality $\left(  \ref{2.105}\right)  .$ This completes the proof.
\end{proof}

\begin{remark}
\label{r30}\textit{Using Remark \ref{r11}, Theorem \ref{t24}} \textit{reduces
to Theorem \ref{A38} as }$x=a,$ $y=y^{\prime}=\frac{a+b}{2},x^{\prime}=b,$
$\alpha=\frac{1}{2}$\ \textit{and }$p\left(  s\right)  =\frac{g\left(
s\right)  }{2}$ $\left(  s\in\left[  a,b\right]  \right)  .$
\end{remark}

\section{Applications}

In this section, we shall establish a large number of applications for the
extended Beta function and the extended special means.

\subsection{Applications for the Extended Beta Function}

Throughout this subsection, let $\rho,\beta>0,r\geq1,a=0,b=1,x+x^{\prime
}=y+y^{\prime}=1$ $\left(  x,y,y^{\prime},x^{\prime}\in\left(  0,1\right)
\right)  ,y=\left(  1-\alpha\right)  x+\alpha x^{\prime}$ $\left(
0<\alpha\leq\frac{1}{2}\right)  ,f\left(  s\right)  =s^{r}$ $\left(
s\in\left[  0,1\right]  \right)  $ and $p\left(  s\right)  =s^{\rho-1}\left(
1-s\right)  ^{\rho-1}$ $\left(  s\in\left[  x,x^{\prime}\right]  \right)  .$
Then $f$ is convex on $\left[  0,1\right]  $ and $p$ is symmetic to
$\frac{x+x^{\prime}}{2}.$

Let us recall the \textit{Beta function}%
\[
B\left(  \rho,\beta\right)  =\int_{0}^{1}s^{\rho-1}\left(  1-s\right)
^{\beta-1}ds.
\]

It is natural to consider the \textit{extended incomplete} \textit{Beta
function}%
\[
B\left(  \rho,\beta;c,d\right)  =\int_{c}^{d}s^{\rho-1}\left(  1-s\right)
^{\beta-1}ds,\text{ \ }\left(  \left[  c,d\right]  \subset\left[  0,1\right]
\right)  .
\]

We have for all $\left[  c,d\right]  \subset\left[  0,1\right]  ,$ that%
\[
B\left(  \rho+r,\rho;c,d\right)  =\int_{c}^{d}s^{\rho-1}\left(  1-s\right)
^{\rho-1}s^{r}ds.
\]

\begin{remark}
\label{r31}From the Section 2, we get%
\[
\int_{\Omega}f\left(  s\right)  p\left(  s\right)  ds=B\left(  \rho
+r,\rho;x,y\right)  +B\left(  \rho+r,\rho;y^{\prime},x^{\prime}\right)  ,
\]%
\[
\int_{\Omega}p\left(  s\right)  ds=2B\left(  \rho,\rho;x,y\right)  ,\text{
}\int_{x}^{x^{\prime}}p\left(  s\right)  ds=2B\left(  \rho,\rho;x,\frac{1}%
{2}\right)
\]
and the following convex functions on $\left[  0,1\right]  $:%
\[
G_{1}\left(  t\right)  =\frac{1}{2}\left[  \left(  tx+\left(  1-t\right)
y\right)  ^{r}+\left(  tx^{\prime}+\left(  1-t\right)  y^{\prime}\right)
^{r}\right]  .
\]%
\[
G_{2}\left(  t\right)  =\frac{1}{2}\left[  \left(  tx+\left(  1-t\right)
y^{\prime}\right)  ^{r}+\left(  tx^{\prime}+\left(  1-t\right)  y\right)
^{r}\right]  .
\]%
\[
G_{3}\left(  t\right)  =\frac{1}{2}\left[  \left(  ty+\left(  1-t\right)
y^{\prime}\right)  ^{r}+\left(  ty^{\prime}+\left(  1-t\right)  y\right)
^{r}\right]  .
\]%
\[
Hp_{1}\left(  t\right)  =\int_{x}^{y}\left(  s-s^{2}\right)  ^{\rho-1}\left[
\left(  ts+\left(  1-t\right)  y\right)  ^{r}+\left(  t\left(  1-s\right)
+\left(  1-t\right)  y^{\prime}\right)  ^{r}\right]  ds.
\]%
\[
Hp_{2}\left(  t\right)  =\int_{x}^{y}\left(  s-s^{2}\right)  ^{\rho-1}\left[
\left(  ts+\left(  1-t\right)  y^{\prime}\right)  ^{r}+\left(  t\left(
1-s\right)  +\left(  1-t\right)  y\right)  ^{r}\right]  ds.
\]%
\[
H_{1}\left(  t\right)  =\frac{1}{2\left(  y-x\right)  }\int_{x}^{y}\left[
\left(  ts+\left(  1-t\right)  y\right)  ^{r}+\left(  t\left(  1-s\right)
+\left(  1-t\right)  y^{\prime}\right)  ^{r}\right]  ds.
\]%
\[
H_{2}\left(  t\right)  =\frac{1}{2\left(  y-x\right)  }\int_{x}^{y}\left[
\left(  ts+\left(  1-t\right)  y^{\prime}\right)  ^{r}+\left(  t\left(
1-s\right)  +\left(  1-t\right)  y\right)  ^{r}\right]  ds.
\]%
\[
Fp_{1}\left(  t\right)  =\int_{\Omega}\int_{\Omega}\left(  s-s^{2}\right)
^{\rho-1}\left(  u-u^{2}\right)  ^{\rho-1}\left(  ts+\left(  1-t\right)
u\right)  ^{r}dsdu.
\]%
\[
Pp_{1}\left(  t\right)  =\int_{x}^{y}\left(  s-s^{2}\right)  ^{\rho-1}\left[
\left(  tx+\left(  1-t\right)  s\right)  ^{r}+\left(  tx^{\prime}+\left(
1-t\right)  \left(  1-s\right)  \right)  ^{r}\right]  ds.
\]%
\[
Lp_{1}\left(  t\right)  =\frac{1}{2}%
{\displaystyle\int\nolimits_{\Omega}}
\left(  s-s^{2}\right)  ^{\rho-1}\left[  \left(  tx+\left(  1-t\right)
s\right)  ^{r}+\left(  tx^{\prime}+\left(  1-t\right)  s\right)  ^{r}\right]
ds.
\]%
\[
Q_{1}\left(  t\right)  =\frac{1}{2}\left[  \left(  tx+\left(  1-t\right)
x^{\prime}\right)  ^{r}+\left(  tx^{\prime}+\left(  1-t\right)  x\right)
^{r}\right]  .
\]%
\begin{align*}
Ip_{1}\left(  t\right)   &  =%
{\displaystyle\int\nolimits_{x}^{x^{\prime}}}
\left(  s-s^{2}\right)  ^{\rho-1}\left[  \left(  t\left(  \left(
1-\alpha\right)  x+\alpha s\right)  +\left(  1-t\right)  y\right)  ^{r}\right.
\\
&  \text{ \ \ \ \ \ \ \ \ \ \ \ \ \ \ }+\left.  \left(  t\left(  \left(
1-\alpha\right)  x^{\prime}+\alpha s\right)  +\left(  1-t\right)  y^{\prime
}\right)  ^{r}\right]  ds.
\end{align*}%
\begin{align*}
Ip_{2}\left(  t\right)   &  =%
{\displaystyle\int\nolimits_{x}^{x^{\prime}}}
\left(  s-s^{2}\right)  ^{\rho-1}\left[  \left(  t\left(  \left(
1-\alpha\right)  x+\alpha s\right)  +\left(  1-t\right)  y^{\prime}\right)
^{r}\right. \\
&  \text{ \ \ \ \ \ \ \ \ \ \ \ \ \ }+\left.  \left(  t\left(  \left(
1-\alpha\right)  x^{\prime}+\alpha s\right)  +\left(  1-t\right)  y\right)
^{r}\right]  ds.
\end{align*}%
\begin{align*}
Jp\left(  t\right)   &  =%
{\displaystyle\int\nolimits_{x}^{x^{\prime}}}
\left(  s-s^{2}\right)  ^{\rho-1}\left[  \left(  t\left(  \left(
1-\alpha\right)  x+\alpha s\right)  +\left(  1-t\right)  \frac{x+y}{2}\right)
^{r}\right. \\
&  \text{ \ \ \ \ \ \ \ \ \ \ \ }+\left.  \left(  t\left(  \left(
1-\alpha\right)  x^{\prime}+\alpha s\right)  +\left(  1-t\right)
\frac{x^{\prime}+y^{\prime}}{2}\right)  ^{r}\right]  ds.
\end{align*}%
\begin{align*}
Mp\left(  t\right)   &  =%
{\displaystyle\int\nolimits_{x}^{\frac{1}{2}}}
\left(  s-s^{2}\right)  ^{\rho-1}\left[  \left(  tx+\left(  1-t\right)
\left(  \left(  1-\alpha\right)  x+\alpha s\right)  \right)  ^{r}\right. \\
&  \text{ \ \ \ \ \ \ \ \ \ \ \ \ \ \ \ \ \ \ \ \ }+\left.  \left(
ty^{\prime}+\left(  1-t\right)  \left(  \left(  1-\alpha\right)  x^{\prime
}+\alpha s\right)  \right)  ^{r}\right]  ds\\
&  +%
{\displaystyle\int\nolimits_{\frac{1}{2}}^{x^{\prime}}}
\left(  s-s^{2}\right)  ^{\rho-1}\left[  \left(  ty+\left(  1-t\right)
\left(  \left(  1-\alpha\right)  x+\alpha s\right)  \right)  ^{r}\right. \\
&  \text{ \ \ \ \ \ \ \ \ \ \ \ \ \ \ \ \ \ \ \ \ \ }+\left.  \left(
tx^{\prime}+\left(  1-t\right)  \left(  \left(  1-\alpha\right)  x^{\prime
}+\alpha s\right)  \right)  ^{r}\right]  ds.
\end{align*}%
\begin{align*}
Np\left(  t\right)   &  =%
{\displaystyle\int\nolimits_{x}^{x^{\prime}}}
\left(  s-s^{2}\right)  ^{\rho-1}\left[  \left(  tx+\left(  1-t\right)
\left(  \left(  1-\alpha\right)  x+\alpha s\right)  \right)  ^{r}\right. \\
&  \text{ \ \ \ \ \ \ \ \ \ \ \ \ \ \ \ \ \ }+\left.  \left(  tx^{\prime
}+\left(  1-t\right)  \left(  \left(  1-\alpha\right)  x^{\prime}+\alpha
s\right)  \right)  ^{r}\right]  ds.
\end{align*}%
\begin{align*}
Sp\left(  t\right)   &  =%
{\displaystyle\int\nolimits_{x}^{x^{\prime}}}
\frac{1}{2}\left(  s-s^{2}\right)  ^{\rho-1}\left[  \left(  tx+\left(
1-t\right)  \left(  \left(  1-\alpha\right)  x+\alpha s\right)  \right)
^{r}\right. \\
&  \text{ \ \ \ \ \ \ \ \ \ \ \ \ \ \ \ \ \ \ \ \ \ }+\left(  tx+\left(
1-t\right)  \left(  \left(  1-\alpha\right)  x^{\prime}+\alpha s\right)
\right)  ^{r}\\
&  \text{ \ \ \ \ \ \ \ \ \ \ \ \ \ \ \ \ \ \ \ \ \ }+\left(  tx^{\prime
}+\left(  1-t\right)  \left(  \left(  1-\alpha\right)  x+\alpha s\right)
\right)  ^{r}\\
&  \text{ \ \ \ \ \ \ \ \ \ \ \ \ \ \ \ \ \ \ \ }+\left.  \left(  tx^{\prime
}+\left(  1-t\right)  \left(  \left(  1-\alpha\right)  x^{\prime}+\alpha
s\right)  \right)  ^{r}\right]  ds.
\end{align*}%
\begin{align*}
&  Kp\left(  t\right)  =%
{\displaystyle\int\nolimits_{x}^{x^{\prime}}}
{\displaystyle\int\nolimits_{x}^{x^{\prime}}}
\left(  s-s^{2}\right)  ^{\rho-1}\left(  u-u^{2}\right)  ^{\rho-1}\\
&  \text{ \ \ \ \ \ \ \ \ \ }\times\left[  \left(  t\left(  \left(
1-\alpha\right)  x+\alpha s\right)  +\left(  1-t\right)  \left(  \left(
1-\alpha\right)  x+\alpha u\right)  \right)  ^{r}\right. \\
&  \text{ \ \ \ \ \ \ \ \ \ \ \ \ }+\left(  t\left(  \left(  1-\alpha\right)
x+\alpha s\right)  +\left(  1-t\right)  \left(  \left(  1-\alpha\right)
x^{\prime}+\alpha u\right)  \right)  ^{r}\\
&  \text{ \ \ \ \ \ \ \ \ \ \ \ \ }+\left(  t\left(  \left(  1-\alpha\right)
x^{\prime}+\alpha s\right)  +\left(  1-t\right)  \left(  \left(
1-\alpha\right)  x+\alpha u\right)  \right)  ^{r}\\
&  \text{ \ \ \ \ \ \ \ \ \ \ \ \ }+\left.  \left(  t\left(  \left(
1-\alpha\right)  x^{\prime}+\alpha s\right)  +\left(  1-t\right)  \left(
\left(  1-\alpha\right)  x^{\prime}+\alpha u\right)  \right)  ^{r}\right]
dsdu.
\end{align*}

\end{remark}

Using Theorems \ref{t1} --\ \ref{t24} and Remark \ref{r31}, we have the
following propositions of the extended Beta function:

\begin{proposition}
\label{p1}$Hp_{1}$\ is increasing on $\left[  0,1\right]  $\ and the following
inequalities hold for all $t\in\left[  0,1\right]  :$
\begin{align*}
&  \left[  y^{r}+\left(  y^{\prime}\right)  ^{r}\right]  B\left(  \rho
,\rho;x,y\right) \\
&  =Hp_{1}\left(  0\right)  \leq Hp_{1}\left(  t\right)  \leq Hp_{1}\left(
1\right) \\
&  =B\left(  \rho+r,\rho;x,y\right)  +B\left(  \rho+r,\rho;y^{\prime
},x^{\prime}\right)  \ .
\end{align*}%
\begin{align*}
&  Hp_{1}\left(  t\right) \\
&  \leq t\left[  B\left(  \rho+r,\rho;x,y\right)  +B\left(  \rho
+r,\rho;y^{\prime},x^{\prime}\right)  \right] \\
&  \text{ \ \ \ }+\left(  1-t\right)  \cdot\left[  y^{r}+\left(  y^{\prime
}\right)  ^{r}\right]  B\left(  \rho,\rho;x,y\right) \\
&  \leq B\left(  \rho+r,\rho;x,y\right)  +B\left(  \rho+r,\rho;y^{\prime
},x^{\prime}\right) \\
&  \leq\left[  x^{r}+\left(  x^{\prime}\right)  ^{r}\right]  B\left(
\rho,\rho;x,y\right)  .
\end{align*}%
\begin{align*}
&  2^{1-r}B\left(  \rho,\rho;x,y\right) \\
&  \leq Hp_{2}\left(  t\right) \\
&  \leq t\left[  B\left(  \rho+r,\rho;x,y\right)  +B\left(  \rho
+r,\rho;y^{\prime},x^{\prime}\right)  \right] \\
&  \text{ \ \ \ }+\left(  1-t\right)  \cdot\left[  y^{r}+\left(  y^{\prime
}\right)  ^{r}\right]  B\left(  \rho,\rho;x,y\right) \\
&  \leq B\left(  \rho+r,\rho;x,y\right)  +B\left(  \rho+r,\rho;y^{\prime
},x^{\prime}\right)  .
\end{align*}%
\[
Hp_{2}\left(  t\right)  \leq Hp_{1}\left(  t\right)  .
\]

\end{proposition}

\begin{proposition}
\label{p2}The following inequalities hold:%
\begin{align*}
&  \left[  y^{r}+\left(  y^{\prime}\right)  ^{r}\right]  B\left(  \rho
,\rho;x,y\right) \\
&  \leq\int_{\frac{x+y}{2}}^{y}2s^{r}\left(  2s-y\right)  ^{\rho-1}\left(
1-2s+y\right)  ^{\rho-1}ds\\
&  \text{ \ \ \ \ }+\int_{y^{\prime}}^{\frac{x^{\prime}+y^{\prime}}{2}}%
2s^{r}\left(  2s-y^{\prime}\right)  ^{\rho-1}\left(  1-2s+y^{\prime}\right)
^{\rho-1}ds\\
&  \leq\int_{0}^{1}Hp_{1}\left(  t\right)  dt\\
&  \leq\frac{1}{2}\left[  \left[  y^{r}+\left(  y^{\prime}\right)
^{r}\right]  B\left(  \rho,\rho;x,y\right)  +B\left(  \rho+r,\rho;x,y\right)
+B\left(  \rho+r,\rho;y^{\prime},x^{\prime}\right)  \right]  .
\end{align*}%
\begin{align*}
&  2^{1-r}B\left(  \rho,\rho;x,y\right) \\
&  \leq\int_{\frac{x+y^{\prime}}{2}}^{\frac{1}{2}}2s^{r}\left(  2s-y^{\prime
}\right)  ^{\rho-1}\left(  1-2s+y^{\prime}\right)  ^{\rho-1}ds\\
&  \text{ \ \ }+\int_{\frac{1}{2}}^{\frac{y+x^{\prime}}{2}}2s^{r}\left(
2s-y\right)  ^{\rho-1}\left(  1-2s+y\right)  ^{\rho-1}ds\\
&  \leq\int_{0}^{1}Hp_{2}\left(  t\right)  dt\\
&  \leq\frac{\left[  y^{r}+\left(  y^{\prime}\right)  ^{r}\right]  B\left(
\rho,\rho;x,y\right)  +B\left(  \rho+r,\rho;x,y\right)  +B\left(  \rho
+r,\rho;y^{\prime},x^{\prime}\right)  }{2}.
\end{align*}%
\begin{align*}
0  &  \leq B\left(  \rho+r,\rho;x,y\right)  +B\left(  \rho+r,\rho;y^{\prime
},x^{\prime}\right)  -Hp_{1}\left(  t\right) \\
&  \leq\left(  1-t\right)  \left[  \left(  x^{r}+\left(  x^{\prime}\right)
^{r}\right)  \left(  y-x\right)  \right. \\
&  \left.  -\frac{1}{r+1}\left(  \left(  x^{\prime}\right)  ^{r+1}%
+y^{r+1}-x^{r+1}-\left(  y^{\prime}\right)  ^{r+1}\right)  \right]
\max\left\{  \left(  x-x^{2}\right)  ^{\rho-1},\left(  y-y^{2}\right)
^{\rho-1}\right\}
\end{align*}
for all $t\in\left[  0,1\right]  .$%
\begin{align*}
0  &  \leq\left[  x^{r}+\left(  x^{\prime}\right)  ^{r}\right]  B\left(
\rho,\rho;x,y\right)  -Hp_{1}\left(  t\right) \\
&  \leq r\left(  y-x\right)  \left(  \left(  x^{\prime}\right)  ^{r-1}%
-x^{r-1}\right)  B\left(  \rho,\rho;x,y\right)  \text{ \ \ \ }\left(
t\in\left[  0,1\right]  \right)  .
\end{align*}%
\begin{align*}
0  &  \leq Hp_{1}\left(  t\right)  -\left[  y^{r}+\left(  y^{\prime}\right)
^{r}\right]  B\left(  \rho,\rho;x,y\right) \\
&  \leq r\left(  y-x\right)  \left(  \left(  x^{\prime}\right)  ^{r-1}%
-x^{r-1}\right)  B\left(  \rho,\rho;x,y\right)  \text{ \ \ \ }\left(
t\in\left[  0,1\right]  \right)  .
\end{align*}

\end{proposition}

\begin{proposition}
\label{p3}$Pp_{1}$\ is increasing on $\left[  0,1\right]  $\ and the following
inequalities hold for all $t\in\left[  0,1\right]  :$%
\begin{align*}
&  B\left(  \rho+r,\rho;x,y\right)  +B\left(  \rho+r,\rho;y^{\prime}%
,x^{\prime}\right) \\
&  =Pp_{1}\left(  0\right)  \leq Pp_{1}\left(  t\right)  \leq Pp_{1}\left(
1\right) \\
&  =\left[  x^{r}+\left(  x^{\prime}\right)  ^{r}\right]  B\left(  \rho
,\rho;x,y\right)  .
\end{align*}%
\begin{align*}
Pp_{1}\left(  t\right)   &  \leq\left(  1-t\right)  \left[  B\left(
\rho+r,\rho;x,y\right)  +B\left(  \rho+r,\rho;y^{\prime},x^{\prime}\right)
\right] \\
&  \text{ \ \ \ }+t\cdot\left[  x^{r}+\left(  x^{\prime}\right)  ^{r}\right]
B\left(  \rho,\rho;x,y\right) \\
&  \leq\left[  x^{r}+\left(  x^{\prime}\right)  ^{r}\right]  B\left(
\rho,\rho;x,y\right)  .
\end{align*}%
\[
Hp_{1}\left(  t\right)  \leq Pp_{1}\left(  t\right)  .\text{ }%
\]

\end{proposition}

\begin{proposition}
\label{p4}The following inequalities hold:%
\begin{align*}
&  B\left(  \rho+r,\rho;x,y\right)  +B\left(  \rho+r,\rho;y^{\prime}%
,x^{\prime}\right) \\
&  \leq%
{\displaystyle\int\nolimits_{x}^{\frac{x+y}{2}}}
2s^{r}\left(  2s-x\right)  ^{\rho-1}\left(  1-2s+x\right)  ^{\rho-1}ds\\
&  \text{ \ \ \ \ }+\int_{\frac{x^{\prime}+y^{\prime}}{2}}^{x^{\prime}}%
2s^{r}\left(  2s-x^{\prime}\right)  ^{\rho-1}\left(  1-2s+x^{\prime}\right)
^{\rho-1}ds\\
&  \leq\int_{0}^{1}Pp_{1}\left(  t\right)  dt\\
&  \leq\frac{\left[  x^{r}+\left(  x^{\prime}\right)  ^{r}\right]  B\left(
\rho,\rho;x,y\right)  +B\left(  \rho+r,\rho;x,y\right)  +B\left(  \rho
+r,\rho;y^{\prime},x^{\prime}\right)  }{2}.
\end{align*}%
\begin{align*}
0  &  \leq t\left[  \frac{1}{r+1}\left(  \left(  x^{\prime}\right)
^{r+1}+y^{r+1}-x^{r+1}-\left(  y^{\prime}\right)  ^{r+1}\right)  \right. \\
&  \text{ \ \ \ \ \ }\left.  -\left(  y^{r}+\left(  y^{\prime}\right)
^{r}\right)  \left(  y-x\right)  \right]  \cdot\min\left\{  \left(
x-x^{2}\right)  ^{\rho-1},\left(  y-y^{2}\right)  ^{\rho-1}\right\} \\
&  \leq Pp_{1}\left(  t\right)  -B\left(  \rho+r,\rho;x,y\right)  -B\left(
\rho+r,\rho;y^{\prime},x^{\prime}\right)  \text{ \ \ \ }\left(  t\in\left[
0,1\right]  \right)  .
\end{align*}%
\begin{align*}
0  &  \leq Pp_{1}\left(  t\right)  -\left[  y^{r}+\left(  y^{\prime}\right)
^{r}\right]  B\left(  \rho,\rho;x,y\right) \\
&  \leq r\left(  y-x\right)  \left(  \left(  x^{\prime}\right)  ^{r-1}%
-x^{r-1}\right)  B\left(  \rho,\rho;x,y\right)  \text{ \ \ \ }\left(
t\in\left[  0,1\right]  \right)  .
\end{align*}%
\begin{align*}
0  &  \leq\left[  x^{r}+\left(  x^{\prime}\right)  ^{r}\right]  B\left(
\rho,\rho;x,y\right)  -Pp_{1}\left(  t\right) \\
&  \leq r\left(  y-x\right)  \left(  \left(  x^{\prime}\right)  ^{r-1}%
-x^{r-1}\right)  B\left(  \rho,\rho;x,y\right)  \text{ \ \ \ }\left(
t\in\left[  0,1\right]  \right)  .
\end{align*}%
\begin{align*}
0  &  \leq Pp_{1}\left(  t\right)  -Hp_{1}\left(  t\right) \\
&  \leq r\left(  y-x\right)  \left(  \left(  x^{\prime}\right)  ^{r-1}%
-x^{r-1}\right)  B\left(  \rho,\rho;x,y\right)  \text{ \ \ \ }\left(
t\in\left[  0,1\right]  \right)  .
\end{align*}

\end{proposition}

\begin{proposition}
\label{p5}$Fp_{1}$\textit{\ is symmetric about }$\frac{1}{2},$%
\textit{\ decreasing on }$\left[  0,\frac{1}{2}\right]  $\textit{\ and
increasing on }$\left[  \frac{1}{2},1\right]  .$ \textit{The following
identities and inequalities hold:}%
\begin{align*}
&  \sup\limits_{t\in\left[  0,1\right]  }Fp_{1}\left(  t\right) \\
&  =2\left[  B\left(  \rho+r,\rho;x,y\right)  +B\left(  \rho+r,\rho;y^{\prime
},x^{\prime}\right)  \right]  B\left(  \rho,\rho;x,y\right)  .
\end{align*}%
\[
\inf\limits_{t\in\left[  0,1\right]  }Fp_{1}\left(  t\right)  =\int_{\Omega
}\int_{\Omega}\left(  s-s^{2}\right)  ^{\rho-1}\left(  u-u^{2}\right)
^{\rho-1}\left(  \frac{s+u}{2}\right)  ^{r}dsdu.
\]%
\[
\left[  Hp_{1}\left(  t\right)  +Hp_{2}\left(  t\right)  \right]  B\left(
\rho,\rho;x,y\right)  \leq Fp_{1}\left(  t\right)  \text{\ \ \ }\left(
t\in\left[  0,1\right]  \right)  .
\]%
\begin{align*}
&  \left[  y^{r}+2^{1-r}+\left(  y^{\prime}\right)  ^{r}\right]  B^{2}\left(
\rho,\rho;x,y\right) \\
&  \leq\int_{\Omega}\int_{\Omega}\left(  s-s^{2}\right)  ^{\rho-1}\left(
u-u^{2}\right)  ^{\rho-1}\left(  \frac{s+u}{2}\right)  ^{r}dsdu.
\end{align*}

\end{proposition}

\begin{proposition}
\label{p6}The following inequalities hold:%
\[
Hp_{1}\left(  t\right)  \leq2G_{1}\left(  t\right)  B\left(  \rho
,\rho;x,y\right)  \leq Pp_{1}\left(  t\right)  \text{\ \ \ }\left(
t\in\left[  0,1\right]  \right)  .
\]%
\begin{align*}
&  \int_{\frac{x+y}{2}}^{y}2s^{r}\left(  2s-y\right)  ^{\rho-1}\left(
1-2s+y\right)  ^{\rho-1}ds\\
&  +\int_{y^{\prime}}^{\frac{x^{\prime}+y^{\prime}}{2}}2s^{r}\left(
2s-y^{\prime}\right)  ^{\rho-1}\left(  1-2s+y^{\prime}\right)  ^{\rho-1}ds\\
&  \leq\left[  \left(  \frac{x+y}{2}\right)  ^{r}+\left(  \frac{x^{\prime
}+y^{\prime}}{2}\right)  ^{r}\right]  B\left(  \rho,\rho;x,y\right) \\
&  \leq\frac{1}{2}\int_{x}^{y}\left(  s-s^{2}\right)  ^{\rho-1}\left[  \left(
\frac{s+x}{2}\right)  ^{r}+\left(  \frac{x+2y-s}{2}\right)  ^{r}\right]  ds\\
&  +\frac{1}{2}\int_{y^{\prime}}^{x^{\prime}}\left(  s-s^{2}\right)  ^{\rho
-1}\left[  \left(  \frac{s+x^{\prime}}{2}\right)  ^{r}+\left(  \frac
{x^{\prime}+2y^{\prime}-s}{2}\right)  ^{r}\right]  ds\\
&  \leq\frac{x^{r}+y^{r}+\left(  y^{\prime}\right)  ^{r}+\left(  x^{\prime
}\right)  ^{r}}{2}B\left(  \rho,\rho;x,y\right)  .
\end{align*}%
\begin{align*}
&  \int_{\frac{x+y^{\prime}}{2}}^{\frac{1}{2}}2s^{r}\left(  2s-y^{\prime
}\right)  ^{\rho-1}\left(  1-2s+y^{\prime}\right)  ^{\rho-1}ds\\
&  +\int_{\frac{1}{2}}^{\frac{x^{\prime}+y}{2}}2s^{r}\left(  2s-y\right)
^{\rho-1}\left(  1-2s+y\right)  ^{\rho-1}ds\\
&  \leq\left[  \left(  \frac{x+y^{\prime}}{2}\right)  ^{r}+\left(
\frac{x^{\prime}+y}{2}\right)  ^{r}\right]  B\left(  \rho,\rho;x,y\right) \\
&  \leq\frac{1}{2}\int_{x}^{y}\left(  s-s^{2}\right)  ^{\rho-1}\left[  \left(
\frac{s+x}{2}\right)  ^{r}+\left(  \frac{x+2y^{\prime}-s}{2}\right)
^{r}\right]  ds\\
&  +\frac{1}{2}\int_{y^{\prime}}^{x^{\prime}}\left(  s-s^{2}\right)  ^{\rho
-1}\left[  \left(  \frac{s+x^{\prime}}{2}\right)  ^{r}+\left(  \frac
{x^{\prime}+2y-s}{2}\right)  ^{r}\right]  ds\\
&  \leq\frac{x^{r}+y^{r}+\left(  y^{\prime}\right)  ^{r}+\left(  x^{\prime
}\right)  ^{r}}{2}B\left(  \rho,\rho;x,y\right)  .
\end{align*}%
\begin{align*}
0  &  \leq Hp_{1}\left(  t\right)  -\left[  y^{r}+\left(  y^{\prime}\right)
^{r}\right]  B\left(  \rho,\rho;x,y\right) \\
&  \leq2G_{1}\left(  t\right)  B\left(  \rho,\rho;x,y\right)  -Hp_{1}\left(
t\right)  \text{\ \ \ }\left(  t\in\left[  0,1\right]  \right)  .
\end{align*}%
\begin{align*}
0  &  \leq Pp_{1}\left(  t\right)  -2G_{1}\left(  t\right)  B\left(  \rho
,\rho;x,y\right) \\
&  \leq\left[  x^{r}+\left(  x^{\prime}\right)  ^{r}\right]  B\left(
\rho,\rho;x,y\right)  -Pp_{1}\left(  t\right)  \text{\ \ \ }\left(
t\in\left[  0,1\right]  \right)  .
\end{align*}

\end{proposition}

\begin{proposition}
\label{p7}The following inequalities hold:%
\[
\left[  G_{1}\left(  t\right)  +G_{2}\left(  t\right)  \right]  B\left(
\rho,\rho;x,y\right)  \leq Lp_{1}\left(  t\right)  \leq Pp_{1}\left(
t\right)  \text{\ }\left(  t\in\left[  0,1\right]  \right)  .\text{ }%
\]%
\[
\sup\limits_{t\in\left[  0,1\right]  }Lp_{1}\left(  t\right)  =Lp_{1}\left(
1\right)  =\left[  x^{r}+\left(  x^{\prime}\right)  ^{r}\right]  B\left(
\rho,\rho;x,y\right)  .
\]%
\begin{align*}
&  \left[  Hp_{1}\left(  1-t\right)  +Hp_{2}\left(  1-t\right)  \right]
B\left(  \rho,\rho;x,y\right) \\
&  \leq Fp_{1}\left(  t\right)  \leq2Lp_{1}\left(  t\right)  B\left(
\rho,\rho;x,y\right)  \text{\ \ \ }\left(  t\in\left[  0,1\right]  \right)  .
\end{align*}%
\begin{align*}
&  \left[  Hp_{1}\left(  t\right)  +Hp_{2}\left(  t\right)  \right]  B\left(
\rho,\rho;x,y\right) \\
&  \leq Fp_{1}\left(  t\right)  \leq2Lp_{1}\left(  t\right)  B\left(
\rho,\rho;x,y\right)  \text{\ \ \ }\left(  t\in\left[  0,1\right]  \right)  .
\end{align*}%
\begin{align*}
&  \frac{Hp_{1}\left(  t\right)  +Hp_{2}\left(  t\right)  +Hp_{1}\left(
1-t\right)  +Hp_{2}\left(  1-t\right)  }{2}B\left(  \rho,\rho;x,y\right) \\
&  \leq Fp_{1}\left(  t\right)  \leq2Lp_{1}\left(  t\right)  B\left(
\rho,\rho;x,y\right)  \text{\ \ \ \ \ \ \ \ \ \ \ \ \ \ \ }\left(  t\in\left[
0,1\right]  \right)  .
\end{align*}

\end{proposition}

\begin{proposition}
\label{p8}$Ip_{1}$\ is increasing on $\left[  0,1\right]  $\ and the following
inequalities hold for all $t\in\left[  0,1\right]  :$%
\begin{align*}
&  2\left[  y^{r}+\left(  y^{\prime}\right)  ^{r}\right]  B\left(  \rho
,\rho;x,\frac{1}{2}\right) \\
&  =Ip_{1}\left(  0\right)  \leq Ip_{1}\left(  t\right)  \leq Ip_{1}\left(
1\right) \\
&  =\int\nolimits_{x}^{x^{\prime}}\left(  s-s^{2}\right)  ^{\rho-1}\left[
\left(  \left(  1-\alpha\right)  x+\alpha s\right)  ^{r}+\left(  \left(
1-\alpha\right)  x^{\prime}+\alpha s\right)  ^{r}\right]  ds.
\end{align*}%
\begin{align*}
Ip_{1}\left(  t\right)   &  \leq2\left(  1-t\right)  \left[  y^{r}+\left(
y^{\prime}\right)  ^{r}\right]  B\left(  \rho,\rho;x,\frac{1}{2}\right) \\
&  +t\int\nolimits_{x}^{x^{\prime}}\left(  s-s^{2}\right)  ^{\rho-1}\left[
\left(  \left(  1-\alpha\right)  x+\alpha s\right)  ^{r}+\left(  \left(
1-\alpha\right)  x^{\prime}+\alpha s\right)  ^{r}\right]  ds\\
&  \leq\int\nolimits_{x}^{x^{\prime}}\left(  s-s^{2}\right)  ^{\rho-1}\left[
\left(  \left(  1-\alpha\right)  x+\alpha s\right)  ^{r}+\left(  \left(
1-\alpha\right)  x^{\prime}+\alpha s\right)  ^{r}\right]  ds\\
&  \leq2\left[  x^{r}+\left(  x^{\prime}\right)  ^{r}\right]  B\left(
\rho,\rho;x,\frac{1}{2}\right)  .
\end{align*}

\end{proposition}

\begin{proposition}
\label{p9}\textit{Let }$m=\frac{y^{\prime}-y}{2\left(  y^{\prime}-x\right)  }
$ \ and $m^{\prime}=\left\{
\begin{array}
[c]{cc}%
\frac{1}{2}, & y\neq y^{\prime}\\
0, & y=y^{\prime}%
\end{array}
\right.  .$ Then:
\end{proposition}

\begin{enumerate}
\item $Ip_{2}$\textit{\ is decreasing on }$\left[  0,m\right]  $\textit{\ and
increasing on }$\left[  m^{\prime},1\right]  .$

\item \textit{The following inequalities hold for all }$t\in\left[
0,1\right]  :$%
\begin{align*}
&  2^{2-r}B\left(  \rho,\rho;x,\frac{1}{2}\right) \\
&  \leq Ip_{2}\left(  t\right)  \leq2\left(  1-t\right)  \left[  y^{r}+\left(
y^{\prime}\right)  ^{r}\right]  B\left(  \rho,\rho;x,\frac{1}{2}\right) \\
&  +t\int\nolimits_{x}^{x^{\prime}}\left(  s-s^{2}\right)  ^{\rho-1}\left[
\left(  \left(  1-\alpha\right)  x+\alpha s\right)  ^{r}+\left(  \left(
1-\alpha\right)  x^{\prime}+\alpha s\right)  ^{r}\right]  ds\\
&  \leq\int\nolimits_{x}^{x^{\prime}}\left(  s-s^{2}\right)  ^{\rho-1}\left[
\left(  \left(  1-\alpha\right)  x+\alpha s\right)  ^{r}+\left(  \left(
1-\alpha\right)  x^{\prime}+\alpha s\right)  ^{r}\right]  ds\\
&  \leq2\left[  x^{r}+\left(  x^{\prime}\right)  ^{r}\right]  B\left(
\rho,\rho;x,\frac{1}{2}\right)  .
\end{align*}%
\[
Ip_{2}\left(  t\right)  \leq Ip_{1}\left(  t\right)  .
\]

\end{enumerate}

\begin{proposition}
\label{p10}$Jp$\ is increasing on $\left[  0,1\right]  $\ and the following
inequality holds for all $t\in\left[  0,1\right]  :$%
\begin{align*}
&  2\left[  \left(  \frac{x+y}{2}\right)  ^{r}+\left(  \frac{x^{\prime
}+y^{\prime}}{2}\right)  ^{r}\right]  B\left(  \rho,\rho;x,\frac{1}{2}\right)
\\
&  =Jp\left(  0\right)  \leq Jp\left(  t\right)  \leq Jp\left(  1\right) \\
&  =\int\nolimits_{x}^{x^{\prime}}\left(  s-s^{2}\right)  ^{\rho-1}\left[
\left(  \left(  1-\alpha\right)  x+\alpha s\right)  ^{r}+\left(  \left(
1-\alpha\right)  x^{\prime}+\alpha s\right)  ^{r}\right]  ds.
\end{align*}

\end{proposition}

\begin{proposition}
\label{p11}\textit{For all }$t\in\left[  0,1\right]  ,$ we have $Ip_{1}\left(
t\right)  \leq Jp\left(  t\right)  .$
\end{proposition}

\begin{proposition}
\label{p12}$Mp$\ is increasing on $\left[  0,1\right]  $\ and the following
inequalities hold for all $t\in\left[  0,1\right]  :$%
\begin{align*}
&  \int\nolimits_{x}^{x^{\prime}}\left(  s-s^{2}\right)  ^{\rho-1}\left[
\left(  \left(  1-\alpha\right)  x+\alpha s\right)  ^{r}+\left(  \left(
1-\alpha\right)  x^{\prime}+\alpha s\right)  ^{r}\right]  ds\\
&  =Mp\left(  0\right)  \leq Mp\left(  t\right)  \leq Mp\left(  1\right) \\
&  =\left[  x^{r}+y^{r}+\left(  y^{\prime}\right)  ^{r}+\left(  x^{\prime
}\right)  ^{r}\right]  B\left(  \rho,\rho;x,\frac{1}{2}\right)  .
\end{align*}%
\begin{align*}
&  Mp\left(  t\right) \\
&  \leq\left(  1-t\right)  \int\nolimits_{x}^{x^{\prime}}\left(
s-s^{2}\right)  ^{\rho-1}\left[  \left(  \left(  1-\alpha\right)  x+\alpha
s\right)  ^{r}+\left(  \left(  1-\alpha\right)  x^{\prime}+\alpha s\right)
^{r}\right]  ds\\
&  +t\left[  x^{r}+y^{r}+\left(  y^{\prime}\right)  ^{r}+\left(  x^{\prime
}\right)  ^{r}\right]  B\left(  \rho,\rho;x,\frac{1}{2}\right) \\
&  \leq\left[  x^{r}+y^{r}+\left(  y^{\prime}\right)  ^{r}+\left(  x^{\prime
}\right)  ^{r}\right]  B\left(  \rho,\rho;x,\frac{1}{2}\right) \\
&  \leq2\left[  x^{r}+\left(  x^{\prime}\right)  ^{r}\right]  B\left(
\rho,\rho;x,\frac{1}{2}\right)  .
\end{align*}

\end{proposition}

\begin{proposition}
\label{p13}$Np$\ is increasing on $\left[  0,1\right]  $\ and the following
inequalities hold for all $t\in\left[  0,1\right]  :$%
\begin{align*}
&  \int\nolimits_{x}^{x^{\prime}}\left(  s-s^{2}\right)  ^{\rho-1}\left[
\left(  \left(  1-\alpha\right)  x+\alpha s\right)  ^{r}+\left(  \left(
1-\alpha\right)  x^{\prime}+\alpha s\right)  ^{r}\right]  ds\\
&  =Np\left(  0\right)  \leq Np\left(  t\right)  \leq Np\left(  1\right)
=2\left[  x^{r}+\left(  x^{\prime}\right)  ^{r}\right]  B\left(  \rho
,\rho;x,\frac{1}{2}\right)  .
\end{align*}%
\begin{align*}
&  Np\left(  t\right) \\
&  \leq\left(  1-t\right)  \int\nolimits_{x}^{x^{\prime}}\left(
s-s^{2}\right)  ^{\rho-1}\left[  \left(  \left(  1-\alpha\right)  x+\alpha
s\right)  ^{r}+\left(  \left(  1-\alpha\right)  x^{\prime}+\alpha s\right)
^{r}\right]  ds\\
&  +2t\left[  x^{r}+\left(  x^{\prime}\right)  ^{r}\right]  B\left(  \rho
,\rho;x,\frac{1}{2}\right) \\
&  \leq2\left[  x^{r}+\left(  x^{\prime}\right)  ^{r}\right]  B\left(
\rho,\rho;x,\frac{1}{2}\right)  .
\end{align*}

\end{proposition}

\begin{proposition}
\label{p14}\textit{For all }$t\in\left[  0,1\right]  ,$ we have $Mp\left(
t\right)  \leq Np\left(  t\right)  .$
\end{proposition}

\begin{proposition}
\label{p15}The following inequalities hold:%
\begin{align*}
&  2\left[  y^{r}+\left(  y^{\prime}\right)  ^{r}\right]  B\left(  \rho
,\rho;x,\frac{1}{2}\right) \\
&  \leq\frac{2}{\alpha}\left[  \int_{\frac{x+y}{2}}^{y}s^{r}\left(  \frac
{1}{\alpha}\left(  2s-x-y\right)  +x\right)  ^{\rho-1}\left(  1-\frac
{1}{\alpha}\left(  2s-x-y\right)  -x\right)  ^{\rho-1}ds\right. \\
&  \text{ \ }+\left.  \int_{y^{\prime}}^{\frac{x^{\prime}+y^{\prime}}{2}}%
s^{r}\left(  \frac{1}{\alpha}\left(  2s-y^{\prime}-x^{\prime}\right)
+x^{\prime}\right)  ^{\rho-1}\left(  1-\frac{1}{\alpha}\left(  2s-y^{\prime
}-x^{\prime}\right)  -x^{\prime}\right)  ^{\rho-1}ds\right] \\
&  \leq\int_{0}^{1}Ip_{1}\left(  t\right)  dt\\
&  \leq\frac{1}{2}\left[  2\left[  y^{r}+\left(  y^{\prime}\right)
^{r}\right]  B\left(  \rho,\rho;x,\frac{1}{2}\right)  \right. \\
&  \text{ }+\left.  \int\nolimits_{x}^{x^{\prime}}\left(  s-s^{2}\right)
^{\rho-1}\left[  \left(  \left(  1-\alpha\right)  x+\alpha s\right)
^{r}+\left(  \left(  1-\alpha\right)  x^{\prime}+\alpha s\right)  ^{r}\right]
ds\right]  .
\end{align*}%
\begin{align*}
0  &  \leq\int\nolimits_{x}^{x^{\prime}}\left(  s-s^{2}\right)  ^{\rho
-1}\left[  \left(  \left(  1-\alpha\right)  x+\alpha s\right)  ^{r}+\left(
\left(  1-\alpha\right)  x^{\prime}+\alpha s\right)  ^{r}\right]
ds-Ip_{1}\left(  t\right) \\
&  \leq\frac{1-t}{\alpha}\left[  \left[  x^{r}+\left(  x^{\prime}\right)
^{r}\right]  \left(  y-x\right)  -\frac{\left(  x^{\prime}\right)
^{r+1}+y^{r+1}-x^{r+1}-\left(  y^{\prime}\right)  ^{r+1}}{r+1}\right] \\
&  \text{ \ \ \ \ \ \ }\times\max\left\{  \left(  x-x^{2}\right)  ^{\rho
-1},\left(  \frac{1}{4}\right)  ^{\rho-1}\right\}  \text{ \ \ }\left(
t\in\left[  0,1\right]  \right)  .
\end{align*}%
\begin{align*}
0  &  \leq2\left[  x^{r}+\left(  x^{\prime}\right)  ^{r}\right]  B\left(
\rho,\rho;x,\frac{1}{2}\right)  -Ip_{1}\left(  t\right) \\
&  \leq2r\left(  y-x\right)  \left(  \left(  x^{\prime}\right)  ^{r-1}%
-x^{r-1}\right)  B\left(  \rho,\rho;x,\frac{1}{2}\right)  \text{ \ \ }\left(
t\in\left[  0,1\right]  \right)  .
\end{align*}%
\begin{align*}
0  &  \leq Ip_{1}\left(  t\right)  -2\left[  y^{r}+\left(  y^{\prime}\right)
^{r}\right]  B\left(  \rho,\rho;x,\frac{1}{2}\right) \\
&  \leq2r\left(  y-x\right)  \left(  \left(  x^{\prime}\right)  ^{r-1}%
-x^{r-1}\right)  B\left(  \rho,\rho;x,\frac{1}{2}\right)  \text{ \ \ }\left(
t\in\left[  0,1\right]  \right)  .
\end{align*}

\end{proposition}

\begin{proposition}
\label{p16}The following inequalities hold:%
\begin{align*}
&  2^{2-r}B\left(  \rho,\rho;x,\frac{1}{2}\right) \\
&  \leq\frac{2}{\alpha}\left[  \int_{\frac{x+y^{\prime}}{2}}^{\frac{1}{2}%
}s^{r}\left(  \frac{1}{\alpha}\left(  2s-x-y^{\prime}\right)  +x\right)
^{\rho-1}\left(  1-\frac{1}{\alpha}\left(  2s-x-y^{\prime}\right)  -x\right)
^{\rho-1}ds\right. \\
&  \text{ \ }+\left.  \int_{\frac{1}{2}}^{\frac{x^{\prime}+y}{2}}s^{r}\left(
\frac{1}{\alpha}\left(  2s-y-x^{\prime}\right)  +x^{\prime}\right)  ^{\rho
-1}\left(  1-\frac{1}{\alpha}\left(  2s-y-x^{\prime}\right)  -x^{\prime
}\right)  ^{\rho-1}ds\right] \\
&  \leq\int_{0}^{1}Ip_{2}\left(  t\right)  dt\\
&  \leq\frac{1}{2}\left[  2\left[  y^{r}+\left(  y^{\prime}\right)
^{r}\right]  B\left(  \rho,\rho;x,\frac{1}{2}\right)  \right. \\
&  \text{ }+\left.  \int\nolimits_{x}^{x^{\prime}}\left(  s-s^{2}\right)
^{\rho-1}\left[  \left(  \left(  1-\alpha\right)  x+\alpha s\right)
^{r}+\left(  \left(  1-\alpha\right)  x^{\prime}+\alpha s\right)  ^{r}\right]
ds\right]  .
\end{align*}%
\begin{align*}
0  &  \leq\int\nolimits_{x}^{x^{\prime}}\left(  s-s^{2}\right)  ^{\rho
-1}\left[  \left(  \left(  1-\alpha\right)  x+\alpha s\right)  ^{r}+\left(
\left(  1-\alpha\right)  x^{\prime}+\alpha s\right)  ^{r}\right]
ds-Ip_{2}\left(  t\right) \\
&  \leq\frac{1-t}{\alpha}\left[  \left[  x^{r}+\left(  x^{\prime}\right)
^{r}\right]  \left(  y^{\prime}-x\right)  -\left[  y^{r}+\left(  y^{\prime
}\right)  ^{r}\right]  \left(  y^{\prime}-y\right)  \right. \\
&  -\left.  \frac{\left(  x^{\prime}\right)  ^{r+1}+y^{r+1}-x^{r+1}-\left(
y^{\prime}\right)  ^{r+1}}{r+1}\right]  \times\max\left\{  \left(
x-x^{2}\right)  ^{\rho-1},\left(  \frac{1}{4}\right)  ^{\rho-1}\right\}
\text{\ }\left(  t\in\left[  0,1\right]  \right)  .
\end{align*}%
\begin{align*}
0  &  \leq Ip_{2}\left(  t\right)  -2^{2-r}B\left(  \rho,\rho;x,\frac{1}%
{2}\right) \\
&  \leq r\left(  2x^{\prime}-1\right)  \left(  \left(  x^{\prime}\right)
^{r-1}-x^{r-1}\right)  B\left(  \rho,\rho;x,\frac{1}{2}\right)  \text{
\ \ }\left(  t\in\left[  0,1\right]  \right)  .
\end{align*}%
\begin{align*}
0  &  \leq2\left[  x^{r}+\left(  x^{\prime}\right)  ^{r}\right]  B\left(
\rho,\rho;x,\frac{1}{2}\right)  -Ip_{2}\left(  t\right) \\
&  \leq r\left(  2x^{\prime}-1\right)  \left(  \left(  x^{\prime}\right)
^{r-1}-x^{r-1}\right)  B\left(  \rho,\rho;x,\frac{1}{2}\right)  \text{
\ \ }\left(  t\in\left[  0,1\right]  \right)  .
\end{align*}

\end{proposition}

\begin{proposition}
\label{p17}The following inequalities hold for all $t\in\left[  0,1\right]  :$%
\[
Ip_{1}\left(  t\right)  \leq4G_{1}\left(  t\right)  B\left(  \rho,\rho
;x,\frac{1}{2}\right)  .
\]%
\begin{align*}
0  &  \leq Ip_{1}\left(  t\right)  -2\left[  y^{r}+\left(  y^{\prime}\right)
^{r}\right]  B\left(  \rho,\rho;x,\frac{1}{2}\right) \\
&  \leq2\left(  x^{\prime}-x\right)  \left[  G_{1}\left(  t\right)
-H_{1}\left(  t\right)  \right]  \max\left\{  \left(  x-x^{2}\right)
^{\rho-1},\left(  \frac{1}{4}\right)  ^{\rho-1}\right\}  .
\end{align*}

\end{proposition}

\begin{proposition}
\label{p18}Let $m,$ $m^{\prime}$\ \textit{be defined as in Proposition
\ref{p9}. Then }the following inequalities hold:%
\[
Ip_{2}\left(  t\right)  \leq\left(  \geq\right)  4G_{2}\left(  t\right)
B\left(  \rho,\rho;x,\frac{1}{2}\right)  \text{ \ \ \textit{as }}t\in\left[
m^{\prime},1\right]  \text{ \ }\left(  t\in\left[  0,m\right]  \right)  .
\]%
\begin{align*}
0  &  \leq Ip_{2}\left(  t\right)  -2^{2-r}B\left(  \rho,\rho;x,\frac{1}%
{2}\right) \\
&  \leq\frac{2}{\alpha}\left[  \left(  y^{\prime}-x\right)  G_{2}\left(
t\right)  -\left(  y^{\prime}-y\right)  G_{3}\left(  t\right)  -\left(
y-x\right)  H_{2}\left(  t\right)  \right] \\
&  \text{ \ \ \ \ \ \ \ \ \ }\times\max\left\{  \left(  x-x^{2}\right)
^{\rho-1},\left(  \frac{1}{4}\right)  ^{\rho-1}\right\}  \text{ \ \ \ \ }%
\left(  t\in\left[  m^{\prime},1\right]  \right)  .
\end{align*}

\end{proposition}

\begin{proposition}
\label{p19}Let $m^{\prime}$\ \textit{be defined as in Proposition \ref{p9}.
Then} we have:%
\begin{align*}
2  &  \left[  G_{1}\left(  t\right)  +G_{2}\left(  t\right)  \right]  B\left(
\rho,\rho;x,\frac{1}{2}\right) \\
&  \leq Sp\left(  t\right) \\
&  \leq\left(  1-t\right)  \int\nolimits_{x}^{x^{\prime}}\left(
s-s^{2}\right)  ^{\rho-1}\left[  \left(  \left(  1-\alpha\right)  x+\alpha
s\right)  ^{r}+\left(  \left(  1-\alpha\right)  x^{\prime}+\alpha s\right)
^{r}\right]  ds\\
&  \qquad\qquad+2t\left[  x^{r}+\left(  x^{\prime}\right)  ^{r}\right]
B\left(  \rho,\rho;x,\frac{1}{2}\right) \\
&  \leq2\left[  x^{r}+\left(  x^{\prime}\right)  ^{r}\right]  B\left(
\rho,\rho;x,\frac{1}{2}\right)  \text{ \ }\left(  t\in\left[  0,1\right]
\right)  .
\end{align*}%
\[
\frac{Ip_{1}\left(  1-t\right)  +Ip_{2}\left(  1-t\right)  }{2}\leq Sp\left(
t\right)  \text{ \ }\left(  t\in\left[  0,1\right]  \right)  .
\]%
\[
\frac{Ip_{1}\left(  t\right)  +Ip_{2}\left(  t\right)  +Ip_{1}\left(
1-t\right)  +Ip_{2}\left(  1-t\right)  }{4}\leq Sp\left(  t\right)  \text{
\ \ }\left(  t\in\left[  m^{\prime},1\right]  \right)  .
\]%
\[
\sup\limits_{t\in\left[  0,1\right]  }Sp\left(  t\right)  =2\left[
x^{r}+\left(  x^{\prime}\right)  ^{r}\right]  B\left(  \rho,\rho;x,\frac{1}%
{2}\right)  .
\]

\end{proposition}

\begin{proposition}
\label{p20}
\end{proposition}

\begin{enumerate}
\item $Q_{1}$\textit{\ is symmetric about }$\frac{1}{2},$\textit{\ decreasing
on }$\left[  0,\frac{1}{2}\right]  $\textit{\ and increasing on }$\left[
\frac{1}{2},1\right]  .$

\item \textit{The following inequalities hold:}%
\[
G_{1}\left(  \frac{t}{\alpha}\right)  \leq Q_{1}\left(  t\right)  \text{
\quad}\left(  t\in\left[  0,\frac{\alpha}{2}\right]  \right)  .
\]%
\[
G_{1}\left(  \frac{t}{\alpha}\right)  \geq Q_{1}\left(  t\right)  \quad\left(
t\in\left[  \frac{\alpha}{2},\alpha\right]  \right)  .
\]%
\[
G_{1}\left(  \frac{1-t}{\alpha}\right)  \geq Q_{1}\left(  t\right)  \text{
\quad}\left(  t\in\left[  1-\alpha,1-\frac{\alpha}{2}\right]  \right)  .
\]%
\[
G_{1}\left(  \frac{1-t}{\alpha}\right)  \leq Q_{1}\left(  t\right)
\quad\left(  t\in\left[  1-\frac{\alpha}{2},1\right]  \right)  .
\]

\end{enumerate}

\begin{proposition}
\label{p21}Let $m^{\prime}$\ \textit{be defined as in Proposition \ref{p9}.
Then} we have:%
\begin{align*}
0  &  \leq Np\left(  t\right)  -4G_{1}\left(  t\right)  B\left(  \rho
,\rho;x,\frac{1}{2}\right) \\
&  \leq2\left[  x^{r}+\left(  x^{\prime}\right)  ^{r}\right]  B\left(
\rho,\rho;x,\frac{1}{2}\right)  -Np\left(  t\right)  \text{ \ }\left(
t\in\left[  0,1\right]  \right)  .
\end{align*}%
\begin{align*}
0  &  \leq Lp_{1}\left(  t\right)  -\frac{Hp_{1}\left(  t\right)
+Hp_{2}\left(  t\right)  }{2}\\
&  \leq\frac{r\left(  4x^{\prime}-y-3y^{\prime}\right)  }{4}\left(  \left(
x^{\prime}\right)  ^{r-1}-x^{r-1}\right)  B\left(  \rho,\rho;x,y\right)
\text{ \ }\left(  t\in\left[  0,1\right]  \right)  .
\end{align*}%
\begin{align*}
0  &  \leq Pp_{1}-Lp_{1}\left(  t\right) \\
&  \leq\frac{r\left(  4x^{\prime}-y-3y^{\prime}\right)  }{4}\left(  \left(
x^{\prime}\right)  ^{r-1}-x^{r-1}\right)  B\left(  \rho,\rho;x,y\right)
\text{ \ }\left(  t\in\left[  0,1\right]  \right)  .
\end{align*}%
\begin{align*}
0  &  \leq Np\left(  t\right)  -Ip_{1}\left(  t\right) \\
&  \leq2r\left(  y-x\right)  \left(  \left(  x^{\prime}\right)  ^{r-1}%
-x^{r-1}\right)  B\left(  \rho,\rho;x,\frac{1}{2}\right)  \text{ \ }\left(
t\in\left[  0,1\right]  \right)  .
\end{align*}%
\begin{align*}
0  &  \leq Sp\left(  t\right)  -\frac{Ip_{1}\left(  t\right)  +Ip_{2}\left(
t\right)  }{2}\\
&  \leq\frac{r\left(  4x^{\prime}-y-3y^{\prime}\right)  }{2}\left(  \left(
x^{\prime}\right)  ^{r-1}-x^{r-1}\right)  B\left(  \rho,\rho;x,\frac{1}%
{2}\right)  \text{ \ }\left(  t\in\left[  m^{\prime},1\right]  \right)  .
\end{align*}

\end{proposition}

\begin{proposition}
\label{p22}The following inequalities hold:%
\begin{align*}
Hp_{1}\left(  t\right)   &  \leq2Q_{1}\left(  t\right)  B\left(  \rho
,\rho;x,y\right) \\
&  \leq\left[  x^{r}+\left(  x^{\prime}\right)  ^{r}\right]  \text{ }B\left(
\rho,\rho;x,y\right)  \qquad\left(  t\in\left[  0,\frac{\alpha}{1+\alpha
}\right]  \right)  .
\end{align*}%
\begin{align*}
2^{1-r}B\left(  \rho,\rho;x,y\right)   &  \leq2Q_{1}\left(  t\right)  B\left(
\rho,\rho;x,y\right) \\
&  \leq Pp_{1}\left(  t\right)  \text{ \qquad}\left(  t\in\left[  \frac
{\alpha}{1+\alpha},1\right]  \right)  .
\end{align*}%
\begin{align*}
0  &  \leq Sp\left(  t\right)  -2\left(  G_{1}\left(  t\right)  +G_{2}\left(
t\right)  \right)  B\left(  \rho,\rho;x,\frac{1}{2}\right) \\
&  \leq\left[  x^{r}+\left(  x^{\prime}\right)  ^{r}+2Q_{1}\left(  t\right)
\right]  B\left(  \rho,\rho;x,\frac{1}{2}\right)  -Sp\left(  t\right)  \text{
\ }\left(  t\in\left[  0,1\right]  \right)  .
\end{align*}

\end{proposition}

\begin{proposition}
\label{p23}$Kp$\textit{\ is symmetric about }$\frac{1}{2},$%
\textit{\ decreasing on }$\left[  0,\frac{1}{2}\right]  $\textit{\ and
increasing on }$\left[  \frac{1}{2},1\right]  .$ The following identities and
inequalities hold:
\end{proposition}

\begin{align*}
&  \sup\limits_{t\in\left[  0,1\right]  }Kp\left(  t\right)  =Kp\left(
0\right)  =Kp\left(  1\right) \\
&  =%
{\displaystyle\int\nolimits_{x}^{x^{\prime}}}
4\left(  s-s^{2}\right)  ^{\rho-1}\left[  \left(  \left(  1-\alpha\right)
x+\alpha s\right)  ^{r}+\left(  \left(  1-\alpha\right)  x^{\prime}+\alpha
s\right)  ^{r}\right]  ds\ B\left(  \rho,\rho;x,\frac{1}{2}\right)  .
\end{align*}%
\begin{align*}
\inf\limits_{t\in\left[  0,1\right]  }Kp\left(  t\right)   &  =Kp\left(
\frac{1}{2}\right) \\
&  =%
{\displaystyle\int\nolimits_{x}^{x^{\prime}}}
{\displaystyle\int\nolimits_{x}^{x^{\prime}}}
\left(  s-s^{2}\right)  ^{\rho-1}\left(  u-u^{2}\right)  ^{\rho-1}\left[
\left(  \left(  1-\alpha\right)  x+\alpha\frac{s+u}{2}\right)  ^{r}\right. \\
&  +2\left(  \frac{1-\alpha}{2}+\alpha\frac{s+u}{2}\right)  ^{r}\\
&  +\left.  \left(  \left(  1-\alpha\right)  x^{\prime}+\alpha\frac{s+u}%
{2}\right)  ^{r}\right]  dsdu.
\end{align*}%
\[
2\left[  Ip_{1}\left(  t\right)  +Ip_{2}\left(  t\right)  \right]  B\left(
\rho,\rho;x,\frac{1}{2}\right)  \leq Kp\left(  t\right)  \text{ \ }\left(
t\in\left[  0,1\right]  \right)  .
\]%
\[
4\left[  y^{r}+\left(  \frac{1}{2}\right)  ^{1-r}+\left(  y^{\prime}\right)
^{r}\right]  B^{2}\left(  \rho,\rho;x,\frac{1}{2}\right)  \leq Kp\left(
\frac{1}{2}\right)  .
\]

\begin{proposition}
\label{p24}The following inequality holds for all $t\in\left[  0,1\right]  :$%
\begin{align*}
0  &  \leq Kp\left(  t\right)  -2\left[  Ip_{1}\left(  t\right)
+Ip_{2}\left(  t\right)  \right]  B\left(  \rho,\rho;x,\frac{1}{2}\right) \\
&  \leq4Sp\left(  1-t\right)  B\left(  \rho,\rho;x,\frac{1}{2}\right)
-Kp\left(  t\right)  .
\end{align*}

\end{proposition}

\subsection{Applications for the Extended Special Means}

Throughout this subsection, let $0<\alpha\leq\frac{1}{2},0<a\leq x<y\leq
y^{\prime}<x^{\prime}\leq b,$ $x+x^{\prime}=y+y^{\prime},\ \Omega=\left[
x,y\right]  \cup\left[  y^{\prime},x^{\prime}\right]  ,$ $y=\left(
1-\alpha\right)  x+\alpha x^{\prime}$ and let $p:\Omega\rightarrow\left(
0,\infty\right)  $ be integrable and symmetic to $\frac{x+x^{\prime}}{2}.$

\bigskip Let us recall the following special means of the two real numbers $u
$ and $v:$

\begin{enumerate}
\item \textit{The Arithmetic mean}%
\[
A\left(  u,v\right)  :=\frac{u+v}{2}.
\]

\item The Harmonic mean%
\[
H\left(  u,v\right)  :=\frac{2uv}{u+v}\text{ \ }\left(  u,v>0\right)  .
\]

\item \textit{The Identric mean}%
\[
I\left(  u,v\right)  :=\frac{1}{e}\left(  \frac{v^{v}}{u^{u}}\right)
^{\frac{1}{v-u}}\text{ \ }\left(  0<u<v\right)  .
\]

\item \textit{The Logarithmic mean}%
\[
L\left(  u,v\right)  :=\frac{v-u}{\ln v-\ln u}\text{ \ }\left(  0<u<v\right)
.
\]

\item \textit{The }$r$\textit{-Logarithmic mean}%
\[
L_{r}\left(  u,v\right)  :=\left[  \frac{v^{r+1}-u^{r+1}}{\left(  r+1\right)
\left(  v-u\right)  }\right]  ^{\frac{1}{r}}\text{ \ }\left(  0<u<v,\text{
}r\in R\backslash\left\{  -1,0\right\}  \right)  .
\]

\end{enumerate}

It is well known that $L_{r}\left(  u,v\right)  $ is monotonically increasing
over $r\in R,$ denoting $L_{0}\left(  u,v\right)  =I\left(  u,v\right)  $ and
$L_{-1}\left(  u,v\right)  =L\left(  u,v\right)  .$

Let the function $q:\left[  u,v\right]  \rightarrow\left(  0,\infty\right)  $
be integrable. It is natural to consider the following extended means:

\begin{enumerate}
\item \textit{The extended Identric mean}%
\[
I\left(  q\left(  s\right)  ;u,v\right)  =I\left(  q;u,v\right)  :=e^{\int%
_{u}^{v}\ln sq\left(  s\right)  ds}\text{ }\left(  0<u<v\right)  .
\]

\item \textit{The extended }$r$\textit{-Logarithmic mean}%
\[
L_{r}\left(  q\left(  s\right)  ;u,v\right)  =L_{r}\left(  q;u,v\right)
:=\left[  \int_{u}^{v}s^{r}q\left(  s\right)  ds\right]  ^{\frac{1}{r}}\left(
0<u<v,\text{ }r\in R\backslash\left\{  0\right\}  \right)  .
\]

\end{enumerate}

It is clear that%
\[
I\left(  \frac{1}{v-u};u,v\right)  =I\left(  u,v\right)  ,\text{ }%
L_{-1}\left(  \frac{1}{v-u};u,v\right)  =L\left(  u,v\right)
\]
and%
\[
L_{r}\left(  \frac{1}{v-u};u,v\right)  =L_{r}\left(  u,v\right)  \left(  r\in
R\backslash\left\{  -1,0\right\}  \right)  .
\]

\subsubsection{Applications for the Extended Identric Mean}

Throughout this subsubsection, let $f\left(  s\right)  =-\ln s$ $\left(
s\in\left[  a,b\right]  \right)  .$

\begin{remark}
\label{r32}From the Section 2, we get%
\[
\int_{\Omega}f\left(  s\right)  p\left(  s\right)  ds=-\ln I\left(
p;x,y\right)  -\ln I\left(  p;y^{\prime},x^{\prime}\right)  ,
\]%
\[
\frac{1}{2\left(  y-x\right)  }\int_{\Omega}f\left(  s\right)  ds=-A\left(
\ln I\left(  x,y\right)  ,\ln I\left(  y^{\prime},x^{\prime}\right)  \right)
,
\]%
\begin{align*}
&  \frac{1}{y-x}\int_{x}^{y}\left[  f\left(  \frac{s+x}{2}\right)  +f\left(
\frac{x+2y-s}{2}\right)  \right]  ds\\
&  =\frac{1}{y-x}\int_{x}^{y}f\left(  s\right)  ds=-\ln I\left(  x,y\right)  ,
\end{align*}%
\begin{align*}
&  \frac{1}{y-x}\int_{y^{\prime}}^{x^{\prime}}\left[  f\left(  \frac
{s+x^{\prime}}{2}\right)  +f\left(  \frac{x^{\prime}+2y^{\prime}-s}{2}\right)
\right]  ds\\
&  =\frac{1}{y-x}\int_{y^{\prime}}^{x^{\prime}}f\left(  s\right)  ds=-\ln
I\left(  y^{\prime},x^{\prime}\right)  ,
\end{align*}%
\begin{align*}
&  \frac{2}{y-x}\left[  \int_{x}^{y}f\left(  \frac{s+x}{2}\right)  +f\left(
\frac{x+2y^{\prime}-s}{2}\right)  ds\right] \\
&  =\frac{2}{y-x}\left[  \int_{x}^{\frac{x+y}{2}}f\left(  s\right)
ds+\int_{\frac{x+2y^{\prime}-y}{2}}^{y^{\prime}}f\left(  s\right)  ds\right]
\\
&  =-\ln I\left(  x,\frac{y-x}{2}\right)  -\ln I\left(  \frac{x+2y^{\prime}%
-y}{2},y^{\prime}\right)  ,
\end{align*}%
\begin{align*}
&  \frac{2}{y-x}\int_{y^{\prime}}^{x^{\prime}}\left[  f\left(  \frac
{s+x^{\prime}}{2}\right)  +f\left(  \frac{x^{\prime}+2y-s}{2}\right)  \right]
ds\\
&  =\frac{2}{y-x}\left[  \int_{\frac{x^{\prime}+y^{\prime}}{2}}^{x^{\prime}%
}f\left(  s\right)  ds+\int_{y}^{\frac{x^{\prime}+2y-y^{\prime}}{2}}f\left(
s\right)  ds\right] \\
&  =-\ln I\left(  \frac{x^{\prime}+y^{\prime}}{2},x^{\prime}\right)  -\ln
I\left(  y,\frac{x^{\prime}+2y-y^{\prime}}{2}\right)
\end{align*}
and the following convex functions on $\left[  0,1\right]  $:%
\[
G_{1}\left(  t\right)  =-A\left(  \ln\left(  tx+\left(  1-t\right)  y\right)
,\ln\left(  tx^{\prime}+\left(  1-t\right)  y^{\prime}\right)  \right)  .
\]%
\[
G_{2}\left(  t\right)  =-A\left(  \ln\left(  tx+\left(  1-t\right)  y^{\prime
}\right)  ,\ln\left(  tx^{\prime}+\left(  1-t\right)  y\right)  \right)  .
\]%
\[
G_{3}\left(  t\right)  =-A\left(  \ln\left(  ty+\left(  1-t\right)  y^{\prime
}\right)  ,\ln\left(  ty^{\prime}+\left(  1-t\right)  y\right)  \right)  .
\]%
\[
Hp_{1}\left(  t\right)  =-\int_{x}^{y}\left[  \ln\left(  ts+\left(
1-t\right)  y\right)  +\ln\left(  t\left(  y+y^{\prime}-s\right)  +\left(
1-t\right)  y^{\prime}\right)  \right]  p\left(  s\right)  ds.
\]%
\[
Hp_{2}\left(  t\right)  =-\int_{x}^{y}\left[  \ln\left(  ts+\left(
1-t\right)  y^{\prime}\right)  +\ln\left(  t\left(  y+y^{\prime}-s\right)
+\left(  1-t\right)  y\right)  \right]  p\left(  s\right)  ds.
\]%
\[
H_{1}\left(  t\right)  =\frac{-1}{2\left(  y-x\right)  }\int_{x}^{y}\left[
\ln\left(  ts+\left(  1-t\right)  y\right)  +\ln\left(  t\left(  y+y^{\prime
}-s\right)  +\left(  1-t\right)  y^{\prime}\right)  \right]  ds.
\]%
\[
H_{2}\left(  t\right)  =\frac{-1}{2\left(  y-x\right)  }\int_{x}^{y}\left[
\ln\left(  ts+\left(  1-t\right)  y^{\prime}\right)  +\ln\left(  t\left(
y+y^{\prime}-s\right)  +\left(  1-t\right)  y\right)  \right]  ds.
\]%
\[
Fp_{1}\left(  t\right)  =-\int_{\Omega}\int_{\Omega}\ln\left(  ts+\left(
1-t\right)  u\right)  p\left(  s\right)  p\left(  u\right)  dsdu.
\]%
\[
F_{1}\left(  t\right)  =\frac{-1}{4\left(  y-x\right)  ^{2}}\int_{\Omega}%
\int_{\Omega}\ln\left(  ts+\left(  1-t\right)  u\right)  dsdu.
\]%
\[
Pp_{1}\left(  t\right)  =-\int_{x}^{y}\left[  \ln\left(  tx+\left(
1-t\right)  s\right)  +\ln\left(  tx^{\prime}+\left(  1-t\right)  \left(
x+x^{\prime}-s\right)  \right)  \right]  p\left(  s\right)  ds.
\]%
\[
P_{1}\left(  t\right)  =\frac{-1}{2\left(  y-x\right)  }\int_{x}^{y}\left[
\ln\left(  tx+\left(  1-t\right)  s\right)  +\ln\left(  tx^{\prime}+\left(
1-t\right)  \left(  x+x^{\prime}-s\right)  \right)  \right]  ds.
\]%
\[
Lp_{1}\left(  t\right)  =\frac{-1}{2}%
{\displaystyle\int\nolimits_{\Omega}}
\left[  \ln\left(  tx+\left(  1-t\right)  s\right)  +\ln\left(  tx^{\prime
}+\left(  1-t\right)  s\right)  \right]  p\left(  s\right)  ds.
\]%
\[
L_{1}\left(  t\right)  =\frac{-1}{4\left(  y-x\right)  }%
{\displaystyle\int\nolimits_{\Omega}}
\left[  \ln\left(  tx+\left(  1-t\right)  s\right)  +\ln\left(  tx^{\prime
}+\left(  1-t\right)  s\right)  \right]  ds.
\]%
\[
Q_{1}\left(  t\right)  =-A\left(  \ln\left(  tx+\left(  1-t\right)  x^{\prime
}\right)  ,\ln\left(  tx^{\prime}+\left(  1-t\right)  x\right)  \right)  .
\]%
\begin{align*}
Ip_{1}\left(  t\right)   &  =-%
{\displaystyle\int\nolimits_{x}^{x^{\prime}}}
\left[  \ln\left(  t\left(  \left(  1-\alpha\right)  x+\alpha s\right)
+\left(  1-t\right)  y\right)  \right. \\
&  \text{ \ \ \ \ }+\left.  \ln\left(  t\left(  \left(  1-\alpha\right)
x^{\prime}+\alpha s\right)  +\left(  1-t\right)  y^{\prime}\right)  \right]
p\left(  s\right)  ds.
\end{align*}%
\begin{align*}
Ip_{2}\left(  t\right)   &  =-%
{\displaystyle\int\nolimits_{x}^{x^{\prime}}}
\left[  \ln\left(  t\left(  \left(  1-\alpha\right)  x+\alpha s\right)
+\left(  1-t\right)  y^{\prime}\right)  \right. \\
&  \text{ \ \ \ \ \ }+\left.  \ln\left(  t\left(  \left(  1-\alpha\right)
x^{\prime}+\alpha s\right)  +\left(  1-t\right)  y\right)  \right]  p\left(
s\right)  ds.
\end{align*}%
\begin{align*}
Jp\left(  t\right)   &  =-%
{\displaystyle\int\nolimits_{x}^{x^{\prime}}}
\left[  \ln\left(  t\left(  \left(  1-\alpha\right)  x+\alpha s\right)
+\left(  1-t\right)  \frac{x+y}{2}\right)  \right. \\
&  \text{\ }+\left.  \ln\left(  t\left(  \left(  1-\alpha\right)  x^{\prime
}+\alpha s\right)  +\left(  1-t\right)  \frac{x^{\prime}+y^{\prime}}%
{2}\right)  \right]  p\left(  s\right)  ds.
\end{align*}%
\begin{align*}
Mp\left(  t\right)   &  =-%
{\displaystyle\int\nolimits_{x}^{\frac{x+x^{\prime}}{2}}}
\left[  \ln\left(  tx+\left(  1-t\right)  \left(  \left(  1-\alpha\right)
x+\alpha s\right)  \right)  \right. \\
&  \text{ \ \ \ \ \ \ \ \ \ \ }+\left.  \ln\left(  ty^{\prime}+\left(
1-t\right)  \left(  \left(  1-\alpha\right)  x^{\prime}+\alpha s\right)
\right)  \right]  p\left(  s\right)  ds\\
&  +%
{\displaystyle\int\nolimits_{\frac{x+x^{\prime}}{2}}^{x^{\prime}}}
\left[  \ln\left(  ty+\left(  1-t\right)  \left(  \left(  1-\alpha\right)
x+\alpha s\right)  \right)  \right. \\
&  \text{ \ \ \ \ \ \ \ \ \ \ \ }+\left.  \ln\left(  tx^{\prime}+\left(
1-t\right)  \left(  \left(  1-\alpha\right)  x^{\prime}+\alpha s\right)
\right)  \right]  p\left(  s\right)  ds.
\end{align*}%
\begin{align*}
Np\left(  t\right)   &  =-%
{\displaystyle\int\nolimits_{x}^{x^{\prime}}}
\left[  \ln\left(  tx+\left(  1-t\right)  \left(  \left(  1-\alpha\right)
x+\alpha s\right)  \right)  \right. \\
&  \text{ \ \ \ \ \ }+\left.  \ln\left(  tx^{\prime}+\left(  1-t\right)
\left(  \left(  1-\alpha\right)  x^{\prime}+\alpha s\right)  \right)  \right]
p\left(  s\right)  ds.
\end{align*}%
\begin{align*}
Sp\left(  t\right)   &  =-%
{\displaystyle\int\nolimits_{x}^{x^{\prime}}}
\frac{1}{2}\left[  \ln\left(  tx+\left(  1-t\right)  \left(  \left(
1-\alpha\right)  x+\alpha s\right)  \right)  \right. \\
&  \text{ \ \ \ \ \ \ \ \ \ }+\ln\left(  tx+\left(  1-t\right)  \left(
\left(  1-\alpha\right)  x^{\prime}+\alpha s\right)  \right) \\
&  \text{ \ \ \ \ \ \ \ \ \ }+\ln\left(  tx^{\prime}+\left(  1-t\right)
\left(  \left(  1-\alpha\right)  x+\alpha s\right)  \right) \\
&  \text{ \ \ \ \ \ \ \ \ \ }+\left.  \ln\left(  tx^{\prime}+\left(
1-t\right)  \left(  \left(  1-\alpha\right)  x^{\prime}+\alpha s\right)
\right)  \right]  p\left(  s\right)  ds.
\end{align*}

\end{remark}

Using Theorems \ref{t1} --\ \ref{t24}, Remark \ref{r3} and Remark \ref{r32},
we have the following propositions of the extended identric mean:

\begin{proposition}
\label{p25}$Hp_{1},H_{1}$\ is increasing on $\left[  0,1\right]  $\ and the
following inequalities hold for all $t\in\left[  0,1\right]  :$
\begin{align*}
&  -A\left(  \ln y,\ln y^{\prime}\right)
{\displaystyle\int\nolimits_{\Omega}}
p\left(  s\right)  ds\\
&  =Hp_{1}\left(  0\right)  \leq Hp_{1}\left(  t\right)  \text{\ }\leq
Hp_{1}\left(  1\right) \\
&  =-\ln I\left(  p;x,y\right)  -\ln I\left(  p;y^{\prime},x^{\prime}\right)
\ .
\end{align*}%
\begin{align*}
&  -A\left(  \ln y,\ln y^{\prime}\right) \\
&  =H_{1}\left(  0\right)  \leq H_{1}\left(  t\right)  \leq H_{1}\left(
1\right) \\
&  =-A\left(  \ln I\left(  x,y\right)  ,\ln I\left(  y^{\prime},x^{\prime
}\right)  \right)
\end{align*}
for $p\left(  s\right)  =\frac{1}{2\left(  y-x\right)  }$ on $\left[
x.x^{\prime}\right]  .$%
\begin{align*}
&  Hp_{1}\left(  t\right)  \leq-t\left[  \ln I\left(  p;x,y\right)  +\ln
I\left(  p;y^{\prime},x^{\prime}\right)  \right] \\
&  \text{ \ \ \ \ \ \ \ \ \ \ \ }-\left(  1-t\right)  A\left(  \ln y,\ln
y^{\prime}\right)
{\displaystyle\int\nolimits_{\Omega}}
p\left(  s\right)  ds\\
&  \leq-\ln I\left(  p;x,y\right)  -\ln I\left(  p;y^{\prime},x^{\prime
}\right)  \ \\
&  \leq-A\left(  \ln x,\ln x^{\prime}\right)
{\displaystyle\int\nolimits_{\Omega}}
p\left(  s\right)  ds.
\end{align*}%
\begin{align*}
&  H_{1}\left(  t\right)  \leq-tA\left(  \ln I\left(  x,y\right)  ,\ln
I\left(  y^{\prime},x^{\prime}\right)  \right)  -\left(  1-t\right)  A\left(
\ln y,\ln y^{\prime}\right) \\
&  \leq-A\left(  \ln I\left(  x,y\right)  ,\ln I\left(  y^{\prime},x^{\prime
}\right)  \right)  \leq-A\left(  \ln x,\ln x^{\prime}\right)
\end{align*}
for $p\left(  s\right)  =\frac{1}{2\left(  y-x\right)  }$ on $\left[
x.x^{\prime}\right]  .$%
\begin{align*}
&  -\ln A\left(  y,y^{\prime}\right)
{\displaystyle\int\nolimits_{\Omega}}
p\left(  s\right)  ds\\
&  \leq Hp_{2}\left(  t\right) \\
&  \leq-t\left[  \ln I\left(  p;x,y\right)  +\ln I\left(  p;y^{\prime
},x^{\prime}\right)  \right] \\
&  \text{ \ \ \ }-\left(  1-t\right)  A\left(  \ln y,\ln y^{\prime}\right)
{\displaystyle\int\nolimits_{\Omega}}
p\left(  s\right)  ds\\
&  \leq-\ln I\left(  p;x,y\right)  -\ln I\left(  p;y^{\prime},x^{\prime
}\right)  .
\end{align*}%
\begin{align*}
&  -\ln A\left(  y,y^{\prime}\right)  \leq H_{2}\left(  t\right) \\
&  \leq-tA\left(  \ln I\left(  x,y\right)  ,\ln I\left(  y^{\prime},x^{\prime
}\right)  \right)  -\left(  1-t\right)  A\left(  \ln y,\ln y^{\prime}\right)
\\
&  \leq-A\left(  \ln I\left(  x,y\right)  ,\ln I\left(  y^{\prime},x^{\prime
}\right)  \right)
\end{align*}
for $p\left(  s\right)  =\frac{1}{2\left(  y-x\right)  }$ on $\left[
x.x^{\prime}\right]  .$%
\[
Hp_{2}\left(  t\right)  \leq Hp_{1}\left(  t\right)  .
\]%
\[
H_{2}\left(  t\right)  \leq H_{1}\left(  t\right)  \text{\ \ as }p\left(
s\right)  =\frac{1}{2\left(  y-x\right)  }\text{ on }\left[  x.x^{\prime
}\right]  .
\]

\end{proposition}

\begin{proposition}
\label{p26}The following inequalities hold:%
\begin{align*}
&  -A\left(  \ln y,\ln y^{\prime}\right)
{\displaystyle\int\nolimits_{\Omega}}
p\left(  s\right)  ds\\
&  \leq-\ln I\left(  p\left(  2s-y\right)  ;\frac{x+y}{2},y\right)  ^{2}-\ln
I\left(  p\left(  2s-y^{\prime}\right)  ;y^{\prime},\frac{x^{\prime}%
+y^{\prime}}{2}\right)  ^{2}\\
&  \leq\int_{0}^{1}Hp_{1}\left(  t\right)  dt\\
&  \leq-A\left(  A\left(  \ln y,\ln y^{\prime}\right)
{\displaystyle\int\nolimits_{\Omega}}
p\left(  s\right)  ds,\ln I\left(  p;x,y\right)  +\ln I\left(  p;y^{\prime
},x^{\prime}\right)  \right)  .
\end{align*}%
\begin{align*}
&  -A\left(  \ln y,\ln y^{\prime}\right) \\
&  \leq-A\left(  \ln I\left(  \frac{x+y}{2},y\right)  ,\ln I\left(  y^{\prime
},\frac{x^{\prime}+y^{\prime}}{2}\right)  \right) \\
&  \leq\int_{0}^{1}H_{1}\left(  t\right)  dt\leq-A\left(  A\left(  \ln y,\ln
y^{\prime}\right)  ,A\left(  \ln I\left(  x,y\right)  ,\ln I\left(  y^{\prime
},x^{\prime}\right)  \right)  \right)
\end{align*}
for $p\left(  s\right)  =\frac{1}{2\left(  y-x\right)  }$ on $\left[
x.x^{\prime}\right]  .$%
\begin{align*}
-  &  \ln A\left(  y,y^{\prime}\right)
{\displaystyle\int\nolimits_{\Omega}}
p\left(  s\right)  ds\\
&  \leq-\ln I\left(  p\left(  2s-y^{\prime}\right)  ;\frac{x+y^{\prime}}%
{2},\frac{y+y^{\prime}}{2}\right)  ^{2}+\ln I\left(  p\left(  2s-y\right)
;\frac{y+y^{\prime}}{2},\frac{y+x^{\prime}}{2}\right)  ^{2}\\
&  \leq\int_{0}^{1}Hp_{2}\left(  t\right)  dt\\
&  \leq-A\left(  A\left(  \ln y,\ln y^{\prime}\right)
{\displaystyle\int\nolimits_{\Omega}}
p\left(  s\right)  ds,\ln I\left(  p;x,y\right)  +\ln I\left(  p;y^{\prime
},x^{\prime}\right)  \right)  .
\end{align*}%
\begin{align*}
&  -\ln A\left(  y,y^{\prime}\right) \\
&  \leq-A\left(  \ln I\left(  \frac{x+y^{\prime}}{2},\frac{y+y^{\prime}}%
{2}\right)  ,\ln I\left(  \frac{y+y^{\prime}}{2},\frac{y+x^{\prime}}%
{2}\right)  \right) \\
&  \leq\int_{0}^{1}H_{2}\left(  t\right)  dt\leq-A\left(  A\left(  \ln y,\ln
y^{\prime}\right)  ,A\left(  \ln I\left(  x,y\right)  ,\ln I\left(  y^{\prime
},x^{\prime}\right)  \right)  \right)
\end{align*}
for $p\left(  s\right)  =\frac{1}{2\left(  y-x\right)  }$ on $\left[
x.x^{\prime}\right]  .$%
\begin{align*}
0  &  \leq-\ln I\left(  p;x,y\right)  -\ln I\left(  p;y^{\prime},x^{\prime
}\right)  -Hp_{1}\left(  t\right) \\
&  \leq2\left(  y-x\right)  \left(  1-t\right)  \left[  A\left(  \ln I\left(
x,y\right)  ,\ln I\left(  y^{\prime},x^{\prime}\right)  \right)  -A\left(  \ln
x,\ln x^{\prime}\right)  \right]  \sup\limits_{s\in\Omega}p\left(  s\right)
\end{align*}
for all $t\in\left[  0,1\right]  .$%
\begin{align*}
0  &  \leq-A\left(  \ln I\left(  x,y\right)  ,\ln I\left(  y^{\prime
},x^{\prime}\right)  \right)  -H_{1}\left(  t\right) \\
&  \leq\left(  1-t\right)  \left[  A\left(  \ln I\left(  x,y\right)  ,\ln
I\left(  y^{\prime},x^{\prime}\right)  \right)  -A\left(  \ln x,\ln x^{\prime
}\right)  \right]
\end{align*}
\ for $t\in\left[  0,1\right]  $ and $p\left(  s\right)  =\frac{1}{2\left(
y-x\right)  }$ on $\left[  x.x^{\prime}\right]  .$%
\begin{align*}
0  &  \leq-A\left(  \ln x,\ln x^{\prime}\right)  \int_{\Omega}p\left(
s\right)  ds-Hp_{1}\left(  t\right) \\
&  \leq\frac{\left(  y-x\right)  \left(  x^{\prime}-x\right)  }{2xx^{\prime}%
}\int_{\Omega}p\left(  s\right)  ds\text{ \ }\left(  t\in\left[  0,1\right]
\right)  .
\end{align*}%
\[
0\leq-A\left(  \ln x,\ln x^{\prime}\right)  -H_{1}\left(  t\right)  \leq
\frac{\left(  y-x\right)  \left(  x^{\prime}-x\right)  }{2xx^{\prime}}%
\]
for $t\in\left[  0,1\right]  $ and $p\left(  s\right)  =\frac{1}{2\left(
y-x\right)  }$ on $\left[  x.x^{\prime}\right]  .$%
\begin{align*}
0  &  \leq Hp_{1}\left(  t\right)  +A\left(  \ln y,\ln y^{\prime}\right)
\int_{\Omega}p\left(  s\right)  ds\\
&  \leq\frac{\left(  y-x\right)  \left(  x^{\prime}-x\right)  }{2xx^{\prime}%
}\int_{\Omega}p\left(  s\right)  ds\text{ \ }\left(  t\in\left[  0,1\right]
\right)  .
\end{align*}%
\[
0\leq H_{1}\left(  t\right)  +A\left(  \ln y,\ln y^{\prime}\right)  \leq
\frac{\left(  y-x\right)  \left(  x^{\prime}-x\right)  }{2xx^{\prime}}%
\]
for $t\in\left[  0,1\right]  $ and $p\left(  s\right)  =\frac{1}{2\left(
y-x\right)  }$ on $\left[  x.x^{\prime}\right]  .$
\end{proposition}

\begin{proposition}
\label{p27}$Pp_{1}$\ is increasing on $\left[  0,1\right]  $\ and the
following inequalities hold for all $t\in\left[  0,1\right]  :$%
\begin{align*}
&  -\ln I\left(  p;x,y\right)  -\ln I\left(  p;y^{\prime},x^{\prime}\right) \\
&  =Pp_{1}\left(  0\right)  \leq Pp_{1}\left(  t\right)  \leq Pp_{1}\left(
1\right)  =-A\left(  \ln x,\ln x^{\prime}\right)  \int_{\Omega}p\left(
s\right)  ds.
\end{align*}%
\begin{align*}
&  -\ln I\left(  x,y\right)  -\ln I\left(  y^{\prime},x^{\prime}\right)
=P_{1}\left(  0\right)  \leq P_{1}\left(  t\right) \\
&  \leq P_{1}\left(  1\right)  =-A\left(  \ln x,\ln x^{\prime}\right)
\end{align*}
for $p\left(  s\right)  =\frac{1}{2\left(  y-x\right)  }$ on $\left[
x.x^{\prime}\right]  .$%
\begin{align*}
Pp_{1}\left(  t\right)   &  \leq-\left(  1-t\right)  \left(  \ln I\left(
p;x,y\right)  +\ln I\left(  p;y^{\prime},x^{\prime}\right)  \right) \\
&  \text{ \ \ }-tA\left(  \ln x,\ln x^{\prime}\right)  \int_{\Omega}p\left(
s\right)  ds\leq-A\left(  \ln x,\ln x^{\prime}\right)  \int_{\Omega}p\left(
s\right)  ds.
\end{align*}%
\begin{align*}
P_{1}\left(  t\right)   &  \leq-\left(  1-t\right)  \left(  \ln I\left(
x,y\right)  +\ln I\left(  y^{\prime},x^{\prime}\right)  \right)  -tA\left(
\ln x,\ln x^{\prime}\right) \\
&  \leq-A\left(  \ln x,\ln x^{\prime}\right)
\end{align*}
for $p\left(  s\right)  =\frac{1}{2\left(  y-x\right)  }$ on $\left[
x.x^{\prime}\right]  .$%
\[
Hp_{1}\left(  t\right)  \leq Pp_{1}\left(  t\right)  .\text{ }%
\]%
\[
H_{1}\left(  t\right)  \leq P_{1}\left(  t\right)  \text{\ \ as }p\left(
s\right)  =\frac{1}{2\left(  y-x\right)  }\text{ on }\left[  x.x^{\prime
}\right]  .\text{ }%
\]

\end{proposition}

\begin{proposition}
\label{p28}The following inequalities hold:%
\begin{align*}
&  -\ln I\left(  p;x,y\right)  -\ln I\left(  p;y^{\prime},x^{\prime}\right) \\
&  \leq-\ln I\left(  p\left(  2s-x\right)  ;x,\frac{x+y}{2}\right)  ^{2}-\ln
I\left(  p\left(  2s-x^{\prime}\right)  ;\frac{x^{\prime}+y^{\prime}}%
{2},x^{\prime}\right)  ^{2}\\
&  \leq\int_{0}^{1}Pp_{1}\left(  t\right)  dt\\
&  \leq-A\left(  A\left(  \ln x,\ln x^{\prime}\right)
{\displaystyle\int\nolimits_{\Omega}}
p\left(  s\right)  ds,\ln I\left(  p;x,y\right)  +\ln I\left(  p;y^{\prime
},x^{\prime}\right)  \right)  .
\end{align*}%
\begin{align*}
&  -A\left(  \ln I\left(  x,y\right)  +\ln I\left(  y^{\prime},x^{\prime
}\right)  \right) \\
&  \leq-A\left(  \ln I\left(  x,\frac{x+y}{2}\right)  ,\ln I\left(
\frac{x^{\prime}+y^{\prime}}{2},x^{\prime}\right)  \right) \\
&  \leq\int_{0}^{1}P_{1}\left(  t\right)  dt\leq-A\left(  A\left(  \ln x,\ln
x^{\prime}\right)  ,A\left(  \ln I\left(  x,y\right)  ,\ln I\left(  y^{\prime
},x^{\prime}\right)  \right)  \right)
\end{align*}
for $p\left(  s\right)  =\frac{1}{2\left(  y-x\right)  }$ on $\left[
x.x^{\prime}\right]  .$%
\begin{align*}
0  &  \leq2\left(  y-x\right)  t\left[  A\left(  \ln y,\ln y^{\prime}\right)
-A\left(  \ln I\left(  x,y\right)  ,\ln I\left(  y^{\prime},x^{\prime}\right)
\right)  \right]  \cdot\inf\limits_{s\in\Omega}p\left(  s\right) \\
&  \leq Pp_{1}\left(  t\right)  +\ln I\left(  p;x,y\right)  +\ln I\left(
p;y^{\prime},x^{\prime}\right)  \text{\ \ \ \ \ }\left(  t\in\left[
0,1\right]  \right)  .
\end{align*}%
\begin{align*}
0  &  \leq t\left[  A\left(  \ln y,\ln y^{\prime}\right)  -A\left(  \ln
I\left(  x,y\right)  ,\ln I\left(  y^{\prime},x^{\prime}\right)  \right)
\right] \\
&  \leq P_{1}\left(  t\right)  +A\left(  \ln I\left(  x,y\right)  ,\ln
I\left(  y^{\prime},x^{\prime}\right)  \right)
\end{align*}
for $t\in\left[  0,1\right]  $ and $p\left(  s\right)  =\frac{1}{2\left(
y-x\right)  }$ on $\left[  x.x^{\prime}\right]  .$%
\begin{align*}
0  &  \leq Pp_{1}\left(  t\right)  +A\left(  \ln y,\ln y^{\prime}\right)
\int_{\Omega}p\left(  s\right)  ds\\
&  \leq\frac{\left(  y-x\right)  \left(  x^{\prime}-x\right)  }{2xx^{\prime}%
}\int_{\Omega}p\left(  s\right)  ds\text{ \ }\left(  t\in\left[  0,1\right]
\right)  .
\end{align*}%
\[
0\leq P_{1}\left(  t\right)  +A\left(  \ln y,\ln y^{\prime}\right)  \leq
\frac{\left(  y-x\right)  \left(  x^{\prime}-x\right)  }{2xx^{\prime}}%
\]
for $t\in\left[  0,1\right]  $ and $p\left(  s\right)  =\frac{1}{2\left(
y-x\right)  }$ on $\left[  x.x^{\prime}\right]  .$%
\begin{align*}
0  &  \leq-A\left(  \ln x,\ln x^{\prime}\right)  \int_{\Omega}p\left(
s\right)  ds-Pp_{1}\left(  t\right) \\
&  \leq\frac{\left(  y-x\right)  \left(  x^{\prime}-x\right)  }{2xx^{\prime}%
}\int_{\Omega}p\left(  s\right)  ds\text{ \ }\left(  t\in\left[  0,1\right]
\right)  .
\end{align*}%
\[
0\leq-A\left(  \ln x,\ln x^{\prime}\right)  -P_{1}\left(  t\right)  \leq
\frac{\left(  y-x\right)  \left(  x^{\prime}-x\right)  }{2xx^{\prime}}%
\]
for $t\in\left[  0,1\right]  $ and $p\left(  s\right)  =\frac{1}{2\left(
y-x\right)  }$ on $\left[  x.x^{\prime}\right]  .$%
\[
0\leq Pp_{1}\left(  t\right)  -Hp_{1}\left(  t\right)  \leq\frac{\left(
y-x\right)  \left(  x^{\prime}-x\right)  }{2xx^{\prime}}\int_{\Omega}p\left(
s\right)  ds\text{ \ }\left(  t\in\left[  0,1\right]  \right)  .
\]%
\[
0\leq P_{1}\left(  t\right)  -H_{1}\left(  t\right)  \leq\frac{\left(
y-x\right)  \left(  x^{\prime}-x\right)  }{2xx^{\prime}}%
\]
for $t\in\left[  0,1\right]  $ and $p\left(  s\right)  =\frac{1}{2\left(
y-x\right)  }$ on $\left[  x.x^{\prime}\right]  .$
\end{proposition}

\begin{proposition}
\label{p29}$Fp_{1}$\textit{\ is symmetric about }$\frac{1}{2},$%
\textit{\ decreasing on }$\left[  0,\frac{1}{2}\right]  $\textit{\ and
increasing on }$\left[  \frac{1}{2},1\right]  .$ \textit{The following
identities and inequalities hold:}%
\[
\sup\limits_{t\in\left[  0,1\right]  }Fp_{1}\left(  t\right)  =-\left[  \ln
I\left(  p;x,y\right)  +\ln I\left(  p;y^{\prime},x^{\prime}\right)  \right]
\int_{\Omega}p\left(  s\right)  ds.
\]%
\[
\sup\limits_{t\in\left[  0,1\right]  }F_{1}\left(  t\right)  =-A\left(  \ln
I\left(  x,y\right)  ,\ln I\left(  y^{\prime},x^{\prime}\right)  \right)
\]
for $p\left(  s\right)  =\frac{1}{2\left(  y-x\right)  }$ on $\left[
x.x^{\prime}\right]  .$%
\[
\inf\limits_{t\in\left[  0,1\right]  }Fp_{1}\left(  t\right)  =-\int_{\Omega
}\int_{\Omega}\ln A\left(  s,u\right)  p\left(  s\right)  p\left(  u\right)
dsdu.
\]%
\[
\inf\limits_{t\in\left[  0,1\right]  }F_{1}\left(  t\right)  =\frac
{-1}{4\left(  y-x\right)  ^{2}}\int_{\Omega}\int_{\Omega}\ln A\left(
s,u\right)  dsdu
\]
for $p\left(  s\right)  =\frac{1}{2\left(  y-x\right)  }$ on $\left[
x.x^{\prime}\right]  .$%
\[
A\left(  Hp_{1}\left(  t\right)  ,Hp_{2}\left(  t\right)  \right)
\int_{\Omega}p\left(  s\right)  ds\leq Fp_{1}\left(  t\right)  \text{\ \ \ }%
\left(  t\in\left[  0,1\right]  \right)  .
\]%
\[
A\left(  H_{1}\left(  t\right)  ,H_{2}\left(  t\right)  \right)  \leq
F_{1}\left(  t\right)
\]
for $t\in\left[  0,1\right]  $ and $p\left(  s\right)  =\frac{1}{2\left(
y-x\right)  }$ on $\left[  x.x^{\prime}\right]  .$%
\begin{align*}
&  -A\left(  A\left(  \ln y,\ln y\right)  ,\ln A\left(  y,y^{\prime}\right)
\right)  \left(  \int_{\Omega}p\left(  s\right)  ds\right)  ^{2}\\
&  \leq-\int_{\Omega}\int_{\Omega}\ln A\left(  s,u\right)  p\left(  s\right)
p\left(  u\right)  dsdu.
\end{align*}%
\[
-A\left(  A\left(  \ln y,\ln y\right)  ,\ln A\left(  y,y^{\prime}\right)
\right)  \leq\frac{-1}{4\left(  y-x\right)  ^{2}}\int_{\Omega}\int_{\Omega}\ln
A\left(  s,u\right)  dsdu
\]
for and $p\left(  s\right)  =\frac{1}{2\left(  y-x\right)  }$ on $\left[
x.x^{\prime}\right]  .$
\end{proposition}

\begin{proposition}
\label{p30}The following inequalities hold:%
\[
Hp_{1}\left(  t\right)  \leq G_{1}\left(  t\right)  \int_{\Omega}p\left(
s\right)  ds\leq Pp_{1}\left(  t\right)  \text{\ }\left(  t\in\left[
0,1\right]  \right)  .
\]%
\[
H_{1}\left(  t\right)  \leq G_{1}\left(  t\right)  \leq P_{1}\left(  t\right)
\]
for $t\in\left[  0,1\right]  $ and $p\left(  s\right)  =\frac{1}{2\left(
y-x\right)  }$ on $\left[  x.x^{\prime}\right]  .$%
\begin{align*}
&  -\ln I\left(  p\left(  2s-y\right)  ;\frac{x+y}{2},y\right)  ^{2}-\ln
I\left(  p\left(  2s-y^{\prime}\right)  ;y^{\prime},\frac{x^{\prime}%
+y^{\prime}}{2}\right)  ^{2}\\
&  \leq-A\left(  \ln A\left(  x,y\right)  ,\ln A\left(  x^{\prime},y^{\prime
}\right)  \right)  \int_{\Omega}p\left(  s\right)  ds\\
&  \leq-\int_{x}^{y}A\left(  \ln\left(  \frac{s+x}{2}\right)  ,\ln\left(
\frac{x+2y-s}{2}\right)  \right)  p\left(  s\right)  ds\\
&  -\int_{y^{\prime}}^{x^{\prime}}A\left(  \ln\left(  \frac{s+x^{\prime}}%
{2}\right)  ,\ln\left(  \frac{x^{\prime}+2y^{\prime}-s}{2}\right)  \right)
p\left(  s\right)  ds\\
&  \leq-A\left(  A\left(  \ln x,\ln y\right)  ,A\left(  \ln y^{\prime},\ln
x^{\prime}\right)  \right)  \int_{\Omega}p\left(  s\right)  ds.
\end{align*}%
\begin{align*}
&  -A\left(  \ln I\left(  \frac{x+y}{2},y\right)  ,\ln I\left(  y^{\prime
},\frac{x^{\prime}+y^{\prime}}{2}\right)  \right) \\
&  \leq-A\left(  \ln A\left(  x,y\right)  ,\ln A\left(  x^{\prime},y^{\prime
}\right)  \right) \\
&  \leq-A\left(  \ln I\left(  x,y\right)  ,\ln I\left(  y^{\prime},x^{\prime
}\right)  \right) \\
&  \leq-A\left(  A\left(  \ln x,\ln y\right)  ,A\left(  \ln y^{\prime},\ln
x^{\prime}\right)  \right)
\end{align*}
for $p\left(  s\right)  =\frac{1}{2\left(  y-x\right)  }$ on $\left[
x.x^{\prime}\right]  .$%
\begin{align*}
&  -\ln I\left(  p\left(  2s-y^{\prime}\right)  ;\frac{x+y^{\prime}}{2}%
,\frac{y+y^{\prime}}{2}\right)  ^{2}-\ln I\left(  p\left(  2s-y\right)
;\frac{y+y^{\prime}}{2},\frac{x^{\prime}+y}{2}\right)  ^{2}\\
&  \leq-A\left(  \ln A\left(  x,y^{\prime}\right)  ,\ln A\left(  x^{\prime
},y\right)  \right)  \int_{\Omega}p\left(  s\right)  ds\\
&  \leq-\int_{x}^{y}A\left(  \ln\left(  \frac{s+x}{2}\right)  ,\ln\left(
\frac{x+2y^{\prime}-s}{2}\right)  \right)  p\left(  s\right)  ds\\
&  -\int_{y^{\prime}}^{x^{\prime}}A\left(  \ln\left(  \frac{s+x^{\prime}}%
{2}\right)  ,\ln\left(  \frac{x^{\prime}+2y-s}{2}\right)  \right)  p\left(
s\right)  ds\\
&  \leq-A\left(  A\left(  \ln x,\ln y\right)  ,A\left(  \ln y^{\prime},\ln
x^{\prime}\right)  \right)  \int_{\Omega}p\left(  s\right)  ds.
\end{align*}%
\begin{align*}
&  -A\left(  \ln I\left(  \frac{x+y^{\prime}}{2},\frac{y+y^{\prime}}%
{2}\right)  ,\ln I\left(  \frac{y+y^{\prime}}{2},\frac{x^{\prime}+y}%
{2}\right)  \right) \\
&  \leq-A\left(  \ln A\left(  x,y^{\prime}\right)  ,\ln A\left(  x^{\prime
},y\right)  \right) \\
&  \leq\frac{-1}{4}\left[  \ln I\left(  x,\frac{y-x}{2}\right)  +\ln I\left(
\frac{x+2y^{\prime}-y}{2},y^{\prime}\right)  \right. \\
&  +\left.  \ln I\left(  \frac{x^{\prime}+y^{\prime}}{2},x^{\prime}\right)
+\ln I\left(  y,\frac{x^{\prime}+2y-y^{\prime}}{2}\right)  \right] \\
&  \leq-A\left(  A\left(  \ln x,\ln y\right)  ,A\left(  \ln y^{\prime},\ln
x^{\prime}\right)  \right)
\end{align*}
for $p\left(  s\right)  =\frac{1}{2\left(  y-x\right)  }$ on $\left[
x.x^{\prime}\right]  .$%
\begin{align*}
0  &  \leq Hp_{1}\left(  t\right)  +A\left(  \ln y,\ln y^{\prime}\right)
\int_{\Omega}p\left(  s\right)  ds\\
&  \leq G_{1}\left(  t\right)  \int_{\Omega}p\left(  s\right)  ds-Hp_{1}%
\left(  t\right)  \text{\ }\left(  t\in\left[  0,1\right]  \right)  .
\end{align*}
\textit{\ }%
\[
0\leq H_{1}\left(  t\right)  +A\left(  \ln y,\ln y^{\prime}\right)  \leq
G_{1}\left(  t\right)  -H_{1}\left(  t\right)
\]
for $t\in\left[  0,1\right]  $ and $p\left(  s\right)  =\frac{1}{2\left(
y-x\right)  }$ on $\left[  x.x^{\prime}\right]  .$%
\begin{align*}
0  &  \leq Pp_{1}\left(  t\right)  -G_{1}\left(  t\right)  \int_{\Omega
}p\left(  s\right)  ds\\
&  \leq-A\left(  \ln x,\ln x^{\prime}\right)  \int_{\Omega}p\left(  s\right)
ds-Pp_{1}\left(  t\right)  \text{\ }\left(  t\in\left[  0,1\right]  \right)  .
\end{align*}%
\begin{align*}
0  &  \leq P_{1}\left(  t\right)  -G_{1}\left(  t\right) \\
&  \leq-A\left(  \ln x,\ln x^{\prime}\right)  -P_{1}\left(  t\right)
\end{align*}
for $t\in\left[  0,1\right]  $ and $p\left(  s\right)  =\frac{1}{2\left(
y-x\right)  }$ on $\left[  x.x^{\prime}\right]  .$
\end{proposition}

\begin{proposition}
\label{p31}The following inequalities hold:%
\[
A\left(  G_{1}\left(  t\right)  ,G_{2}\left(  t\right)  \right)  \int_{\Omega
}p\left(  s\right)  ds\leq Lp_{1}\left(  t\right)  \leq Pp_{1}\left(
t\right)  \text{\ }\left(  t\in\left[  0,1\right]  \right)  .\text{ }%
\]%
\[
A\left(  G_{1}\left(  t\right)  ,G_{2}\left(  t\right)  \right)  \leq
L_{1}\left(  t\right)  \leq P_{1}\left(  t\right)  \text{ }%
\]
for $t\in\left[  0,1\right]  $ and $p\left(  s\right)  =\frac{1}{2\left(
y-x\right)  }$ on $\left[  x.x^{\prime}\right]  .$%
\[
\sup\limits_{t\in\left[  0,1\right]  }Lp_{1}\left(  t\right)  =Lp_{1}\left(
1\right)  =-A\left(  \ln x,\ln x^{\prime}\right)  \int_{\Omega}p\left(
s\right)  ds.
\]%
\[
\sup\limits_{t\in\left[  0,1\right]  }L_{1}\left(  t\right)  =L_{1}\left(
1\right)  =-A\left(  \ln x,\ln x^{\prime}\right)
\]
for $p\left(  s\right)  =\frac{1}{2\left(  y-x\right)  }$ on $\left[
x.x^{\prime}\right]  .$%
\begin{align*}
&  A\left(  Hp_{1}\left(  1-t\right)  ,Hp_{2}\left(  1-t\right)  \right)
\int_{\Omega}p\left(  s\right)  ds\\
&  \leq Fp_{1}\left(  t\right)  \leq Lp_{1}\left(  t\right)  \int_{\Omega
}p\left(  s\right)  ds\text{\ \ \ }\left(  t\in\left[  0,1\right]  \right)  .
\end{align*}%
\[
A\left(  H_{1}\left(  1-t\right)  ,H_{2}\left(  1-t\right)  \right)  \leq
F_{1}\left(  t\right)  \leq L_{1}\left(  t\right)
\]
for $t\in\left[  0,1\right]  $ and $p\left(  s\right)  =\frac{1}{2\left(
y-x\right)  }$ on $\left[  x.x^{\prime}\right]  .$%
\begin{align*}
&  A\left(  Hp_{1}\left(  t\right)  ,Hp_{2}\left(  t\right)  \right)
\int_{\Omega}p\left(  s\right)  ds\\
&  \leq Fp_{1}\left(  t\right)  \leq Lp_{1}\left(  t\right)  \int_{\Omega
}p\left(  s\right)  ds\text{\ \ \ }\left(  t\in\left[  0,1\right]  \right)  .
\end{align*}%
\[
A\left(  H_{1}\left(  t\right)  ,H_{2}\left(  t\right)  \right)  \leq
F_{1}\left(  t\right)  \leq L_{1}\left(  t\right)  \text{\ \ }\left(
t\in\left[  0,1\right]  \right)
\]
for $t\in\left[  0,1\right]  $ and $p\left(  s\right)  =\frac{1}{2\left(
y-x\right)  }$ on $\left[  x.x^{\prime}\right]  .$%
\begin{align*}
&  A\left(  A\left(  Hp_{1}\left(  t\right)  ,Hp_{2}\left(  t\right)  \right)
,A\left(  Hp_{1}\left(  1-t\right)  ,Hp_{2}\left(  1-t\right)  \right)
\right)  \int_{\Omega}p\left(  s\right)  ds\\
&  \leq Fp_{1}\left(  t\right)  \leq Lp_{1}\left(  t\right)  \int_{\Omega
}p\left(  s\right)  ds\text{\ \ \ }\left(  t\in\left[  0,1\right]  \right)  .
\end{align*}%
\[
A\left(  A\left(  H_{1}\left(  t\right)  ,H_{2}\left(  t\right)  \right)
,A\left(  H_{1}\left(  1-t\right)  ,H_{2}\left(  1-t\right)  \right)  \right)
\leq F_{1}\left(  t\right)  \leq Lp_{1}\left(  t\right)
\]
for $t\in\left[  0,1\right]  $ and $p\left(  s\right)  =\frac{1}{2\left(
y-x\right)  }$ on $\left[  x.x^{\prime}\right]  .$%
\begin{align*}
0  &  \leq Fp_{1}\left(  t\right)  -A\left(  Hp_{1}\left(  t\right)
,Hp_{2}\left(  t\right)  \right)  \int_{\Omega}p\left(  s\right)  ds\\
&  \leq Lp_{1}\left(  1-t\right)  \int_{\Omega}p\left(  s\right)
ds-Fp_{1}\left(  t\right)  \text{\ \ }\left(  t\in\left[  0,1\right]  \right)
\end{align*}
as $p\left(  s\right)  =p\left(  x+x^{\prime}-s\right)  $ $\left(  s\in\left[
x,y\right]  \right)  $ and $p\left(  s\right)  =p\left(  x+y-s\right)  $
$\left(  s\in\left[  x,\frac{x+y}{2}\right]  \right)  .$%
\[
0\leq F_{1}\left(  t\right)  -A\left(  H_{1}\left(  t\right)  ,H_{2}\left(
t\right)  \right)  \leq L_{1}\left(  1-t\right)  -F_{1}\left(  t\right)
\]
for $t\in\left[  0,1\right]  $ and $p\left(  s\right)  =\frac{1}{2\left(
y-x\right)  }$ on $\left[  x.x^{\prime}\right]  .$
\end{proposition}

\begin{proposition}
\label{p32}$Ip_{1}$\ is increasing on $\left[  0,1\right]  $\ and the
following inequalities hold for all $t\in\left[  0,1\right]  :$%
\begin{align*}
&  -2A\left(  \ln y,\ln y^{\prime}\right)  \int_{x}^{x^{\prime}}p\left(
s\right)  ds\\
&  =Ip_{1}\left(  0\right)  \leq Ip_{1}\left(  t\right)  \leq Ip_{1}\left(
1\right) \\
&  =-\int\nolimits_{x}^{x^{\prime}}\left[  \ln\left(  \left(  1-\alpha\right)
x+\alpha s\right)  +\ln\left(  \left(  1-\alpha\right)  x^{\prime}+\alpha
s\right)  \right]  p\left(  s\right)  ds.
\end{align*}%
\begin{align*}
Ip_{1}\left(  t\right)   &  \leq-2\left(  1-t\right)  A\left(  \ln y,\ln
y^{\prime}\right)  \int_{x}^{x^{\prime}}p\left(  s\right)  ds\\
&  -t\int\nolimits_{x}^{x^{\prime}}\left[  \ln\left(  \left(  1-\alpha\right)
x+\alpha s\right)  +\ln\left(  \left(  1-\alpha\right)  x^{\prime}+\alpha
s\right)  \right]  p\left(  s\right)  ds\\
&  \leq-\int\nolimits_{x}^{x^{\prime}}\left[  \ln\left(  \left(
1-\alpha\right)  x+\alpha s\right)  +\ln\left(  \left(  1-\alpha\right)
x^{\prime}+\alpha s\right)  \right]  p\left(  s\right)  ds\\
&  \leq-2A\left(  \ln x,\ln x^{\prime}\right)  \int_{x}^{x^{\prime}}p\left(
s\right)  ds.
\end{align*}

\end{proposition}

\begin{proposition}
\label{p33}Let $m,$ $m^{\prime}$\ \textit{be defined as in Proposition
\ref{p9}. Then}
\end{proposition}

\begin{enumerate}
\item $Ip_{2}$\textit{\ is decreasing on }$\left[  0,m\right]  $\textit{\ and
increasing on }$\left[  m^{\prime},1\right]  .$

\item \textit{The following inequalities hold for all }$t\in\left[
0,1\right]  :$%
\begin{align*}
&  -2\ln A\left(  x,x^{\prime}\right)  \int_{x}^{x^{\prime}}p\left(  s\right)
ds\\
&  \leq Ip_{2}\left(  t\right)  \leq-2\left(  1-t\right)  A\left(  \ln y,\ln
y^{\prime}\right)  \int_{x}^{x^{\prime}}p\left(  s\right)  ds\\
&  -t\int\nolimits_{x}^{x^{\prime}}\left[  \ln\left(  \left(  1-\alpha\right)
x+\alpha s\right)  +\ln\left(  \left(  1-\alpha\right)  x^{\prime}+\alpha
s\right)  \right]  p\left(  s\right)  ds\\
&  \leq-\int\nolimits_{x}^{x^{\prime}}\left[  \ln\left(  \left(
1-\alpha\right)  x+\alpha s\right)  +\ln\left(  \left(  1-\alpha\right)
x^{\prime}+\alpha s\right)  \right]  p\left(  s\right)  ds\\
&  \leq-2A\left(  \ln x,\ln x^{\prime}\right)  \int_{x}^{x^{\prime}}p\left(
s\right)  ds.
\end{align*}%
\[
Ip_{2}\left(  t\right)  \leq Ip_{1}\left(  t\right)  .
\]

\end{enumerate}

\begin{proposition}
\label{p34}$Jp$\ is increasing on $\left[  0,1\right]  $\ and the following
inequality holds for all $t\in\left[  0,1\right]  :$%
\begin{align*}
&  -2A\left(  \ln A\left(  x,y\right)  ,\ln A\left(  x^{\prime},y^{\prime
}\right)  \right)  \int_{x}^{x^{\prime}}p\left(  s\right)  ds\\
&  =Jp\left(  0\right)  \leq Jp\left(  t\right)  \leq Jp\left(  1\right) \\
&  =-\int\nolimits_{x}^{x^{\prime}}\left[  \ln\left(  \left(  1-\alpha\right)
x+\alpha s\right)  +\ln\left(  \left(  1-\alpha\right)  x^{\prime}+\alpha
s\right)  \right]  p\left(  s\right)  ds.
\end{align*}

\end{proposition}

\begin{proposition}
\label{p35}\textit{For all }$t\in\left[  0,1\right]  ,$ we have $Ip_{1}\left(
t\right)  \leq Jp\left(  t\right)  .$
\end{proposition}

\begin{proposition}
\label{p36}$Mp$\ is increasing on $\left[  0,1\right]  $\ and the following
inequalities hold for all $t\in\left[  0,1\right]  :$%
\begin{align*}
&  -\int\nolimits_{x}^{x^{\prime}}\left[  \ln\left(  \left(  1-\alpha\right)
x+\alpha s\right)  +\ln\left(  \left(  1-\alpha\right)  x^{\prime}+\alpha
s\right)  \right]  p\left(  s\right)  ds\\
&  =Mp\left(  0\right)  \leq Mp\left(  t\right)  \leq Mp\left(  1\right) \\
&  =-A\left(  A\left(  \ln x,\ln y\right)  ,A\left(  \ln y^{\prime},\ln
x^{\prime}\right)  \right)  \int\nolimits_{x}^{x^{\prime}}p\left(  s\right)
ds.
\end{align*}%
\begin{align*}
&  Mp\left(  t\right) \\
&  \leq-\left(  1-t\right)  \int\nolimits_{x}^{x^{\prime}}\left[  \ln\left(
\left(  1-\alpha\right)  x+\alpha s\right)  +\ln\left(  \left(  1-\alpha
\right)  x^{\prime}+\alpha s\right)  \right]  p\left(  s\right)  ds\\
&  -tA\left(  A\left(  \ln x,\ln y\right)  ,A\left(  \ln y^{\prime},\ln
x^{\prime}\right)  \right)  \int\nolimits_{x}^{x^{\prime}}p\left(  s\right)
ds\\
&  \leq-A\left(  A\left(  \ln x,\ln y\right)  ,A\left(  \ln y^{\prime},\ln
x^{\prime}\right)  \right)  \int\nolimits_{x}^{x^{\prime}}p\left(  s\right)
ds\\
&  \leq-2A\left(  \ln x,\ln x^{\prime}\right)  \int_{x}^{x^{\prime}}p\left(
s\right)  ds.
\end{align*}

\end{proposition}

\begin{proposition}
\label{p37}$Np$\ is increasing on $\left[  0,1\right]  $\ and the following
inequalities hold for all $t\in\left[  0,1\right]  :$%
\begin{align*}
&  -\int\nolimits_{x}^{x^{\prime}}\left[  \ln\left(  \left(  1-\alpha\right)
x+\alpha s\right)  +\ln\left(  \left(  1-\alpha\right)  x^{\prime}+\alpha
s\right)  \right]  p\left(  s\right)  ds\\
&  =Np\left(  0\right)  \leq Np\left(  t\right)  \leq Np\left(  1\right)
=-2A\left(  \ln x,\ln x^{\prime}\right)  \int_{x}^{x^{\prime}}p\left(
s\right)  ds.
\end{align*}%
\begin{align*}
&  Np\left(  t\right) \\
&  \leq-\left(  1-t\right)  \int\nolimits_{x}^{x^{\prime}}\left[  \ln\left(
\left(  1-\alpha\right)  x+\alpha s\right)  +\ln\left(  \left(  1-\alpha
\right)  x^{\prime}+\alpha s\right)  \right]  p\left(  s\right)  ds\\
&  -2tA\left(  \ln x,\ln x^{\prime}\right)  \int_{x}^{x^{\prime}}p\left(
s\right)  ds\\
&  \leq-2A\left(  \ln x,\ln x^{\prime}\right)  \int_{x}^{x^{\prime}}p\left(
s\right)  ds.
\end{align*}

\end{proposition}

\begin{proposition}
\label{p38}\textit{For all }$t\in\left[  0,1\right]  ,$ we have $Mp\left(
t\right)  \leq Np\left(  t\right)  .$
\end{proposition}

\begin{proposition}
\label{p39}The following inequalities hold:%
\begin{align*}
&  -2A\left(  \ln y,\ln y^{\prime}\right)  \int_{x}^{x^{\prime}}p\left(
s\right)  ds\\
&  \leq\frac{-2}{\alpha}\left[  \ln I\left(  p\left(  \frac{1}{\alpha}\left(
2s-x-y\right)  +x\right)  ;\frac{x+y}{2},y\right)  \right. \\
&  \text{ \ \ \ \ \ \ \ \ \ \ }+\left.  \ln I\left(  p\left(  \frac{1}{\alpha
}\left(  2s-y^{\prime}-x^{\prime}\right)  +x^{\prime}\right)  ;y^{\prime
},\frac{x^{\prime}+y^{\prime}}{2}\right)  \right] \\
&  \leq\int_{0}^{1}Ip_{1}\left(  t\right)  dt\\
&  \leq-A\left(  \ln y,\ln y^{\prime}\right)  \int_{x}^{x^{\prime}}p\left(
s\right)  ds\\
&  \text{ }-\int\nolimits_{x}^{x^{\prime}}A\left(  \ln\left(  \left(
1-\alpha\right)  x+\alpha s\right)  ,\ln\left(  \left(  1-\alpha\right)
x^{\prime}+\alpha s\right)  \right)  p\left(  s\right)  ds.
\end{align*}%
\begin{align*}
0  &  \leq-\int\nolimits_{x}^{x^{\prime}}\left[  \ln\left(  \left(
1-\alpha\right)  x+\alpha s\right)  +\ln\left(  \left(  1-\alpha\right)
x^{\prime}+\alpha s\right)  \right]  p\left(  s\right)  ds-Ip_{1}\left(
t\right) \\
&  \leq\frac{2\left(  1-t\right)  \left(  y-x\right)  }{\alpha}\left[
-A\left(  \ln x,\ln x^{\prime}\right)  +A\left(  \ln I\left(  x,y\right)  ,\ln
I\left(  y^{\prime},x^{\prime}\right)  \right)  \right]  \left\Vert
p\right\Vert _{\infty}%
\end{align*}
for $t\in\left[  0,1\right]  .$%
\begin{align*}
0  &  \leq-2A\left(  \ln x,\ln x^{\prime}\right)  \int_{x}^{x^{\prime}%
}p\left(  s\right)  ds-Ip_{1}\left(  t\right) \\
&  \leq\frac{\left(  y-x\right)  \left(  x^{\prime}-x\right)  }{xx^{\prime}%
}\int_{x}^{x^{\prime}}p\left(  s\right)  ds\text{ \ \ }\left(  t\in\left[
0,1\right]  \right)  .
\end{align*}%
\begin{align*}
0  &  \leq Ip_{1}\left(  t\right)  +2A\left(  \ln y,\ln y^{\prime}\right)
\int_{x}^{x^{\prime}}p\left(  s\right)  ds\\
&  \leq\frac{\left(  y-x\right)  \left(  x^{\prime}-x\right)  }{xx^{\prime}%
}\int_{x}^{x^{\prime}}p\left(  s\right)  ds\text{ \ \ }\left(  t\in\left[
0,1\right]  \right)  .
\end{align*}

\end{proposition}

\begin{proposition}
\label{p40}The following inequalities hold:%
\begin{align*}
&  -2\ln A\left(  y,y^{\prime}\right)  \int_{x}^{x^{\prime}}p\left(  s\right)
ds\\
&  \leq\frac{-2}{\alpha}\left[  \ln I\left(  p\left(  \frac{1}{\alpha}\left(
2s-x-y^{\prime}\right)  +x\right)  ;\frac{x+y^{\prime}}{2},\frac{y+y^{\prime}%
}{2}\right)  \right. \\
&  \text{ \ \ \ \ \ \ \ \ }+\left.  \ln I\left(  p\left(  \frac{1}{\alpha
}\left(  2s-y-x^{\prime}\right)  +x^{\prime}\right)  ;\frac{y+y^{\prime}}%
{2},\frac{x^{\prime}+y}{2}\right)  \right] \\
&  \leq\int_{0}^{1}Ip_{2}\left(  t\right)  dt\\
&  \leq-A\left(  \ln y,\ln y^{\prime}\right)  \int_{x}^{x^{\prime}}p\left(
s\right)  ds\\
&  \text{ }-\int\nolimits_{x}^{x^{\prime}}A\left(  \ln\left(  \left(
1-\alpha\right)  x+\alpha s\right)  ,\ln\left(  \left(  1-\alpha\right)
x^{\prime}+\alpha s\right)  \right)  p\left(  s\right)  ds.
\end{align*}%
\begin{align*}
0  &  \leq-\int\nolimits_{x}^{x^{\prime}}\left[  \ln\left(  \left(
1-\alpha\right)  x+\alpha s\right)  +\ln\left(  \left(  1-\alpha\right)
x^{\prime}+\alpha s\right)  \right]  p\left(  s\right)  ds-Ip_{2}\left(
t\right) \\
&  \leq\frac{2\left(  1-t\right)  \left(  y-x\right)  }{\alpha}\left[
-A\left(  \ln x,\ln x^{\prime}\right)  \frac{y^{\prime}-x}{y-x}+A\left(  \ln
y,\ln y^{\prime}\right)  \frac{y^{\prime}-y}{y-x}\right. \\
&  \left.  +A\left(  \ln I\left(  x,y\right)  ,\ln I\left(  y^{\prime
},x^{\prime}\right)  \right)  \right]  \left\Vert p\right\Vert _{\infty
}\text{\ \ \ }\left(  t\in\left[  0,1\right]  \right)  .
\end{align*}%
\begin{align*}
0  &  \leq Ip_{2}\left(  t\right)  +2\ln A\left(  y,y^{\prime}\right)
\int_{x}^{x^{\prime}}p\left(  s\right)  ds\\
&  \leq\frac{\left(  2x^{\prime}-y-y^{\prime}\right)  \left(  x^{\prime
}-x\right)  }{2xx^{\prime}}\int_{x}^{x^{\prime}}p\left(  s\right)  ds\text{
\ \ }\left(  t\in\left[  0,1\right]  \right)  .
\end{align*}%
\begin{align*}
0  &  \leq-2A\left(  \ln x,\ln x^{\prime}\right)  \int_{x}^{x^{\prime}%
}p\left(  s\right)  ds-Ip_{2}\left(  t\right) \\
&  \leq\frac{\left(  2x^{\prime}-y-y^{\prime}\right)  \left(  x^{\prime
}-x\right)  }{2xx^{\prime}}\int_{x}^{x^{\prime}}p\left(  s\right)  ds\text{
\ \ }\left(  t\in\left[  0,1\right]  \right)  .
\end{align*}

\end{proposition}

\begin{proposition}
\label{p41}The following inequalities hold for all $t\in\left[  0,1\right]  :$%
\[
Ip_{1}\left(  t\right)  \leq2G_{1}\left(  t\right)  \int_{x}^{x^{\prime}%
}p\left(  s\right)  ds.
\]%
\begin{align*}
0  &  \leq Ip_{1}\left(  t\right)  +2A\left(  \ln y,\ln y^{\prime}\right)
\int_{x}^{x^{\prime}}p\left(  s\right)  ds\\
&  \leq2\left(  x^{\prime}-x\right)  \left[  G_{1}\left(  t\right)
-H_{1}\left(  t\right)  \right]  \left\Vert p\right\Vert _{\infty}.
\end{align*}

\end{proposition}

\begin{proposition}
\label{p42}Let $m,$ $m^{\prime}$\ \textit{be defined as in Proposition
\ref{p9}. Then }the following inequalities hold:%
\[
Ip_{2}\left(  t\right)  \leq\left(  \geq\right)  2G_{2}\left(  t\right)
\int_{x}^{x^{\prime}}p\left(  s\right)  ds\text{ \ \ \textit{as }}t\in\left[
m^{\prime},1\right]  \text{ \ }\left(  t\in\left[  0,m\right]  \right)  .
\]%
\begin{align*}
0  &  \leq Ip_{2}\left(  t\right)  +2\ln A\left(  y,y^{\prime}\right)
\int_{x}^{x^{\prime}}p\left(  s\right)  ds\\
&  \leq\frac{2}{\alpha}\left[  \left(  y^{\prime}-x\right)  G_{2}\left(
t\right)  -2\left(  y^{\prime}-y\right)  G_{3}\left(  t\right)  -\left(
y-x\right)  H_{2}\left(  t\right)  \right]  \left\Vert p\right\Vert _{\infty
}\text{\ }\left(  t\in\left[  0,1\right]  \right)  .
\end{align*}

\end{proposition}

\begin{proposition}
\label{p43}Let $m^{\prime}$\ \textit{be defined as in Proposition \ref{p9}.
Then} we have:%
\begin{align*}
&  2A\left(  G_{1}\left(  t\right)  ,G_{2}\left(  t\right)  \right)  \int%
_{x}^{x^{\prime}}p\left(  s\right)  ds\\
&  \leq Sp\left(  t\right) \\
&  \leq-\left(  1-t\right)  \int\nolimits_{x}^{x^{\prime}}\left[  \ln\left(
\left(  1-\alpha\right)  x+\alpha s\right)  +\ln\left(  \left(  1-\alpha
\right)  x^{\prime}+\alpha s\right)  \right]  p\left(  s\right)  ds\\
&  \qquad\qquad-2tA\left(  \ln x,\ln x^{\prime}\right)  \int_{x}^{x^{\prime}%
}p\left(  s\right)  ds\\
&  \leq-2A\left(  \ln x,\ln x^{\prime}\right)  \int_{x}^{x^{\prime}}p\left(
s\right)  ds\text{ \ }\left(  t\in\left[  0,1\right]  \right)  .
\end{align*}%
\[
A\left(  Ip_{1}\left(  1-t\right)  ,Ip_{2}\left(  1-t\right)  \right)  \leq
Sp\left(  t\right)  \text{ \ }\left(  t\in\left[  0,1\right]  \right)  .
\]%
\[
A\left(  A\left(  Ip_{1}\left(  t\right)  ,Ip_{2}\left(  t\right)  \right)
,A\left(  Ip_{1}\left(  1-t\right)  ,Ip_{2}\left(  1-t\right)  \right)
\right)  \leq Sp\left(  t\right)  \text{ \ \ }\left(  t\in\left[  m^{\prime
},1\right]  \right)  .
\]%
\[
\sup\limits_{t\in\left[  0,1\right]  }Sp\left(  t\right)  =-2A\left(  \ln
x,\ln x^{\prime}\right)  \int_{x}^{x^{\prime}}p\left(  s\right)  ds\text{ }.
\]

\end{proposition}

\begin{proposition}
\label{p44}
\end{proposition}

\begin{enumerate}
\item $Q_{1}$\textit{\ is symmetric about }$\frac{1}{2},$\textit{\ decreasing
on }$\left[  0,\frac{1}{2}\right]  $\textit{\ and increasing on }$\left[
\frac{1}{2},1\right]  .$

\item \textit{The following inequalities hold:}%
\[
G_{1}\left(  \frac{t}{\alpha}\right)  \leq Q_{1}\left(  t\right)  \text{
\quad}\left(  t\in\left[  0,\frac{\alpha}{2}\right]  \right)  .
\]%
\[
G_{1}\left(  \frac{t}{\alpha}\right)  \geq Q_{1}\left(  t\right)  \quad\left(
t\in\left[  \frac{\alpha}{2},\alpha\right]  \right)  .
\]%
\[
G_{1}\left(  \frac{1-t}{\alpha}\right)  \geq Q_{1}\left(  t\right)  \text{
\quad}\left(  t\in\left[  1-\alpha,1-\frac{\alpha}{2}\right]  \right)  .
\]%
\[
G_{1}\left(  \frac{1-t}{\alpha}\right)  \leq Q_{1}\left(  t\right)
\quad\left(  t\in\left[  1-\frac{\alpha}{2},1\right]  \right)  .
\]

\end{enumerate}

\begin{proposition}
\label{p45}Let $m^{\prime}$\ \textit{be defined as in Proposition \ref{p9}.
Then} we have:%
\begin{align*}
0  &  \leq Np\left(  t\right)  -2G_{1}\left(  t\right)  \int_{x}^{x^{\prime}%
}p\left(  s\right)  ds\\
&  \leq-2A\left(  \ln x,\ln x^{\prime}\right)  \int_{x}^{x^{\prime}}p\left(
s\right)  ds\text{ }-Np\left(  t\right)  \text{ \ }\left(  t\in\left[
0,1\right]  \right)  .
\end{align*}%
\begin{align*}
0  &  \leq Lp_{1}\left(  t\right)  -A\left(  Hp_{1}\left(  t\right)
,Hp_{2}\left(  t\right)  \right) \\
&  \leq\frac{\left(  4x^{\prime}-y-3y^{\prime}\right)  \left(  x^{\prime
}-x\right)  }{8xx^{\prime}}\int_{\Omega}p\left(  s\right)  ds\text{ \ }\left(
t\in\left[  0,1\right]  \right)  .
\end{align*}%
\begin{align*}
0  &  \leq Pp_{1}-Lp_{1}\left(  t\right) \\
&  \leq\frac{\left(  4x^{\prime}-y-3y^{\prime}\right)  \left(  x^{\prime
}-x\right)  }{8xx^{\prime}}\int_{\Omega}p\left(  s\right)  ds\text{ \ }\left(
t\in\left[  0,1\right]  \right)  .
\end{align*}%
\begin{align*}
0  &  \leq Np\left(  t\right)  -Ip_{1}\left(  t\right) \\
&  \leq\frac{\left(  y-x\right)  \left(  x^{\prime}-x\right)  }{xx^{\prime}%
}\int_{x}^{x^{\prime}}p\left(  s\right)  ds\text{ \ }\left(  t\in\left[
0,1\right]  \right)  .
\end{align*}%
\begin{align*}
0  &  \leq Sp\left(  t\right)  -A\left(  Ip_{1}\left(  t\right)
,Ip_{2}\left(  t\right)  \right) \\
&  \leq\frac{\left(  4x^{\prime}-y-3y^{\prime}\right)  \left(  x^{\prime
}-x\right)  }{4xx^{\prime}}\int_{x}^{x^{\prime}}p\left(  s\right)  ds\text{
\ }\left(  t\in\left[  m^{\prime},1\right]  \right)  .
\end{align*}

\end{proposition}

\begin{proposition}
\label{p46}The following inequalities hold:%
\begin{align*}
Hp_{1}\left(  t\right)   &  \leq Q_{1}\left(  t\right)  \int_{\Omega}p\left(
s\right)  ds\\
&  \leq-A\left(  \ln x,\ln x^{\prime}\right)  \int_{\Omega}p\left(  s\right)
ds\qquad\left(  t\in\left[  0,\frac{\alpha}{1+\alpha}\right]  \right)  .
\end{align*}%
\begin{align*}
-\ln A\left(  x,x^{\prime}\right)  \int_{\Omega}p\left(  s\right)  ds  &  \leq
Q_{1}\left(  t\right)  \int_{\Omega}p\left(  s\right)  ds\\
&  \leq Pp_{1}\left(  t\right)  \text{ \qquad}\left(  t\in\left[  \frac
{\alpha}{1+\alpha},1\right]  \right)  .
\end{align*}%
\begin{align*}
0  &  \leq Sp\left(  t\right)  -2A\left(  G_{1}\left(  t\right)  ,G_{2}\left(
t\right)  \right)  \int_{x}^{x^{\prime}}p\left(  s\right)  ds\\
&  \leq\left[  -A\left(  \ln x,\ln x^{\prime}\right)  +Q_{1}\left(  t\right)
\right]  \int_{x}^{x^{\prime}}p\left(  s\right)  ds-Sp\left(  t\right)  \text{
\ }\left(  t\in\left[  0,1\right]  \right)  .
\end{align*}

\end{proposition}

\begin{proposition}
\label{p47}$Kp$\textit{\ is symmetric about }$\frac{1}{2},$%
\textit{\ decreasing on }$\left[  0,\frac{1}{2}\right]  $\textit{\ and
increasing on }$\left[  \frac{1}{2},1\right]  .$ The following identities and
inequalities hold:
\end{proposition}

\begin{align*}
&  \sup\limits_{t\in\left[  0,1\right]  }Kp\left(  t\right)  =Kp\left(
0\right)  =Kp\left(  1\right) \\
&  =-2%
{\displaystyle\int\nolimits_{x}^{x^{\prime}}}
\left[  \ln\left(  \left(  1-\alpha\right)  x+\alpha s\right)  +\ln\left(
\left(  1-\alpha\right)  x^{\prime}+\alpha s\right)  \right]  p\left(
s\right)  ds\ \int_{x}^{x^{\prime}}p\left(  s\right)  ds.
\end{align*}%
\begin{align*}
\inf\limits_{t\in\left[  0,1\right]  }Kp\left(  t\right)   &  =Kp\left(
\frac{1}{2}\right) \\
&  =-%
{\displaystyle\int\nolimits_{x}^{x^{\prime}}}
{\displaystyle\int\nolimits_{x}^{x^{\prime}}}
\left[  \ln\left(  \left(  1-\alpha\right)  x+\alpha\frac{s+u}{2}\right)
\right. \\
&  +2\ln\left(  \left(  1-\alpha\right)  \frac{x+x^{\prime}}{2}+\alpha
\frac{s+u}{2}\right) \\
&  +\left.  \ln\left(  \left(  1-\alpha\right)  x^{\prime}+\alpha\frac{s+u}%
{2}\right)  \right]  p\left(  s\right)  p\left(  u\right)  dsdu.
\end{align*}%
\[
2A\left(  Ip_{1}\left(  t\right)  ,Ip_{2}\left(  t\right)  \right)  \int%
_{x}^{x^{\prime}}p\left(  s\right)  ds\leq Kp\left(  t\right)  \text{
\ }\left(  t\in\left[  0,1\right]  \right)  .
\]%
\[
-\left[  \ln y+\ln A\left(  x,x^{\prime}\right)  +\ln y^{\prime}\right]
\left[  \int_{x}^{x^{\prime}}p\left(  s\right)  ds\right]  ^{2}\leq Kp\left(
\frac{1}{2}\right)  .
\]

\begin{proposition}
\label{p48}The following inequality holds for all $t\in\left[  0,1\right]  :$%
\begin{align*}
0  &  \leq Kp\left(  t\right)  -2A\left(  Ip_{1}\left(  t\right)
,Ip_{2}\left(  t\right)  \right)  \int_{x}^{x^{\prime}}p\left(  s\right)  ds\\
&  \leq2Sp\left(  1-t\right)  \int_{x}^{x^{\prime}}p\left(  s\right)
ds-Kp\left(  t\right)  .
\end{align*}

\end{proposition}

\subsubsection{Applications for the Extended $r$-Logarithmic Mean}

Throughout this subsubsection, let $r\in\left(  -\infty,0\right)  \cup\left[
1,\infty\right)  $ and $f\left(  s\right)  =s^{r}$ $\left(  s\in\left[
a,b\right]  \right)  .$

\begin{remark}
\label{r33}From the Section 2, we get%
\[
\int_{\Omega}f\left(  s\right)  p\left(  s\right)  ds=L_{r}^{r}\left(
p;x,y\right)  +L_{r}^{r}\left(  p;y^{\prime},x^{\prime}\right)  ,
\]%
\[
\frac{1}{2\left(  y-x\right)  }\int_{\Omega}f\left(  s\right)  ds=\left\{
\begin{array}
[c]{cc}%
A\left(  L_{r}^{r}\left(  x,y\right)  ,L_{r}^{r}\left(  y^{\prime},x^{\prime
}\right)  \right)  & \text{as }r\neq-1\text{ }\\
H^{-1}\left(  L\left(  x,y\right)  ,L\left(  y^{\prime},x^{\prime}\right)
\right)  & \text{as }r=-1\text{ }%
\end{array}
\right.  \text{ },
\]%
\begin{align*}
&  \frac{1}{y-x}\int_{x}^{y}\left[  f\left(  \frac{s+x}{2}\right)  +f\left(
\frac{x+2y-s}{2}\right)  \right]  ds\\
&  =\frac{1}{y-x}\int_{x}^{y}f\left(  s\right)  ds=\left\{
\begin{array}
[c]{cc}%
L_{r}^{r}\left(  x,y\right)  & \text{as }r\neq-1\text{ }\\
L^{-1}\left(  x,y\right)  & \text{as }r=-1
\end{array}
\right.  \text{ \ },
\end{align*}%
\begin{align*}
&  \frac{1}{y-x}\int_{y^{\prime}}^{x^{\prime}}\left[  f\left(  \frac
{s+x^{\prime}}{2}\right)  +f\left(  \frac{x^{\prime}+2y^{\prime}-s}{2}\right)
\right]  ds\\
&  =\frac{1}{y-x}\int_{y^{\prime}}^{x^{\prime}}f\left(  s\right)  ds=\left\{
\begin{array}
[c]{cc}%
L_{r}^{r}\left(  y^{\prime},x^{\prime}\right)  & \text{as }r\neq-1\text{ }\\
L^{-1}\left(  y^{\prime},x^{\prime}\right)  & \text{as }r=-1
\end{array}
\right.  ,
\end{align*}%
\begin{align*}
&  \frac{1}{y-x}\left[  \int_{x}^{y}f\left(  \frac{s+x}{2}\right)  +f\left(
\frac{x+2y^{\prime}-s}{2}\right)  ds\right] \\
&  =\frac{2}{y-x}\left[  \int_{x}^{\frac{x+y}{2}}f\left(  s\right)
ds+\int_{\frac{x+2y^{\prime}-y}{2}}^{y^{\prime}}f\left(  s\right)  ds\right]
\\
&  =\left\{
\begin{array}
[c]{cc}%
L_{r}^{r}\left(  x,\frac{x+y}{2}\right)  +L_{r}^{r}\left(  \frac{x+2y^{\prime
}-y}{2},y^{\prime}\right)  & \text{as }r\neq-1\text{ }\\
L^{-1}\left(  x,\frac{x+y}{2}\right)  +L^{-1}\left(  \frac{x+2y^{\prime}-y}%
{2},y^{\prime}\right)  & \text{as }r=-1
\end{array}
\right.  \text{ },
\end{align*}%
\begin{align*}
&  \frac{1}{y-x}\int_{y^{\prime}}^{x^{\prime}}\left[  f\left(  \frac
{s+x^{\prime}}{2}\right)  +f\left(  \frac{x^{\prime}+2y-s}{2}\right)  \right]
ds\\
&  =\frac{2}{y-x}\left[  \int_{\frac{x^{\prime}+y^{\prime}}{2}}^{x^{\prime}%
}f\left(  s\right)  ds+\int_{y}^{\frac{x^{\prime}+2y-y^{\prime}}{2}}f\left(
s\right)  ds\right] \\
&  =\left\{
\begin{array}
[c]{cc}%
L_{r}^{r}\left(  \frac{x^{\prime}+y^{\prime}}{2},x^{\prime}\right)  +L_{r}%
^{r}\left(  y,\frac{x^{\prime}+2y-y^{\prime}}{2}\right)  & \text{as }%
r\neq-1\text{ }\\
L^{-1}\left(  \frac{x^{\prime}+y^{\prime}}{2},x^{\prime}\right)
+L^{-1}\left(  y,\frac{x^{\prime}+2y-y^{\prime}}{2}\right)  & \text{as }r=-1
\end{array}
\right.
\end{align*}
and the following convex functions on $\left[  0,1\right]  $:%
\[
G_{1}\left(  t\right)  =A\left(  \left(  tx+\left(  1-t\right)  y\right)
^{r},\left(  tx^{\prime}+\left(  1-t\right)  y^{\prime}\right)  ^{r}\right)
.
\]%
\[
G_{2}\left(  t\right)  =A\left(  \left(  tx+\left(  1-t\right)  y^{\prime
}\right)  ^{r},\left(  tx^{\prime}+\left(  1-t\right)  y\right)  ^{r}\right)
.
\]%
\[
G_{3}\left(  t\right)  =A\left(  \left(  ty+\left(  1-t\right)  y^{\prime
}\right)  ^{r},\left(  ty^{\prime}+\left(  1-t\right)  y\right)  ^{r}\right)
.
\]%
\[
Hp_{1}\left(  t\right)  =\int_{x}^{y}\left[  \left(  ts+\left(  1-t\right)
y\right)  ^{r}+\left(  t\left(  y+y^{\prime}-s\right)  +\left(  1-t\right)
y^{\prime}\right)  ^{r}\right]  p\left(  s\right)  ds.
\]%
\[
Hp_{2}\left(  t\right)  =\int_{x}^{y}\left[  \left(  ts+\left(  1-t\right)
y^{\prime}\right)  ^{r}+\left(  t\left(  y+y^{\prime}-s\right)  +\left(
1-t\right)  y\right)  ^{r}\right]  p\left(  s\right)  ds.
\]%
\[
H_{1}\left(  t\right)  =\frac{1}{2\left(  y-x\right)  }\int_{x}^{y}\left[
\left(  ts+\left(  1-t\right)  y\right)  ^{r}+\left(  t\left(  y+y^{\prime
}-s\right)  +\left(  1-t\right)  y^{\prime}\right)  ^{r}\right]  ds.
\]%
\[
H_{2}\left(  t\right)  =\frac{1}{2\left(  y-x\right)  }\int_{x}^{y}\left[
\left(  ts+\left(  1-t\right)  y^{\prime}\right)  ^{r}+\left(  t\left(
y+y^{\prime}-s\right)  +\left(  1-t\right)  y\right)  ^{r}\right]  ds.
\]%
\[
Fp_{1}\left(  t\right)  =\int_{\Omega}\int_{\Omega}\left(  ts+\left(
1-t\right)  u\right)  ^{r}p\left(  s\right)  p\left(  u\right)  dsdu.
\]%
\[
F_{1}\left(  t\right)  =\frac{1}{4\left(  y-x\right)  ^{2}}\int_{\Omega}%
\int_{\Omega}\left(  ts+\left(  1-t\right)  u\right)  ^{r}dsdu.
\]%
\[
Pp_{1}\left(  t\right)  =\int_{x}^{y}\left[  \left(  tx+\left(  1-t\right)
s\right)  ^{r}+\left(  tx^{\prime}+\left(  1-t\right)  \left(  x+x^{\prime
}-s\right)  \right)  ^{r}\right]  p\left(  s\right)  ds.
\]%
\[
P_{1}\left(  t\right)  =\frac{1}{2\left(  y-x\right)  }\int_{x}^{y}\left[
\left(  tx+\left(  1-t\right)  s\right)  ^{r}+\left(  tx^{\prime}+\left(
1-t\right)  \left(  x+x^{\prime}-s\right)  \right)  ^{r}\right]  ds.
\]%
\[
Lp_{1}\left(  t\right)  =\frac{1}{2}%
{\displaystyle\int\nolimits_{\Omega}}
\left[  \left(  tx+\left(  1-t\right)  s\right)  ^{r}+\left(  tx^{\prime
}+\left(  1-t\right)  s\right)  ^{r}\right]  p\left(  s\right)  ds.
\]%
\[
L_{1}\left(  t\right)  =\frac{1}{4\left(  y-x\right)  }%
{\displaystyle\int\nolimits_{\Omega}}
\left[  \left(  tx+\left(  1-t\right)  s\right)  ^{r}+\left(  tx^{\prime
}+\left(  1-t\right)  s\right)  ^{r}\right]  ds.
\]%
\[
Q_{1}\left(  t\right)  =A\left(  \left(  tx+\left(  1-t\right)  x^{\prime
}\right)  ^{r},\left(  tx^{\prime}+\left(  1-t\right)  x\right)  ^{r}\right)
.
\]%
\begin{align*}
Ip_{1}\left(  t\right)   &  =%
{\displaystyle\int\nolimits_{x}^{x^{\prime}}}
\left[  \left(  t\left(  \left(  1-\alpha\right)  x+\alpha s\right)  +\left(
1-t\right)  y\right)  ^{r}\right. \\
&  \text{ \ \ \ \ }+\left.  \left(  t\left(  \left(  1-\alpha\right)
x^{\prime}+\alpha s\right)  +\left(  1-t\right)  y^{\prime}\right)
^{r}\right]  p\left(  s\right)  ds.
\end{align*}%
\begin{align*}
Ip_{2}\left(  t\right)   &  =%
{\displaystyle\int\nolimits_{x}^{x^{\prime}}}
\left[  \left(  t\left(  \left(  1-\alpha\right)  x+\alpha s\right)  +\left(
1-t\right)  y^{\prime}\right)  ^{r}\right. \\
&  \text{ \ \ \ \ \ }+\left.  \left(  t\left(  \left(  1-\alpha\right)
x^{\prime}+\alpha s\right)  +\left(  1-t\right)  y\right)  ^{r}\right]
p\left(  s\right)  ds.
\end{align*}%
\begin{align*}
Jp\left(  t\right)   &  =%
{\displaystyle\int\nolimits_{x}^{x^{\prime}}}
\left[  \left(  t\left(  \left(  1-\alpha\right)  x+\alpha s\right)  +\left(
1-t\right)  \frac{x+y}{2}\right)  ^{r}\right. \\
&  \text{\ }+\left.  \left(  t\left(  \left(  1-\alpha\right)  x^{\prime
}+\alpha s\right)  +\left(  1-t\right)  \frac{x^{\prime}+y^{\prime}}%
{2}\right)  ^{r}\right]  p\left(  s\right)  ds.
\end{align*}%
\begin{align*}
Mp\left(  t\right)   &  =%
{\displaystyle\int\nolimits_{x}^{\frac{x+x^{\prime}}{2}}}
\left[  \left(  tx+\left(  1-t\right)  \left(  \left(  1-\alpha\right)
x+\alpha s\right)  \right)  ^{r}\right. \\
&  \text{ \ \ \ \ \ \ \ \ \ \ }+\left.  \left(  ty^{\prime}+\left(
1-t\right)  \left(  \left(  1-\alpha\right)  x^{\prime}+\alpha s\right)
\right)  ^{r}\right]  p\left(  s\right)  ds\\
&  +%
{\displaystyle\int\nolimits_{\frac{x+x^{\prime}}{2}}^{x^{\prime}}}
\left[  \left(  ty+\left(  1-t\right)  \left(  \left(  1-\alpha\right)
x+\alpha s\right)  \right)  ^{r}\right. \\
&  \text{ \ \ \ \ \ \ \ \ \ \ \ }+\left.  \left(  tx^{\prime}+\left(
1-t\right)  \left(  \left(  1-\alpha\right)  x^{\prime}+\alpha s\right)
\right)  ^{r}\right]  p\left(  s\right)  ds.
\end{align*}%
\begin{align*}
Np\left(  t\right)   &  =%
{\displaystyle\int\nolimits_{x}^{x^{\prime}}}
\left[  \left(  tx+\left(  1-t\right)  \left(  \left(  1-\alpha\right)
x+\alpha s\right)  \right)  ^{r}\right. \\
&  \text{ \ \ \ \ \ }+\left.  \left(  tx^{\prime}+\left(  1-t\right)  \left(
\left(  1-\alpha\right)  x^{\prime}+\alpha s\right)  \right)  ^{r}\right]
p\left(  s\right)  ds.
\end{align*}%
\begin{align*}
Sp\left(  t\right)   &  =%
{\displaystyle\int\nolimits_{x}^{x^{\prime}}}
\frac{1}{2}\left[  \left(  tx+\left(  1-t\right)  \left(  \left(
1-\alpha\right)  x+\alpha s\right)  \right)  ^{r}\right. \\
&  \text{ \ \ \ \ \ \ \ \ \ }+\left(  tx+\left(  1-t\right)  \left(  \left(
1-\alpha\right)  x^{\prime}+\alpha s\right)  \right)  ^{r}\\
&  \text{ \ \ \ \ \ \ \ \ \ }+\left(  tx^{\prime}+\left(  1-t\right)  \left(
\left(  1-\alpha\right)  x+\alpha s\right)  \right)  ^{r}\\
&  \text{ \ \ \ \ \ \ \ \ \ }+\left.  \left(  tx^{\prime}+\left(  1-t\right)
\left(  \left(  1-\alpha\right)  x^{\prime}+\alpha s\right)  \right)
^{r}\right]  p\left(  s\right)  ds.
\end{align*}

\end{remark}

Using Theorems \ref{t1} --\ \ref{t24}, Remark \ref{r3} and Remark \ref{r33},
we have the following propositions of the extended identric mean:

\begin{proposition}
\label{p49}$Hp_{1},H_{1}$\ is increasing on $\left[  0,1\right]  $\ and the
following inequalities hold for all $t\in\left[  0,1\right]  :$
\begin{align*}
&  A\left(  y^{r},\left(  y^{\prime}\right)  ^{r}\right)
{\displaystyle\int\nolimits_{\Omega}}
p\left(  s\right)  ds\\
&  =Hp_{1}\left(  0\right)  \leq Hp_{1}\left(  t\right)  \text{\ }\leq
Hp_{1}\left(  1\right) \\
&  =L_{r}^{r}\left(  p;x,y\right)  +L_{r}^{r}\left(  p;y^{\prime},x^{\prime
}\right)  .
\end{align*}%
\begin{align*}
&  A\left(  y^{r},\left(  y^{\prime}\right)  ^{r}\right) \\
&  =H_{1}\left(  0\right)  \leq H_{1}\left(  t\right)  \leq H_{1}\left(
1\right) \\
&  =A\left(  L_{r}^{r}\left(  x,y\right)  ,L_{r}^{r}\left(  y^{\prime
},x^{\prime}\right)  \right)
\end{align*}
for $r\neq-1$ and $p\left(  s\right)  =\frac{1}{2\left(  y-x\right)  }$ on
$\left[  x.x^{\prime}\right]  .$%
\begin{align*}
&  H^{-1}\left(  y,y^{\prime}\right) \\
&  =H_{1}\left(  0\right)  \leq H_{1}\left(  t\right)  \leq H_{1}\left(
1\right) \\
&  =H^{-1}\left(  L\left(  x,y\right)  ,L\left(  y^{\prime},x^{\prime}\right)
\right)
\end{align*}
for $r=-1$ and $p\left(  s\right)  =\frac{1}{2\left(  y-x\right)  }$ on
$\left[  x.x^{\prime}\right]  .$%
\begin{align*}
&  Hp_{1}\left(  t\right)  \leq t\left[  L_{r}^{r}\left(  p;x,y\right)
^{r}+L_{r}^{r}\left(  p;y^{\prime},x^{\prime}\right)  \right] \\
&  \text{ \ \ \ \ \ \ \ \ \ \ \ }+\left(  1-t\right)  A\left(  y^{r},\left(
y^{\prime}\right)  ^{r}\right)
{\displaystyle\int\nolimits_{\Omega}}
p\left(  s\right)  ds\\
&  \leq L_{r}^{r}\left(  p;x,y\right)  +L_{r}^{r}\left(  p;y^{\prime
},x^{\prime}\right)  \ \\
&  \leq A\left(  x^{r},\left(  x^{\prime}\right)  ^{r}\right)
{\displaystyle\int\nolimits_{\Omega}}
p\left(  s\right)  ds.
\end{align*}%
\begin{align*}
&  H_{1}\left(  t\right)  \leq tA\left(  L_{r}^{r}\left(  x,y\right)
,L_{r}^{r}\left(  y^{\prime},x^{\prime}\right)  \right)  +\left(  1-t\right)
A\left(  y^{r},\left(  y^{\prime}\right)  ^{r}\right) \\
&  \leq A\left(  L_{r}^{r}\left(  x,y\right)  ,L_{r}^{r}\left(  y^{\prime
},x^{\prime}\right)  \right)  \leq A\left(  x^{r},\left(  x^{\prime}\right)
^{r}\right)
\end{align*}
for $r\neq-1$ and $p\left(  s\right)  =\frac{1}{2\left(  y-x\right)  }$ on
$\left[  x.x^{\prime}\right]  .$%
\begin{align*}
&  H_{1}\left(  t\right)  \leq tH^{-1}\left(  L\left(  x,y\right)  ,L\left(
y^{\prime},x^{\prime}\right)  \right)  +\left(  1-t\right)  H^{-1}\left(
y,y^{\prime}\right) \\
&  \leq H^{-1}\left(  L\left(  x,y\right)  ,L\left(  y^{\prime},x^{\prime
}\right)  \right)  \leq H^{-1}\left(  x,x^{\prime}\right)
\end{align*}
for $r=-1$ and $p\left(  s\right)  =\frac{1}{2\left(  y-x\right)  }$ on
$\left[  x.x^{\prime}\right]  .$%
\begin{align*}
&  A\left(  y^{r},\left(  y^{\prime}\right)  ^{r}\right)
{\displaystyle\int\nolimits_{\Omega}}
p\left(  s\right)  ds\\
&  \leq Hp_{2}\left(  t\right) \\
&  \leq t\left[  L_{r}^{r}\left(  p;x,y\right)  +L_{r}^{r}\left(  p;y^{\prime
},x^{\prime}\right)  \right] \\
&  \text{ \ \ \ }+\left(  1-t\right)  A\left(  y^{r},\left(  y^{\prime
}\right)  ^{r}\right)
{\displaystyle\int\nolimits_{\Omega}}
p\left(  s\right)  ds\\
&  \leq L_{r}^{r}\left(  p;x,y\right)  +L_{r}^{r}\left(  p;y^{\prime
},x^{\prime}\right)  \
\end{align*}%
\begin{align*}
&  A\left(  y^{r},\left(  y^{\prime}\right)  ^{r}\right)  \leq H_{2}\left(
t\right) \\
&  \leq tA\left(  L_{r}^{r}\left(  x,y\right)  ,L_{r}^{r}\left(  y^{\prime
},x^{\prime}\right)  \right)  +\left(  1-t\right)  A\left(  y^{r},\left(
y^{\prime}\right)  ^{r}\right) \\
&  \leq A\left(  L_{r}^{r}\left(  x,y\right)  ,L_{r}^{r}\left(  y^{\prime
},x^{\prime}\right)  \right)
\end{align*}
for $r\neq-1$ and $p\left(  s\right)  =\frac{1}{2\left(  y-x\right)  }$ on
$\left[  x.x^{\prime}\right]  .$%
\begin{align*}
&  H^{-1}\left(  y,y^{\prime}\right)  \leq H_{2}\left(  t\right) \\
&  \leq tH^{-1}\left(  L\left(  x,y\right)  ,L\left(  y^{\prime},x^{\prime
}\right)  \right)  +\left(  1-t\right)  H^{-1}\left(  y,y^{\prime}\right) \\
&  \leq H^{-1}\left(  L\left(  x,y\right)  ,L\left(  y^{\prime},x^{\prime
}\right)  \right)
\end{align*}
for $r=-1$ and $p\left(  s\right)  =\frac{1}{2\left(  y-x\right)  }$ on
$\left[  x.x^{\prime}\right]  .$%
\[
Hp_{2}\left(  t\right)  \leq Hp_{1}\left(  t\right)  .
\]%
\[
H_{2}\left(  t\right)  \leq H_{1}\left(  t\right)  \text{\ \ as }p\left(
s\right)  =\frac{1}{2\left(  y-x\right)  }\text{ on }\left[  x.x^{\prime
}\right]  .
\]

\end{proposition}

\begin{proposition}
\label{p50}The following inequalities hold:%
\begin{align*}
&  A\left(  y^{r},\left(  y^{\prime}\right)  ^{r}\right)
{\displaystyle\int\nolimits_{\Omega}}
p\left(  s\right)  ds\\
&  \leq2L_{r}^{r}\left(  p\left(  2s-y\right)  ;\frac{x+y}{2},y\right)
+2L_{r}^{r}\left(  p\left(  2s-y^{\prime}\right)  ;y^{\prime},\frac{x^{\prime
}+y^{\prime}}{2}\right) \\
&  \leq\int_{0}^{1}Hp_{1}\left(  t\right)  dt\\
&  \leq A\left(  A\left(  y^{r},\left(  y^{\prime}\right)  ^{r}\right)
{\displaystyle\int\nolimits_{\Omega}}
p\left(  s\right)  ds,L_{r}^{r}\left(  p;x,y\right)  +L_{r}^{r}\left(
p;y^{\prime},x^{\prime}\right)  \right)  .
\end{align*}%
\begin{align*}
&  A\left(  y^{r},\left(  y^{\prime}\right)  ^{r}\right) \\
&  \leq A\left(  L_{r}^{r}\left(  \frac{x+y}{2},y\right)  ,L_{r}^{r}\left(
y^{\prime},\frac{x^{\prime}+y^{\prime}}{2}\right)  \right) \\
&  \leq\int_{0}^{1}H_{1}\left(  t\right)  dt\leq A\left(  A\left(
y^{r},\left(  y^{\prime}\right)  ^{r}\right)  ,A\left(  L_{r}^{r}\left(
x,y\right)  ,L_{r}^{r}\left(  y^{\prime},x^{\prime}\right)  \right)  \right)
\end{align*}
for $r\neq-1$ and $p\left(  s\right)  =\frac{1}{2\left(  y-x\right)  }$ on
$\left[  x.x^{\prime}\right]  .$%
\begin{align*}
&  H^{-1}\left(  y,y^{\prime}\right) \\
&  \leq H^{-1}\left(  L\left(  \frac{x+y}{2},y\right)  ,L\left(  y^{\prime
},\frac{x^{\prime}+y^{\prime}}{2}\right)  \right) \\
&  \leq\int_{0}^{1}H_{1}\left(  t\right)  dt\leq H^{-1}\left(  H\left(
y,y^{\prime}\right)  ,H\left(  L\left(  x,y\right)  ,L\left(  y^{\prime
},x^{\prime}\right)  \right)  \right)
\end{align*}
for $r=-1$ and $p\left(  s\right)  =\frac{1}{2\left(  y-x\right)  }$ on
$\left[  x.x^{\prime}\right]  .$%
\begin{align*}
&  A^{r}\left(  y,y^{\prime}\right)
{\displaystyle\int\nolimits_{\Omega}}
p\left(  s\right)  ds\\
&  \leq2L_{r}^{r}\left(  p\left(  2s-y^{\prime}\right)  ;\frac{x+y^{\prime}%
}{2},\frac{y+y^{\prime}}{2}\right)  +2L_{r}^{r}\left(  p\left(  2s-y\right)
;\frac{y+y^{\prime}}{2},\frac{y+x^{\prime}}{2}\right) \\
&  \leq\int_{0}^{1}Hp_{2}\left(  t\right)  dt\\
&  \leq A\left(  A\left(  y^{r},\left(  y^{\prime}\right)  ^{r}\right)
{\displaystyle\int\nolimits_{\Omega}}
p\left(  s\right)  ds,L_{r}^{r}\left(  p;x,y\right)  +L_{r}^{r}\left(
p;y^{\prime},x^{\prime}\right)  \right)  .
\end{align*}%
\begin{align*}
&  A^{r}\left(  y,y^{\prime}\right) \\
&  \leq A\left(  L_{r}^{r}\left(  \frac{x+y^{\prime}}{2},\frac{y+y^{\prime}%
}{2}\right)  ,L_{r}^{r}\left(  \frac{y+y^{\prime}}{2},\frac{y+x^{\prime}}%
{2}\right)  \right) \\
&  \leq\int_{0}^{1}H_{2}\left(  t\right)  dt\leq A\left(  A\left(
y^{r},\left(  y^{\prime}\right)  ^{r}\right)  ,A\left(  L_{r}^{r}\left(
x,y\right)  ,L_{r}^{r}\left(  y^{\prime},x^{\prime}\right)  \right)  \right)
\end{align*}
for $r\neq-1$ and $p\left(  s\right)  =\frac{1}{2\left(  y-x\right)  }$ on
$\left[  x.x^{\prime}\right]  .$%
\begin{align*}
&  A^{-1}\left(  y,y^{\prime}\right) \\
&  \leq H^{-1}\left(  L\left(  \frac{x+y^{\prime}}{2},\frac{y+y^{\prime}}%
{2}\right)  ,L\left(  \frac{y+y^{\prime}}{2},\frac{y+x^{\prime}}{2}\right)
\right) \\
&  \leq\int_{0}^{1}H_{2}\left(  t\right)  dt\leq H^{-1}\left(  H\left(
y,y^{\prime}\right)  ,H\left(  L\left(  x,y\right)  ,L\left(  y^{\prime
},x^{\prime}\right)  \right)  \right)
\end{align*}
for $r=-1$ and $p\left(  s\right)  =\frac{1}{2\left(  y-x\right)  }$ on
$\left[  x.x^{\prime}\right]  .$%
\begin{align*}
0  &  \leq L_{r}^{r}\left(  p;x,y\right)  +L_{r}^{r}\left(  p;y^{\prime
},x^{\prime}\right)  -Hp_{1}\left(  t\right) \\
&  \leq2\left(  y-x\right)  \left(  1-t\right)  \left[  A\left(  x^{r},\left(
x^{\prime}\right)  ^{r}\right)  -A\left(  L_{r}^{r}\left(  x,y\right)
,L_{r}^{r}\left(  y^{\prime},x^{\prime}\right)  \right)  \right]
\sup\limits_{s\in\Omega}p\left(  s\right)
\end{align*}
for $t\in\left[  0,1\right]  $ and $r\neq-1.$%
\begin{align*}
0  &  \leq L_{-1}^{-1}\left(  p;x,y\right)  +L_{-1}^{-1}\left(  p;y^{\prime
},x^{\prime}\right)  -Hp_{1}\left(  t\right) \\
&  \leq2\left(  y-x\right)  \left(  1-t\right)  \left[  H^{-1}\left(
x,x^{\prime}\right)  -H^{-1}\left(  L\left(  x,y\right)  ,L\left(  y^{\prime
},x^{\prime}\right)  \right)  \right]  \sup\limits_{s\in\Omega}p\left(
s\right)
\end{align*}
for $t\in\left[  0,1\right]  $ and $r=-1.$%
\begin{align*}
0  &  \leq A\left(  L_{r}^{r}\left(  x,y\right)  ,L_{r}^{r}\left(  y^{\prime
},x^{\prime}\right)  \right)  -H_{1}\left(  t\right) \\
&  \leq\left(  1-t\right)  \left[  A\left(  x^{r},\left(  x^{\prime}\right)
^{r}\right)  -A\left(  L_{r}^{r}\left(  x,y\right)  ,L_{r}^{r}\left(
y^{\prime},x^{\prime}\right)  \right)  \right]
\end{align*}
\ for $t\in\left[  0,1\right]  ,r\neq-1$ and $p\left(  s\right)  =\frac
{1}{2\left(  y-x\right)  }$ on $\left[  x.x^{\prime}\right]  .$%
\begin{align*}
0  &  \leq H^{-1}\left(  L\left(  x,y\right)  ,L\left(  y^{\prime},x^{\prime
}\right)  \right)  -H_{1}\left(  t\right) \\
&  \leq\left(  1-t\right)  \left[  H^{-1}\left(  x,x^{\prime}\right)
-H^{-1}\left(  L\left(  x,y\right)  ,L\left(  y^{\prime},x^{\prime}\right)
\right)  \right]
\end{align*}
\ for $t\in\left[  0,1\right]  ,r=-1$ and $p\left(  s\right)  =\frac
{1}{2\left(  y-x\right)  }$ on $\left[  x.x^{\prime}\right]  .$%
\begin{align*}
0  &  \leq A\left(  x^{r},\left(  x^{\prime}\right)  ^{r}\right)  \int%
_{\Omega}p\left(  s\right)  ds-Hp_{1}\left(  t\right) \\
&  \leq\frac{r\left(  y-x\right)  \left(  \left(  x^{\prime}\right)
^{r-1}-x^{r-1}\right)  }{2}\int_{\Omega}p\left(  s\right)  ds\text{ \ }\left(
t\in\left[  0,1\right]  \right)  .
\end{align*}%
\[
0\leq A\left(  x^{r},\left(  x^{\prime}\right)  ^{r}\right)  -H_{1}\left(
t\right)  \leq\frac{r\left(  y-x\right)  \left(  \left(  x^{\prime}\right)
^{r-1}-x^{r-1}\right)  }{2}%
\]
for $t\in\left[  0,1\right]  ,r\neq-1$ and $p\left(  s\right)  =\frac
{1}{2\left(  y-x\right)  }$ on $\left[  x.x^{\prime}\right]  .$%
\[
0\leq H^{-1}\left(  x,x^{\prime}\right)  -H_{1}\left(  t\right)  \leq
\frac{\left(  y-x\right)  \left(  x^{-2}-\left(  x^{\prime}\right)
^{-2}\right)  }{2}%
\]
for $t\in\left[  0,1\right]  ,r=-1$ and $p\left(  s\right)  =\frac{1}{2\left(
y-x\right)  }$ on $\left[  x.x^{\prime}\right]  .$%
\begin{align*}
0  &  \leq Hp_{1}\left(  t\right)  -A\left(  y^{r},\left(  y^{\prime}\right)
^{r}\right)  \int_{\Omega}p\left(  s\right)  ds\\
&  \leq\frac{r\left(  y-x\right)  \left(  \left(  x^{\prime}\right)
^{r-1}-x^{r-1}\right)  }{2}\int_{\Omega}p\left(  s\right)  ds\text{ \ }\left(
t\in\left[  0,1\right]  \right)  .
\end{align*}%
\[
0\leq H_{1}\left(  t\right)  -A\left(  y^{r},\left(  y^{\prime}\right)
^{r}\right)  \leq\frac{r\left(  y-x\right)  \left(  \left(  x^{\prime}\right)
^{r-1}-x^{r-1}\right)  }{2}%
\]
for $t\in\left[  0,1\right]  ,r\neq-1$ and $p\left(  s\right)  =\frac
{1}{2\left(  y-x\right)  }$ on $\left[  x.x^{\prime}\right]  .$%
\[
0\leq H_{1}\left(  t\right)  -H^{-1}\left(  y,y^{\prime}\right)  \leq
\frac{\left(  y-x\right)  \left(  x^{-2}-\left(  x^{\prime}\right)
^{-2}\right)  }{2}%
\]
for $t\in\left[  0,1\right]  ,r=-1$ and $p\left(  s\right)  =\frac{1}{2\left(
y-x\right)  }$ on $\left[  x.x^{\prime}\right]  .$
\end{proposition}

\begin{proposition}
\label{p51}$Pp_{1}$\ is increasing on $\left[  0,1\right]  $\ and the
following inequalities hold for all $t\in\left[  0,1\right]  :$%
\begin{align*}
&  L_{r}^{r}\left(  p;x,y\right)  +L_{r}^{r}\left(  p;y^{\prime},x^{\prime
}\right) \\
&  =Pp_{1}\left(  0\right)  \leq Pp_{1}\left(  t\right)  \leq Pp_{1}\left(
1\right)  =A\left(  x^{r},\left(  x^{\prime}\right)  ^{r}\right)  \int%
_{\Omega}p\left(  s\right)  ds.
\end{align*}%
\begin{align*}
&  A\left(  L_{r}^{r}\left(  x,y\right)  ,L_{r}^{r}\left(  y^{\prime
},x^{\prime}\right)  \right)  =P_{1}\left(  0\right)  \leq P_{1}\left(
t\right) \\
&  \leq P_{1}\left(  1\right)  =A\left(  x^{r},\left(  x^{\prime}\right)
^{r}\right)
\end{align*}
for $r\neq-1$ and $p\left(  s\right)  =\frac{1}{2\left(  y-x\right)  }$ on
$\left[  x.x^{\prime}\right]  .$%
\begin{align*}
&  H^{-1}\left(  L\left(  x,y\right)  ,L\left(  y^{\prime},x^{\prime}\right)
\right)  =P_{1}\left(  0\right)  \leq P_{1}\left(  t\right) \\
&  \leq P_{1}\left(  1\right)  =H^{-1}\left(  x,x^{\prime}\right)
\end{align*}
for $r=-1$ and $p\left(  s\right)  =\frac{1}{2\left(  y-x\right)  }$ on
$\left[  x.x^{\prime}\right]  .$%
\begin{align*}
Pp_{1}\left(  t\right)   &  \leq\left(  1-t\right)  \left(  L_{r}^{r}\left(
p;x,y\right)  +L_{r}^{r}\left(  p;y^{\prime},x^{\prime}\right)  \right) \\
&  \text{ \ \ }+tA\left(  x^{r},\left(  x^{\prime}\right)  ^{r}\right)
\int_{\Omega}p\left(  s\right)  ds\leq A\left(  x^{r},\left(  x^{\prime
}\right)  ^{r}\right)  \int_{\Omega}p\left(  s\right)  ds.
\end{align*}%
\begin{align*}
P_{1}\left(  t\right)   &  \leq\left(  1-t\right)  A\left(  L_{r}^{r}\left(
x,y\right)  ,L_{r}^{r}\left(  y^{\prime},x^{\prime}\right)  \right)
+tA\left(  x^{r},\left(  x^{\prime}\right)  ^{r}\right) \\
&  \leq A\left(  x^{r},\left(  x^{\prime}\right)  ^{r}\right)
\end{align*}
for $r\neq-1$ and $p\left(  s\right)  =\frac{1}{2\left(  y-x\right)  }$ on
$\left[  x.x^{\prime}\right]  .$%
\begin{align*}
P_{1}\left(  t\right)   &  \leq\left(  1-t\right)  H^{-1}\left(  L\left(
x,y\right)  ,L\left(  y^{\prime},x^{\prime}\right)  \right)  +tH^{-1}\left(
x,x^{\prime}\right) \\
&  \leq H^{-1}\left(  x,x^{\prime}\right)
\end{align*}
for $r=-1$ and $p\left(  s\right)  =\frac{1}{2\left(  y-x\right)  }$ on
$\left[  x.x^{\prime}\right]  .$%
\[
Hp_{1}\left(  t\right)  \leq Pp_{1}\left(  t\right)  .\text{ }%
\]%
\[
H_{1}\left(  t\right)  \leq P_{1}\left(  t\right)
\]
for $p\left(  s\right)  =\frac{1}{2\left(  y-x\right)  }$ on $\left[
x.x^{\prime}\right]  .$
\end{proposition}

\begin{proposition}
\label{p52}The following inequalities hold:%
\begin{align*}
&  L_{r}^{r}\left(  p;x,y\right)  +L_{r}^{r}\left(  p;y^{\prime},x^{\prime
}\right) \\
&  \leq2L_{r}^{r}\left(  p\left(  2s-x\right)  ;x,\frac{x+y}{2}\right)
+2L_{r}^{r}\left(  p\left(  2s-x^{\prime}\right)  ;\frac{x^{\prime}+y^{\prime
}}{2},x^{\prime}\right) \\
&  \leq\int_{0}^{1}Pp_{1}\left(  t\right)  dt\\
&  \leq A\left(  A\left(  x^{r},\left(  x^{\prime}\right)  ^{r}\right)
{\displaystyle\int\nolimits_{\Omega}}
p\left(  s\right)  ds,L_{r}^{r}\left(  p;x,y\right)  +L_{r}^{r}\left(
p;y^{\prime},x^{\prime}\right)  \right)  .
\end{align*}%
\begin{align*}
&  A\left(  L_{r}^{r}\left(  x,y\right)  ,L_{r}\left(  y^{\prime},x^{\prime
}\right)  \right) \\
&  \leq A\left(  L_{r}^{r}\left(  x,\frac{x+y}{2}\right)  ,L_{r}^{r}\left(
\frac{x^{\prime}+y^{\prime}}{2},x^{\prime}\right)  \right) \\
&  \leq\int_{0}^{1}P_{1}\left(  t\right)  dt\leq A\left(  A\left(
x^{r},\left(  x^{\prime}\right)  ^{r}\right)  ,A\left(  L_{r}^{r}\left(
x,y\right)  ,L_{r}^{r}\left(  y^{\prime},x^{\prime}\right)  \right)  \right)
\end{align*}
for $r\neq-1$ and $p\left(  s\right)  =\frac{1}{2\left(  y-x\right)  }$ on
$\left[  x.x^{\prime}\right]  .$%
\begin{align*}
&  H^{-1}\left(  L\left(  x,y\right)  ,L\left(  y^{\prime},x^{\prime}\right)
\right) \\
&  \leq H^{-1}\left(  L\left(  x,\frac{x+y}{2}\right)  ,L\left(
\frac{x^{\prime}+y^{\prime}}{2},x^{\prime}\right)  \right) \\
&  \leq\int_{0}^{1}P_{1}\left(  t\right)  dt\leq H^{-1}\left(  H\left(
x,x^{\prime}\right)  ,H\left(  L\left(  x,y\right)  ,L\left(  y^{\prime
},x^{\prime}\right)  \right)  \right)
\end{align*}
for $r=-1$ and $p\left(  s\right)  =\frac{1}{2\left(  y-x\right)  }$ on
$\left[  x.x^{\prime}\right]  .$%
\begin{align*}
0  &  \leq2\left(  y-x\right)  t\left[  A\left(  L_{r}^{r}\left(  x,y\right)
,L_{r}^{r}\left(  y^{\prime},x^{\prime}\right)  \right)  -A\left(
y^{r},\left(  y^{\prime}\right)  ^{r}\right)  \right]  \cdot\inf
\limits_{s\in\Omega}p\left(  s\right) \\
&  \leq Pp_{1}\left(  t\right)  -L_{r}^{r}\left(  p;x,y\right)  -L_{r}%
^{r}\left(  p;y^{\prime},x^{\prime}\right)
\end{align*}
for $t\in\left[  0,1\right]  $ and $r\neq-1.$%
\begin{align*}
0  &  \leq2\left(  y-x\right)  t\left[  H^{-1}\left(  L\left(  x,y\right)
,L\left(  y^{\prime},x^{\prime}\right)  \right)  -H^{-1}\left(  y,y^{\prime
}\right)  \right]  \cdot\inf\limits_{s\in\Omega}p\left(  s\right) \\
&  \leq Pp_{1}\left(  t\right)  -L_{-1}^{-1}\left(  p;x,y\right)  -L_{-1}%
^{-1}\left(  p;y^{\prime},x^{\prime}\right)
\end{align*}
for $t\in\left[  0,1\right]  $ and $r=-1.$%
\begin{align*}
0  &  \leq t\left[  A\left(  L_{r}^{r}\left(  x,y\right)  ,L_{r}^{r}\left(
y^{\prime},x^{\prime}\right)  \right)  -A\left(  y^{r},\left(  y^{\prime
}\right)  ^{r}\right)  \right] \\
&  \leq P_{1}\left(  t\right)  -A\left(  L_{r}^{r}\left(  x,y\right)
,L_{r}^{r}\left(  y^{\prime},x^{\prime}\right)  \right)
\end{align*}
for $t\in\left[  0,1\right]  ,$ $r\neq-1$ and $p\left(  s\right)  =\frac
{1}{2\left(  y-x\right)  }$ on $\left[  x.x^{\prime}\right]  .$%
\begin{align*}
0  &  \leq t\left[  H^{-1}\left(  L\left(  x,y\right)  ,L\left(  y^{\prime
},x^{\prime}\right)  \right)  -H^{-1}\left(  y,y^{\prime}\right)  \right] \\
&  \leq P_{1}\left(  t\right)  -H^{-1}\left(  L\left(  x,y\right)  ,L\left(
y^{\prime},x^{\prime}\right)  \right)
\end{align*}
for $t\in\left[  0,1\right]  ,$ $r=-1$ and $p\left(  s\right)  =\frac
{1}{2\left(  y-x\right)  }$ on $\left[  x.x^{\prime}\right]  .$%
\begin{align*}
0  &  \leq Pp_{1}\left(  t\right)  -A\left(  y^{r},\left(  y^{\prime}\right)
^{r}\right)  \int_{\Omega}p\left(  s\right)  ds\\
&  \leq\frac{r\left(  y-x\right)  \left(  \left(  x^{\prime}\right)
^{r-1}-x^{r-1}\right)  }{2}\int_{\Omega}p\left(  s\right)  ds\text{ \ }\left(
t\in\left[  0,1\right]  \right)  .
\end{align*}%
\[
0\leq P_{1}\left(  t\right)  -A\left(  y^{r},\left(  y^{\prime}\right)
^{r}\right)  \leq\frac{r\left(  y-x\right)  \left(  \left(  x^{\prime}\right)
^{r-1}-x^{r-1}\right)  }{2}%
\]
for $t\in\left[  0,1\right]  ,$ $r\neq-1$ and $p\left(  s\right)  =\frac
{1}{2\left(  y-x\right)  }$ on $\left[  x.x^{\prime}\right]  .$%
\[
0\leq P_{1}\left(  t\right)  -H^{-1}\left(  y,y^{\prime}\right)  \leq
\frac{\left(  y-x\right)  \left(  x^{-2}-\left(  x^{\prime}\right)
^{-2}\right)  }{2}%
\]
for $t\in\left[  0,1\right]  ,$ $r=-1$ and $p\left(  s\right)  =\frac
{1}{2\left(  y-x\right)  }$ on $\left[  x.x^{\prime}\right]  .$%
\begin{align*}
0  &  \leq A\left(  x^{r},\left(  x^{\prime}\right)  ^{r}\right)  \int%
_{\Omega}p\left(  s\right)  ds-Pp_{1}\left(  t\right) \\
&  \leq\frac{r\left(  y-x\right)  \left(  \left(  x^{\prime}\right)
^{r-1}-x^{r-1}\right)  }{2}\int_{\Omega}p\left(  s\right)  ds\text{ \ }\left(
t\in\left[  0,1\right]  \right)  .
\end{align*}%
\[
0\leq A\left(  x^{r},\left(  x^{\prime}\right)  ^{r}\right)  -P_{1}\left(
t\right)  \leq\frac{r\left(  y-x\right)  \left(  \left(  x^{\prime}\right)
^{r-1}-x^{r-1}\right)  }{2}%
\]
for $t\in\left[  0,1\right]  ,$ $r\neq-1$ and $p\left(  s\right)  =\frac
{1}{2\left(  y-x\right)  }$ on $\left[  x.x^{\prime}\right]  .$%
\[
0\leq H^{-1}\left(  x,x^{\prime}\right)  -P_{1}\left(  t\right)  \leq
\frac{\left(  y-x\right)  \left(  x^{-2}-\left(  x^{\prime}\right)
^{-2}\right)  }{2}%
\]
for $t\in\left[  0,1\right]  ,$ $r=-1$ and $p\left(  s\right)  =\frac
{1}{2\left(  y-x\right)  }$ on $\left[  x.x^{\prime}\right]  .$%
\begin{align*}
0  &  \leq Pp_{1}\left(  t\right)  -Hp_{1}\left(  t\right) \\
&  \leq\frac{r\left(  y-x\right)  \left(  \left(  x^{\prime}\right)
^{r-1}-x^{r-1}\right)  }{2}\int_{\Omega}p\left(  s\right)  ds\text{ \ }\left(
t\in\left[  0,1\right]  \right)  .
\end{align*}%
\[
0\leq P_{1}\left(  t\right)  -H_{1}\left(  t\right)  \leq\frac{r\left(
y-x\right)  \left(  \left(  x^{\prime}\right)  ^{r-1}-x^{r-1}\right)  }{2}%
\]
for $t\in\left[  0,1\right]  $ and $p\left(  s\right)  =\frac{1}{2\left(
y-x\right)  }$ on $\left[  x.x^{\prime}\right]  .$
\end{proposition}

\begin{proposition}
\label{p53}$Fp_{1}$\textit{\ is symmetric about }$\frac{1}{2},$%
\textit{\ decreasing on }$\left[  0,\frac{1}{2}\right]  $\textit{\ and
increasing on }$\left[  \frac{1}{2},1\right]  .$ \textit{The following
identities and inequalities hold:}%
\[
\sup\limits_{t\in\left[  0,1\right]  }Fp_{1}\left(  t\right)  =\left(
L_{r}^{r}\left(  p;x,y\right)  +L_{r}^{r}\left(  p;y^{\prime},x^{\prime
}\right)  \right)  \int_{\Omega}p\left(  s\right)  ds.
\]%
\[
\sup\limits_{t\in\left[  0,1\right]  }F_{1}\left(  t\right)  =A\left(
L_{r}^{r}\left(  x,y\right)  ,L_{r}^{r}\left(  y^{\prime},x^{\prime}\right)
\right)
\]
for $r\neq-1$ and $p\left(  s\right)  =\frac{1}{2\left(  y-x\right)  }$ on
$\left[  x.x^{\prime}\right]  .$%
\[
\sup\limits_{t\in\left[  0,1\right]  }F_{1}\left(  t\right)  =H^{-1}\left(
L\left(  x,y\right)  ,L\left(  y^{\prime},x^{\prime}\right)  \right)
\]
for $r=-1$ and $p\left(  s\right)  =\frac{1}{2\left(  y-x\right)  }$ on
$\left[  x.x^{\prime}\right]  .$%
\[
\inf\limits_{t\in\left[  0,1\right]  }Fp_{1}\left(  t\right)  =\int_{\Omega
}\int_{\Omega}A^{r}\left(  s,u\right)  p\left(  s\right)  p\left(  u\right)
dsdu.
\]%
\[
\inf\limits_{t\in\left[  0,1\right]  }F_{1}\left(  t\right)  =\frac
{1}{4\left(  y-x\right)  ^{2}}\int_{\Omega}\int_{\Omega}A^{r}\left(
s,u\right)  dsdu
\]
for $p\left(  s\right)  =\frac{1}{2\left(  y-x\right)  }$ on $\left[
x.x^{\prime}\right]  .$%
\[
A\left(  Hp_{1}\left(  t\right)  ,Hp_{2}\left(  t\right)  \right)
\int_{\Omega}p\left(  s\right)  ds\leq Fp_{1}\left(  t\right)  \text{\ \ \ }%
\left(  t\in\left[  0,1\right]  \right)  .
\]%
\[
A\left(  H_{1}\left(  t\right)  ,H_{2}\left(  t\right)  \right)  \leq
F_{1}\left(  t\right)
\]
for $t\in\left[  0,1\right]  $ and $p\left(  s\right)  =\frac{1}{2\left(
y-x\right)  }$ on $\left[  x.x^{\prime}\right]  .$%
\begin{align*}
&  A\left(  A\left(  y^{r},\left(  y^{\prime}\right)  ^{r}\right)
,A^{r}\left(  y,y^{\prime}\right)  \right)  \left(  \int_{\Omega}p\left(
s\right)  ds\right)  ^{2}\\
&  \leq\int_{\Omega}\int_{\Omega}A^{r}\left(  s,u\right)  p\left(  s\right)
p\left(  u\right)  dsdu
\end{align*}%
\[
A\left(  A\left(  y^{r},\left(  y^{\prime}\right)  ^{r}\right)  ,A^{r}\left(
y,y^{\prime}\right)  \right)  \leq\frac{1}{4\left(  y-x\right)  ^{2}}%
\int_{\Omega}\int_{\Omega}A^{r}\left(  s,u\right)  dsdu
\]
for $r\neq-1$ and $p\left(  s\right)  =\frac{1}{2\left(  y-x\right)  }$ on
$\left[  x.x^{\prime}\right]  .$%
\[
H^{-1}\left(  H\left(  y,y^{\prime}\right)  ,A\left(  y,y^{\prime}\right)
\right)  \leq\frac{1}{4\left(  y-x\right)  ^{2}}\int_{\Omega}\int_{\Omega
}A^{-1}\left(  s,u\right)  dsdu
\]
for $r=-1$ and $p\left(  s\right)  =\frac{1}{2\left(  y-x\right)  }$ on
$\left[  x.x^{\prime}\right]  .$
\end{proposition}

\begin{proposition}
\label{p54}The following inequalities hold:%
\[
Hp_{1}\left(  t\right)  \leq G_{1}\left(  t\right)  \int_{\Omega}p\left(
s\right)  ds\leq Pp_{1}\left(  t\right)  \text{\ }\left(  t\in\left[
0,1\right]  \right)  .
\]%
\[
H_{1}\left(  t\right)  \leq G_{1}\left(  t\right)  \leq P_{1}\left(  t\right)
\]
for $t\in\left[  0,1\right]  $ and $p\left(  s\right)  =\frac{1}{2\left(
y-x\right)  }$ on $\left[  x.x^{\prime}\right]  .$%
\begin{align*}
&  2L_{r}^{r}\left(  p\left(  2s-y\right)  ;\frac{x+y}{2},y\right)
+2L_{r}^{r}\left(  p\left(  2s-y^{\prime}\right)  ;y^{\prime},\frac{x^{\prime
}+y^{\prime}}{2}\right) \\
&  \leq A\left(  A^{r}\left(  x,y\right)  ,A^{r}\left(  x^{\prime},y^{\prime
}\right)  \right)  \int_{\Omega}p\left(  s\right)  ds\\
&  \leq\int_{x}^{y}A\left(  \left(  \frac{s+x}{2}\right)  ^{r},\left(
\frac{x+2y-s}{2}\right)  ^{r}\right)  p\left(  s\right)  ds\\
&  +\int_{y^{\prime}}^{x^{\prime}}A\left(  \left(  \frac{s+x}{2}\right)
^{r},\left(  \frac{x^{\prime}+2y^{\prime}-s}{2}\right)  ^{r}\right)  p\left(
s\right)  ds\\
&  \leq A\left(  A\left(  x^{r},\left(  x^{\prime}\right)  ^{r}\right)
,A\left(  y^{r},\left(  y^{\prime}\right)  ^{r}\right)  \right)  \int_{\Omega
}p\left(  s\right)  ds.
\end{align*}%
\begin{align*}
&  A\left(  L_{r}^{r}\left(  \frac{x+y}{2},y\right)  ,L_{r}^{r}\left(
y^{\prime},\frac{x^{\prime}+y^{\prime}}{2}\right)  \right) \\
&  \leq A\left(  A^{r}\left(  x,y\right)  ,A^{r}\left(  x^{\prime},y^{\prime
}\right)  \right) \\
&  \leq A\left(  L_{r}^{r}\left(  x,y\right)  ,L_{r}^{r}\left(  x^{\prime
},y^{\prime}\right)  \right) \\
&  \leq A\left(  A\left(  x^{r},\left(  x^{\prime}\right)  ^{r}\right)
,A\left(  y^{r},\left(  y^{\prime}\right)  ^{r}\right)  \right)
\end{align*}
for $r\neq-1$ and $p\left(  s\right)  =\frac{1}{2\left(  y-x\right)  }$ on
$\left[  x.x^{\prime}\right]  .$%
\begin{align*}
&  H^{-1}\left(  L\left(  \frac{x+y}{2},y\right)  ,L\left(  y^{\prime}%
,\frac{x^{\prime}+y^{\prime}}{2}\right)  \right) \\
&  \leq H^{-1}\left(  A\left(  x,y\right)  ,A\left(  x^{\prime},y^{\prime
}\right)  \right) \\
&  \leq H^{-1}\left(  L\left(  x,y\right)  ,L\left(  x^{\prime},y^{\prime
}\right)  \right) \\
&  \leq H^{-1}\left(  H\left(  x,x^{\prime}\right)  ,H\left(  y,y^{\prime
}\right)  \right)
\end{align*}
for $r=-1$ and $p\left(  s\right)  =\frac{1}{2\left(  y-x\right)  }$ on
$\left[  x.x^{\prime}\right]  .$%
\begin{align*}
&  2L_{r}^{r}\left(  p\left(  2s-y^{\prime}\right)  ;\frac{x+y^{\prime}}%
{2},\frac{y+y^{\prime}}{2}\right)  +2L_{r}^{r}\left(  p\left(  2s-y\right)
;\frac{y+y^{\prime}}{2},\frac{x^{\prime}+y}{2}\right) \\
&  \leq A\left(  A^{r}\left(  x,y^{\prime}\right)  ,A^{r}\left(  x^{\prime
},y\right)  \right)  \int_{\Omega}p\left(  s\right)  ds\\
&  \leq\int_{x}^{y}A\left(  \left(  \frac{s+x}{2}\right)  ^{r},\left(
\frac{x+2y^{\prime}-s}{2}\right)  ^{r}\right)  p\left(  s\right)  ds\\
&  +\int_{y^{\prime}}^{x^{\prime}}A\left(  \left(  \frac{s+x^{\prime}}%
{2}\right)  ^{r},\left(  \frac{x^{\prime}+2y-s}{2}\right)  ^{r}\right)
p\left(  s\right)  ds\\
&  \leq A\left(  A\left(  x^{r},\left(  x^{\prime}\right)  ^{r}\right)
,A\left(  y^{r},\left(  y^{\prime}\right)  ^{r}\right)  \right)  \int_{\Omega
}p\left(  s\right)  ds.
\end{align*}%
\begin{align*}
&  A\left(  L_{r}^{r}\left(  \frac{x+y^{\prime}}{2},\frac{y+y^{\prime}}%
{2}\right)  ,L_{r}^{r}\left(  \frac{y+y^{\prime}}{2},\frac{x^{\prime}+y}%
{2}\right)  \right) \\
&  \leq A\left(  A^{r}\left(  x,y^{\prime}\right)  ,A^{r}\left(  x^{\prime
},y\right)  \right) \\
&  \leq\frac{1}{4}\left[  L_{r}^{r}\left(  x,\frac{x+y}{2}\right)  +L_{r}%
^{r}\left(  \frac{x+2y^{\prime}-y}{2},y^{\prime}\right)  \right. \\
&  +\left.  L_{r}^{r}\left(  \frac{x^{\prime}+y^{\prime}}{2},x^{\prime
}\right)  +L_{r}^{r}\left(  y,\frac{x^{\prime}+2y-y^{\prime}}{2}\right)
\right] \\
&  \leq A\left(  A\left(  x^{r},\left(  x^{\prime}\right)  ^{r}\right)
,A\left(  y^{r},\left(  y^{\prime}\right)  ^{r}\right)  \right)
\end{align*}
for $r\neq-1$ and $p\left(  s\right)  =\frac{1}{2\left(  y-x\right)  }$ on
$\left[  x.x^{\prime}\right]  .$%
\begin{align*}
&  H^{-1}\left(  L\left(  \frac{x+y^{\prime}}{2},\frac{y+y^{\prime}}%
{2}\right)  ,L\left(  \frac{y+y^{\prime}}{2},\frac{x^{\prime}+y}{2}\right)
\right) \\
&  \leq H^{-1}\left(  A\left(  x,y^{\prime}\right)  ,A\left(  x^{\prime
},y\right)  \right) \\
&  \leq\frac{1}{2}H\left(  L\left(  x,\frac{x+y}{2}\right)  ,L\left(
\frac{x+2y^{\prime}-y}{2},y^{\prime}\right)  \right) \\
&  +\frac{1}{2}H\left(  L\left(  \frac{x^{\prime}+y^{\prime}}{2},x^{\prime
}\right)  ,L\left(  y,\frac{x^{\prime}+2y-y^{\prime}}{2}\right)  \right) \\
&  \leq H^{-1}\left(  H\left(  x,x^{\prime}\right)  ,H\left(  y,y^{\prime
}\right)  \right)
\end{align*}
for $r=-1$ and $p\left(  s\right)  =\frac{1}{2\left(  y-x\right)  }$ on
$\left[  x.x^{\prime}\right]  .$%
\begin{align*}
0  &  \leq Hp_{1}\left(  t\right)  -A\left(  y^{r},\left(  y^{\prime}\right)
^{r}\right)  \int_{\Omega}p\left(  s\right)  ds\\
&  \leq G_{1}\left(  t\right)  \int_{\Omega}p\left(  s\right)  ds-Hp_{1}%
\left(  t\right)  \text{\ }\left(  t\in\left[  0,1\right]  \right)  .
\end{align*}
\textit{\ }%
\[
0\leq H_{1}\left(  t\right)  -A\left(  y^{r},\left(  y^{\prime}\right)
^{r}\right)  \leq G_{1}\left(  t\right)  -H_{1}\left(  t\right)
\]
for $t\in\left[  0,1\right]  ,$ $r\neq-1$ and $p\left(  s\right)  =\frac
{1}{2\left(  y-x\right)  }$ on $\left[  x.x^{\prime}\right]  .$%
\[
0\leq H_{1}\left(  t\right)  -H^{-1}\left(  y,y^{\prime}\right)  \leq
G_{1}\left(  t\right)  -H_{1}\left(  t\right)
\]
for $t\in\left[  0,1\right]  ,$ $r=-1$ and $p\left(  s\right)  =\frac
{1}{2\left(  y-x\right)  }$ on $\left[  x.x^{\prime}\right]  .$%
\begin{align*}
0  &  \leq Pp_{1}\left(  t\right)  -G_{1}\left(  t\right)  \int_{\Omega
}p\left(  s\right)  ds\\
&  \leq A\left(  x^{r},\left(  x^{\prime}\right)  ^{r}\right)  \int_{\Omega
}p\left(  s\right)  ds-Pp_{1}\left(  t\right)  \text{\ }\left(  t\in\left[
0,1\right]  \right)  .
\end{align*}%
\[
0\leq P_{1}\left(  t\right)  -G_{1}\left(  t\right)  \leq A\left(
x^{r},\left(  x^{\prime}\right)  ^{r}\right)  -P_{1}\left(  t\right)
\]
for $t\in\left[  0,1\right]  ,$ $r\neq-1$ and $p\left(  s\right)  =\frac
{1}{2\left(  y-x\right)  }$ on $\left[  x.x^{\prime}\right]  .$%
\[
0\leq P_{1}\left(  t\right)  -G_{1}\left(  t\right)  \leq H^{-1}\left(
x,x^{\prime}\right)  -P_{1}\left(  t\right)
\]
for $t\in\left[  0,1\right]  ,$ $r=-1$ and $p\left(  s\right)  =\frac
{1}{2\left(  y-x\right)  }$ on $\left[  x.x^{\prime}\right]  .$
\end{proposition}

\begin{proposition}
\label{p55}The following inequalities hold:%
\[
A\left(  G_{1}\left(  t\right)  ,G_{2}\left(  t\right)  \right)  \int_{\Omega
}p\left(  s\right)  ds\leq Lp_{1}\left(  t\right)  \leq Pp_{1}\left(
t\right)  \text{\ }\left(  t\in\left[  0,1\right]  \right)  .\text{ }%
\]%
\[
A\left(  G_{1}\left(  t\right)  ,G_{2}\left(  t\right)  \right)  \leq
L_{1}\left(  t\right)  \leq P_{1}\left(  t\right)  \text{ }%
\]
for $t\in\left[  0,1\right]  $ and $p\left(  s\right)  =\frac{1}{2\left(
y-x\right)  }$ on $\left[  x.x^{\prime}\right]  .$%
\[
\sup\limits_{t\in\left[  0,1\right]  }Lp_{1}\left(  t\right)  =Lp_{1}\left(
1\right)  =A\left(  x^{r},\left(  x^{\prime}\right)  ^{r}\right)  \int%
_{\Omega}p\left(  s\right)  ds.
\]%
\[
\sup\limits_{t\in\left[  0,1\right]  }L_{1}\left(  t\right)  =L_{1}\left(
1\right)  =A\left(  x^{r},\left(  x^{\prime}\right)  ^{r}\right)
\]
for $r\neq-1$ and $p\left(  s\right)  =\frac{1}{2\left(  y-x\right)  }$ on
$\left[  x.x^{\prime}\right]  .$%
\[
\sup\limits_{t\in\left[  0,1\right]  }L_{1}\left(  t\right)  =L_{1}\left(
1\right)  =H^{-1}\left(  x,x^{\prime}\right)
\]
for $r=-1$ and $p\left(  s\right)  =\frac{1}{2\left(  y-x\right)  }$ on
$\left[  x.x^{\prime}\right]  .$%
\begin{align*}
&  A\left(  Hp_{1}\left(  1-t\right)  ,Hp_{2}\left(  1-t\right)  \right)
\int_{\Omega}p\left(  s\right)  ds\\
&  \leq Fp_{1}\left(  t\right)  \leq Lp_{1}\left(  t\right)  \int_{\Omega
}p\left(  s\right)  ds\text{\ \ \ }\left(  t\in\left[  0,1\right]  \right)  .
\end{align*}%
\[
A\left(  H_{1}\left(  1-t\right)  ,H_{2}\left(  1-t\right)  \right)  \leq
F_{1}\left(  t\right)  \leq L_{1}\left(  t\right)  \text{\ \ }\left(
t\in\left[  0,1\right]  \right)
\]
for $t\in\left[  0,1\right]  $ and $p\left(  s\right)  =\frac{1}{2\left(
y-x\right)  }$ on $\left[  x.x^{\prime}\right]  .$%
\begin{align*}
&  A\left(  Hp_{1}\left(  t\right)  ,Hp_{2}\left(  t\right)  \right)
\int_{\Omega}p\left(  s\right)  ds\\
&  \leq Fp_{1}\left(  t\right)  \leq Lp_{1}\left(  t\right)  \int_{\Omega
}p\left(  s\right)  ds\text{\ \ \ }\left(  t\in\left[  0,1\right]  \right)  .
\end{align*}%
\[
A\left(  H_{1}\left(  t\right)  ,H_{2}\left(  t\right)  \right)  \leq
F_{1}\left(  t\right)  \leq L_{1}\left(  t\right)  \text{\ \ }\left(
t\in\left[  0,1\right]  \right)
\]
for $t\in\left[  0,1\right]  $ and $p\left(  s\right)  =\frac{1}{2\left(
y-x\right)  }$ on $\left[  x.x^{\prime}\right]  .$%
\begin{align*}
&  A\left(  A\left(  Hp_{1}\left(  1-t\right)  ,Hp_{2}\left(  1-t\right)
\right)  ,A\left(  Hp_{1}\left(  t\right)  ,Hp_{2}\left(  t\right)  \right)
\right)  \int_{\Omega}p\left(  s\right)  ds\\
&  \leq Fp_{1}\left(  t\right)  \leq Lp_{1}\left(  t\right)  \int_{\Omega
}p\left(  s\right)  ds\text{\ \ \ }\left(  t\in\left[  0,1\right]  \right)  .
\end{align*}%
\[
A\left(  A\left(  H_{1}\left(  1-t\right)  ,H_{2}\left(  1-t\right)  \right)
,A\left(  H_{1}\left(  t\right)  ,H_{2}\left(  t\right)  \right)  \right)
\leq F_{1}\left(  t\right)  \leq L_{1}\left(  t\right)
\]
for $t\in\left[  0,1\right]  $ and $p\left(  s\right)  =\frac{1}{2\left(
y-x\right)  }$ on $\left[  x.x^{\prime}\right]  .$%
\begin{align*}
0  &  \leq Fp_{1}\left(  t\right)  -A\left(  Hp_{1}\left(  t\right)
,Hp_{2}\left(  t\right)  \right)  \int_{\Omega}p\left(  s\right)  ds\\
&  \leq Lp_{1}\left(  1-t\right)  \int_{\Omega}p\left(  s\right)
ds-Fp_{1}\left(  t\right)
\end{align*}
as $p\left(  s\right)  =p\left(  x+x^{\prime}-s\right)  $ $\left(  s\in\left[
x,y\right]  \right)  $ and $p\left(  s\right)  =p\left(  x+y-s\right)  $
$\left(  s\in\left[  x,\frac{x+y}{2}\right]  \right)  .$%
\[
0\leq F_{1}\left(  t\right)  -A\left(  H_{1}\left(  t\right)  ,H_{2}\left(
t\right)  \right)  \leq L_{1}\left(  1-t\right)  -F_{1}\left(  t\right)
\]
for $t\in\left[  0,1\right]  $ and $p\left(  s\right)  =\frac{1}{2\left(
y-x\right)  }$ on $\left[  x.x^{\prime}\right]  .$
\end{proposition}

\begin{proposition}
\label{p56}$Ip_{1}$\ is increasing on $\left[  0,1\right]  $\ and the
following inequalities hold for all $t\in\left[  0,1\right]  :$%
\begin{align*}
&  2A\left(  y^{r},\left(  y^{\prime}\right)  ^{r}\right)  \int_{x}%
^{x^{\prime}}p\left(  s\right)  ds\\
&  =Ip_{1}\left(  0\right)  \leq Ip_{1}\left(  t\right)  \leq Ip_{1}\left(
1\right) \\
&  =\int\nolimits_{x}^{x^{\prime}}\left[  \left(  \left(  1-\alpha\right)
x+\alpha s\right)  ^{r}+\left(  \left(  1-\alpha\right)  x^{\prime}+\alpha
s\right)  ^{r}\right]  p\left(  s\right)  ds.
\end{align*}%
\begin{align*}
Ip_{1}\left(  t\right)   &  \leq2\left(  1-t\right)  A\left(  y^{r},\left(
y^{\prime}\right)  ^{r}\right)  \int_{x}^{x^{\prime}}p\left(  s\right)  ds\\
&  +t\int\nolimits_{x}^{x^{\prime}}\left[  \left(  \left(  1-\alpha\right)
x+\alpha s\right)  ^{r}+\left(  \left(  1-\alpha\right)  x^{\prime}+\alpha
s\right)  ^{r}\right]  p\left(  s\right)  ds\\
&  \leq\int\nolimits_{x}^{x^{\prime}}\left[  \left(  \left(  1-\alpha\right)
x+\alpha s\right)  ^{r}+\left(  \left(  1-\alpha\right)  x^{\prime}+\alpha
s\right)  ^{r}\right]  p\left(  s\right)  ds\\
&  \leq2A\left(  x^{r},\left(  x^{\prime}\right)  ^{r}\right)  \int%
_{x}^{x^{\prime}}p\left(  s\right)  ds.
\end{align*}

\end{proposition}

\begin{proposition}
\label{p57}Let $m,$ $m^{\prime}$\ \textit{be defined as in Proposition
\ref{p9}. Then}
\end{proposition}

\begin{enumerate}
\item $Ip_{2}$\textit{\ is decreasing on }$\left[  0,m\right]  $\textit{\ and
increasing on }$\left[  m^{\prime},1\right]  .$

\item \textit{The following inequalities hold for all }$t\in\left[
0,1\right]  :$%
\begin{align*}
&  2A^{r}\left(  x,x^{\prime}\right)  \int_{x}^{x^{\prime}}p\left(  s\right)
ds\\
&  \leq Ip_{2}\left(  t\right)  \leq2\left(  1-t\right)  A\left(
y^{r},\left(  y^{\prime}\right)  ^{r}\right)  \int_{x}^{x^{\prime}}p\left(
s\right)  ds\\
&  +t\int\nolimits_{x}^{x^{\prime}}\left[  \left(  \left(  1-\alpha\right)
x+\alpha s\right)  ^{r}+\left(  \left(  1-\alpha\right)  x^{\prime}+\alpha
s\right)  ^{r}\right]  p\left(  s\right)  ds\\
&  \leq\int\nolimits_{x}^{x^{\prime}}\left[  \left(  \left(  1-\alpha\right)
x+\alpha s\right)  ^{r}+\left(  \left(  1-\alpha\right)  x^{\prime}+\alpha
s\right)  ^{r}\right]  p\left(  s\right)  ds\\
&  \leq2A\left(  x^{r},\left(  x^{\prime}\right)  ^{r}\right)  \int%
_{x}^{x^{\prime}}p\left(  s\right)  ds.
\end{align*}%
\[
Ip_{2}\left(  t\right)  \leq Ip_{1}\left(  t\right)  .
\]

\end{enumerate}

\begin{proposition}
\label{p58}$Jp$\ is increasing on $\left[  0,1\right]  $\ and the following
inequality holds for all $t\in\left[  0,1\right]  :$%
\begin{align*}
&  2A\left(  A^{r}\left(  x,y\right)  ,A^{r}\left(  x^{\prime},y^{\prime
}\right)  \right)  \int_{x}^{x^{\prime}}p\left(  s\right)  ds\\
&  =Jp\left(  0\right)  \leq Jp\left(  t\right)  \leq Jp\left(  1\right) \\
&  =\int\nolimits_{x}^{x^{\prime}}\left[  \left(  \left(  1-\alpha\right)
x+\alpha s\right)  ^{r}+\left(  \left(  1-\alpha\right)  x^{\prime}+\alpha
s\right)  ^{r}\right]  p\left(  s\right)  ds.
\end{align*}

\end{proposition}

\begin{proposition}
\label{p59}\textit{For all }$t\in\left[  0,1\right]  ,$ we have $Ip_{1}\left(
t\right)  \leq Jp\left(  t\right)  .$
\end{proposition}

\begin{proposition}
\label{p60}$Mp$\ is increasing on $\left[  0,1\right]  $\ and the following
inequalities hold for all $t\in\left[  0,1\right]  :$%
\begin{align*}
&  \int\nolimits_{x}^{x^{\prime}}\left[  \left(  \left(  1-\alpha\right)
x+\alpha s\right)  ^{r}+\left(  \left(  1-\alpha\right)  x^{\prime}+\alpha
s\right)  ^{r}\right]  p\left(  s\right)  ds\\
&  =Mp\left(  0\right)  \leq Mp\left(  t\right)  \leq Mp\left(  1\right) \\
&  =A\left(  A\left(  x^{r},y^{r}\right)  ,A\left(  \left(  y^{\prime}\right)
^{r},\left(  x^{\prime}\right)  ^{r}\right)  \right)  \int\nolimits_{x}%
^{x^{\prime}}p\left(  s\right)  ds.
\end{align*}%
\begin{align*}
&  Mp\left(  t\right) \\
&  \leq\left(  1-t\right)  \int\nolimits_{x}^{x^{\prime}}\left[  \left(
\left(  1-\alpha\right)  x+\alpha s\right)  ^{r}+\left(  \left(
1-\alpha\right)  x^{\prime}+\alpha s\right)  ^{r}\right]  p\left(  s\right)
ds\\
&  +tA\left(  A\left(  x^{r},y^{r}\right)  ,A\left(  \left(  y^{\prime
}\right)  ^{r},\left(  x^{\prime}\right)  ^{r}\right)  \right)  \int%
\nolimits_{x}^{x^{\prime}}p\left(  s\right)  ds\\
&  \leq A\left(  A\left(  x^{r},y^{r}\right)  ,A\left(  \left(  y^{\prime
}\right)  ^{r},\left(  x^{\prime}\right)  ^{r}\right)  \right)  \int%
\nolimits_{x}^{x^{\prime}}p\left(  s\right)  ds\\
&  \leq2A\left(  x^{r},\left(  x^{\prime}\right)  ^{r}\right)  \int%
_{x}^{x^{\prime}}p\left(  s\right)  ds.
\end{align*}

\end{proposition}

\begin{proposition}
\label{p61}$Np$\ is increasing on $\left[  0,1\right]  $\ and the following
inequalities hold for all $t\in\left[  0,1\right]  :$%
\begin{align*}
&  \int\nolimits_{x}^{x^{\prime}}\left[  \left(  \left(  1-\alpha\right)
x+\alpha s\right)  ^{r}+\left(  \left(  1-\alpha\right)  x^{\prime}+\alpha
s\right)  ^{r}\right]  p\left(  s\right)  ds\\
&  =Np\left(  0\right)  \leq Np\left(  t\right)  \leq Np\left(  1\right)
=2A\left(  x^{r},\left(  x^{\prime}\right)  ^{r}\right)  \int_{x}^{x^{\prime}%
}p\left(  s\right)  ds.
\end{align*}%
\begin{align*}
&  Np\left(  t\right) \\
&  \leq\left(  1-t\right)  \int\nolimits_{x}^{x^{\prime}}\left[  \left(
\left(  1-\alpha\right)  x+\alpha s\right)  ^{r}+\left(  \left(
1-\alpha\right)  x^{\prime}+\alpha s\right)  ^{r}\right]  p\left(  s\right)
ds\\
&  +2tA\left(  x^{r},\left(  x^{\prime}\right)  ^{r}\right)  \int%
_{x}^{x^{\prime}}p\left(  s\right)  ds\\
&  \leq2A\left(  x^{r},\left(  x^{\prime}\right)  ^{r}\right)  \int%
_{x}^{x^{\prime}}p\left(  s\right)  ds.
\end{align*}

\end{proposition}

\begin{proposition}
\label{p62}\textit{For all }$t\in\left[  0,1\right]  ,$ we have $Mp\left(
t\right)  \leq Np\left(  t\right)  .$
\end{proposition}

\begin{proposition}
\label{p63}The following inequalities hold:%
\begin{align*}
2  &  A\left(  y^{r},\left(  y^{\prime}\right)  ^{r}\right)  \int%
_{x}^{x^{\prime}}p\left(  s\right)  ds\\
&  \leq\frac{2}{\alpha}\left[  L_{r}^{r}\left(  p\left(  \frac{1}{\alpha
}\left(  2s-x-y\right)  +x\right)  ;\frac{x+y}{2},y\right)  \right. \\
&  \text{ \ \ \ \ \ \ \ \ \ \ }+\left.  L_{r}^{r}\left(  p\left(  \frac
{1}{\alpha}\left(  2s-y^{\prime}-x^{\prime}\right)  +x^{\prime}\right)
;y^{\prime},\frac{x^{\prime}+y^{\prime}}{2}\right)  \right] \\
&  \leq\int_{0}^{1}Ip_{1}\left(  t\right)  dt\\
&  \leq\frac{1}{2}\left[  2A\left(  y^{r},\left(  y^{\prime}\right)
^{r}\right)  \int_{x}^{x^{\prime}}p\left(  s\right)  ds\right. \\
&  \text{ }+\left.  \int\nolimits_{x}^{x^{\prime}}\left[  \left(  \left(
1-\alpha\right)  x+\alpha s\right)  ^{r}+\left(  \left(  1-\alpha\right)
x^{\prime}+\alpha s\right)  ^{r}\right]  p\left(  s\right)  ds\right]  .
\end{align*}%
\begin{align*}
0  &  \leq\int\nolimits_{x}^{x^{\prime}}\left[  \left(  \left(  1-\alpha
\right)  x+\alpha s\right)  ^{r}+\left(  \left(  1-\alpha\right)  x^{\prime
}+\alpha s\right)  ^{r}\right]  p\left(  s\right)  ds-Ip_{1}\left(  t\right)
\\
&  \leq\frac{2\left(  1-t\right)  \left(  y-x\right)  }{\alpha}\left[
A\left(  x^{r},\left(  x^{\prime}\right)  ^{r}\right)  -A\left(  L_{r}%
^{r}\left(  x,y\right)  ,L_{r}^{r}\left(  y^{\prime},x^{\prime}\right)
\right)  \right]  \left\Vert p\right\Vert _{\infty}%
\end{align*}
for $r\neq-1.$%
\begin{align*}
0  &  \leq\int\nolimits_{x}^{x^{\prime}}\left[  \left(  \left(  1-\alpha
\right)  x+\alpha s\right)  ^{-1}+\left(  \left(  1-\alpha\right)  x^{\prime
}+\alpha s\right)  ^{-1}\right]  p\left(  s\right)  ds-Ip_{1}\left(  t\right)
\\
&  \leq\frac{2\left(  1-t\right)  \left(  y-x\right)  }{\alpha}\left[
H^{-1}\left(  x,x^{\prime}\right)  -H^{-1}\left(  L\left(  x,y\right)
,L\left(  y^{\prime},x^{\prime}\right)  \right)  \right]  \left\Vert
p\right\Vert _{\infty}%
\end{align*}
for $r=-1.$%
\begin{align*}
0  &  \leq2A\left(  x^{r},\left(  x^{\prime}\right)  ^{r}\right)  \int%
_{x}^{x^{\prime}}p\left(  s\right)  ds-Ip_{1}\left(  t\right) \\
&  \leq r\left(  y-x\right)  \left(  \left(  x^{\prime}\right)  ^{r-1}%
-x^{r-1}\right)  \int_{x}^{x^{\prime}}p\left(  s\right)  ds\text{ \ \ }\left(
t\in\left[  0,1\right]  \right)  .
\end{align*}%
\begin{align*}
0  &  \leq Ip_{1}\left(  t\right)  -2A\left(  y^{r},\left(  y^{\prime}\right)
^{r}\right)  \int_{x}^{x^{\prime}}p\left(  s\right)  ds\\
&  \leq r\left(  y-x\right)  \left(  \left(  x^{\prime}\right)  ^{r-1}%
-x^{r-1}\right)  \int_{x}^{x^{\prime}}p\left(  s\right)  ds\text{ \ \ }\left(
t\in\left[  0,1\right]  \right)  .
\end{align*}

\end{proposition}

\begin{proposition}
\label{p64}The following inequalities hold:%
\begin{align*}
&  2A^{r}\left(  y,y^{\prime}\right)  \int_{x}^{x^{\prime}}p\left(  s\right)
ds\\
&  \leq\frac{2}{\alpha}\left[  L_{r}^{r}\left(  p\left(  \frac{1}{\alpha
}\left(  2s-x-y^{\prime}\right)  +x\right)  ;\frac{x+y^{\prime}}{2}%
,\frac{y+y^{\prime}}{2}\right)  \right. \\
&  \text{ \ \ \ \ \ \ \ \ }+\left.  L_{r}^{r}\left(  p\left(  \frac{1}{\alpha
}\left(  2s-y-x^{\prime}\right)  +x^{\prime}\right)  ;\frac{y+y^{\prime}}%
{2},\frac{x^{\prime}+y}{2}\right)  \right] \\
&  \leq\int_{0}^{1}Ip_{2}\left(  t\right)  dt\\
&  \leq A\left(  y^{r},\left(  y^{\prime}\right)  ^{r}\right)  \int%
_{x}^{x^{\prime}}p\left(  s\right)  ds\\
&  \text{ }+\frac{1}{2}\int\nolimits_{x}^{x^{\prime}}\left[  \left(  \left(
1-\alpha\right)  x+\alpha s\right)  ^{r}+\left(  \left(  1-\alpha\right)
x^{\prime}+\alpha s\right)  ^{r}\right]  p\left(  s\right)  ds.
\end{align*}%
\begin{align*}
0  &  \leq\int\nolimits_{x}^{x^{\prime}}\left[  \left(  \left(  1-\alpha
\right)  x+\alpha s\right)  ^{r}+\left(  \left(  1-\alpha\right)  x^{\prime
}+\alpha s\right)  ^{r}\right]  p\left(  s\right)  ds-Ip_{2}\left(  t\right)
\\
&  \leq\frac{2\left(  1-t\right)  \left(  y-x\right)  }{\alpha}\left[
A\left(  x^{r},\left(  x^{\prime}\right)  ^{r}\right)  \frac{y^{\prime}%
-x}{y-x}-A\left(  y^{r},\left(  y^{\prime}\right)  ^{r}\right)  \frac
{y^{\prime}-y}{y-x}\right. \\
&  \left.  -A\left(  L_{r}^{r}\left(  x,y\right)  ,L_{r}^{r}\left(  y^{\prime
},x^{\prime}\right)  \right)  \right]  \left\Vert p\right\Vert _{\infty
}\text{\ }%
\end{align*}
for $t\in\left[  0,1\right]  $ and $r\neq-1.$%
\begin{align*}
0  &  \leq\int\nolimits_{x}^{x^{\prime}}\left[  \left(  \left(  1-\alpha
\right)  x+\alpha s\right)  ^{-1}+\left(  \left(  1-\alpha\right)  x^{\prime
}+\alpha s\right)  ^{-1}\right]  p\left(  s\right)  ds-Ip_{2}\left(  t\right)
\\
&  \leq\frac{2\left(  1-t\right)  \left(  y-x\right)  }{\alpha}\left[
H^{-1}\left(  x,x^{\prime}\right)  \frac{y^{\prime}-x}{y-x}-H^{-1}\left(
y,y^{\prime}\right)  \frac{y^{\prime}-y}{y-x}\right. \\
&  \left.  -H^{-1}\left(  L\left(  x,y\right)  ,L\left(  y^{\prime},x^{\prime
}\right)  \right)  \right]  \left\Vert p\right\Vert _{\infty}\text{\ }%
\end{align*}
for $t\in\left[  0,1\right]  $ and $r=-1.$%
\begin{align*}
0  &  \leq Ip_{2}\left(  t\right)  -2A^{r}\left(  y,y^{\prime}\right)
\int_{x}^{x^{\prime}}p\left(  s\right)  ds\\
&  \leq\frac{r\left(  2x^{\prime}-y-y^{\prime}\right)  \left(  \left(
x^{\prime}\right)  ^{r-1}-x^{r-1}\right)  }{2}\int_{x}^{x^{\prime}}p\left(
s\right)  ds\text{ \ \ }\left(  t\in\left[  0,1\right]  \right)  .
\end{align*}%
\begin{align*}
0  &  \leq2A\left(  x^{r},\left(  x^{\prime}\right)  ^{r}\right)  \int%
_{x}^{x^{\prime}}p\left(  s\right)  ds-Ip_{2}\left(  t\right) \\
&  \leq\frac{r\left(  2x^{\prime}-y-y^{\prime}\right)  \left(  \left(
x^{\prime}\right)  ^{r-1}-x^{r-1}\right)  }{2}\int_{x}^{x^{\prime}}p\left(
s\right)  ds\text{ \ \ }\left(  t\in\left[  0,1\right]  \right)  .
\end{align*}

\end{proposition}

\begin{proposition}
\label{p65}The following inequalities hold for all $t\in\left[  0,1\right]  :$%
\[
Ip_{1}\left(  t\right)  \leq2G_{1}\left(  t\right)  \int_{x}^{x^{\prime}%
}p\left(  s\right)  ds.
\]%
\begin{align*}
0  &  \leq Ip_{1}\left(  t\right)  -2A\left(  y^{r},\left(  y^{\prime}\right)
^{r}\right)  \int_{x}^{x^{\prime}}p\left(  s\right)  ds\\
&  \leq2\left(  x^{\prime}-x\right)  \left[  G_{1}\left(  t\right)
-H_{1}\left(  t\right)  \right]  \left\Vert p\right\Vert _{\infty
}\text{\ \ \ }\left(  t\in\left[  0,1\right]  \right)  .
\end{align*}

\end{proposition}

\begin{proposition}
\label{p66}Let $m,$ $m^{\prime}$\ \textit{be defined as in Proposition
\ref{p9}. Then }the following inequalities hold:%
\[
Ip_{2}\left(  t\right)  \leq\left(  \geq\right)  2G_{2}\left(  t\right)
\int_{x}^{x^{\prime}}p\left(  s\right)  ds\text{ \ \ \textit{as }}t\in\left[
m^{\prime},1\right]  \text{ \ }\left(  t\in\left[  0,m\right]  \right)  .
\]%
\begin{align*}
0  &  \leq Ip_{2}\left(  t\right)  -2A^{r}\left(  y,y^{\prime}\right)
\int_{x}^{x^{\prime}}p\left(  s\right)  ds\\
&  \leq\frac{2}{\alpha}\left[  \left(  y^{\prime}-x\right)  G_{2}\left(
t\right)  -2\left(  y^{\prime}-y\right)  G_{3}\left(  t\right)  -\left(
y-x\right)  H_{2}\left(  t\right)  \right]  \left\Vert p\right\Vert _{\infty
}\text{\ }\left(  t\in\left[  0,1\right]  \right)  .
\end{align*}

\end{proposition}

\begin{proposition}
\label{p67}Let $m^{\prime}$\ \textit{be defined as in Proposition \ref{p9}.
Then} we have:%
\begin{align*}
&  2A\left(  G_{1}\left(  t\right)  ,G_{2}\left(  t\right)  \right)  \int%
_{x}^{x^{\prime}}p\left(  s\right)  ds\\
&  \leq Sp\left(  t\right) \\
&  \leq\left(  1-t\right)  \int\nolimits_{x}^{x^{\prime}}\left[  \left(
\left(  1-\alpha\right)  x+\alpha s\right)  ^{r}+\left(  \left(
1-\alpha\right)  x^{\prime}+\alpha s\right)  ^{r}\right]  p\left(  s\right)
ds\\
&  \qquad\qquad+2tA\left(  x^{r},\left(  x^{\prime}\right)  ^{r}\right)
\int_{x}^{x^{\prime}}p\left(  s\right)  ds\\
&  \leq2A\left(  x^{r},\left(  x^{\prime}\right)  ^{r}\right)  \int%
_{x}^{x^{\prime}}p\left(  s\right)  ds\text{ \ }\left(  t\in\left[
0,1\right]  \right)  .
\end{align*}%
\[
A\left(  Ip_{1}\left(  1-t\right)  ,Ip_{2}\left(  1-t\right)  \right)  \leq
Sp\left(  t\right)  \text{ \ }\left(  t\in\left[  0,1\right]  \right)  .
\]%
\[
A\left(  A\left(  Ip_{1}\left(  t\right)  ,Ip_{2}\left(  t\right)  \right)
,A\left(  Ip_{1}\left(  1-t\right)  ,Ip_{2}\left(  1-t\right)  \right)
\right)  \leq Sp\left(  t\right)  \text{ \ \ }\left(  t\in\left[  m^{\prime
},1\right]  \right)  .
\]%
\[
\sup\limits_{t\in\left[  0,1\right]  }Sp\left(  t\right)  =2A\left(
x^{r},\left(  x^{\prime}\right)  ^{r}\right)  \int_{x}^{x^{\prime}}p\left(
s\right)  ds\text{ }.
\]

\end{proposition}

\begin{proposition}
\label{p68}
\end{proposition}

\begin{enumerate}
\item $Q_{1}$\textit{\ is symmetric about }$\frac{1}{2},$\textit{\ decreasing
on }$\left[  0,\frac{1}{2}\right]  $\textit{\ and increasing on }$\left[
\frac{1}{2},1\right]  .$

\item \textit{The following inequalities hold:}%
\[
G_{1}\left(  \frac{t}{\alpha}\right)  \leq Q_{1}\left(  t\right)  \text{
\quad}\left(  t\in\left[  0,\frac{\alpha}{2}\right]  \right)  .
\]%
\[
G_{1}\left(  \frac{t}{\alpha}\right)  \geq Q_{1}\left(  t\right)  \quad\left(
t\in\left[  \frac{\alpha}{2},\alpha\right]  \right)  .
\]%
\[
G_{1}\left(  \frac{1-t}{\alpha}\right)  \geq Q_{1}\left(  t\right)  \text{
\quad}\left(  t\in\left[  1-\alpha,1-\frac{\alpha}{2}\right]  \right)  .
\]%
\[
G_{1}\left(  \frac{1-t}{\alpha}\right)  \leq Q_{1}\left(  t\right)
\quad\left(  t\in\left[  1-\frac{\alpha}{2},1\right]  \right)  .
\]

\end{enumerate}

\begin{proposition}
\label{p69}Let $m^{\prime}$\ \textit{be defined as in Proposition \ref{p9}.
Then} we have:%
\begin{align*}
0  &  \leq Np\left(  t\right)  -2G_{1}\left(  t\right)  \int_{x}^{x^{\prime}%
}p\left(  s\right)  ds\\
&  \leq2A\left(  x^{r},\left(  x^{\prime}\right)  ^{r}\right)  \int%
_{x}^{x^{\prime}}p\left(  s\right)  ds\text{ }-Np\left(  t\right)  \text{
\ }\left(  t\in\left[  0,1\right]  \right)  .
\end{align*}%
\begin{align*}
0  &  \leq Lp_{1}\left(  t\right)  -A\left(  Hp_{1}\left(  t\right)
,Hp_{2}\left(  t\right)  \right) \\
&  \leq\frac{r\left(  4x^{\prime}-y-3y^{\prime}\right)  \left(  \left(
x^{\prime}\right)  ^{r-1}-x^{r-1}\right)  }{8}\int_{\Omega}p\left(  s\right)
ds\text{ \ }\left(  t\in\left[  0,1\right]  \right)  .
\end{align*}%
\begin{align*}
0  &  \leq Pp_{1}-Lp_{1}\left(  t\right) \\
&  \leq\frac{r\left(  4x^{\prime}-y-3y^{\prime}\right)  \left(  \left(
x^{\prime}\right)  ^{r-1}-x^{r-1}\right)  }{8}\int_{\Omega}p\left(  s\right)
ds\text{ \ }\left(  t\in\left[  0,1\right]  \right)  .
\end{align*}%
\begin{align*}
0  &  \leq Np\left(  t\right)  -Ip_{1}\left(  t\right) \\
&  \leq r\left(  y-x\right)  \left(  \left(  x^{\prime}\right)  ^{r-1}%
-x^{r-1}\right)  \int_{x}^{x^{\prime}}p\left(  s\right)  ds\text{ \ }\left(
t\in\left[  0,1\right]  \right)  .
\end{align*}%
\begin{align*}
0  &  \leq Sp\left(  t\right)  -A\left(  Ip_{1}\left(  t\right)
,Ip_{2}\left(  t\right)  \right) \\
&  \leq\frac{r\left(  4x^{\prime}-y-3y^{\prime}\right)  \left(  \left(
x^{\prime}\right)  ^{r-1}-x^{r-1}\right)  }{4}\int_{x}^{x^{\prime}}p\left(
s\right)  ds\text{ \ }\left(  t\in\left[  m^{\prime},1\right]  \right)  .
\end{align*}

\end{proposition}

\begin{proposition}
\label{p70}The following inequalities hold:%
\begin{align*}
Hp_{1}\left(  t\right)   &  \leq Q_{1}\left(  t\right)  \int_{\Omega}p\left(
s\right)  ds\\
&  \leq A\left(  x^{r},\left(  x^{\prime}\right)  ^{r}\right)  \int_{\Omega
}p\left(  s\right)  ds\qquad\left(  t\in\left[  0,\frac{\alpha}{1+\alpha
}\right]  \right)  .
\end{align*}%
\begin{align*}
A^{r}\left(  x,x^{\prime}\right)  \int_{\Omega}p\left(  s\right)  ds  &  \leq
Q_{1}\left(  t\right)  \int_{\Omega}p\left(  s\right)  ds\\
&  \leq Pp_{1}\left(  t\right)  \text{ \qquad}\left(  t\in\left[  \frac
{\alpha}{1+\alpha},1\right]  \right)  .
\end{align*}%
\begin{align*}
0  &  \leq Sp\left(  t\right)  -2A\left(  G_{1}\left(  t\right)  ,G_{2}\left(
t\right)  \right)  \int_{x}^{x^{\prime}}p\left(  s\right)  ds\\
&  \leq\left[  A\left(  x^{r},\left(  x^{\prime}\right)  ^{r}\right)
+Q_{1}\left(  t\right)  \right]  \int_{x}^{x^{\prime}}p\left(  s\right)
ds-Sp\left(  t\right)  \text{ \ }\left(  t\in\left[  0,1\right]  \right)  .
\end{align*}

\end{proposition}

\begin{proposition}
\label{p71}$Kp$\textit{\ is symmetric about }$\frac{1}{2},$%
\textit{\ decreasing on }$\left[  0,\frac{1}{2}\right]  $\textit{\ and
increasing on }$\left[  \frac{1}{2},1\right]  .$ The following identities hold:
\end{proposition}

\begin{align*}
&  \sup\limits_{t\in\left[  0,1\right]  }Kp\left(  t\right)  =Kp\left(
0\right)  =Kp\left(  1\right) \\
&  =2%
{\displaystyle\int\nolimits_{x}^{x^{\prime}}}
\left[  \left(  \left(  1-\alpha\right)  x+\alpha s\right)  ^{r}+\left(
\left(  1-\alpha\right)  x^{\prime}+\alpha s\right)  ^{r}\right]  p\left(
s\right)  ds\ \int_{x}^{x^{\prime}}p\left(  s\right)  ds.
\end{align*}%
\begin{align*}
\inf\limits_{t\in\left[  0,1\right]  }Kp\left(  t\right)   &  =Kp\left(
\frac{1}{2}\right) \\
&  =%
{\displaystyle\int\nolimits_{x}^{x^{\prime}}}
{\displaystyle\int\nolimits_{x}^{x^{\prime}}}
\left[  \left(  \left(  1-\alpha\right)  x+\alpha\frac{s+u}{2}\right)
^{r}\right. \\
&  +2\left(  \left(  1-\alpha\right)  \frac{x+x^{\prime}}{2}+\alpha\frac
{s+u}{2}\right)  ^{r}\\
&  +\left.  \left(  \left(  1-\alpha\right)  x^{\prime}+\alpha\frac{s+u}%
{2}\right)  ^{r}\right]  p\left(  s\right)  p\left(  u\right)  dsdu.
\end{align*}%
\[
2A\left(  Ip_{1}\left(  t\right)  ,Ip_{2}\left(  t\right)  \right)  \int%
_{x}^{x^{\prime}}p\left(  s\right)  ds\leq Kp\left(  t\right)  \text{
\ }\left(  t\in\left[  0,1\right]  \right)  .
\]%
\[
\left[  y^{r}+A^{r}\left(  x,x^{\prime}\right)  +\left(  y^{\prime}\right)
^{r}\right]  \left[  \int_{x}^{x^{\prime}}p\left(  s\right)  ds\right]
^{2}\leq Kp\left(  \frac{1}{2}\right)  .
\]

\begin{proposition}
\label{p72}The following inequality holds for all $t\in\left[  0,1\right]  :$%
\begin{align*}
0  &  \leq Kp\left(  t\right)  -2A\left(  Ip_{1}\left(  t\right)
,Ip_{2}\left(  t\right)  \right)  \int_{x}^{x^{\prime}}p\left(  s\right)  ds\\
&  \leq2Sp\left(  1-t\right)  \int_{x}^{x^{\prime}}p\left(  s\right)
ds-Kp\left(  t\right)  .
\end{align*}

\end{proposition}

\end{document}